\documentclass[10pt,oneside,reqno]{amsart} 

\usepackage[utf8]{inputenc}
\usepackage[T1]{fontenc}
\usepackage{lmodern}
\usepackage{amsmath}
\allowdisplaybreaks
\usepackage{enumitem}
\usepackage{amsfonts}
\usepackage{amssymb}
\usepackage{amsthm}
\usepackage{tikz}
\usepackage{stmaryrd}
\usepackage{comment}
\usepackage[toc]{appendix}
\usepackage{thmtools}
\usepackage{thm-restate}
\usetikzlibrary{arrows.meta, positioning}
\newtheorem{theorem}{Theorem}
\newtheorem{cor}{Corollary}
\newtheorem{conj}{Conjecture}
\newtheorem{prop}{Proposition}
\newtheorem{example}{Example}
\theoremstyle{definition}

\theoremstyle{definition}
\newtheorem{lem}{Lemma}
\theoremstyle{definition}
\newtheorem{defi}{Definition}
\theoremstyle{remark}
\newtheorem{rem}{Remark}
\newtheorem{notation}{Notation}
\theoremstyle{remark}

\theoremstyle{plain}
\newtheoremstyle{styleintro}
  {\topsep}   
  {\topsep}   
  {\itshape} 
  {0pt}       
  {\bfseries} 
  {.}         
  {5pt plus 1pt minus 1pt} 
  {\thmname{#1}\thmnote{ #3}} 

\theoremstyle{styleintro}
\newtheorem*{thmintro}{Theorem}
\newtheorem*{conjintro}{Conjecture}
\newcommand{\Z}{\mathbb{Z}}
\newcommand{\N}{\mathbb{N}}
\newcommand{\Q}{\mathbb{Q}}

\newcommand{\C}{\mathbb{C}}

\newcommand{\kkk}{\mathbb F}

\newcommand*\fonction[5]{
#1  \left\{\begin{alignedat}{2}  &#2  & &\longrightarrow      #3\\
                                &#4 &   &\longmapsto  #5
\end{alignedat} \right. \kern-\nulldelimiterspace}

 \usepackage[all]{xy}
  
\newcommand{\eqn}[0]{\begin{array}{rcl}}
\newcommand{\eqnend}[0]{\end{array} }  	

\newcommand{\Wqt}{\mathcal{W}_{q,t}}

\newcommand{\W}{\mathcal{W}}

\newcommand{\Hqt}{\mathcal{H}_{q,t}}
\newcommand{\sld}{\mathfrak{sl}}
\newcommand{\g}{\mathfrak{g}}
\newcommand{\al}{\alpha}
\newcommand{\om}{\omega}
\newcommand{\el}{\ell}
\newcommand{\on}{\operatorname}
\newcommand{\nc}{\newcommand}

\nc{\ch}{\mbox{ch}}
\nc{\pa}{\partial}
\nc{\arr}{\rightarrow}
\nc{\larr}{\longrightarrow}
\nc{\ri}{\rangle}
\nc{\lef}{\langle}
\nc{\la}{\lambda}
\nc{\ep}{\epsilon}
\nc{\su}{\widehat{\mathfrak{sl}_2}}
\nc{\sw}{{\mathfrak s}{\mathfrak l}}
\nc{\h}{{\mathfrak h}}
\nc{\n}{{\mathfrak n}}
\nc{\G}{\widehat{\g}}
\nc{\De}{\Delta_+}
\nc{\gt}{\widetilde{\g}}
\nc{\Ga}{\Gamma}
\nc{\one}{{\mathbf 1}}
\nc{\z}{{\mathfrak Z}}
\nc{\zz}{{\mathcal Z}}
\nc{\Hh}{{\mathcal H}_\beta}
\nc{\qp}{q^{\frac{k}{2}}}
\nc{\qm}{q^{-\frac{k}{2}}}
\nc{\La}{\Lambda}
\nc{\wt}{\widetilde}
\nc{\qn}{\frac{[m]_q^2}{[2m]_q}}
\nc{\cri}{_{\on{cr}}}
\nc{\kk}{h^\vee}
\nc{\F}{\mathcal{F}}
\nc{\sun}{\widehat{\sw}_N}
\nc{\hh}{{\mathbf H}_{q,t}(\g)}
\nc{\HH}{{\mathcal H}_{q,t}}
\nc{\ca}{\wt{{\mathcal A}}_{h,k}(\sw_2)}
\nc{\si}{\sigma}
\nc{\gl}{\widehat{{\mathfrak g}{\mathfrak l}}_2}
\nc{\s}{T}
\nc{\bi}{\bibitem}
\nc{\WW}{\W_\beta}
\nc{\scr}{{\mathbf S}}
\nc{\ab}{{\mathbf a}}
\nc{\rr}{r}
\nc{\ol}{\overline}
\nc{\con}{qt^{-1} + q^{-1}t}
\nc{\den}{q^{\el-1} t^{-\el+1}+ q^{-\el+1} t^{\el-1}}
\nc{\ds}{\displaystyle}
\nc{\B}{B}
\nc{\A}{A^{(2)}_{2\el}}
\nc{\GG}{{\mathcal G}}
\nc{\UU}{{\mathcal U}}
\nc{\MM}{{\mathcal M}}
\nc{\CC}{{\mathcal C}}
\nc{\GL}{^L\G}
\nc{\dzz}{\frac{dz}{z}}
\nc{\Res}{\on{Res}}
\nc{\HHqt}{\mathbf{H}_{q,t}}
\nc{\HHqtz}{\mathbf{H}_{q,t}^0}
\nc{\WWqt}{\mathbf{W}_{q,t}(\g)}
\usepackage{pgfplots}
\pgfplotsset{compat=1.15}
\usepackage{mathrsfs}
\usetikzlibrary{arrows}
\usepackage[hypertexnames=false]{hyperref}
\hypersetup{
    colorlinks=true,  
    linkcolor=blue,   
    citecolor=red,    
    urlcolor=blue,    
}

\author{Hicham Assakaf}
\address{\parbox{\linewidth}{Universit\'e Paris-Cité and Sorbonne Universit\'e, CNRS,
  IMJ-PRG, F-75006, Paris, France\\ \textit{Email address}: \textnormal{assakaf@imj-prg.fr}}}
\title{Fundamental fields in the deformed $W$-algebras}

\begin{document}
\newcommand{\mlimt}{
The limit ${\mathbf{W}}_{q,1}^0(\mathfrak{g})$ is embedded in the intersection of the kernels of the operators $S_i$:
$$ {\mathbf{W}}_{q,1}^0(\mathfrak{g}) \hookrightarrow \bigcap_{i \in I} \text{Ker}(S_i). $$
}
\newcommand{\conjjj}{
    \begin{enumerate}
        \item Let $m\in\mathbf{M}$ be a dominant monomial. Then the algorithm starting from $m$ terminates in finitely many steps (with or without a failure).
        \item Let $m \in \mathbf{M}$ be a dominant regular generic monomial. If the algorithm starting from $m$ creates a non-regular monomial, then there exists no field in $\mathbf{W}_{q,t}(\mathfrak{g})$ with $m$ as its unique dominant monomial.
    \end{enumerate}
}

\begin{abstract}
   Let $\mathfrak{g}$ be a simple Lie algebra. 
   Frenkel and Reshetikhin introduced the deformed $W$-algebra $\WWqt$ in \cite{MR1646483}. 
      In this work, we give a formal version of this definition.
   In this setting,
   we reformulate and prove the well-definedness of an algorithm \cite{fjm1,kimurahiggsing} inspired by the Frenkel--Mukhin algorithm \cite{FM1} which, starting from a given dominant monomial $m$ satisfying some degree conditions,
   produces elements of the deformed $W$-algebra. 
   Then, we apply this algorithm to construct explicitly some specific elements of $\WWqt$. 
   In particular, we apply this to prove a conjecture of Frenkel and Reshetikhin in \cite{MR1646483} in types $B_\el$, $C_\el$, and for some nodes in other types.
   This framework opens up new possibilities for studying explicitly fields in the deformed $W$-algebra $\WWqt$.
\end{abstract}
\subjclass[2020]{Primary 17B37, 17B69; Secondary 81R10, 17B67}
\maketitle
\tableofcontents

\section{Introduction}

\subsubsection*{A brief historical review of $W$-algebras}
The study of $W$-algebras and their deformations has been a central theme at the intersection of conformal field theory, integrable systems, and representation theory. 
Let us present a quick review of their story (see \cite{bouwknegt1993w} for a more complete review). \\

The first $W$-algebra was introduced by Zamolodchikov in \cite{zamolodchikov1995infinite} as an extension of the Virasoro algebra.
Then, Fateev and Lukyanov generalize this construction to $\sld_N$ and define the $W_N$-algebras \cite{fateev1988models}.
In \cite{feigin1990quantization}, Feigin and Frenkel prove that this $W_N$-algebra can be obtained by a Drinfeld-Sokolov reduction of the affine algebra $\widehat{\sld_N}$ with respect to a principal nilpotent element of $\sld_N$. They derive explicitly this $W_N$-algebra as a BRST-cohomology algebra. 
For all simple Lie algebra $\g$, the quantum Drinfeld-Sokolov reduction of $\widehat\g$ at level $k$ gives a generalization of the $W_N$-algebra denoted $W_k(\g)$. 
At the critical level $k=-h^\vee$, the $W$-algebra becomes the center of the universal affine vertex algebra $V_{-h^\vee}(\mathfrak{g})$ at the critical level \cite{frenkel1991affine}. 
Moreover, the affine $W$-algebras $W_k(\g)$ present a remarkable duality. If ${}^L\g$ is the Langlands dual of $\g$, then $W_k(\g)$ and $W_{k'}({}^L\g)$ are isomorphic if $r(k+h^\vee)(k'+{}^Lh^\vee) =1$, 
where $r$ is the maximal number of edges connecting two vertices in the Dynkin diagram of $\g$ (\cite{feiginfrenkel1992affine}).
Later on, Awata, Kubo, Odake and Shiraishi defined a two-parameters deformation of the Virasoro algebra \cite{quantumvirasoro_AKOS} 
called $q$-Virasoro. 
Then the same authors define in parallel to Feigin and Frenkel the $q$-$W_N$-algebras \cite{awata1996quantum,feigin1996quantum} which are deformations of the $W_N$-algebras. 
Finally, Frenkel and Reshetikhin extend this definition and define the two parameters deformed $W$-algebra $\WWqt$ 
associated to a simple Lie algebra $\g$ in \cite{MR1646483}.
These deformed $W$-algebras are defined as the intersection of the kernels of screening operators acting on a 
double deformed Heisenberg algebra $\mathbf{H}_{q,t}(\mathfrak{g})$. 
In \cite{MR1646483}, Frenkel and Reshetikhin highlight that the deformed $W$-algebra $\mathbf{W}_{q,t}(\g)$ is remarkably connected with the analytic Bethe Ansatz in integrable
models associated with the quantum affine algebras $U_q(\widehat{\g})$, $U_t(^L\widehat{\g})$ and $U_t(^L(\widehat{{}^L\g}))$, 
where $\widehat{\g}$ denotes the corresponding affine Kac-Moody algebra, ${}^L\g$ is the Langlands dual of $\g$, and $U_q(\widehat{\g})$ is the quantum affine algebra associated to $\g$ defined by Drinfeld and Jimbo in \cite{Drinfeldx,jimbo1985q}.\\
In \cite{kimurapestunquiverwalg, kimura2018fractional}, Kimura and Pestun provide a quiver gauge theoretic construction of Frenkel and Reshetikhin's deformed $W$-algebras, 
providing a geometric point of view on the deformed $W$-algebras. 
In \cite{negut}, Negu\c{t} defines another deformed $W$-algebra in type $A$, that deforms the $W$-algebra of $\mathfrak{gl}_{nr}$ with respect to a rectangular nilpotent.
This definition generalizes the definition of Frenkel and Reshetikhin in type $A$. In~\cite{fjm0}, Feigin Jimbo Mukhin and Vilkoviskiy construct the currents in the deformed $W$-algebras with (co)modules over the quantum toroidal $\mathfrak{gl}_1$ algebra. This construction extends the definition of Frenkel and Reshetikhin in supersymmetric and twisted cases.
Deformations of the $W$-algebra are also studied in \cite{sevostyanov,lodin,avan,nieri,fjm1,fjm}.\\
While the original construction provided a powerful framework for understanding quantum integrable systems, 
the precise algebraic nature of these algebras and their behavior under classical limits present significant technical challenges. 
This paper addresses these issues, and provides new tools to compute elements of this deformed $W$-algebra explicitly.\\

\subsubsection*{The deformed $W$-algebra in a formal context}
Let $\g$ be a simple Lie algebra of rank $\el$, and let $I=\llbracket1,\el\rrbracket$. In this work, we define the deformed
$W$-algebras $\mathbf{W}_{q,t}(\g)$ of \cite{MR1646483} in a formal setting: we work over the ring of formal power series
$K=\Q\llbracket h,\beta\rrbracket$, with $q=e^h$ and $t=e^{h\beta}$. This choice has two advantages. First, the parameters
are automatically generic: for instance $q^at^b\neq1$ for $(a,b)\neq(0,0)$, and the powers $q^\xi$, $\xi\in\Q[\beta]$, are
well defined. Second, the classical limit $t\to1$ becomes an algebraic operation: we specialize the deformed $W$-algebra
at $t=1$ by taking the quotient by $\beta$.

We recall the definitions of the double deformed Heisenberg algebra $\mathcal{H}_{q,t}(\g)$ and the screening operators in \cite{MR1646483} adapted in our context.
Then, $\HHqt(\g)$ denotes the vector space generated by monomials in (respectively fundamental weight-type and simple root-type) variables $Y_i(za),A_i(za)\in \Hqt(\g)\llbracket z^{\pm1}\rrbracket$, $i\in I$, $a \in q^\Z t^\Z$ with respect to the normally ordered product and their derivatives.
Finally, the deformed $W$-algebra $\WWqt$ denotes the vector subspace of $\HHqt(\g)$ generated by the fields commuting with screening operators $S_i^{\pm}$. In this article, we restrict our study to the \textit{generic} case (that is, fields such that $Y_i(za)$ only appear in its monomials with power $0,\pm1$), as it is sufficient to prove the conjecture below in new cases.  \\
If $\omega_1,\ldots,\omega_\el$ are the fundamental weights of the simple Lie algebra $\g$, then there exist \textit{fundamental} representations $V_{\omega_i}$ of the quantum affine algebra $U_q(\widehat{\g})$ (see \cite{chari1995guide} for more details). In \cite{MR1646483}, Frenkel and Reshetikhin formulate the following conjecture:
\begin{conjintro}[{\hyperref[conj1]{A}}]
For each $i = 1, \ldots, \ell$, there exists a field $T_i(z)$ in $\mathbf{W}_{q,t}(\mathfrak{g})$, such that $T_i(z) = Y_i(z) +$ the sum of elementary terms of the form
\[
    c(q,t): Y_i(z) A_{i_1}(zq^{a_1}t^{b_1})^{-1} \cdots A_{i_k}(zq^{a_k}t^{b_k})^{-1}:
\]
(where $c(q,1)$ is a positive integer independent of $q$). Furthermore, the set of weights of these terms counted with multiplicity $c(q,1)$ is the set of weights of the finite-dimensional irreducible representation $V_{\omega_i}$ of $U_q(\widehat{\mathfrak{g}})$ with highest weight $\omega_i$, 
where the weight of $:\prod_{j}Y_{i_j}(za_j)^{\varepsilon_j}:$
is $\sum_{j} \varepsilon_j \omega_{i_j}.$
\end{conjintro}
This would establish a strong link between the deformed $W$-algebra $\WWqt$ and the representation theory of the quantum affine algebra $U_q(\widehat{\g})$. 
This conjecture is proved for $\g$ of type $A_\el$ (\cite{feigin1996quantum,awata1996quantum,MR1646483}). 
Frenkel and Reshetikhin proved this conjecture in all classical types for $i=1$ in \cite{MR1646483}. 
In \cite{fjm0}, Feigin, Jimbo, Mukhin, and Vilkoviskiy generalize this result for all classical types for $i=1$ to the supersymmetric case.
Bouwknegt and Pilch proved this conjecture for all simple Lie algebras $\g$ of rank $2$ (Appendix A in \cite{MR1633032}). 
Feigin, Jimbo and Mukhin compute fundamental fields in type $\mathfrak{gl}(m|n)$ (see Section 2.4 in \cite{fjm}).\\
However, as far as we know this conjecture has not been proved yet for the other classical types because of the lack of tools to compute systematically the fields in $\WWqt$. Indeed, one can try to guess the monomials which will appear in the fundamental fields, but there exists no rigorous general proof that there exists coefficients such that the sum gives a field. We prove it in this paper. An analogous result was proved by Feigin, Jimbo and Mukhin \cite[Theorem 6.1]{fjm1}: every finite qq-character gives a sum of vertex operators commuting with the screening operators, but only when all screening currents are fermionic, i.e.\ $S_i(z)S_i(w)=(z-w):S_i(z)S_i(w):$, which excludes $\mathbf{W}_{q,t}(\mathfrak{g})$ where all the screening currents are bosonic.

We draw attention to the fact that the elements of this deformed $W$-algebra correspond to Nekrasov's $qq$-characters \cite{nekrasov2016bps,kimurapestunquiverwalg}.
These elements must be distinguished from Nakajima's $(q,t)$-characters \cite{nakajimatanalogsqch,hernandez2003algebraicapproachqtcharacters} and from the interpolating $(q,t)$-characters of Frenkel and Hernandez \cite{frenkelhernandez2011langlands}. 
The latter are commutative algebraic Laurent polynomials defined to emulate the behavior of the fields in $\WWqt$.

\subsubsection*{An algorithm to compute explicitly fields in $\WWqt$}
In this document, the reformulation of the definition of $\WWqt$ allows us to construct explicitly its elements using an algorithm inspired by the Frenkel-Mukhin algorithm \cite{FM1} for computing $q$-characters. 
This algorithm is similar to the one used and defined by Feigin, Jimbo, and Mukhin in \cite[Section 2.3]{fjm1} or by Kimura in \cite[Section 2.4]{kimurahiggsing} to compute $qq$-characters, however, its definition is adapted to our context.

Our algorithm works step by step, starting from a \textit{dominant} monomial (that is, a monomial in $(Y_i(za))_{i\in \llbracket 1,\el\rrbracket, a\in  q^\Z t^\Z}$), and at each step, 
multiplying it by variables $A_{i}(za)^{-1}$ until we get an \textit{antidominant} monomial in $(Y_i(za)^{-1})_{i,a}$. 
This procedure is inspired by the Frenkel-Mukhin algorithm procedure for $q$-characters, and is observed to be efficient for computing $qq$-characters 
(see \cite{MR1633032}, \cite[Lemma 2.1]{fjm0}, \cite[Section 2.3]{fjm1}, or again \cite[Section 2.4]{kimurahiggsing}).

However, the intricate structure of the double deformed Heisenberg algebra implies that unlike the set of $q$-characters, the deformed $W$-algebra is not a ring, and we cannot give a complete characterization of the fields in a fixed $\sld_2$-direction. 
Thus, our algorithm differs from the Frenkel-Mukhin algorithm. 
The first difference is the fact that at each step and for each monomial, we have to isolate each \textit{admissible} variable $Y_i(za)$, multiply the monomial by $A_i(zaq^{-r_i}t)^{-1}$, and then define a coefficient for the new monomial.

The second and main difference is the fact that the coefficients are not defined as maxima as in \cite{FM1}, but they are defined as residues of rational functions in $q^{\pm1},t^{\pm1}$.

Our algorithm is inspired by the one defined by Feigin, Jimbo, Mukhin \cite[Section 2.3, Theorem 6.1]{fjm1} and reformulates the one defined by Kimura \cite[Section 2.4]{kimurahiggsing}. This reformulation is necessary to prove its well-definedness and to prove the main theorems of this paper.

We prove the following theorem which makes this algorithm a key tool to compute elements in $\WWqt$:
\begin{thmintro}[{\hyperref[theoremalgowork]{A}}](Theorem~\ref{theoremalgowork})
    The algorithm is well-defined. Moreover, if it does not \textit{fail} and ends in finitely many steps, then it gives a field $T(z)\in\WWqt$.
\end{thmintro}
\vskip 3 pt
As an application, we prove Conjecture \hyperref[conj1]{A} in new cases: 
\vskip 3 pt
\begin{thmintro}[{\hyperref[theoremefinal]{B}}](Theorem~\ref{theoremefinal})
    Conjecture \hyperref[conj1]{A} holds in types $A_\ell$ (for all $i\in I$), $B_\ell$ (for all $i\in I$), $C_\ell$ (for all $i\in I$), $D_\ell$ (for $i=1,\ell-1,\ell$), $E_6$ (for $i=1,5$), $E_7$ (for $i=6$), $F_4$ (for $i=1,4$) and $G_2$ (for $i=1,2$). 
\end{thmintro}
\vskip 3 pt
Thus, this algorithm provides a powerful tool to construct explicitly fields in $\WWqt$.

Finally, we conjecture that for any dominant monomial $m$, our algorithm terminates in finitely many steps without failing if and only if there exists a generic field in $\WWqt$, without derivatives, and having $m$ as its unique dominant monomial. 
Furthermore, the following conjecture highlights a potential sufficient condition for our algorithm to work (that is, not  to fail and to end in finitely many steps) based on the thinness of representations of the quantum affine algebra $U_q(\widehat{\g})$.
\vskip 3 pt
\begin{conjintro}[{\hyperref[conjjjjj]{B}}]
Let $m$ be a dominant monomial whose $t=1$ specialization is the unique dominant monomial of a special thin representation of $U_q(\widehat{\g})$, that is a representation whose $q$-character has a single dominant monomial and all coefficients equal to 1. 
Then, the algorithm successfully generates an element of $\WWqt$ from $m$.
\end{conjintro}
\vskip 3 pt
Theorem {\hyperref[theoremefinal]{B}} proves this for fundamental fields.

\subsubsection*{Structure of the paper}

This article is organized as follows:\\ 
In Section \ref{section21}, we recall the definition of the double deformed Heisenberg algebra $\Hqt(\g)$ introduced by Frenkel and Reshetikhin in \cite{MR1646483}, its Fock representation in our context, and the completion of $\Hqt(\g)$. 
We prove that the Fock representation is faithful.\\

In Section \ref{sect3}, we recall the definition of the screening currents and of the deformed $W$-algebra $\WWqt$ as the intersection of the kernels of the screening operators acting on $\HHqt(\g)$. We prove that the subspace of fields without derivatives $\mathbf{H}_{q,t}^0(\g)\subset\mathbf{H}_{q,t}(\g)$ is isomorphic to a Laurent polynomial algebra $K[Y_{i,a}^{\pm1}]$, where $Y_{i,a}$ is identified with $Y_i(za)$ with respect to a specific product.
Then, we compute commutation relations of the screening currents with the generators of $\mathbf{H}_{q,t}^0(\g)$, and discuss the coefficients of the fields in $\WWqt$.\\

In Section \ref{sect5}, we present the central algorithmic result, including a graph-based representation of the algorithm inspired by \cite{fjm1,frenkel2008qcharactersrepresentationsquantumaffine}, and explicit examples of fields in $\WWqt$. These results are the first main results of this article. \\

In Section \ref{sect6}, we apply our algorithmic results to prove Conjecture \hyperref[conj1]{A} in types $B_\el,C_\el$, and in other new cases. This proof is the second main result of this article.\\

Finally, in Section \ref{sect7}, we formulate conjectures on the behavior of our algorithm and on a possible sufficient condition for our algorithm to work based on the thinness of representations of the quantum affine algebra $U_q(\widehat{\g})$.

\addtocontents{toc}{\protect\setcounter{tocdepth}{-1}}
\subsection*{Acknowledgements}
I would like to thank David Hernandez for his careful, repeated readings of this paper. 
I also want to thank Edward Frenkel, Taro Kimura, and Evgeny Mukhin for their interest in this work and for their very insightful comments on a preliminary version of this paper. 
\addtocontents{toc}{\protect\setcounter{tocdepth}{2}}

\section{The double deformed Heisenberg algebra $\Hqt(\g)$}\label{section21}
In this section, we recall the definition of the double deformed Heisenberg algebra $\Hqt(\g)$. 
This definition is introduced by Frenkel and Reshetikhin for a simple Lie algebra $\g$ of any type in \cite{MR1646483}. In type $A_\el$, the definition appears in \cite{Asai_1996,awata1996quantum,feigin1996quantum}. Here we give the same definition, however, to justify the exponential and logarithm, we define the parameters $q,t$ as formal power series in $h$.
\vspace{10 pt}\\
Let $h,\beta$ be generic variables. Let $(q,t)=(e^h,e^{h\beta})\in\Q\llbracket h,\beta\rrbracket^2$. 
These are formal variables. Moreover, if $\gamma\in\Q$, then $q^\gamma:=e^{h\gamma}$ and $t^\gamma:=e^{h\beta\gamma}$ are well defined as formal variables. We can write $t=q^\beta$.
We can now work on the ring $K=\Q\llbracket h,\beta\rrbracket $.\\ 
\begin{notation}
    For all $n\in\Z$, we set
    $$
    [n]_q = \frac{q^n - q^{-n}}{q - q^{-1}}.
    $$
\end{notation} 
\begin{rem}
    Note that $[n]_q\in\Q\llbracket h\rrbracket $ and $[n]_q\in n+h\Q\llbracket h\rrbracket$.
\end{rem}
Let $\g$ be a simple Lie algebra of rank $\ell$. Let $I=\{1,\ldots,\el\}$. Let $(\cdot,\cdot)$ be the invariant inner product on $\g$ normalized such that for all maximal root $\alpha$, $(\alpha,\alpha)=2$ (see \cite{kac1990infinite}). Let $\{ \al_1,\ldots,\al_\el \}$ and
$\{ \om_1,\ldots,\om_\el \}$ be the sets of simple roots and of fundamental
weights of $\g$, respectively. We have:
$$
(\al_i,\om_j) = \frac{(\al_i,\al_i)}{2} \delta_{i,j}.
$$
Let $r$ be the maximal number of edges connecting two vertices of the
Dynkin diagram of $\g$. Thus, $r=1$ for simply-laced $\g$, $r=2$
for $B_\el, C_\el, F_4$, and $r=3$ for $G_2$. Set
$$
D = \on{diag}(\rr_1,\ldots,\rr_\el),
$$
where
\begin{equation}    \label{di}
\rr_i = r \frac{(\al_i,\al_i)}{2}.
\end{equation}
All $\rr_i$'s are integers; for simply-laced $\g$, $D$ is the identity
matrix.

Now let $C = (C_{ij})_{1\leq i,j\leq \el}$ be the {\em Cartan matrix} of
$\g$. 
Denote by $\mathcal{I} = 2\mathrm{Id}_\ell-C$ the {\em incidence matrix}, and let $\B = (\B_{ij})_{1\leq i,j\leq \el}=DC$ be the symmetrized Cartan matrix:
$$
\B_{ij} = r (\al_i,\al_j).
$$
In \cite{MR1646483}, the authors define $\el \times \el$ matrices $C(q,t)$, $D(q,t)$, and $\B(q,t)$ with coefficients in $K$ by the
formulas
\begin{align}    \label{qtc}
C_{ij}(q,t) &= (q^{\rr_i} t^{-1} + q^{-\rr_i} t) \delta_{i,j} - [\mathcal{I}_{ij}]_q,
\\ D(q,t) &= \on{diag}([\rr_1]_q,\ldots,[\rr_\el]_q), \label{qtd} \\
\B(q,t) &= D(q,t) C(q,t). \notag
\end{align}
Thus,
\begin{equation}    \label{qts}
\B_{ij}(q,t) = [\rr_i]_q \left((q^{\rr_i} t^{-1} + q^{-\rr_i} t)
\delta_{i,j} - [\mathcal{I}_{ij}]_q \right).
\end{equation}
It is easy to see that the matrix $\B(q,t)$ is symmetric. For simply-laced
$\g$,
$$
C_{ij}(q,t) = \B_{ij}(q,t) = (q t^{-1} + q^{-1} t) \delta_{i,j} - \mathcal{I}_{ij}.
$$
\begin{lem}\label{invertiblecartan}
    For all $n\geq0$, the matrix $C(q^n,t^n)\in \text{Mat}_{\el\times \el}(K)$ is invertible and its inverse lies in $\text{Mat}_{\el\times \el}(K)$. 
\end{lem}
\begin{proof}
    Let $n\in\N$. We consider $d=\det(C(q^n,t^n))\in K$. 
    If $d=0$ then its image in the quotient $\overline d\in K/h$ is null. 
    However, $\overline{d}=\overline{\det C}$. 
    As $\g$ is a simple Lie algebra, the Cartan matrix $C$ is invertible. Thus $\det C\in\Q\backslash\{0\}$ and $\overline{d}\neq0$. 
    Thus $C(q^n,t^n)$ is invertible.
    Moreover, $d=d_0+f(h,\beta)$ where $f(h,\beta)$ is a formal power series in $h,\beta$ with no constant terms and $d_0\in\Q^*$. 
    Hence, $d$ is invertible in  $\Q\llbracket h,\beta\rrbracket $, and $$C(q^n,t^n)^{-1}=\frac{1}{\det C(q^n,t^n)}\text{Com}(C(q^n,t^n))^T\in\text{Mat}_{\el\times \el}(K).$$
\end{proof}
\subsection{Heisenberg algebra $\Hqt(\g)$}
In this subsection, we present the definition of the double deformed Heisenberg algebra as introduced in \cite{MR1646483}. This definition appears in type $A_\el$ in \cite{quantumvirasoro_AKOS}. 
\begin{defi}
Let $\HH(\g)$ be the (double-deformed Heisenberg) algebra over the ring $K$ with generators $x_i[n]$, $q^{\xi a_i[0]}$ and $e^{\gamma Q_i} $, with $i\in I$, $\xi\in\Q[\beta],\gamma\in\beta^{-1}\Q[\beta]$, $n\in\Z\backslash\{0\}$ and relations
\begin{equation}    \label{a}
[x_i[n],x_j[m]] = n \frac{t^n - t^{-n} }{q^n - q^{-n}}\B_{ij}(q^n,t^n)
\delta_{n,-m},
\end{equation}
\begin{equation*}    
[x_i[n],e^{\gamma Q_j}] = 0, \quad n\neq0,
\end{equation*}
\begin{equation*}    
[e^{\gamma Q_j},e^{\gamma' Q_{j'}}] = 0, \quad \gamma,\gamma'\in\beta^{-1}\Q[\beta], j,j'\in I,
\end{equation*}
\begin{equation*}    
q^{\xi a_i[0]}e^{\gamma Q_j} = q^{\xi\gamma\beta B_{ij}}e^{\gamma Q_j}q^{\xi a_i[0]},
\end{equation*}
\begin{equation*}    
[q^{\xi a_i[0]},x] = 0,\qquad x\in\langle x_j[n]\rangle_{j\in I,n\neq00}.
\end{equation*}
$$\forall \xi,\xi'\in\Q[\beta],\ \forall\gamma,\gamma'\in\beta^{-1}\Q[\beta],\ \forall j,j'\in I,\qquad
q^{\xi a_j[0]}q^{\xi' a_{j'}[0]}=q^{\xi' a_{j'}[0]}q^{\xi a_j[0]},\qquad
q^{\xi a_j[0]}q^{\xi' a_j[0]}=q^{(\xi+\xi')a_j[0]},$$
$$e^{\gamma Q_j}e^{\gamma' Q_j}=e^{(\gamma+\gamma')Q_j},\qquad q^{0\cdot a_j[0]}=e^{0\cdot Q_j}=1,$$
\end{defi}

We construct for all $i\in I$, $m\in\Z\backslash\{0\}$,
$$a_i[m]:=\frac{q^m-q^{-m}}{m}x_i[m].$$
It implies the following commutation relation:
$$[a_i[n],a_j[m]] = \frac{1}{n} (q^n - q^{-n}) (t^n - t^{-n}) \B_{ij}(q^n,t^n)
\delta_{n,-m}.$$

\begin{rem}\label{gendiff}
   The $(a_i[m])_{i\in I,m\ne0}$ are the generators of the Heisenberg algebra $\Hqt(\g)$ defined in \cite{MR1646483}. 
   Here, because we consider the ring $\Q\llbracket h,\beta\rrbracket$ we slightly change the definition and the Frenkel and Reshetikhin's Heisenberg algebra is a strict subalgebra of the Heisenberg algebra of this paper.
   If we work in the field $Frac\left(\Q\llbracket h,\beta\rrbracket\right)$ instead of $\Q\llbracket h,\beta\rrbracket $, then both are equal.
\end{rem}
\vskip 1cm

The generators $a_i[n]$ are ``simple root''-type elements of $\HH(\g)$. As the deformed Cartan matrix is invertible, there is
a unique set of ``fundamental weight''-type elements, $y_i[n], q^{\xi y_i[0]},t^{\xi y_i[0]}$ $
i=1,\ldots,\el; n \in \Z\backslash\{0\}$, $\xi\in\Q[\beta]$ satisfying:
\begin{equation}    \label{expr}
\forall 1\leq j\leq\el, \qquad a_j[n] = \sum_{i=1}^\el C_{ij}(q^n,t^n) y_i[n], \qquad 
q^{\xi a_j[0]} = \prod_{i=1}^\el q^{\xi C_{ij} y_i[0]},
\end{equation}
\begin{equation}    \label{ay}
[a_i[n],y_j[m]] = \frac{1}{n} (q^{\rr_i n} - q^{-\rr_i n})(t^n - t^{-n})
\delta_{i,j} \delta_{n,-m}.
\end{equation}
They satisfy the following commutation relations:
\begin{equation}    \label{y}
[y_i[n],y_j[m]] = \frac{1}{n} (q^{n} - q^{-n}) (t^n - t^{-n})
M_{ij}(q^n,t^n) \delta_{n,-m},
\end{equation}
where $(M_{ij}(q,t))_{1\leq i,j\leq \el}$ is the matrix
$M(q,t) = D(q,t) C(q,t)^{-1} $.

\subsection{Fock representation $\pi_\mu$ of $\HH(\g)$}
in this subsection  we present the construction of the Fock representations of the double deformed Heisenberg algebra. 
This representation is introduced in \cite{MR1646483,MR1633032}. 
Let $\mathfrak{h}$ be the Cartan subalgebra of $\g$ and $\mathfrak{h}^*$ the $\Q$-span of the fundamental weights. We define 
$$P:=\Q[\beta]\otimes_\Q\mathfrak{h}^*=\{\sum_{i=1}^k\gamma_i\mu_i~~;~~\gamma_i\in\Q[\beta],~~\mu_i \in\mathfrak{h}^*\}.$$
Let us define the following subalgebras of $\Hqt(\g)$
$$H^-:=\langle x_i[n] \rangle_{i\in I,n<0},\quad H^+:=\langle x_i[n],q^{\xi a_i[0]}\rangle_{i\in I,\xi\in\Q[\beta], n>0},$$
$$\text{and}\quad H:=\langle x_i[n],q^{\xi a_i[0]} \rangle_{i\in I,\xi\in\Q[\beta], n\neq0}$$
Let $\mu=\sum_{i=1}^k\gamma_i\mu_i\in P$, with $\gamma_i\in \Q[\beta]$, $\mu_i\in\mathfrak{h}^*$. Let $K_\mu$ be a K-module generated by a vector $v_\mu$. We construct a structure of $H^+$-module on $K_\mu$ with the following actions:
\begin{align*}
    \forall j\in I,n>0\qquad x_j[n] v_{\mu} &= 0,\\
    \forall j\in I \qquad q^{\xi a_j[0]}v_\mu&=q^{\xi(\mu,\alpha_j)}v_\mu,
\end{align*}
with $(\mu,\alpha_j):=\sum_{i=1}^k\gamma_i(\mu_i,\alpha_j).$
We can now define the Fock representation $\pi_\mu$ as follows:
$$\pi_\mu:=\text{Ind}_{H^+}^{H}K_\mu=H\otimes_{H^+}K_\mu,$$
The set $\pi_\mu$ is a representation of $H$ and has a Poincaré-Birkhoff-Witt basis:  
$$(x_{i_1}[n_1]x_{i_2}[n_2]\dots x_{i_m}[n_m]v_\mu)_{n_j < n_{j+1}<0\text{, or }n_j = n_{j+1}\text{ and }i_j \le i_{j+1}}$$
Thus, the direct sum $\bigoplus_{\mu\in P}\pi_\mu$ is a representation of $H$. 
To extend it as a representation of $\HH(\g)$ we have to define how the shift generators act on the elements of each $\pi_\mu$. 
The shift generators commute with $(x_i[n])_{n\neq0}$. Hence it is sufficient to define how it acts on each $v_\mu$. 
Let $\mu\in P$, $\gamma\in\beta^{-1}\Q[\beta]$, and $i\in I$. Thus $\mu+\gamma \beta r\al_i\in P$ and we define 
$$e^{\gamma Q_i}\cdot v_\mu:=v_{\mu+\gamma\beta r\al_i}.$$

The following proposition proves that this definition is compatible with the commutation relations of $\HH(\g)$ and that it gives a well-defined faithful representation of $\HH(\g)$. 
This may be well-known to experts, but we write a proof for completeness.

\begin{prop}\label{fockrep}
    The vector space $\bigoplus_{\mu\in P}\pi_\mu$ is a well-defined, faithful representation of $\HH(\g)$.
\end{prop}

\begin{proof}
For $n\geq1$, set $c_n:=n\frac{t^n-t^{-n}}{q^n-q^{-n}}\in K$.
Then $[x_i[n],x_j[-m]]=\delta_{n,m}c_n\B_{ij}(q^n,t^n)$ for $n,m>0$.
Let $\mathcal S:=K[x_{j,-n}]_{j\in I,\,n\geq1}$, and let $\partial_{i,n}$ be the $K$-derivation of $\mathcal S$ such that
$\partial_{i,n}(x_{j,-m})=\delta_{n,m}c_n\B_{ij}(q^n,t^n)$.
For $f\in\mathcal S$, we still write $f$ for its image in $H^-$ under $x_{j,-n}\mapsto x_j[-n]$, and $fv_\mu:=f\cdot v_\mu$.
By the PBW property, $f\mapsto fv_\mu$ is a $K$-linear isomorphism $\mathcal S\to\pi_\mu$, and for $n>0$
\begin{equation}\label{eq:fockmodel}
x_j[-n]\cdot fv_\mu=(x_{j,-n}f)v_\mu,\qquad x_j[n]\cdot fv_\mu=(\partial_{j,n}f)v_\mu,\qquad
q^{\xi a_j[0]}\cdot fv_\mu=q^{\xi(\mu,\al_j)}fv_\mu .
\end{equation}
The second identity holds because $x_j[n]v_\mu=0$ and $\operatorname{ad}x_j[n]$ is a derivation sending $x_k[-m]$ to
$\delta_{n,m}c_n\B_{jk}(q^n,t^n)$. The third holds because $q^{\xi a_j[0]}$ commutes with $H^-$.
In particular, the subalgebra genertaed by the $x_i[n]$ acts on $\pi_\mu\simeq\mathcal S$ by operators that do not depend
on $\mu$.

Let us prove the well-definedness:
For $s\in P$, let $\tau_s$ be the $K$-linear map $fv_\mu\mapsto fv_{\mu+s}$ on $\bigoplus_\mu\pi_\mu$, and let $e^{\gamma Q_i}$
act by $\tau_{\gamma\beta r\al_i}$ (note that $\gamma\beta r\al_i\in P$).
The relations that do not involve shift generators hold because each $\pi_\mu$ is an $H$-module.
By \eqref{eq:fockmodel}, each $\tau_s$ commutes with every $x_j[n]$, $n\neq0$, and $\tau_s\tau_{s'}=\tau_{s+s'}$, $\tau_0=\mathrm{id}$.
Finally, for $s=\gamma\beta r\al_j$ we have $(s,\al_i)=\gamma\beta\B_{ij}$, hence
$$
q^{\xi a_i[0]}\tau_s(fv_\mu)=q^{\xi(\mu+s,\al_i)}fv_{\mu+s}=q^{\xi\gamma\beta \B_{ij}}\,\tau_s\,q^{\xi a_i[0]}(fv_\mu).
$$
So all defining relations hold, and we obtain an algebra morphism $\rho:\HH(\g)\to\operatorname{End}(\bigoplus_\mu\pi_\mu)$.

Let us prove that this morphism is injective.
For $\gamma\in(\beta^{-1}\Q[\beta])^I$ and $\xi\in\Q[\beta]^I$, to simplify the notations, set
$e^{\gamma\cdot Q}:=\prod_ie^{\gamma_iQ_i}$, $q^{\xi\cdot a[0]}:=\prod_iq^{\xi_ia_i[0]}$ and
$s(\gamma):=\beta r\sum_i\gamma_i\al_i\in P$.
For a finitely supported $\mathbf k=(k_{i,n})\in\N^{I\times\N_{>0}}$, set $x^{\mathbf k}:=\prod_{i,n}x_i[n]^{k_{i,n}}$, and
$\partial^{\mathbf k}:=\prod_{i,n}\partial_{i,n}^{k_{i,n}}$ and $z^{\mathbf k}:=\prod_{i,n}z_{i,-n}^{k_{i,n}}$, with $z_{i,-n}$ defined below; also set $|\mathbf k|:=\sum k_{i,n}$ and $\mathbf k!:=\prod k_{i,n}!$.
The defining relations allow us to normal-order any product of generators: positive modes commute with each other,
$[x_i[n],x_j[m]]\in K$, the shift generators commute with the $x_j[n]$, $n\neq0$, and moving a shift generator past a zero mode
only produces a scalar factor. Hence every $X\in\HH(\g)$ can be written as
\begin{equation}\label{eq:normalform}
X=\sum_{\gamma\in\Gamma}\sum_{\xi\in\Xi}\sum_{\mathbf k\in A}e^{\gamma\cdot Q}\,q^{\xi\cdot a[0]}\,f_{\gamma,\xi,\mathbf k}\,x^{\mathbf k},
\qquad f_{\gamma,\xi,\mathbf k}\in\mathcal S,
\end{equation}
where $\Gamma$, $\Xi$ and $A$ are finite sets.
Assume $\rho(X)=0$. We prove that all $f_{\gamma,\xi,\mathbf k}$ vanish, so that $X=0$.

We prove the cancellation of the elements of $\Gamma$.
The map $s$ is injective: if $s(\gamma)=s(\gamma')$, linear independence of the $\al_i$ gives
$r\beta\gamma_i=r\beta\gamma'_i$ in $\Q[\beta]$, hence $\gamma_i=\gamma'_i$.
The term of \eqref{eq:normalform} indexed by $\gamma$ maps $\pi_\mu$ into $\pi_{\mu+s(\gamma)}$.
Projecting $\rho(X)(\pi_\mu)$ onto $\pi_{\mu+s(\gamma)}$ and using that $\tau_{s(\gamma)}$ is injective, we find that
$Y_\gamma:=\sum_{\xi,\mathbf k}q^{\xi\cdot a[0]}f_{\gamma,\xi,\mathbf k}x^{\mathbf k}$ acts by $0$ on every $\pi_\mu$.

Now we consider the elements of $\Xi$.
Fix $\gamma$ and drop it from the notation. By \eqref{eq:fockmodel}, for every $\mu\in P$, $Y_\gamma$ acts on
$\pi_\mu\simeq\mathcal S$ as
\[
\textstyle\sum_{\xi\in\Xi}q^{\sum_i\xi_i(\mu,\al_i)}\sum_{\mathbf k}f_{\xi,\mathbf k}\partial^{\mathbf k},
\]
and the operators $\sum_{\mathbf k}f_{\xi,\mathbf k}\partial^{\mathbf k}\in\operatorname{End}_K(\mathcal S)$ do not depend on $\mu$.
For $j\in I$, set $\varpi_j:=\frac{2}{(\al_j,\al_j)}\om_j\in P$; since $(\varpi_j,\al_i)=\delta_{ij}$, replacing $\mu$ by
$\mu+\varpi_j$ multiplies $q^{\sum_i\xi_i(\mu,\al_i)}$ by $q^{\xi_j}$.
Suppose that some $\sum_{\mathbf k}f_{\xi,\mathbf k}\partial^{\mathbf k}$ is nonzero. Among all families $(\psi_\xi)_{\xi\in\Xi}$
in $\operatorname{End}_K(\mathcal S)$, not all zero, such that $\sum_\xi q^{\sum_i\xi_i(\mu,\al_i)}\psi_\xi=0$ for every
$\mu\in P$, choose one with the smallest number $N$ of nonzero terms.
If $N=1$, then $q^{\sum_i\xi_i(\mu,\al_i)}\psi_\xi=0$ with $q^{\sum_i\xi_i(\mu,\al_i)}\in K^\times$, so $\psi_\xi=0$, which is absurd.
If $N\geq2$, choose $\xi\neq\xi'$ with $\psi_\xi\neq0\neq\psi_{\xi'}$, and choose $j$ with $\xi_j\neq\xi'_j$.
Subtracting $q^{\xi_j}$ times the relation at $\mu$ from the relation at $\mu+\varpi_j$ gives
\[
\textstyle\sum_\eta q^{\sum_i\eta_i(\mu,\al_i)}\,(q^{\eta_j}-q^{\xi_j})\,\psi_\eta=0\qquad\text{for every }\mu\in P.
\]
In this new family the term $\eta=\xi$ vanishes, so it has fewer than $N$ nonzero terms.
Its term $\eta=\xi'$ is $(q^{\xi'_j}-q^{\xi_j})\psi_{\xi'}$, and this is nonzero.
Indeed, $q^{\xi'_j}-q^{\xi_j}=q^{\xi_j}\big(e^{h(\xi'_j-\xi_j)}-1\big)$ is nonzero, since its coefficient of $h$ is
$\xi'_j-\xi_j\neq0$. This contradicts the minimality of $N$.
Hence $\sum_{\mathbf k}f_{\xi,\mathbf k}\partial^{\mathbf k}=0$ for every $\xi$.

Finally consider $A$.
Fix $\xi$ and drop it from the notation. Then $\sum_{\mathbf k}f_{\mathbf k}\partial^{\mathbf k}$ vanishes on $\mathcal S$,
hence on $\mathcal S_L:=L\otimes_K\mathcal S=L[x_{j,-n}]$, where $L:=\operatorname{Frac}K$.
We have $c_n\neq0$ and $\B(q^n,t^n)=D(q^n,t^n)C(q^n,t^n)$, with $\det D(q^n,t^n)=\prod_i[r_i]_{q^n}\neq0$ and
$C(q^n,t^n)$ invertible by Lemma~\ref{invertiblecartan}. Hence $\B(q^n,t^n)\in GL_\el(L)$.
Therefore $z_{i,-n}:=\sum_jc_n^{-1}\big(\B(q^n,t^n)^{-1}\big)_{ji}\,x_{j,-n}$ are new polynomial generators of $\mathcal S_L$, and they
satisfy $\partial_{i,n}(z_{k,m})=\delta_{n,-m}\delta_{i,k}$. In other words, $\partial_{i,n}=\partial/\partial z_{i,-n}$ on $\mathcal S_L$.
Suppose that some $f_{\mathbf k}\neq0$, and choose such a $\mathbf k$ with $|\mathbf k|$ minimal.
For every $\mathbf k'\neq\mathbf k$ with $f_{\mathbf k'}\neq0$, minimality gives $\mathbf k'\not\leq\mathbf k$, so
$\partial^{\mathbf k'}z^{\mathbf k}=0$.
Hence $0=\sum_{\mathbf k'}f_{\mathbf k'}\partial^{\mathbf k'}(z^{\mathbf k})=\mathbf k!\,f_{\mathbf k}\neq0$, a contradiction.
Thus all $f_{\gamma,\xi,\mathbf k}$ vanish, and $X=0$ by \eqref{eq:normalform}.
\end{proof}

Hence, we will write ${x_i[m]}$ for $\rho(x_i[m])$  as elements of $\operatorname{End}(\bigoplus_{\mu\in P}\pi_\mu)$. The previous proposition gives:
$$\Hqt(\g)\hookrightarrow \operatorname{End}(\bigoplus_{\mu\in P}\pi_\mu).$$
In particular, the ${y_i[m]}_{i\in I,m<0}$ are algebraically independent as elements of $\operatorname{End}(\bigoplus_{\mu\in P}\pi_\mu)$, and similarly for $(a_i[n])_{i\in I,n\neq0}$ and $(x_i[n])_{i\in I,n\ne0}$.\\

\subsection{Topology on $\Hqt(\g)$}
To define formal series with coefficients in $\HH(\g)$ and their products, we need a completion of $\HH(\g)$.

By the proof of Proposition~\ref{fockrep}, the normally ordered monomials
$$e^{\gamma\cdot Q}\,q^{\xi\cdot a[0]}\,x_{i_1}[n_1]\cdots x_{i_m}[n_m],\qquad n_1\le\dots\le n_m,\quad n_j\neq0,
\quad i_j\le i_{j+1}\text{ if }n_j=n_{j+1},$$
form a basis of $\HH(\g)$ over $K$. The annihilation degree of such a monomial is $\sum_j\max(0,n_j)$. For
$k\ge0$, let $I_k$ be the subspace spanned by the normally ordered monomials of annihilation degree $\ge k$. Since
positive modes commute with each other, $I_k$ is a left ideal, and $\bigcap_kI_k=0$.

Moreover, $I_{k+n}\,x_j[-n]\subset I_k$ for $n>0$, since moving $x_j[-n]$ to the left of positive modes lowers the
annihilation degree by at most $n$. For the other generators $a$, we have $I_k\,a\subset I_k$. Hence, for every
$Y\in\HH(\g)$ and every $k$, there is $N\ge0$ such that $I_{k+N}Y\subset I_k$.

The completion $\widehat{\mathcal H}_{q,t}(\g):=\varprojlim_k\HH(\g)/I_k$ is the space of formal series
$X=\sum_jc_jM_j$ of normally ordered monomials such that, for every $k$, only finitely many $M_j$ have annihilation degree
$<k$. We denote by $X_{<k}$ the sum of these terms. It is an algebra: $XY$ is defined modulo $I_k$ by
$X_{<k+N}Y_{<k}$, where $I_{k+N}Y_{<k}\subset I_k$.

We grade $\pi_\mu$ by the degree of $f\in\mathcal S$ in $fv_\mu$, where $\deg x_{j,-n}:=n$; a vector of degree $<k$ is a
sum of homogeneous vectors of degree $<k$. By \eqref{eq:fockmodel}, a positive mode $x_j[n]$ lowers the degree by $n$.
In a normally ordered monomial the positive modes act first, so a monomial of annihilation degree $\geq k$ kills the
vectors of degree $<k$, and so does $I_k$. Hence $X\cdot v:=X_{<k}\cdot v$, for $v$ of degree $<k$, does not depend on
$k$ (since $X_{<k'}-X_{<k}\in I_k$ for $k'\geq k$) and defines a representation $\widehat\rho$ of
$\widehat{\mathcal H}_{q,t}(\g)$ on $\bigoplus_\mu\pi_\mu$ which extends $\rho$.

\begin{lem}\label{lem:truncated}
Let $k\geq0$ and let $Y\in\HH(\g)$ be such that $Y\cdot v=0$ for every vector $v$ of degree $<k$. Then $Y\in I_k$.
\end{lem}
\begin{proof}
Let $Y''$ be the sum of the terms of $Y$ of annihilation degree $\geq k$ in the basis of normally ordered monomials,
and set $Y':=Y-Y''$. Then $Y''\in I_k$, so $Y''$ kills the vectors of degree $<k$, and so does $Y'$.
Write $Y'$ as in \eqref{eq:normalform}. Since the annihilation degree of $x^{\mathbf k}$ is $\sum_{i,n}nk_{i,n}$, every
$\mathbf k\in A$ satisfies $\sum_{i,n}nk_{i,n}<k$.
The proof of injectivity in Proposition~\ref{fockrep} applies to $Y'$ with every operator restricted to the vectors
of degree $<k$: shifts and zero modes preserve the degree, the argument on $\Xi$ is unchanged, and in the argument on $A$
the test vector $z^{\mathbf k}$ has degree $\sum_{i,n}nk_{i,n}<k$, because $z_{i,-n}$ is homogeneous of degree $n$.
Hence $Y'=0$, and $Y=Y''\in I_k$.
\end{proof}

\begin{prop}\label{annidegree}
The representation $\widehat\rho:\widehat{\mathcal H}_{q,t}(\g)\to\operatorname{End}(\bigoplus_{\mu\in P}\pi_\mu)$ is faithful.
\end{prop}
\begin{proof}
Let $X\in\widehat{\mathcal H}_{q,t}(\g)$ with $\widehat\rho(X)=0$, and fix $k$. For every $v$ of degree $<k$, we have
$X_{<k}v=\widehat\rho(X)v=0$. By Lemma~\ref{lem:truncated}, $X_{<k}\in I_k$. Since $X_{<k}$ is a combination of normally
ordered monomials of annihilation degree $<k$, this gives $X_{<k}=0$. This holds for every $k$, so $X=0$.
\end{proof}
Since $\bigcap_kI_k=0$, we have $\HH(\g)\subset\widehat{\mathcal H}_{q,t}(\g)$, and we again write $\HH(\g)$ for the completion.

\section{$\mathbf{H}_{q,t}(\g)$ and $\mathbf{W}_{q,t}(\g)$ }\label{sect3}
In this section, we recall the definition of the deformed $W$-algebra $\WWqt$. 
It was introduced by Frenkel and Reshetikhin in \cite{MR1646483}.
We first introduce some formal power series in $\Hqt(\g)$ and then we recall 
the definition of the screening operators all due to Frenkel and Reshetikhin. 
Finally, we recall the definition of the deformed $W$-algebra as the subalgebra of $\Hqt(\g)$ commuting with the screening operators.

\subsection{Some fields in $\Hqt(\g)\llbracket z^{\pm1}\rrbracket $}
in this subsection , we introduce some formal power series in $\Hqt(\g)\llbracket z^{\pm1}\rrbracket$ due to Frenkel and Reshetikhin (\cite{MR1646483}), that will be useful in the definition of the screening operators and the deformed $W$-algebra. 
We begin by recalling the notion of \textit{fields}.
\begin{defi}
    A \textit{field} $\Phi(z)\in \Hqt(\g)\llbracket z^{\pm1}\rrbracket$ is a formal power series $$\Phi(z)=\sum_{n\in\Z}A_nz^{n},$$
    such that for all $x\in\bigoplus_{\mu\in P}\pi_\mu$, there exists $N\in\Z$ such that for all $n\leq N$, $A_nx=0$.
\end{defi}
\vskip0.5cm
For each $i\in\{1,\dots,\el\}$, we introduce the following symbols:
$$
A_i(z) = t^{2(\rho^\vee,\al_i)} q^{-2r (\rho,\al_i) + 2a_i[0]}
\exp \left(\sum_{m\neq 0} a_i[m] z^{-m} \right),
$$
$$
Y_i(z) = t^{2(\rho^\vee,\om_i)} q^{-2r (\rho,\om_i) + 2 y_i[0]}
\exp \left(\sum_{m\neq 0} y_i[m] z^{-m} \right).
$$
Note that $(\rho^\vee,\al_i) = 1, r(\rho,\al_i) = \rr_i$.\\
$:A_i(z):$ and $:Y_i(z):$ are formal series with coefficients in $\Hqt(\g)$. We have:

$$
:A_i(z): = t^{2(\rho^\vee,\al_i)} q^{-2r (\rho,\al_i) + 2a_i[0]}
:\exp \left(\sum_{m\neq 0} a_i[m] z^{-m} \right):\in \Hqt(\g)\llbracket z^{\pm1}\rrbracket ,
$$
$$
:Y_i(z): = t^{2(\rho^\vee,\om_i)} q^{-2r (\rho,\om_i) + 2 y_i[0]}
:\exp \left( \sum_{m\neq 0} y_i[m] z^{-m} \right):\in \Hqt(\g)\llbracket z^{\pm1}\rrbracket ,
$$
where $:\cdot:$ stands for the normal ordered product (see \cite{frenkel2004vertex}). In general, for $\gamma\in\beta^{-1}\Q[\beta]$, $\xi\in\Q[\beta]$,
\begin{align*}
:e^{\gamma Q_i}q^{\xi a_i[0]}\exp\left(\sum_{m\neq 0} x_i[m]z^{-m}\right): ~~&=~~:q^{\xi a_i[0]}e^{\gamma Q_i}\exp\left(\sum_{m\neq 0} x_i[m]z^{-m}\right):, \\[3mm]
&:=~~e^{\gamma Q_i}q^{\xi a_i[0]}\exp\left(\sum_{m< 0} x_i[m]z^{-m}\right)\exp\left(\sum_{m> 0} x_i[m]z^{-m}\right).
\end{align*}
The \textit{Fourier coefficients} of these formal series lie in $\Hqt(\g)$. The following relation is straightforward: for all $i\in I$, $a\in  q^\Z t^\Z$,
{\footnotesize
$$\begin{aligned}
\colon A_i(za) \colon \;=&\colon Y_i(zaq^{-r_i}t)\,Y_i(zaq^{r_i}t^{-1})
\prod_{\underset{C_{ji}=-1}{j}} Y_j(za)^{-1}\\
 &\times\prod_{\underset{C_{ji}=-2}{j}} Y_j(zaq^{-1})^{-1}\,Y_j(zaq)^{-1}
\prod_{\underset{C_{ji}=-3}{j}} Y_j(zaq^{-2})^{-1}\,Y_j(za)^{-1}\,Y_j(zaq^{2})^{-1}\colon
\end{aligned}
$$}
\begin{rem}
It is not the same as defining directly $A_i(z)$ and $Y_i(z)$ as the formal series with the normal ordering. Indeed, a priori, we have:
    $$:~:Y_i(z)::Y_i(za):~:~~\neq~~:Y_i(z)Y_i(za):$$
\end{rem}

\subsection{Screening operators}\label{sectionscreening}
in this subsection  we define the screening currents introduced by Frenkel and Reshetikhin in \cite{MR1646483}.
 The $(-1)^{\text{th}}$ Fourier coefficients of the screening currents are called the screening operators, 
 and the deformed $W$-algebra is a set of fields commuting with the screening operators.\\
 We define $\Hqt'(\g):=K[\beta^{-1}]\otimes_K\Hqt(\g)$.
\vskip1cm

For each $i\in\{1,\dots,\el\}$, $m\in\Z\backslash\{0\}$ define the modes $s_i^{\pm}[m]$ for $m\in\Z$ by the formulas
\begin{equation}    \label{sm+}
s^+_i[m] = \frac{a_i[m]}{q^{m \rr_i}-q^{-m \rr_i}}=x_i[m]\frac{q^{m }-q^{-m }}{m(q^{m \rr_i}-q^{-m \rr_i})}\in\Hqt(\g), \quad m\neq 0,
\end{equation}
\begin{equation}    \label{sm-}
s^-_i[m] = \frac{a_i[m]}{t^m-t^{-m}}=x_i[m]\frac{q^{m}-q^{-m}}{m(t^{m}-t^{-m})}\in\Hqt'(\g), \quad m\neq 0.
\end{equation}

\begin{rem}
    The element $s_i^-[m]$ does not belong to $\Hqt(\g)$ as it involves negative powers in $\beta$.
\end{rem}

\vskip 1cm

Now define the following symbols:
\begin{align}   
S_i^+(z) & = e^{-Q_i/\rr_i}  \exp \left( \sum_{m\neq 0}
s^+_i[m] z^{-m} \right),\\    
S_i^-(z) & = e^{Q_i/\beta}  \exp \left( - \sum_{m\neq 0}
s^-_i[m] z^{-m} \right)
\end{align}

The {\em screening currents} are the following fields with coefficients in $\Hqt(\g)$ (resp. $\mathcal{H}_{q,t}'(\g)$):
\begin{align}    \label{s+}
:S_i^+(z): & = e^{-Q_i/\rr_i} :\exp \left( \sum_{m\neq 0}
s^+_i[m] z^{-m} \right):=\sum_{m\in\Z}S_{i,m}^{+}z^{-m} \in\mathcal{L}(\pi_0, \pi_{-\beta\al_ir/r_i})\llbracket{z^{\pm1}}\rrbracket,\\    \label{s-}
:S_i^-(z): & = e^{Q_i/\beta} :\exp \left( - \sum_{m\neq 0}
s^-_i[m] z^{-m} \right):=\sum_{m\in\Z}S_{i,m}^{-}z^{-m}\in K[\beta^{-1}]\otimes\mathcal{L}(\pi_0, \pi_{r\al_i})\llbracket z^{\pm1}\rrbracket .
\end{align}
\begin{rem}
In standard references such as \cite{MR1646483} and \cite{MR1633032}, the screening currents include a zero-mode factor $z^{\pm s_i^{\mp}[0]}$ and are defined as follows:
$$:S_i^+(z):  = e^{-Q_i/\rr_i}z^{- s_i^+[0]} :\exp \left( \sum_{m\neq 0}
s^+_i[m] z^{-m} \right):,$$
$$:S_i^-(z):  = e^{Q_i/\beta} z^{ s_i^-[0]}:\exp \left( - \sum_{m\neq 0}
s^-_i[m] z^{-m} \right):.$$
However, when acting on the module $\pi_0$, these zero-modes act as the identity. Consequently, we omit these terms in our definition.
This omission is necessary to ensure mathematical rigor: we require our objects to be formal power series in $z$ with coefficients that are strictly independent of $z$.
Since the status of $z^{\pm s_i^{\mp}[0]}$ as such an object is ambiguous, removing it ensures that the screening currents are well-defined formal series. We will see thet the difference relations below are no longer true in $\Hqt(\g)$. However, they remain true in $\mathcal{L}(\pi_0, \pi_\mu)\llbracket{z^{\pm1}}\rrbracket$. One may naturally wonder if this formalism gives a "weak" version of the deformed $W$-algebras defined by Frenkel and Reshetikhin. This is not the case as in both \cite{MR1646483,MR1633032} the residues are only defined as elements of $\mathcal{L}(\pi_0, \pi_\mu)$. Indeed, the meaning of a residue of screening currents above in $\Hqt(\g)$ is not clear.
\end{rem}
\begin{prop}\label{differencesrelationnelles}
The screening currents satisfy the difference equations in $\mathcal{L}(\pi_0, \pi_\mu)\llbracket{z^{\pm1}}\rrbracket$, with $\mu=-\beta\al_ir/r_i$ (resp. $\mu=r\al_i)$:
\begin{equation}    \label{scr1}
:S^+_i(zq^{-\rr_i}): = t^{-2} q^{2\rr_i}:A_i(z) S^+_i(zq^{\rr_i}):,
\end{equation}
and
\begin{equation}    \label{scr2}
:S^-_i(zt): = t^{-2} q^{2\rr_i}:A_i(z) S^-_i(zt^{-1}):.
\end{equation}
\end{prop}
\begin{proof} Let us prove the first equation.
\begin{align*}
\colon S^+_i(zq^{-r_i})\colon&=e^{-Q_i/r_i}\colon\exp\left(\sum_{m\ne0}s_i^+[m]z^{-m}q^{r_im}\right)\colon\\[3mm]
&=e^{-Q_i/r_i}\colon\exp\left(\sum_{m\ne0}(s_i^+[m]q^{-r_im}+a_i[m])z^{-m}\right)\colon\\[3mm]
&=e^{-Q_i/r_i}q^{2a_i[0]}\colon\exp\left(\sum_{m\ne0}s_i^+[m]q^{-r_im}z^{-m}\right)\exp\left(\sum_{m\ne0}a_i[m]z^{-m}\right)\colon
\end{align*}
as ${q^{2a_i[0]}}_{|\pi_0}=\on{id}_{|\pi_0}$
\begin{align*}
\colon S^+_i(zq^{-r_i})\colon&=e^{-Q_i/r_i}q^{2a_i[0]}\colon\exp\left(\sum_{m\ne0}s_i^+[m]q^{-r_im}z^{-m}\right)\exp\left(\sum_{m\ne0}a_i[m]z^{-m}\right)\colon\\[3mm]
&=t^{-2}q^{2r_i}\colon A_i(z) S^+_i(zq^{\rr_i})\colon
\end{align*}
The proof for $S_i^-(z)$ is exactly analogous.
\end{proof}

\subsection{The deformed chiral algebra $\mathbf{H}_{q,t}(\g)$}
in this subsection , we introduce what Frenkel and Reshetikhin call the \textit{deformed chiral algebra} 
(\cite{MR1646483,frenkel1997towards}) $\mathbf{H}_{q,t}(\g)$. It is a vector space spanned by the monomials on the $Y_i(za)$ and their 
derivatives. The deformed chiral $W$-algebra will be defined as a subspace of $\mathbf{H}_{q,t}(\g)$. 
However, in this paper we will only consider the monomials on the $Y_i(za)$ without any derivative. 
It makes sense as this is a subspace of Frenkel and Reshetikhin's deformed chiral algebra, but the deformed $W$-algebra we obtain is not trivial. For example, the fields Frenkel-Reshetikhin and Bouwknegt-Pilch compute in \cite{MR1646483,MR1633032} lie in this subspace.\\

Let $\mathbf{H}_{q,t}(\g)\subset\Hqt\llbracket z^{\pm1}\rrbracket$ be the $K$-vector space spanned by formal power series of the form 
$$:D_z^{n_1}Y_{i_1}(za_1)^{\epsilon_1}\dots D_z^{n_m}Y_{i_m}(za_m)^{\epsilon_m}:~\in \Hqt(\g)\llbracket z^{\pm1}\rrbracket  $$
for $m\geq1$, $\epsilon_i=\pm1$, $a_1,\ldots,a_m\in q^\Z t^\Z$, $n_1,\ldots,n_m\in\N$, and where $D_z^kf(z):=z^k{\partial_z}^kf(z)$.\\
\begin{lem}
    The normal ordering $:D_z^{n_1}Y_{i_1}(za_1)^{\epsilon_1}\dots D_z^{n_m}Y_{i_m}(za_m)^{\epsilon_m}:$ is independent of the ordering of the factors.
\end{lem}

\begin{defi}\label{defmonomial}
\begin{itemize}
\item Let $\mathbf{H}_{q,t}^0(\g)\subset\mathbf{H}_{q,t}(\g)$ be the $K$-vector subspace of $\HHqt(\g)$ spanned by the formal power series of the form 
$$:Y_{i_1}(za_1)^{\epsilon_1}\dots Y_{i_m}(za_m)^{\epsilon_m}:~\in \Hqt(\g)\llbracket z^{\pm1}\rrbracket , $$
for $m\geq1$, $\epsilon_i=\pm1$, $a_1,\ldots,a_m\in q^\Z t^\Z$.\\
\item Let $\mathbf{M}$ be the following set:
$$\mathbf{M}:=\left\{:Y_{i_1}(za_1)^{\epsilon_1}\dots Y_{i_m}(za_m)^{\epsilon_m}: \right\}~\subset \mathbf{H}_{q,t}^0(\g)$$
for $m\geq1$, $\epsilon_i=\pm1$, $a_1,\ldots,a_m\in q^\Z t^\Z$.\\ 
An element $m\in\mathbf{M}$ is called a \textit{monomial} in the $Y_i(za)^{\pm1}$.
\end{itemize}
\end{defi}

\begin{prop}\label{polynomial}
The map 
\begin{equation} \label{isom}
    \left\{
  \begin{array}{rcl}
    K[Y_{i,a}^{\pm1}]_{i\in I,a\in q^\Z t^\Z} & \longrightarrow &\Hqt\llbracket z^{\pm1}\rrbracket \\[3mm]
    \sum_{p=1}^d\lambda_p\prod_{i=1}^{d_p}Y_{j_{p,i},a_{p,i}}^{n_i}& \longmapsto &  \sum_{p=1}^d\lambda_p\colon \prod_{i=1}^{d_p}Y_{j_{p,i}}(za_{p,i})^{n_i}\colon \\
  \end{array}\right.
\end{equation}
is an injective linear map. In particular, the fields ${Y_i(za)}_{i\in I,a\in q^\Z t^\Z}$ are algebraically independent with respect to the normally ordered product.
\end{prop}

\begin{proof}
    Let $(M_p)_{ p}$ be a finite set of distinct monomials such that $$\sum_{p=1}^N\lambda_pM_p=0$$
    We will prove that for all $p$, $\lambda_p=0$.
    Let $$M_p=:Y_{j_{p,1}}(za_{p,1})^{m_{p,1}}Y_{j_{p,2}}(za_{p,2})^{m_{p,2}}\cdots Y_{j_{p,d_p}}(za_{p,d_p})^{m_{p,d_p}}:$$
    We define the following left-ideal of the algebra $\Hqt(\g)$: 
    $$\Hqt^+(\g):=\langle x_i[n],(q^{\xi y_i[0]}-1),(t^{y_i[0]}-1)\rangle_{i\in I,\xi\in\Q[\beta],n>0}$$
    We define the canonical projection $\Pi:\Hqt(\g)\longrightarrow\Hqt(\g)/\Hqt^+(\g)$. It is a homomorphism of left $\Hqt(\g)$-module. 
    We define its extension to formal series with coefficients in $\Hqt(\g)$ as follows:
    \begin{displaymath}
\Pi:
\left\{
  \begin{array}{rcl}
    \mathbf{H}_{q,t}^0(\g) & \longrightarrow &\mathbf{H}_{q,t}^0(\g) \\
    \sum_{n\in\Z}A_{n}z^n & \longmapsto & \sum_{n\in\Z}\Pi(A_{n})z^n \\
  \end{array}
\right.
\end{displaymath}
Then $$\sum_p\lambda_p\Pi(M_p)=0.$$
Furthermore, for all $i\in I$, for all $a\in  q^\Z t^\Z$,
\begin{align*}
   \Pi(Y_i(za)) &=\Pi\left(t^{2(\rho^\vee,\om_i)} q^{-2r (\rho,\om_i) + 2 y_i[0]}
\exp \left( \sum_{n>0} y_i[-n] (az)^n \right) \exp \left( \sum_{n>0} y_i[n] (az)^{-n} \right)\right)\\
   &=C\exp\left(\sum_{n>0} y_i[-n](az)^n\right),
\end{align*}

with $C=t^{2(\rho^\vee,\om_i)} q^{-2r (\rho,\om_i)}\in K^*$.\\
Thus,
\begin{align*}
    \Pi(M_p)&=C_p\exp\left(\sum_{n>0}\sum_{r=1}^{d_p}m_{p,r}y_{j_{p,r}}[-n]z^n{a_{p,r}}^n\right)\\[3mm]
    &=C_p\left(1+z\sum_{r=1}^{d_p}m_{p,r}y_{j_{p,r}}[-1]{a_{p,r}} +z^2\left[\sum_{r=1}^{d_p}m_{p,r}y_{j_{p,r}}[-2]{a_{p,r}}^2 \right.\right.\\[3mm]
    &\qquad \qquad\qquad 
    +\left.\left. \frac{1}{2}\left(\sum_{r=1}^{d_p}m_{p,r}y_{j_{p,r}}[-1]{a_{p,r}} \right)^2\right]+\dots\right)\\[3mm]
    &=C_p\sum_{n>0}z^n\sum_{\underset{\underset{n_i>0}{n_1+\dots +n_k=n}}{1\leq k\leq n}}\frac{1}{k!}\prod_{j=1}^k\left(\sum_{r=1}^{d_p}m_{p,r}y_{j_{p,r}}[-n_j]{a_{p,r}}^{n_j}\right),
\end{align*}
with $C_p\in K^*$ a non-zero constant depending on $q,t$.\\
It is a linear combination of elements of the form 
$$z^ny_{i_1}[-n_1]^{k_1}y_{i_2}[-n_2]^{k_2}\dots y_{i_s}[-n_s]^{k_s},$$
but the $(y_i[-m])_{i\in I,m>0}$ are algebraically independent in $\Hqt(\g)$. Then for every finite set of triples $(i_u,n_u,k_u)\in I\times \N^*\times \N^*$ such that the $(i_u,n_u)\in I\times \N^*$ are pairwise distinct, $1\leq u\leq s$, 
$$\sum_p\lambda_pC_p\frac{1}{(\sum_uk_u)!}\prod_{u=1}^s\left(\sum_{r|j_{p,r}=i_u}m_{p,r}{a_{p,r}}^{n_u} \right)^{k_u}=0.$$
Then for all finite set of triples $(i_u,n_u,k_u)\in I\times \N^*\times \N^*$ such that the $(i_u,n_u)\in I\times \N^*$ are pairwise distinct, $1\leq u\leq s$, 
$$\sum_p\lambda_pC_p\prod_{u=1}^s\left(\sum_{r|j_{p,r}=i_u}m_{p,r}{a_{p,r}}^{n_u} \right)^{k_u}=0.$$
Let $S_{p,u}=\sum_{r|j_{p,r}=i_u}m_{p,r}{a_{p,r}}^{n_u}$.
We have $$\sum_{b_1,\dots,b_s\in K^*}b_1^{k_1}b_2^{k_2}\dots b_s^{k_s}\sum_{\underset{\forall u,S_{p,u}=b_u}{p}}\lambda_pC_p=0.$$
This equality holds for all $(k_u)\in(\N^*)^s$. It is a Vandermonde system and it implies $$\sum_{\underset{\forall u,S_{p,u}=b_u}{p}}\lambda_pC_p=0$$ for all $(b_1,\dots,b_s)\in(K^*)^s$.

Let $p\in\llbracket 1,N\rrbracket $. We want to conclude that $\lambda_p=0$. So we want to prove that this sum contains at most one term. We assume that there exists $p'$ such that for all set of pairwise distinct couples $(i_u,n_u)_u\in (I\times \N^*)^s$, for all $u\in\llbracket 1,s\rrbracket $, $S_{p,u}=S_{p',u}$. Then for all set of pairwise distinct couples $(i_u,n_u)_u\in (I\times \N^*)^s$, for all $u\in\llbracket 1,s\rrbracket $, 
$$\sum_{r|j_{p,r}=i_u}m_{p,r}{a_{p,r}}^{n_u}=\sum_{r|j_{p',r}=i_u}m_{p',r}{a_{p',r}}^{n_u}.$$ 
Then
$$\sum_{a\in K^*}\sum_{\underset{a_{p,r}=a}{r|j_{p,r}=i_u}}m_{p,r}{a}^{n_u}=\sum_{a\in K^*}\sum_{\underset{a_{p',r}=a}{r|j_{p',r}=i_u}}m_{p',r}{a}^{n_u},$$ 
and we obtain the following Vandermonde system 
$$\sum_{a\in K^*}\left(\sum_{\underset{a_{p,r}=a}{r|j_{p,r}=i_u}}m_{p,r}-\sum_{\underset{a_{p',r}=a}{r|j_{p',r}=i_u}}m_{p',r}\right){a}^{n_u}=0.$$ 
We fix $i_u$ and we vary $n_u$ to obtain for all $u\in\llbracket 1,s\rrbracket $, for all $i_u$,
$$\sum_{\underset{a_{p,r}=a}{r|j_{p,r}=i_u}}m_{p,r}=\sum_{\underset{a_{p',r}=a}{r|j_{p',r}=i_u}}m_{p',r}.$$
Each sum contains at most one term. It implies $M_p=M_{p'}$ and $p=p'$.\\
Finally, we can fix a $s$-tuple $(i_u,n_u,b_u)_u\in(I\times \N^*\times K^*)^s$ such that $\lambda_p$ is the unique term in the sum $\sum_{\underset{\forall u,S_{p,u}=b_u}{p}}\lambda_pC_p=0$. Hence $\lambda_pC_p=0$, and 
$$\lambda_p=0.$$
\end{proof}
\begin{cor}
    \begin{enumerate}
        \item The map \eqref{isom} is an isomorphism of vector spaces:
        $$\mathbf{H}_{q,t}^0(\g)\simeq K[Y_{i,a}^{\pm1}]_{i\in I,a\in  q^\Z t^\Z},$$
        sending $Y_{i,a}^{\pm1}\mapsto Y_i(za)^{\pm1}$.
        \item The $(A_{i}(za))_{i\in I,a\in q^\Z t^\Z}$ are algebraically independent in $\mathbf{H}_{q,t}^0(\g)$ with respect to the normally ordered product.
    \end{enumerate}
\end{cor}
\begin{proof}
    \begin{enumerate}
        \item The map $\eqref{isom}$ is surjective by definition of $\mathbf{H}_{q,t}^0(\g)$.
        \item The proof is the same, replacing $Y_i(za)$ by $A_i(za)$, because of the algebraic independence of the $(a_i[n])_{i\in I,n\in\Z}$ in $\Hqt(\g)$.
    \end{enumerate}
\end{proof}

\begin{defi}
    We also define a degree. 
    We denote by $d_{i,a}$ the degree of a field $\Phi(z)$ by taking its $Y_{i,a}$-degree in the image of $\Phi(z)$ in $K[Y_{j,b}^{\pm1}]_{j\in I, b\in q^\Z t^\Z}$.  
\end{defi}

\subsection{The deformed $W$-algebra $\WWqt$}

We shall now define the deformed $W$-algebra $\WWqt$ as introduced by Frenkel and Reshetikhin in \cite{MR1646483}. 
Firstly, we define the \textit{screening operators} as follows: 
$$S_i^{+}:=S_{i,1}^{+}\in\Hqt(\g),\qquad S_i^{-}:=S_{i,1}^{-}\in\Hqt'(\g),$$
where 
$$S_i^{+}(w)=\sum_{m\in\Z}S_{i,m}^{+}w^{-m}\in\Hqt(\g)\llbracket w^{\pm1}\rrbracket ,\qquad S_i^{-}(w)=\sum_{m\in\Z}S_{i,m}^{-}w^{-m}\in\Hqt'(\g)\llbracket w^{\pm1}\rrbracket .$$ 
It is the Fourier coefficient in front of $w^{-1}$. It is also called \textit{the residue} at $0$ of $S_i^\pm(w)$.
We say that a field $A(z)=\sum A_nz^{-n}\in\Hqt(\g)\llbracket z^{\pm1}\rrbracket $ commutes with an operator $B\in\Hqt'(\g)$ in $\pi_0$ if for all $n\in\Z$, for all $f\in\pi_0$, $[A_n,B]\cdot f=0$.\\
\begin{defi}
Let $\mathbf{W}_{q,t}(\g)$ be the vector subspace of $\mathbf{H}_{q,t}(\g)$ of fields commuting with the screening operators $S_i^\pm$ in $\pi_0$.\\
The \textit{deformed $W$-algebra associated to the Lie algebra} $\g$ is the subalgebra of $\Hqt(\g)$ generated by the Fourier coefficients of the fields in $\mathbf{W}_{q,t}(\g)$ and is denoted $\Wqt(\g)$.\\
By abuse of notation, both $\mathcal{W}_{q,t}(\g)$ and $\WWqt$ are called \textit{the deformed $W$-algebra}. Remark that $\WWqt$ is not an algebra, but $\Wqt(\g)$ is an algebra.
\end{defi}
\begin{rem}
As in \cite{MR1646483}, the commutation is only in $\pi_0$:
$$\forall \Phi\in\WWqt,~\forall i\in I,\qquad [\Phi,S_i^\pm]\cdot\pi_0=0.$$ 
A priori, the fields in $\mathbf{W}_{q,t}(\g)$ do not commute with the screening operators in $\pi_\mu$ for $\mu\neq0$. For instance, the difference relations (see Proposition~\ref{differencesrelationnelles}) are no longer true outside $\pi_0$. One may remark that they are true if we consider the definition of the screening currents in \cite{MR1646483} containing the zero modes $z^{\pm s_i^\mp[0]}$. However, the rigorous meaning of
$z^{\pm s_i^\mp[0]}$ is unclear there: on $\pi_\mu$ it acts by a power $z^{c}$ with $c\in\Q[\beta]$, which is not an integer in
general, so that the residue of $S_i^\pm(w)$ is not defined.
 
From now on, commutation with the screening operators is always understood on $\pi_0$, and we omit it from the notation.
\end{rem}
\subsection{Some Computations}\label{computation}
in this subsection , we provide the fundamental computations required for the next sections. 
In particular, we derive explicit expressions for the coefficients within the fields of $\mathbf{W}_{q,t}(\g)$.
\vskip1cm
\subsubsection{Operator Product Expansion (OPE) and difference relations}\label{opesec}
We recall the following OPEs from \cite{MR1646483,MR1633032} for all $i,j\in I$, $i\ne j$:
{\footnotesize
    \begin{align*}
\text{1.~~}:Y_i(z)::S_i^+(w): &= t^{-2}\!\left(\frac{1-t\frac{w}{z}}{1-t^{-1}\frac{w}{z}}\right):Y_i(z)S_i^+(w): &\quad
\text{2.~~}:S^+_i(w)::Y_i(z): &= \frac{1-t^{-1}\frac{z}{w}}{1-t\frac{z}{w}}:S^+_i(w)Y_i(z):\\
\text{3.~~}:Y_i(z)::S_i^-(w): &= q^{2r_i}\!\left(\frac{1-q^{-r_i}\frac{w}{z}}{1-q^{r_i}\frac{w}{z}}\right):Y_i(z)S_i^-(w): &\quad
\text{4.~~}:S^-_i(w)::Y_i(z): &= \frac{1-q^{r_i}\frac{z}{w}}{1-q^{-r_i}\frac{z}{w}}:S^-_i(w)Y_i(z):\\[2mm]
 \text{5.~~}:Y_i(z)::S_j^\pm(w): &=:Y_i(z)S_j^\pm(w): =:S_j^\pm(w)::Y_i(z): & &(i\neq j)
\end{align*}}
We recall that the normal ordering convention between $e^{\gamma Q_i}$ and $q^{\xi y_i[0]}$ places the shift operators at left.
The first and third (resp. second and fourth) rational functions have to be understood as formal power series in positive powers of $\frac{w}{z}$ (resp. $\frac{z}{w}$). 
We remark that the first two rational functions are the same as third and fourth. \\
To construct elements in $\WWqt$ means constructing elements $\Phi(z)$ such that for all $i\in I$, $$\Res_w[\Phi(z),S_i^\pm(w)]=0.$$ 
Thus, we will use the difference relations \eqref{scr1}, \eqref{scr2} to get:
\begin{equation}\label{diffrely}
\delta\left(\frac{w}{zq^{r_i}}\right):Y_i(z)^{-1}S_i^-(w):=t^{-2} q^{2r_i}\delta\left(\frac{w}{zq^{r_i}}\right):Y_i(z)^{-1}A_i(zq^{r_i}t^{-1})S_i^-(wt^{-2}):
\end{equation}
and
\begin{equation}\label{diffrelyplus}
\delta\left(\frac{w}{zt^{-1}}\right):Y_i(z)^{-1}S_i^+(w):=t^{-2} q^{2r_i}\delta\left(\frac{w}{zt^{-1}}\right):Y_i(z)^{-1}A_i(zq^{r_i}t^{-1})S_i^+(wq^{2r_i}):
\end{equation}
\subsubsection{Commutators with the screening}
We deduce the following commutators. Let $:M(z):$ be a monomial with
$$M(z)=\prod_{k=1}^nY_{i_k}(za_k)^{\varepsilon_km_k},$$
where $i_k\in I$, $\varepsilon_k=\pm1$, $m_k\in\N$ and $a_k\in q^\Z t^\Z$ for all $1\le k\le n$. For $j\in I$, consider the rational functions
\begin{align*}
\mathcal{F}^+_{M,j}(z,w)&=\prod_{k\mid i_k=j}\left(\frac{1-t^{-\varepsilon_k}\frac{za_k}{w}}{1-t^{\varepsilon_k}\frac{za_k}{w}}\right)^{m_k},\\
\mathcal{F}^-_{M,j}(z,w)&=\prod_{k\mid i_k=j}\left(\frac{1-q^{\varepsilon_kr_j}\frac{za_k}{w}}{1-q^{-\varepsilon_kr_j}\frac{za_k}{w}}\right)^{m_k}.
\end{align*}
For a rational function $F$ in $z$ and $w$, denote by $i_{z,w}(F)$ (resp. $i_{w,z}(F)$) the formal power series expansion of $F$ in the region $|w|\ll|z|$ (resp. $|w|\gg|z|$). Then, for all $j\in I$,
\begin{equation*}
[:M(z):,:S^\pm_j(w):]=\Big(i_{z,w}\big(\mathcal{F}^\pm_{M,j}(z,w)\big)-i_{w,z}\big(\mathcal{F}^\pm_{M,j}(z,w)\big)\Big):M(z)S^\pm_j(w):.
\end{equation*}
For example, for $i\in I$ and $M(z)=Y_i(z)$, we get
$$[:Y_i(z):,:S_i^-(w):]=\delta\left(\frac{w}{z}q^{r_i}\right)(q^{2r_i}-1):Y_i(z)S_i^-(w):.$$

To simplify the computations, we now restrict to monomials satisfying some additional hypotheses.

\begin{defi}\label{defgenreg}
Let $:M(z):\in\mathbf{M}$ be a monomial, written in reduced form
$$M(z)=\prod_{i\in I,\ a\in q^\Z t^\Z}Y_i(za)^{d_{i,a}},$$
where $d_{i,a}=d_{i,a}(M)\in\Z$ and $d_{i,a}=0$ for all but finitely many pairs $(i,a)$.
\begin{enumerate}
    \item The monomial $M$ is \emph{generic} if $|d_{i,a}(M)|\le1$ for all $i\in I$ and $a\in q^\Z t^\Z$.
    \item A generic monomial $M$ is \emph{regular} if, for all $i\in I$ and $a\in q^\Z t^\Z$, the following conditions hold:
    \begin{enumerate}[label=(R\arabic*), ref=(R\arabic*)]
        \item\label{R1} if $d_{i,a}(M)=1$, then $d_{i,aq^{-2r_i}}(M)\neq-1$;
        \item\label{R2} if $d_{i,a}(M)=1$, then $d_{i,at^{2}}(M)\neq-1$;
        \item\label{R3} if $d_{i,a}(M)=1$, then $d_{i,aq^{-2r_i}t^{2}}(M)\neq-1$;
        \item\label{R4} if $d_{i,a}(M)=d_{i,aq^{-2r_i}t^{2}}(M)=1$, then $d_{i,aq^{-2r_i}}(M)=d_{i,at^{2}}(M)=1$;
        \item\label{R5} if $d_{i,a}(M)=d_{i,aq^{-2r_i}t^{2}}(M)=-1$, then $d_{i,aq^{-2r_i}}(M)=d_{i,at^{2}}(M)=-1$.
    \end{enumerate}
\end{enumerate}
\end{defi}

From now on, let $:M(z):\in\mathbf{M}$ be a generic and regular monomial. Let $i\in I$. Let
\begin{align*}
R_{M}^{(i)}&=\big\{a\in q^\Z t^\Z \big|\ d_{i,a}(M)=1,\text{ and }d_{i,aq^{-2r_i}}(M)=0\big\},\\
S_{M}^{(i)}&=\big\{b\in q^\Z t^\Z \big|\ d_{i,b}(M)=-1,\text{ and }d_{i,bq^{2r_i}}(M)=0\big\}.
\end{align*}
Then we have
\begin{equation}\label{formulecommut}
[:M(z):,:S^-_i(w):]
=\Bigg[\sum_{a\in R_{M}^{(i)}}C_{M,i,a}^+\,\delta\left(\frac{w}{za}q^{r_i}\right)+\sum_{b\in S_{M}^{(i)}}C_{M,i,b}^-\,\delta\left(\frac{w}{zb}q^{-r_i}\right)\Bigg]:M(z)S^-_i(w):,
\end{equation}
where the coefficients $C_{M,i,a}^+$ and $C_{M,i,b}^-$ are given by the partial fraction decomposition of the rational function $\mathcal{F}^-_{M,i}(z,w)$.

\subsubsection{How to compute the coefficients in the fields in $\WWqt$?}\label{calculcoeffappendix}

In the previous part, we computed the commutators between the monomials and the screening currents. In this subsubsection, we explain the strategy coming from \cite{FM1,MR1633032,fjm0,fjm1,kimurahiggsing} to compute explicitly fields in $\WWqt$.

Recall that a field lies in $\WWqt$ if and only if it commutes with the residues of the screening currents. We look for fields of the form
$$\Phi(z)=\sum_{M}\lambda_M\,:M(z):,$$
where the sum runs over a finite set of monomials $M$, and the coefficients $\lambda_M$ are to be determined in $K^\times$. We assume all monomials to be generic and regular. 

By \eqref{formulecommut}, the commutator of each monomial $:M(z):$ with a screening current is a finite sum of terms. The strategy is the following: each term in the commutator of a monomial $M$ is paired with a term in the commutator of a monomial $N$ following $M$ in the sum, and the coefficient $\lambda_N$ is chosen so that the residues of these two terms cancel each other. In this way, the coefficients are determined step by step, each $\lambda_N$ being fixed by the coefficients of the monomials preceding $N$. Then the residue of the commutator of $\Phi(z)$ with each screening current vanishes, that is, $\Phi(z)\in\WWqt$. Let us do the explicit computation for $M\in\mathbf{M}$, $i\in I$, $a\in R_{M}^{(i)}$ defined as in \eqref{formulecommut}. First define 
$$N:=MA_i(zaq^{-r_i}t)^{-1}.$$
We then have $a\in R_{M}^{(i)}$, and $aq^{-2r_i}t^2\in S_{M}^{(i)}$.
By a straightforward computation, the term in front of $\delta(\frac{w}{z}q^{r_i}a^{-1})$ in $[:M:,S_i^-]$ is 
$$C_{M,i,a}^+=q^{2r_id_i(M)}(1-q^{-2r_i})\prod_{b\ne a}\left(\frac{1-q^{-2r_i}ab^{-1}}{1-ab^{-1}}\right)^{d_{i,b}(M)},$$
and the term in front of $\delta(\frac{w}{z}q^{r_i}t^{-2}a^{-1})$ in $[:N:,S_i^-]$ is 
$$C_{N,i,aq^{-2r_i}t^2}^-=q^{2r_id_i(N)}(1-q^{2r_i})\prod_{b\ne aq^{-2r_i}t^2}\left(\frac{1-q^{-2r_i}t^2ab^{-1}}{1-t^2ab^{-1}}\right)^{d_{i,b}(M)},$$
where 
$$d_i(M):=\sum_c d_{i,c}(M).$$
Therefore, again by a straightforward computation by taking the residue, $\lambda_M,\lambda_N$ have to satisfy:
\begin{equation*}
\begin{split}
   \lambda_{N}&\times q^{2r_i(d_i(N)+1)}(1-q^{2r_i})\prod_{b\ne aq^{-2r_i}t^2}\left(\frac{1-q^{-2r_i}t^2ab^{-1}}{1-t^2ab^{-1}}\right)^{d_{i,b}(N)}\\[4mm]
   &=-\lambda_{M}\times q^{2r_id_i(M)}(1-q^{-2r_i})\prod_{b\ne a}\left(\frac{1-q^{-2r_i}ab^{-1}}{1-ab^{-1}}\right)^{d_{i,b}(M)}.
\end{split}
\end{equation*}
By regularity and genericity of $M$ and $N$, for all $b$ in these products, $b\not\in\{a,aq^{-2r_i}t^2\}$ and $d_{i,b}(N)=d_{i,b}(M)$. Thus,
\begin{equation}\label{coeff}
    \boxed{\lambda_{N}=\lambda_{M}\prod_{b\not\in\{a,aq^{-2r_i}t^2\}}\left(\frac{(1-q^{-2r_i}ab^{-1})(1-t^2ab^{-1})}{(1-ab^{-1})(1-q^{-2r_i}t^2ab^{-1})}\right)^{d_{i,b}(M)}.}
\end{equation}

\begin{rem}
    We recognize a product of $\mathcal{S}$-functions defined by Kimura and Pestun in \cite{kimura2018fractional} (see equation (3.45) in section 3.5.2).
\end{rem}

\begin{prop}\label{coeffwelldefprop}
Let $M$, $N$, $i$ and $a$ be as in \eqref{coeff}, and assume that $\lambda_M\in K^\times$. Assume that every $b$ appearing in the product \eqref{coeff} satisfies
$$ab^{-1}\notin\{1,\,q^{2r_i},\,t^{-2},\,q^{2r_i}t^{-2}\}.$$
Then $\lambda_N\in K^\times$. If moreover $ab^{-1}\notin\{t^{-1},q^{2r_i}t^{-1}\}$ for all such $b$, then $\lambda_N/\lambda_M\in\Q_{>0}+\beta K$.
\end{prop}
\begin{proof}
Recall that $q=e^h$ and $t=q^\beta=e^{h\beta}$, and that $a,b\in q^\Z t^\Z$. Set
$$\varphi(x)=\frac{(1-q^{-2r_i}x)(1-t^2x)}{(1-x)(1-q^{-2r_i}t^2x)}.$$
By \eqref{coeff}, $\lambda_N/\lambda_M$ is the product of the $\varphi(ab^{-1})^{d_{i,b}(M)}$, where $d_{i,b}(M)=\pm1$ by
genericity. For $c\in\Q^\times$, the elements of $c+\beta K$ are invertible in $K$, and $\Q^\times+\beta K$ and
$\Q_{>0}+\beta K$ are subgroups of $K^\times$. Hence it is enough to show that, for each $b$ in the product, $x=ab^{-1}$
satisfies
$$\varphi(x)\in c+\beta K\quad\text{for some } c\in\Q^\times,\ \text{with } c>0 \text{ if } x\notin\{t^{-1},q^{2r_i}t^{-1}\}.$$
Write $x=q^mt^n$, with $m,n\in\Z$, so that
$$\varphi(x)=\frac{1-e^{h(m-2r_i+n\beta)}}{1-e^{h(m-2r_i+(n+2)\beta)}}\cdot\frac{1-e^{h(m+(n+2)\beta)}}{1-e^{h(m+n\beta)}}.$$
Both quotients are of the form $Q(u,v)=\frac{1-e^{hu}}{1-e^{hv}}$ with $u=k+p\beta$, $v=k+p'\beta$ and $k,p,p'\in\Z$.
Write $1-e^{y}=y\,g(y)$, where $g(y)=-\sum_{n\ge0}\frac{y^n}{(n+1)!}$ is invertible in $\Q\llbracket y\rrbracket$. Then
$$Q(u,v)=\frac uv\cdot\frac{g(hu)}{g(hv)},\qquad\text{and}\qquad \frac{g(hu)}{g(hv)}\in1+\beta K,$$
since $u-v=(p-p')\beta$. Hence:
\begin{itemize}
    \item[(a)] if $k\neq0$, then $v$ is invertible in $K$ and $\frac uv=1+\frac{(p-p')\beta}{v}$, so $Q(u,v)\in1+\beta K$;
    \item[(b)] if $k=0$ and $p'\neq0$, then $\frac uv=\frac p{p'}$, so $Q(u,v)\in\frac{p}{p'}+\beta K$.
\end{itemize}
We can now conclude.
\begin{itemize}
    \item If $m\notin\{0,2r_i\}$, both quotients have $k\neq0$, so $\varphi(x)\in1+\beta K$ by (a).
    \item If $m=0$, then $n\notin\{0,-2\}$ since $x\notin\{1,t^{-2}\}$. By (a) and (b), $\varphi(x)\in\frac{n+2}{n}+\beta K$.
    \item If $m=2r_i$, then $n\notin\{0,-2\}$ since $x\notin\{q^{2r_i},q^{2r_i}t^{-2}\}$. By (a) and (b),
    $\varphi(x)\in\frac{n}{n+2}+\beta K$.
\end{itemize}
Finally, for $n\notin\{0,-2\}$, the numbers $\frac{n+2}{n}$ and $\frac{n}{n+2}$ are non-zero, and they are positive unless
$n=-1$, that is, unless $x=t^{-1}$ (when $m=0$) or $x=q^{2r_i}t^{-1}$ (when $m=2r_i$).
\end{proof}

\vskip 1cm
\begin{example}
    In the case of $\g=\mathfrak{sl}_2$, we have $r_1=1$. Let us consider the field 
    $$T(z)=:Y(z)Y(zq^{-2}): +\lambda_2:Y(z)Y(zq^{-4}t^2)^{-1}:+\lambda_3:Y(zq^{-2}t^{2})^{-1}Y(zq^{-4}t^2)^{-1}:.$$
    We have $M_1=:Y(z)Y(zq^{-2}):$ and $M_2=:Y(z)Y(zq^{-4}t^2)^{-1}:$. By \eqref{coeff}, we have 
    $$\lambda_{2}=\lambda_{1}\frac{(1+q^{-2})(1-q^{-2}t^2)}{1-q^{-4}t^2},$$
    then we get $$\lambda_3=\lambda_2\frac{1-q^{-4}t^2}{(1+q^{-2})(1-q^{-2}t^2)}=\lambda_1.$$
    Finally, setting $\lambda_1=1$, we get the field 

    \begin{equation*}
T(z)={:}Y(z)Y(zq^{-2}){:}+\tfrac{(1+q^{-2})(1-q^{-2}t^2)}{1-q^{-4}t^2}{:}Y(z)Y(zq^{-4}t^2)^{-1}{:}+{:}Y(zq^{-2}t^{2})^{-1}Y(zq^{-4}t^2)^{-1}{:},
\end{equation*}

    and we recover the interpolating function introduced in \cite{frenkel1996quantum,MR1646483} and constructed in a more elementary way in \cite{frenkelhernandez2011langlands} (Section 4.3).
\end{example}

\section{Construction of fields in $\mathbf{W}_{q,t}(\mathfrak{g})$}\label{sect5}

In this section, we explicitly construct elements in $\mathbf{W}_{q,t}(\mathfrak{g})$. 
We adapt to our purpose an algorithm which works step by step, starting from a \textit{dominant} monomial (that is, a monomial in $(Y_i(za))_{i\in \llbracket 1,\el\rrbracket, a\in  q^\Z t^\Z}$), and at each step, 
multiplying it by variables $A_{i}(za)^{-1}$ until we get an \textit{antidominant} monomial in $(Y_i(za)^{-1})_{i,a}$. 
This procedure is inspired by the Frenkel-Mukhin algorithm procedure for $q$-characters \cite{FM1}, and is observed to be efficient for computing $qq$-characters 
(see \cite{MR1633032}, Lemma 2.1 of \cite{fjm0}, Section 2.3 of \cite{fjm1}, or again Section 2.4 of \cite{kimurahiggsing}).

First, we describe this algorithm which, given a \textit{generic dominant regular} monomial $m\in\mathbf{M}$ (defined below), produces a field in $\mathbf{H}_{q,t}^0(\g)$. 
Then, we prove that this algorithm is well-defined and that the resulting field lies in $\WWqt$. 
Finally, we give examples of the results of this algorithm for various types of Lie algebras. Let us first define an \textit{admissible} variable or transformation.\\

\begin{defi}
Let $M\in\mathbf{M}$ be a monomial, and $(i,a)\in I\times q^\Z t^\Z$. We say that the variable $Y_i(za)$ (resp. the transformation $A_i(zaq^{-r_i}t)^{-1}$ is admissible in $M$ if:
$$ d_{i,a}(M)>0,\quad\text{and }d_{i,aq^{-2r_i}}(M)=d_{i,at^2}(M)=0.$$
\end{defi}

Let us now define the algorithm.

The Frenkel-Mukhin algorithm is inspired by classical Lie algebra representation theory, where characters are constructed using the Weyl group action by subtracting simple roots from the weights.\\
The two main differences between the Frenkel-Mukhin algorithm and the one we propose are the following: 
\begin{itemize}
    \item First, at each step, we do not define an $i$-expansion as Frenkel and Mukhin do in \cite{FM1}. Instead, we expand each \textit{admissible} variable $Y_{i}(za)$ (defined below) one by one. 
    \item The second and main difference is the definition of the coefficients associated with the new monomials created at each step of the algorithm. 
    Indeed, in \cite{FM1}, the coefficients are defined as maxima, while here we define them explicitly as a quotient of residues of rational functions using formula \eqref{coefalg}.
\end{itemize}

Moreover, similar algorithms have been used in the literature to compute $qq$-characters, but they apply to different context or are formulated differently. 
These differences are motivated by our proof of the well-definedness of the algorithm and the fact that the resulting field lies in $\WWqt$. 

In Section 2.3 of \cite{fjm1}, Feigin,
Jimbo, and Mukhin define an algorithm to compute fields commuting with a set of fermionic screening currents, in the case of a Lie superalgebra.

Our algorithm is almost the same as the one defined by Kimura in Section 2.4 of \cite{kimurahiggsing}.
The difference is that we only treat the case of generic regular monomials: we stop the algorithm as soon as a non-generic or non-regular monomial is produced, while Kimura's algorithm continues in that case. Indeed, in that case we do not know how to prove that the algorithm is well defined, nor that the resulting field lies in $\WWqt$. Restricting to generic regular monomials allows us to prove rigorously that the algorithm is well defined. We use this to prove the existence of new fields in $\WWqt$, and hence Conjecture~\ref{conj1} in new cases. Moreover, the algorithm gives explicitly the set of monomials of the field. This is a significant advantage: if we prove that a given set of monomials satisfies the combinatorics of this algorithm, then this set is exactly the set of monomials of a field in $\WWqt$.

From now on, to simplify notations, we will drop the normal ordering symbol $: \:$. For all monomials $m, m'\in\mathbf{M}$, the notation $mm'$ will always denote the normal ordered product $:mm':$.

\subsection{Description of the algorithm}

\begin{defi}
    We say that a monomial $M\in\mathbf{M}$ is \textit{dominant} if for all $i\in I$, $a \in q^\Z t^\Z$, $d_{i,a}(M)\geq 0$.
\end{defi}

\begin{rem}
    It is well-defined since the variables $Y_{i}(za)$ are algebraically independent in $\Hqt(\g)$ with respect to the normal ordered product.
\end{rem}

\begin{rem}
    Roughly speaking, to be regular means not to contain $Y_i(za)\allowbreak Y_i(zaq^{-2r_i})^{-1}$, $Y_i(za)\allowbreak Y_i(zat^2)^{-1}$ nor $Y_i(za)\allowbreak Y_i(zaq^{-2r_i}t^{2})^{-1}$, and to contain $Y_i(za)\allowbreak Y_i(zaq^{-2r_i}t^2)$ only if it also contains $Y_i(zaq^{-2r_i})\allowbreak Y_i(zat^2)$.
\end{rem}

Now let us describe an algorithm which, given a generic dominant regular monomial $m$ (see Definition~\ref{defgenreg}), if it does not \textit{fail} (see the meaning below) and terminates, gives an element $T(m)\in \WWqt$ such that $m$ is the unique dominant monomial appearing in the expression of $T(m)$:\\
We want to obtain a list of monomials with their associated coefficients.
\begin{itemize}
    \item \textbf{\textit{\underline{Step 0:}}} The first monomial is $m$ with coefficient $1$.
    \item \textbf{\textit{\underline{At each step:}}} At the previous step, we obtained a list of monomials with their coefficients:
    $\lambda_1m_1,\lambda_2m_2,\dots,\lambda_pm_p$. \\
    If any $m_i$ is not regular or not generic then the algorithm \textit{fails} and stops. 
    Else, for each $1\leq i\leq p$, for each admissible $Y_{j}(za)$ in $m_i$, if it was not already created, we create the monomial $m'_{i,(j,a)}=:m_iA_{j}(zaq^{-r_j}t)^{-1}:$. \\
    We shall now compute the coefficient associated to $m'_{i,(j,a)}$. We set: 
    $$ m_i=:Y_{j}(za)Y_{j}(zb_1)\dots Y_{j}(zb_k)Y_j(zc_1)^{-1}\dots Y_j(zc_N)^{-1}n_i:$$
    and
    $$ m'_{i,(j,a)}=:Y_{j}(zaq^{-2r_j}t^{2})^{-1}Y_{j}(zb_1)\dots Y_{j}(zb_k)Y_j(zc_1)^{-1}\dots Y_j(zc_N)^{-1}n'_{i,(j,a)}:.$$
    where $n_i$ and $n'_{i,(j,a)}$ are monomials in the variables $Y_u(zd)$ for $u\neq j$ and $d\in q^\Z t^\Z$.\\
    The coefficient $\lambda_{m'_{i,(j,a)}}$ associated to $m'_{i,(j,a)}$ is defined as in \eqref{coeff}:
    \begin{equation}\label{coefalg}
        \boxed{\lambda_{m'_{i,(j,a)}}=\lambda_{m_i}\prod_{j=1}^k\frac{(1-q^{-2r_i}ab_j^{-1})(1-t^2ab_j^{-1})}{(1-ab_j^{-1})(1-q^{-2r_i}t^2ab_j^{-1})}\prod_{u=1}^N\frac{(1-ac_u^{-1})(1-q^{-2r_i}t^2ac_u^{-1})}{(1-q^{-2r_i}ac_u^{-1})(1-t^2ac_u^{-1})}.}
    \end{equation}
\end{itemize}
We repeat the steps until no new monomial can be created. 
We denote by $T(m)$ the sum of all the monomials created with their associated coefficients. We have seen that this coefficient lies in $K^*$ (the admissibility, regularity and genericity together prove that the spectral parameters satisfy the conditions in Proposition \ref{coeffwelldefprop}).\\

\begin{defi}
Let $m\in\mathbf{M}$ be a dominant generic regular monomial. 
We define $\mathbf{A}(m)$ to be the set of monomials created by the algorithm (including the initial dominant generic regular monomial $m$). 

For all $M\in\mathbf{A}(m)$, $i\in I$, and $a \in q^\Z t^\Z$, we say that the transformation $A_{i}(za)^{-1}$ is \textit{admissible} in $M$ if $Y_i(zaq^{r_i}t^{-1})$ is admissible in $M$.

Let $m$ be a dominant generic regular monomial. We say that $m'\in\mathbf{M}$ \textit{appears} in the algorithm if $$m'\in\mathbf{A}(m).$$
\end{defi}

\begin{example}\label{example1}
     For $\g=\sld_2$, we can construct a field $T(z)\in\WWqt$ starting from the dominant monomial $m=Y(z)Y(zq^{-2})Y(zt^2)$:
     {\small
    $$T(z)=:Y(z)Y(zq^{-2})Y(zt^2):+\lambda_1:Y(z)Y(zq^{-2})Y(zq^{-2}t^4)^{-1}:+$$
    $$+\lambda_2:Y(z)Y(zq^{-4}t^2)^{-1}Y(zt^2):+\lambda_3:Y(z)Y(zq^{-4}t^2)^{-1}Y(zq^{-2}t^4)^{-1}:+$$
    $$+:Y(zq^{-2}t^2)^{-1}Y(zq^{-4}t^2)^{-1}Y(zq^{-2}t^4)^{-1}:,$$
    }    
    where $$\lambda_1=\frac{(q^2t^4-1)(q^2-t^2)}{(q^2-t^4)(q^2t^2-1)}; \quad \lambda_2=\frac{(q^4t^2-1)(q^2-t^2)}{(q^4-t^2)(q^2t^2-1)}; \quad\lambda_3=\frac{(q^2+1)(t^2+1)(q^2-t^2)^2}{(q^4-t^2)(q^2-t^4)}.$$
\end{example}
\begin{rem}\label{faild4}
    One may wonder whether the regularity and genericity of the initial monomial $m$ implies the regularity (resp. genericity) of all the monomials produced by the algorithm. It is not the case as shown by the following example in type $D_4$:\\ 
    Let us consider the following indexation of the Dynkin diagram in type $D_4$:\\
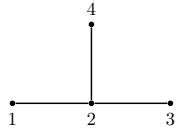
    \begin{center}
    \scalebox{0.7}{
\begin{tikzpicture}[dynkin node/.style={circle, fill=black, inner sep=0pt, minimum size=3pt},connection/.style={thick, draw=black}]

    \node[dynkin node, label=below:2] (2) at (0,0) {};
    
    \node[dynkin node, label=below:1] (1) at (-1.5,0) {};
    \node[dynkin node, label=below:3] (3) at (1.5,0) {};
    \node[dynkin node, label=above:4] (4) at (0,1.5) {};

    \draw[connection] (1) -- (2);
    \draw[connection] (3) -- (2);
    \draw[connection] (4) -- (2);

\end{tikzpicture}}
\end{center}
We consider the initial dominant generic regular monomial: $$m=Y_2(z).$$
\begin{enumerate}[align=left]
    \item[Step 1:] The term $Y_2(z)$ is admissible in $m$. We get the monomial:
        $$m_1=Y_1(zq^{-1}t)Y_2(zq^{-2}t^2)^{-1}Y_3(zq^{-1}t)Y_4(zq^{-1}t).$$
    \item[Step 2:] The term $Y_1(zq^{-1}t)$ is admissible in $m_1$. We get the monomial:
        $$m_2=Y_1(zq^{-3}t^3)^{-1}Y_3(zq^{-1}t)Y_4(zq^{-1}t).$$
    \item[Step 3:] The term $Y_3(zq^{-1}t)$ is admissible in $m_2$. We get the monomial:
        $$m_3=Y_1(zq^{-3}t^3)^{-1}Y_2(zq^{-2}t^2)Y_3(zq^{-3}t^3)^{-1}Y_4(zq^{-1}t).$$
    \item[Step 4.a:] The term $Y_2(zq^{-2}t^2)$ is admissible in $m_3$. We get the monomial:
        $$m_4=Y_2(zq^{-4}t^4)^{-1}Y_4(zq^{-1}t)Y_4(zq^{-3}t^3).$$
        $m_4$ is generic but not regular as it contains $Y_4(zq^{-1}t)Y_4(zq^{-3}t^3)$.
    \item[Step 4.b:] The term $Y_4(zq^{-1}t)$ is admissible in $m_3$. We get the monomial:
        $$m_5=Y_1(zq^{-3}t^3)^{-1}Y_2(zq^{-2}t^2)^{2}Y_3(zq^{-3}t^3)^{-1}Y_4(zq^{-3}t^3)^{-1}.$$
        $m_5$ is regular but not generic as it contains $Y_2(zq^{-2}t^2)^{2}$.
\end{enumerate}
\end{rem}

\subsection{Graph representation of the algorithm}

We can represent the algorithm as a graph where the vertices are the monomials appearing in the algorithm and the edges represent the transformations of the algorithm. 
This definition is the analogue of the graph defined by Feigin, Jimbo and Mukhin for $qq$-characters in \cite{fjm1} which is inspired by the one introduced by Frenkel and Reshetikhin in \cite{frenkel2008qcharactersrepresentationsquantumaffine}. 
However, in this framework, for clarity we omit the coefficients in the graph. But they exist and we can explicitly compute them step by step using the formula \ref{coefalg}.
\begin{defi}
    Let $m$ be a generic dominant regular monomial. We assume the algorithm starting from the monomial $m$ never fails. 
    We define the oriented colored graph $G(m)$ associated to the algorithm starting from $m$ as follows:
    \begin{itemize}
        \item The set of vertices of $G(m)$ is the set of monomials $\mathbf{A}(m)$ appearing in the algorithm starting from $m$.
        \item There is an oriented edge of color $A_i(zaq^{-r_i}t)^{-1}$ from a vertex $m_1$ to a vertex $m_2$ if and only if $Y_i(za)$ is admissible in $m_1$ and $m_2~=~m_1A_i(zaq^{-r_i}t)^{-1}$.
    \end{itemize}
\end{defi}

\begin{rem}
    We will draw all the paths such that the edges are directed from top to bottom so that the upper monomial in the graph $G(m)$ is $m$. 
\end{rem}

\begin{example}\label{examp1}
    \begin{itemize}
        \item For $\g=\sld_2$, for $k\in\Z\backslash\{0,2\}$, and $m=Y_1(z)Y_1(zq^{-2}t^k)$ we obtain the following graph:
{
    \begin{center}   \scalebox{0.7}{   
\begin{tikzpicture}[
    >=Stealth,
    every node/.style={align=center, text width=4cm},
    every edge/.style={->, thick}
]
\node (N1) at (0, 0) {\shortstack{$Y_{1}(zq^{-2}t^{k}) Y_{1}(z)$}};
\node (N2) at (-3, -3) { {$Y_{1}(zq^{-2}t^{k}) Y_{1}(zq^{-2}t^{2})^{-1}$}};
\node (N3) at (3, -3) {\shortstack{$Y_{1}(zq^{-4}t^{k+2})^{-1} Y_{1}(z)$}};
\node (N4) at (0, -6) { \shortstack{$Y_{1}(zq^{-4}t^{k+2})^{-1} Y_{1}(zq^{-2}t^{2})^{-1}$}};
\draw[->] (N1) -- (N2) node[midway, left] {$ A_1(zq^{-1}t)^{-1} $};
\draw[->] (N1) -- (N3) node[midway, right] {$ A_1(zq^{-3}t^{k+1})^{-1} $};
\draw[->] (N2) -- (N4) node[midway, left] {$ A_1(zq^{-3}t^{k+1})^{-1} $};
\draw[->] (N3) -- (N4) node[midway, right] {$ A_1(zq^{-1}t)^{-1} $};
\end{tikzpicture}}
\end{center}}
\item  For $\g=\sld_5$ and $m=Y_1(z)$:\\
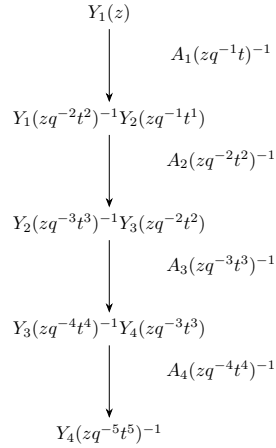
\begin{center}\scalebox{0.7}{
\begin{tikzpicture}[
    >=Stealth,
    every node/.style={align=center, text width=4cm},
    every edge/.style={-{Stealth[scale=1.2]}, thick}
]
\node (N1) at (0, 0) { \shortstack{$Y_{1}(z)$}};
\node (N2) at (0, -2) { \shortstack{$Y_{1}(zq^{-2}t^{2})^{-1}Y_{2}(zq^{-1}t)$}};
\node (N3) at (0, -4) { \shortstack{$Y_{2}(zq^{-3}t^{3})^{-1} Y_{3}(zq^{-2}t^{2})$}};
\node (N4) at (0, -6) { \shortstack{$Y_{3}(zq^{-4}t^{4})^{-1} Y_{4}(zq^{-3}t^{3})$}};
\node (N5) at (0, -8) { \shortstack{$Y_{4}(zq^{-5}t^{5})^{-1}$}};
\draw[->] (N1) -- (N2) node[midway, right, yshift=0.2cm] {$ A_{1}(zq^{-1}t)^{-1} $};
\draw[->] (N2) -- (N3) node[midway, right, yshift=0.2cm] {$ A_{2}(zq^{-2}t^2)^{-1} $};
\draw[->] (N3) -- (N4) node[midway, right, yshift=0.2cm] {$ A_{3}(zq^{-3}t^3)^{-1} $};
\draw[->] (N4) -- (N5) node[midway, right, yshift=0.2cm] {$ A_{4}(zq^{-4}t^4)^{-1} $};
\end{tikzpicture}}
\end{center}
    \end{itemize}
\end{example}

\subsection{Well-definedness of the algorithm} 
It is not obvious that the algorithm described above is well-defined as the coefficient $\lambda_{i,k}$ defined in \eqref{coefalg} 
seems to depend on the path taken to create the monomial $m'_{i,k}$. 
For example, in the first example in Example \ref{examp1}, the monomial $Y_{1}(zq^{-4}t^{k+2})^{-1} \allowbreak Y_{1}(zq^{-2}t^{2})^{-1}$ can be obtained in two different ways. We need to prove that the coefficients obtained are the same. \\
To prove this, we proceed in two steps. 
Firstly, we prove in Corollary \ref{cycle} that if a monomial is reachable via two distinct transformations in directions $i$ and $j$ starting from $M$ and $N$ respectively, then $M$ and $N$ share a common ancestor. 
This ancestor is guaranteed to exist within the subgraph restricted to the arrows coloured by $A_u(zc)^{-1}$ for $c \in q^\Z t^\Z$ and $u\in\{i,j\}$. 
In order to prove this, we have to prove the technical Lemma \ref{algorithmreverse} which formulates a sufficient (and necessary) condition for a monomial to come from a given transformation. 
Finally, we prove that the coefficient obtained by traversing the cycle along the right-hand path is identical to the one obtained along the left-hand path. This proof is a basic computation in a rank $2$ Lie algebra.

After that, we need to prove that if the algorithm does not fail and ends in finitely many steps then the element constructed lies in $\WWqt$.\\

Let us begin with some definitions and the technical lemma.

\begin{defi} 
    \begin{enumerate}
        \item We say that the algorithm \textit{works} if it never fails and ends in finitely many steps.
        \item Let $m$ be a generic dominant regular monomial such that the algorithm starting from $m$ works. 
        For all $X\in \mathbf{A}(m)$, there exist $h\geq0$, $i_1,\ldots i_h \in I$ and $a_1,\ldots a_h \in q^\Z t^\Z$ such that $X = mA_{i_1}(za_1)^{-1}\cdots A_{i_h}(za_h)^{-1}$. We define the \textit{height} of $X$ to be $h$.
        It is well-defined since the variables $A_i(za)$ are algebraically independent in $\Hqt(\g)$ with respect to the normal ordered product.
    \end{enumerate}
\end{defi}

\begin{lem}\label{algorithmreverse}
    Let $m\in\mathbf{M}$ be a dominant regular generic monomial. We assume the algorithm starting from $m$ works. For each monomial $X\in\mathbf{A}(m)$ appearing in the algorithm, for all $i\in I$, for all $a\in q^\Z t^\Z$, if $d_{i,a}(X)=-1$ and $d_{i,aq^{2r_i}}(X),d_{i,at^{-2}}(X)\ne -1$ then there exists a monomial $X'$ also appearing in the algorithm such that $X'A_{i}(zaq^{r_i}t^{-1})^{-1}=X$.
\end{lem}

\begin{rem}
    This lemma is natural and seems not so difficult to prove. 
    The proof is not difficult but very technical and long as we did not find a simpler way to do it.
\end{rem}

\begin{notation}
    Let $m=Y_{i_1}(za_1)^{\varepsilon_1}\ldots Y_{i_k}(za_k)^{\varepsilon_k}$ with $(i_u,a_u)\ne (i_v,a_v)$ for $u\ne v$. We will say that a monomial $X\in\mathbf{M}$ \textit{contains} $m$ (or that $m$ \textit{appears} in $X$) if for all $1\le j\le k$, $d_{i_j,a_j}(X)=\varepsilon_j$.
\end{notation}

\begin{rem}\label{rem:admissible}
Let $M$ be a generic regular monomial, $i\in I$ and $c\in q^{\mathbb{Z}}t^{\mathbb{Z}}$ such that $M$ contains $Y_i(zc)$. By \ref{R1} and \ref{R2}, $M$ contains neither $Y_i(zcq^{-2r_i})^{-1}$ nor $Y_i(zct^{2})^{-1}$. Hence, $Y_i(zc)$ is admissible in $M$ if and only if $M$ contains neither $Y_i(zcq^{-2r_i})$ nor $Y_i(zct^{2})$. We will use this repeatedly in the following proof.
\end{rem}

\begin{proof}
    We prove it by induction on the height of the monomials.\\ 
        \textit{\underline{Height $0$:}} The only monomial with height $0$ is the dominant generic monomial from which we start our algorithm. Hence the property holds at height $0$.\\

         \textit{\underline{Height $h+1$:}} We assume the property is true for all the monomials with heights $h'\leq h\in\N$. 
        Let us prove it for the monomials with height $h+1$. Let $X$ be a monomial with height $h+1$. 
        Let $i\in I$, $a\in q^\Z t^\Z$ such that $d_{i,a}(X)=-1$ and $d_{i,aq^{2r_i}}(X),d_{i,at^{-2}}(X)\ne -1$. 
        The monomial appears in the algorithm and its height is strictly positive, so it has to come from a monomial $X'$ with height $h$. 
        Let $A_j(zbq^{r_j}t^{-1})^{-1}$ be the transformation involved. By definition, $Y_j(zbq^{2r_j}t^{-2})$ is admissible in $X'$, so that $d_{j,bq^{2r_j}t^{-2}}(X')=1$ and $d_{j,bt^{-2}}(X')=d_{j,bq^{2r_j}}(X')=0$. Moreover, $d_{j,b}(X')=0$. Indeed, if $d_{j,b}(X')=1$, then the regularity of $X'$ (\ref{R4} with $c=bq^{2r_j}t^{-2}$) gives $d_{j,bt^{-2}}(X')=1$, which is a contradiction; and if $d_{j,b}(X')=-1$, then $d_{j,b}(X)=d_{j,b}(X')-1=-2$, which contradicts the genericity of $X$. Hence $d_{j,b}(X)=-1$. If $(i,a)=(j,b)$ then we have the result. We assume $(i,a)\neq(j,b)$.\\

        In all this proof we assume the Dynkin subdiagram with nodes $\{i,j\}$ is of type $A_1$, $A_1\times A_1$, or $A_2$. 
        In the other cases the proof has the same structure but the parameters change and the order of transformations differ.\\
        The reader can find a proof in type $B_2$ and $G_2$ in
         Appendix \ref{appendixlem}. We have:
        $$X=X'A_{j}(zbq^{r_j}t^{-1})^{-1}.$$
        Hence,
        $$ \left\{\begin{array}{llll}
            d_{j,bq^{2r_j}t^{-2}}(X')=1\\
            d_{j,b}(X)=-1\\
            d_{j,c}(X)=d_{j,c}(X')\text{ for all $c\not\in\{b,bq^{2r_j}t^{-2}\}$}\\
            d_{k,c}(X)\geq d_{k,c}(X')\text{ for all $k\neq j$, $c\in q^\Z t^\Z$}
        \end{array}\right.$$
        We know that $d_{i,a}(X)=-1$, then $(i,a)\neq (j,bq^{2r_j}t^{-2})$. Moreover, we know that $(i,a)\neq (j,b)$.  
        By the relations listed above, it implies $d_{i,a}(X')\leq-1$. The monomials are all generic, thus 
        $$ d_{i,a}(X')=-1.$$
        Then by the induction assumption, one of the two following assertions is true:\\
        \begin{itemize}
            \item[(a)]  $d_{i,aq^{2r_i}}(X')=-1$  {or}  $d_{i,at^{-2}}(X')= -1$.
            \item[(b)]  There exists a monomial $X''$ appearing in the algorithm such that
        $$X''A_{i}(zaq^{r_i}t^{-1})^{-1}=X'.$$
        \end{itemize}
        \textit{a)} Firstly we assume (a) is true and $X'$ does contain $Y_i(zaq^{2r_i})^{-1}$ (i.e $d_{i,aq^{2r_i}}(X')=-1$). 
        By assumption, $X$ does not. 
        It implies that the transformation $A_j(zbq^{r_j}t^{-1})^{-1}$ simplifies $Y_i(zaq^{2r_i})^{-1}$ and in particular $i\ne j$. 
        Therefore, the expressions depend on the type of the Dynkin diagram generated by the nodes $i$ and $j$. 
        It is clear that this implies that $C_{i,j}\neq0$, and the Lie subalgebra generated by the simple roots $\alpha_i$ and $\alpha_j$ 
        is of type $A_2$, $B_2$ or $G_2$.\\

        Let us do it in type $A_2$. \\
        We obtain $b=aqt$. Let $k$ be the maximal integer such that $Y_i(zaq^{2u})^{-1}$ appears in $X'$ for all $0\le u\le k$. Note that $k\ge1$, as $d_{i,a}(X')=d_{i,aq^{2}}(X')=-1$, and that $Y_i(zaq^{2(k+1)})^{-1}$ does not appear in $X'$.

We claim that $Y_i(zaq^{2v}t^{-2})^{-1}$ does not appear in $X'$ for all $0\le v\le k$. Assume that it appears for some $1\le v\le k$. Then $X'$ contains $Y_i(zaq^{2(v-1)})^{-1}Y_i(zaq^{2v}t^{-2})^{-1}$. The monomial $X'$ has to be regular, so by \ref{R5} (with $c=aq^{2v}t^{-2}$), $X'$ also contains $Y_i(zaq^{2(v-1)}t^{-2})^{-1}$. Iterating this reasoning, we get that $Y_i(zat^{-2})^{-1}$ appears in $X'$. Hence it is enough to prove that $Y_i(zat^{-2})^{-1}$ does not appear in $X'$. If it did, as it does not appear in $X$, it would have to be simplified by the transformation $A_j(zbq^{r_j}t^{-1})^{-1}=A_j(zaq^2)^{-1}$. It is absurd by definition of the fields $(A_j(zc))_{c\in q^{\mathbb{Z}}t^{\mathbb{Z}}}$, as the only variable $Y_i(zc)$ involved in $A_j(zaq^2)^{-1}$ is $Y_i(zaq^2)$. This proves the claim.

Set $X_0:=X'$. We define recursively monomials $X_s$, $1\le s\le k+1$, appearing in the algorithm, such that
$$X_sA_i(zaq^{2k-2s+3}t^{-1})^{-1}=X_{s-1}.$$
Let $1\le s\le k+1$ and assume that $X_0,\dots,X_{s-1}$ have been constructed. We have
$$X_{s-1}=X'A_i(zaq^{2k+1}t^{-1})\cdots A_i(zaq^{2k-2s+5}t^{-1})$$
(empty product if $s=1$), and among the variables $Y_i(zc)$, these factors only involve $Y_i(zaq^{2u})$ and $Y_i(zaq^{2u+2}t^{-2})$ for $k-s+2\le u\le k$. Therefore:
\begin{itemize}
    \item $d_{i,aq^{2(k+1-s)}}(X_{s-1})=d_{i,aq^{2(k+1-s)}}(X')=-1$;
    \item $d_{i,aq^{2(k+2-s)}}(X_{s-1})\neq-1$: for $s=1$ by maximality of $k$, and for $s\ge2$ because $d_{i,aq^{2(k+2-s)}}(X_{s-1})=d_{i,aq^{2(k+2-s)}}(X')+1=0$;
    \item $d_{i,aq^{2(k+1-s)}t^{-2}}(X_{s-1})=d_{i,aq^{2(k+1-s)}t^{-2}}(X')\neq-1$ by the claim above.
\end{itemize}
Since $X_{s-1}$ has height $h-s+1\le h$, we can apply the induction hypothesis to $X_{s-1}$, which gives $X_s$. In particular, we have constructed $X_{k+1}$ such that
$$X_{k+1}A_i(zaqt^{-1})^{-1}=X_k.$$
        Moreover, 
        \begin{equation}\label{eqn1}
            X_{k+1}A_{i}(zaq^{}t^{-1})^{-1}A_{i}(zaq^{3}t^{-1})^{-1}\cdots  A_{i}(zaq^{2k+1}t^{-1})^{-1}=X'.
        \end{equation}
        We recall that $Y_j(zbq^2t^{-2})=Y_{j}(zaq^{3}t^{-1})$ is admissible in $X'$ and $X'A_{j}(zaq^{2})^{-1}=X$. By definition of the algorithm, it implies 
        that $d_{j,aq^3t^{-1}}(X')=1$ and $d_{j,aqt^{-1}}(X')=0$. But the equation $\eqref{eqn1}$ gives
        $$d_{j,aqt^{-1}}(X')=d_{j,aqt^{-1}}(X_{k+1})+1\quad \text{and } \quad d_{j,aq^3t^{-1}}(X')=d_{j,aq^3t^{-1}}(X_{k+1})+1.$$
        Hence, $d_{j,aqt^{-1}}(X_{k+1})=-1$ and $d_{j,aq^3t^{-1}}(X_{k+1})=0$. Thus, $X_{k+1}$ contains $Y_{j}(zaqt^{-1})^{-1}$.
        Let $s\ge0$ be the maximal integer such that $X_{k+1}$ contains the monomial: $$Y_{j}(zaqt^{-1})^{-1}\dots Y_{j}(zaqt^{-2s-1})^{-1}.$$ 
        If for one integer $1\le u\le s$, $X_{k+1}$ contains $Y_{j}(zaq^3t^{-2u-1})^{-1}$, we know that it contains the monomial $Y_{j}(zaqt^{-2u+1})^{-1}$. 
        However, $X_{k+1}$ has to be regular, so $X_{k+1}$ also contains $Y_{j}(zaq^3t^{-2u+1})^{-1}$. 
        Iterating the same argument we would get that $X_{k+1}$ does contain $Y_{j}(zaq^3t^{-1})^{-1}$, which is a contradiction as $d_{j,aq^3t^{-1}}(X_{k+1})=0$. 
        Hence, we can use $s+1$ times the recurrence hypothesis until a monomial $X_{k+s+2}$. We get: 
        $$\forall 0\leq u \leq s, \quad X_{k+u+2}A_j(zaq^2t^{-2(s-u)-2})^{-1}=X_{k+u+1}$$
        At this step, we have constructed a path from $X_{k+ s +2}$ to $X$ with:
        \begin{equation}\label{eqn2}
            \begin{split}
            X_{k+s+2}A_j(zaq^2t^{-2})^{-1}\cdots A_j(zaq^2t^{-2s-2})^{-1} \times\qquad\qquad\qquad \\[5mm]
            \qquad A_i(zaqt^{-1})^{-1}\cdots A_i(zaq^{2k+1}t^{-1})^{-1}A_j(zaq^2)^{-1} = X.
            \end{split}
        \end{equation}
        We get the path in the graph representation of the algorithm shown in Figure~\ref{fig:pathA2}.
            \begin{figure}[htbp]
                \begin{center}
                \begin{tikzpicture}[scale=0.7,
                    >=Stealth,
                    every node/.style={align=center, text width=4cm},
                    every edge/.style={-{Stealth[scale=1.2]}, thick}
                ]
                \node (N1) at (0, 0) { \shortstack{$X_{k+s+2}$}};
                \node (N2) at (0, -1.5) { \shortstack{$\vdots$}};
                \node (N3) at (0, -3) { $X_{k+1}$};
                \node (N5) at (0, -4.5) { \shortstack{$\vdots$}};
                \node (N6) at (0, -6) { \shortstack{$X'$}};
                \node (N7) at (0, -7.5) { \shortstack{$X$}};
                \draw[->] (N1) -- (N2) node[midway, right, yshift=0.2cm] {$ A_{j}(zaq^{2}t^{-2})^{-1}$};
                \draw[->] (N5) -- (N6) node[midway, right, yshift=0.2cm] {$ A_{i}(zaq^{2k+1}t^{-1})^{-1}$};
                \draw[->] (N6) -- (N7) node[midway, right, yshift=0.2cm] {$ A_j(zaq^2)^{-1} $};
                \draw[->] (N2) -- (N3) node[midway, right, yshift=0.2cm] {$ A_{j}(zaq^{2}t^{-2s-2})^{-1} $};
                \draw[->] (N3) -- (N5) node[midway, right, yshift=0.2cm] {$ A_{i}(zaq^{}t^{-1})^{-1} $};
                \end{tikzpicture}
                \end{center} 
                \caption{One path from $X_{k+s+2}$ to $X$} \label{fig:pathA2}
            \end{figure}
        The aim is now to construct another path from $X_{k+s+2}$ to $X$ ending with the transformation $A_i(zaqt^{-1})^{-1}$ in order to conclude.\\
        Finally, 
        \begin{equation}\label{assumption1}
            \text{$X_{k+s+2}$ contains $Y_i(zaq^4t^{-2})\dots Y_i(zaq^{2k+2}t^{-2})Y_j(zaq^3t^{-3})\dots Y_j(zaq^3t^{-2s-3})$.}
        \end{equation} 
        
        To apply the transformation $A_i(zaq^3t^{-1})^{-1}$, we need to check that $Y_i(zaq^{4}t^{-2})$ is admissible in $X_{k+s+2}$, that is, $d_{i,aq^{2}t^{-2}}(X_{k+s+2})=d_{i,aq^{4}}(X_{k+s+2})=0$. By \eqref{assumption1}, $X_{k+s+2}$ contains $Y_i(zaq^{4}t^{-2})$. Since $X_{k+s+2}$ is regular, it contains neither $Y_i(zaq^{2}t^{-2})^{-1}$ nor $Y_i(zaq^{4})^{-1}$.\\

We know that $X_{k+s+2}$ does not contain $Y_i(zaq^{2}t^{-2})$, as it would imply that $X_{k+s+1}$ contains $Y_i(zaq^{2}t^{-2})^2$, which is a contradiction as all monomials are generic by assumption.

Moreover, assume that $X_{k+s+2}$ contains $Y_i(zaq^{4})$. By \eqref{eqn2}, we have
$$X_k=X_{k+s+2}A_j(zaq^2t^{-2})^{-1}\cdots A_j(zaq^2t^{-2s-2})^{-1}A_i(zaqt^{-1})^{-1},$$
and none of these transformations involves $Y_i(zaq^{4})$: the factors $A_j(zaq^2t^{-2u})^{-1}$ only involve $Y_i(zaq^2t^{-2u})$ with $u\geq 1$, and $A_i(zaqt^{-1})^{-1}$ only involves $Y_i(za)^{-1}$ and $Y_i(zaq^2t^{-2})^{-1}$. Hence
$$d_{i,aq^4}(X_k)=d_{i,aq^4}(X_{k+s+2})=1.$$
However, by construction of the monomials $X_s$ (with the convention $X_0:=X'$), we have
$$X_kA_i(zaq^3t^{-1})^{-1}=X_{k-1},$$
so $Y_i(zaq^{4}t^{-2})$ is admissible in $X_k$, which forces $d_{i,aq^4}(X_k)=0$. It is a contradiction.

Thus, $Y_i(zaq^{4}t^{-2})$ is admissible in $X_{k+s+2}$ and the algorithm applies the transformation $A_i(zaq^3t^{-1})^{-1}$ to $X_{k+s+2}$, giving a monomial $Z_{k+s+1}$ such that:
\begin{equation}\label{eqn4}
    X_{k+s+2}A_i(zaq^3t^{-1})^{-1}=Z_{k+s+1}.
\end{equation}
        
                We assume there exists $1\le v\le k-1$ such that the algorithm gives $Z_{u+s+1}$ for all $v< u\le k$ such that:
        $$\forall v< u\le k, \quad Z_{u+s+2}A_i(zaq^{2(k-u+1)+1}t^{-1})^{-1}=Z_{u+s+1},$$
        setting $Z_{k+s+2}=X_{k+s+2}$.
        We want to prove that the algorithm gives $Z_{v+s+1}$ such that:
        $$Z_{v+s+2}A_i(zaq^{2(k-v+1)+1}t^{-1})^{-1}=Z_{v+s+1}.$$
        We know that
        \begin{equation}\label{eqn3}
            Z_{v+s+2}=X_{k+s+2}A_i(zaq^{3}t^{-1})^{-1}\cdots A_i(zaq^{2(k-v)+1}t^{-1})^{-1}.
        \end{equation}
        Then, $Z_{v+s+2}$ contains $Y_i(zaq^{2(k-v+2)}t^{-2})\ldots Y_i(zaq^{2k+2}t^{-2})$.\\
        Let us check that $Y_i(zaq^{2(k-v+2)}t^{-2})$ is admissible in $Z_{v+s+2}$.\\
        It is clear that $Z_{v+s+2}$ does not contain $Y_i(zaq^{2(k-v+1)}t^{-2})$ as it has been simplified by the transformation $A_i(zaq^{2(k-v)+1}t^{-1})^{-1}$.
        Moreover, if $Z_{v+s+2}$ contains $Y_i(zaq^{2(k-v+2)})$ then by \eqref{eqn3}, $X_{k+s+2}$ contains $Y_i(zaq^{2(k-v+2)})$.
        Thus, according to Figure~\ref{fig:pathA2}, $X_v$ contains $Y_i(zaq^{2(k-v+2)})$. 
        However, we read in the same figure that $Y_i(zaq^{2(k-v+2)}t^{-2})$ is admissible in $X_v$. It is a contradiction.
        Hence, $Z_{v+s+2}$ does not contain $Y_i(zaq^{2(k-v+2)})$. Thus, $Y_i(zaq^{2(k-v+2)}t^{-2})$ is admissible in $Z_{v+s+2}$ and the algorithm gives a monomial $Z_{v+s+1}$ such that:
        $$Z_{v+s+2}A_i(zaq^{2(k-v+1)+1}t^{-1})^{-1}=Z_{v+s+1}.$$
        By induction, we get monomials $Z_{s+2},\ldots,Z_{k+s+1}$ verifying the equation \eqref{eqn3} and leading to the following subgraph:\\
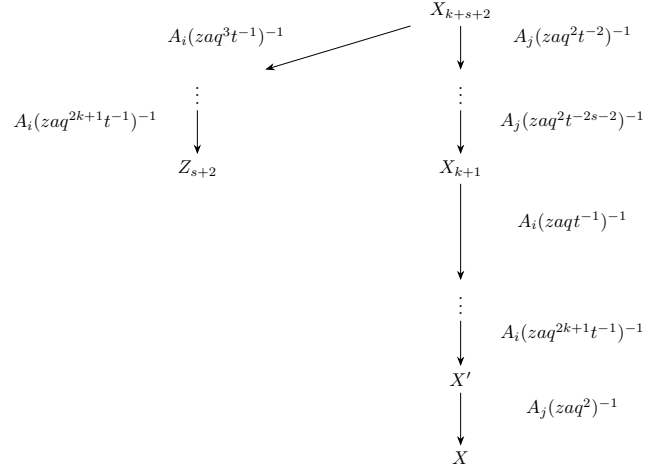
        \begin{center}
\scalebox{0.7}{\begin{tikzpicture}[
    >=Stealth,
    every node/.style={align=center, text width=4cm},
    every edge/.style={-{Stealth[scale=1.2]}, thick}
]
\node (N1) at (0, 0) { \shortstack{$X_{k+s+2}$}};
\node (N2) at (0, -1.5) { \shortstack{$\vdots$}};
\node (N3) at (0, -3) { $X_{k+1}$};
\node (N4) at (-5, -1.5) { \shortstack{$\vdots$}};
\node (N5) at (0, -5.5) { \shortstack{$\vdots$}};
\node (N6) at (0, -7) { \shortstack{$X'$}};
\node (N7) at (0, -8.5) { \shortstack{$X$}};
\node (N9) at (-5, -3) { \shortstack{$Z_{s+2}$}};
\draw[->] (N1) -- (N2) node[midway, right, yshift=0.2cm] {$ A_{j}(zaq^{2}t^{-2})^{-1}$};
\draw[->] (N5) -- (N6) node[midway, right, yshift=0.2cm] {$ A_{i}(zaq^{2k+1}t^{-1})^{-1}$};
\draw[->] (N6) -- (N7) node[midway, right, yshift=0.2cm] {$ A_j(zaq^2)^{-1} $};
\draw[->] (N1) -- (N4) node[midway, left, yshift=0.2cm] {$ A_{i}(zaq^{3}t^{-1})^{-1} $};
\draw[->] (N4) -- (N9) node[midway, left, yshift=0.2cm] {$ A_{i}(zaq^{2k+1}t^{-1})^{-1} $};
\draw[->] (N2) -- (N3) node[midway, right, yshift=0.2cm] {$ A_{j}(zaq^{2}t^{-2s-2})^{-1} $};
\draw[->] (N3) -- (N5) node[midway, right, yshift=0.2cm] {$ A_{i}(zaq^{}t^{-1})^{-1} $};
\end{tikzpicture}}
\end{center} 

       Furthermore, by \eqref{eqn4}, we get:
$$d_{j,aq^3t^{-1}}(Z_{s+2})=d_{j,aq^3t^{-1}}(Z_{k+s+1})=d_{j,aq^3t^{-1}}(X_{k+s+2})+1.$$
But $Y_j(zaq^3t^{-3})$ is admissible in $X_{k+s+2}$, as $X_{k+s+2}A_j(zaq^2t^{-2})^{-1}=X_{k+s+1}$. Hence,
$$d_{j,aq^3t^{-3}}(X_{k+s+2})=1\quad\text{and}\quad d_{j,aq^3t^{-1}}(X_{k+s+2})=0,$$
so that $d_{j,aq^3t^{-1}}(Z_{s+2})=1$. Thus, the monomial $Z_{s+2}$ contains $Y_j(zaq^3t^{-1})\dots Y_j(zaq^3t^{-2s-3})$. We want to check that the term $Y_j(zaq^3t^{-1})$ is admissible in $Z_{s+2}$.\\
If $Z_{s+2}$ contains $Y_j(zaqt^{-1})$, then it contains $Y_j(zaq^3t^{-3})Y_j(zaqt^{-1})$. By regularity, it also contains $Y_j(zaqt^{-3})$. By \eqref{eqn5}, $X_{k+s+2}$ contains $Y_j(zaqt^{-3})$, and $Y_j(zaq^3t^{-3})$ is not admissible in $X_{k+s+2}$. It is a contradiction.\\
If $Z_{s+2}$ contains $Y_j(zaq^3t)$, then by \eqref{eqn5}, $X_{k+s+2}$ contains $Y_j(zaq^3t)$. According to Figure~\ref{fig:pathA2}, this implies that $X'$ contains $Y_j(zaq^3t)$. But $Y_j(zaq^3t^{-1})$ is admissible in $X'$. It is a contradiction.\\
        Hence, $Y_j(zaq^3t^{-1})$ is admissible in $Z_{s+2}$ and the algorithm gives a monomial $Z_{s+1}$ such that:
        \begin{equation}\label{eqn5}
            Z_{s+2}A_j(zaq^2)^{-1}=Z_{s+1}.
        \end{equation}
        We assume there exists $0\le m\le s$ such that we constructed $Z_u$ for all $m<u\le s+1$ such that:
$$\forall m<u\le s+1,\quad Z_{u+1}A_j(zaq^2t^{-2(s+1-u)})^{-1}=Z_u.$$
We want to prove that $Y_j(zaq^3t^{-2(s-m+1)-1})$ is admissible in $Z_{m+1}$, so that the algorithm gives the right monomial $Z_m$.\\
It is clear that $Z_{m+1}$ does not contain $Y_j(zaq^3t^{-2(s-m)-1})$, as it has been simplified by the last transformation $A_j(zaq^2t^{-2(s-m)})^{-1}$.\\
Moreover, if $Z_{m+1}$ contains $Y_j(zaqt^{-2(s-m+1)-1})$, then by construction, $X_{k+s+2}$ contains the term $Y_j(zaqt^{-2(s-m+1)-1})$. According to Figure~\ref{fig:pathA2}, this implies that the monomial $X_{k+m+2}$ contains the term $Y_j(zaqt^{-2(s-m+1)-1})$. However, we read in the same figure that $Y_j(zaq^3t^{-2(s-m+1)-1})$ is admissible in $X_{k+m+2}$. It is a contradiction.\\
        Hence, $Y_j(zaq^3t^{-2(s-m+1)-1})$ is admissible in $Z_{m+1}$ and the algorithm gives a monomial $Z_m$ such that:
        $$Z_{m+1}A_j(zaq^2t^{-2(s-m+1)})^{-1}=Z_{m}.$$
        Now, we have to prove that $Y_i(zaq^2t^{-2})$ is admissible in $Z_0$ to prove the result.\\

        We have rigorously shown the following equality:
        $$ Z_0=X_{k+1}A_i(zaq^3t^{-1})^{-1}\ldots A_i(zaq^{2k+1}t^{-1})^{-1}A_j(zaq^2)^{-1}.$$
        Moreover, $Y_i(zaq^2t^{-2})$ is admissible in $X_{k+1}$. Hence, 
        $$d_{i,aq^2t^{-2}}(X_{k+1})=1;\quad d_{i,aq^2}(X_{k+1})=0; \quad d_{i,at^{-2}}(X_{k+1})=0.$$
        Hence, we get
        $$d_{i,aq^2t^{-2}}(Z_0)=1;\quad d_{i,aq^2}(Z_0)=0; \quad d_{i,at^{-2}}(Z_0)=0.$$
        Thus, $Y_i(zaq^2t^{-2})$ is admissible in $Z_0$ and the algorithm gives the transformation:
        $$Z_0A_i(zaqt^{-1})^{-1}=X.$$
        Hence the result leading to the following graph:
               \begin{center}
\scalebox{0.7}
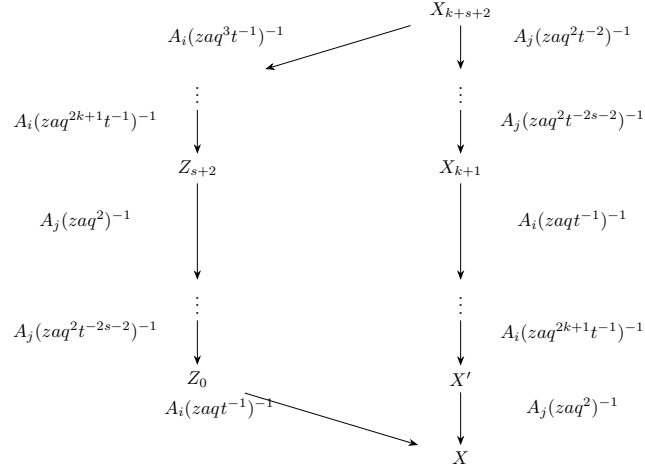
{\begin{tikzpicture}[
    >=Stealth,
    every node/.style={align=center, text width=4cm},
    every edge/.style={-{Stealth[scale=1.2]}, thick}
]
\node (N1) at (0, 0) { \shortstack{$X_{k+s+2}$}};
\node (N2) at (0, -1.5) { \shortstack{$\vdots$}};
\node (N3) at (0, -3) { $X_{k+1}$};
\node (N4) at (-5, -1.5) { \shortstack{$\vdots$}};
\node (N5) at (0, -5.5) { \shortstack{$\vdots$}};
\node (N6) at (0, -7) { \shortstack{$X'$}};
\node (N7) at (0, -8.5) { \shortstack{$X$}};
\node (N8) at (-5, -7) { \shortstack{$Z_0$}};
\node (N9) at (-5, -3) { \shortstack{$Z_{s+2}$}};
\node (N10) at (-5, -5.5) { \shortstack{$\vdots$}};
\draw[->] (N1) -- (N2) node[midway, right, yshift=0.2cm] {$ A_{j}(zaq^{2}t^{-2})^{-1}$};
\draw[->] (N5) -- (N6) node[midway, right, yshift=0.2cm] {$ A_{i}(zaq^{2k+1}t^{-1})^{-1}$};
\draw[->] (N6) -- (N7) node[midway, right, yshift=0.2cm] {$ A_j(zaq^2)^{-1} $};
\draw[->] (N9) -- (N10) node[midway, left, yshift=0.2cm] {$ A_j(zaq^2)^{-1} $};
\draw[->] (N1) -- (N4) node[midway, left, yshift=0.2cm] {$ A_{i}(zaq^{3}t^{-1})^{-1} $};
\draw[->] (N4) -- (N9) node[midway, left, yshift=0.2cm] {$ A_{i}(zaq^{2k+1}t^{-1})^{-1} $};
\draw[->] (N8) -- (N7) node[midway, left, yshift=0.2cm] {$A_{i}(zaq^{}t^{-1})^{-1}$};
\draw[->] (N10) -- (N8) node[midway, left, yshift=0.2cm] {$ A_{j}(zaq^{2}t^{-2s-2})^{-1}$};
\draw[->] (N2) -- (N3) node[midway, right, yshift=0.2cm] {$ A_{j}(zaq^{2}t^{-2s-2})^{-1} $};
\draw[->] (N3) -- (N5) node[midway, right, yshift=0.2cm] {$ A_{i}(zaq^{}t^{-1})^{-1} $};
\end{tikzpicture}}
\end{center}   
In the case $d_{i,at^{-2}}(X')=-1$, we can do a similar reasoning by exchanging the roles of $q^{-1}$ and $t$ in order to get the same result.\\
With a similar reasoning we do it in type $B_2$ and $G_2$ for $(i,j)$ equals $(1,2)$ or $(2,1)$, for $d_{i,aq^{2r_i}}(X')=-1$ and $d_{i,at^{-2}}(X')=-1$ in Appendix \ref{appendixlem}.\\

        \textit{b)} We assume there exists a monomial $X''$ also appearing in the algorithm such that $$X''A_{i}(zaq^{r_i}t^{-1})^{-1}=X',$$
        then $X''$ does contain $Y_i(zaq^{2r_i}t^{-2})$.\\ 

        Here there are two cases to consider:
        \begin{itemize}
            \item We have $d_{j,bq^{2r_j}t^{-2}}(A_{i}(zaq^{r_i}t^{-1})^{-1})=1$ and the term $Y_j(zbq^{2r_j}t^{-2})$ comes from 
            the transformation $A_{i}(zaq^{r_i}t^{-1})^{-1}$.
            \item $X''$ contains $Y_j(zbq^{2r_j}t^{-2})$.
        \end{itemize}

        Firstly, we assume $d_{j,bq^{2r_j}t^{-2}}(A_i(zaq^{r_i}t^{-1})^{-1})=1$ and the term $Y_j(zbq^{2r_j}t^{-2})$ comes from the transformation $A_i(zaq^{r_i}t^{-1})^{-1}$. Hence, $i\neq j$, and we assume the Dynkin subdiagram with nodes $\{i,j\}$ is of type $A_2$. Thus, $bq^2t^{-2}=aqt^{-1}$, and $bqt^{-1}=a$. Hence,
$$d_{i,a}(A_j(zbqt^{-1})^{-1})=1.$$
However, $d_{i,a}(X)=-1$ and $X'A_j(zbqt^{-1})^{-1}=X$. Thus, $d_{i,a}(X')=-2$, which contradicts the genericity. It is absurd.

Hence, we are in the second case and $X''$ contains $Y_j(zbq^{2r_j}t^{-2})$.
Let us prove that $Y_j(zbq^{2r_j}t^{-2})$ is admissible in $X''$.\\
If $i=j$ (so that $r_i=r_j$), then by contradiction we assume that $Y_j(zbq^{2r_j}t^{-2})$ is not admissible in $X''$. We know that $Y_j(zbq^{2r_j}t^{-2})$ is admissible in $X'$, so that the presence of $Y_i(zaq^{2r_i}t^{-2})$ in $X''$ prevents from applying the transformation $A_j(zbq^{r_j}t^{-1})^{-1}$. It implies that $a=bq^{-2r_i}$ or $a=bt^2$. This implies $X$ contains $Y_i(zb)^{-1}=Y_i(zaq^{2r_i})^{-1}$ or $Y_i(zat^{-2})^{-1}$. It is absurd by the initial assumption on $X$. Thus, $Y_j(zbq^{2r_j}t^{-2})$ is admissible in $X''$.\\
If $i\neq j$, then $X''=X'A_i(zaq^{r_i}t^{-1})$, and $A_i(zaq^{r_i}t^{-1})$ only involves non-positive powers of the variables $Y_j(zc)$. Hence $d_{j,bt^{-2}}(X'')\le d_{j,bt^{-2}}(X')=0$ and $d_{j,bq^{2r_j}}(X'')\le d_{j,bq^{2r_j}}(X')=0$, so $X''$ contains neither $Y_j(zbt^{-2})$ nor $Y_j(zbq^{2r_j})$. As $X''$ contains $Y_j(zbq^{2r_j}t^{-2})$, Remark~\ref{rem:admissible} shows that $Y_j(zbq^{2r_j}t^{-2})$ is admissible in $X''$.

Finally, there exists a monomial
$$X_{new}=X''A_j(zbq^{r_j}t^{-1})^{-1}$$
given by the algorithm, and $X_{new}$ contains $Y_i(zaq^{2r_i}t^{-2})$. Indeed, if $i\neq j$, the transformation $A_j(zbq^{r_j}t^{-1})^{-1}$ only involves non-negative powers of the variables $Y_i(zc)$. If $i=j$, note that $b\neq aq^{2r_i}t^{-2}$: otherwise $d_{i,a}(X)=d_{i,aq^{2r_i}t^{-2}}(X)=-1$, and the regularity of $X$ (\ref{R5} with $c=aq^{2r_i}t^{-2}$) would give $d_{i,aq^{2r_i}}(X)=-1$, contradicting the assumption on $X$. Since moreover $b\neq a$, the transformation $A_i(zbq^{r_i}t^{-1})^{-1}$ does not involve $Y_i(zaq^{2r_i}t^{-2})$.\\
Moreover, if $A_i(zaq^{r_i}t^{-1})^{-1}$ is not admissible from $X_{new}$, then there exists a factor $Y_i(zat^{-2})$ (resp. $Y_i(zaq^{2r_i})$) blocking this transformation. But this would imply that this factor also appears in $X$, as we have the following equality of monomials:
$$X=X_{new}A_i(zaq^{r_i}t^{-1})^{-1}.$$
We remark that a priori, this equality is not sufficient to get an arrow from $X_{new}$ to $X$. This equality implies:
$$d_{i,aq^{2r_i}}(X_{new})=d_{i,aq^{2r_i}}(X),$$
and the same for $at^{-2}$. Hence $X$ contains $Y_i(zc)$ and $Y_i(zct^2)^{-1}$ with $c=at^{-2}$ (resp. $Y_i(zc)$ and $Y_i(zcq^{-2r_i})^{-1}$ with $c=aq^{2r_i}$), which contradicts the regularity of $X$. It is impossible. Hence $Y_i(zaq^{2r_i}t^{-2})$ is admissible in $X_{new}$.\\
        Finally we get the following paths:
        \begin{center}
\scalebox{0.7}{\begin{tikzpicture}[
    >=Stealth,
    every node/.style={align=center, text width=4cm},
    every edge/.style={-{Stealth[scale=1.2]}, thick}
]
\node (N1) at (0, 0) { \shortstack{$X''$}};
\node (N2) at (0, -3) { \shortstack{$X'$}};
\node (N3) at (0, -6) { \shortstack{$X$}};
\node (N4) at (-3, -3) { \shortstack{$X_{new}$}};
\draw[->] (N1) -- (N2) node[midway, right, yshift=0.2cm] {$ A_{i}(zaq^{r_i}t^{-1})^{-1}$};
\draw[->] (N2) -- (N3) node[midway, right, yshift=0.2cm] {$ A_j(zbq^{r_j}t^{-1})^{-1} $};
\draw[->] (N4) -- (N3) node[midway, left, yshift=0.2cm] {$ A_{i}(zaq^{r_i}t^{-1})^{-1} $};
\draw[->] (N1) -- (N4) node[midway, left, yshift=0.2cm] {$ A_j(zbq^{r_j}t^{-1})^{-1}$};
\end{tikzpicture}}
\end{center}
and hence the result.
\end{proof}

\begin{cor}\label{cycle}
        If a monomial $M$ produced by the algorithm comes from two different transformations: $$M=M_0A_i(za)^{-1}=N_0A_j(zb)^{-1},$$
    then there exists $k\geq0$ and $2k+1$ monomials $N$, $M_1,\ldots,M_k$ and $N_1,\ldots,N_k$ such that:
    $$N\rightarrow M_k\rightarrow \cdots \rightarrow M_1\rightarrow M_0\rightarrow M$$
    and 
    $$N\rightarrow N_k\rightarrow \cdots \rightarrow N_1\rightarrow N_0\rightarrow M$$
    are both paths produced by the algorithm. Moreover, all the transformations $A_u(zc)^{-1}$ involved satisfy $u\in\{i,j\}$.
\end{cor}
\begin{proof}
    This is proved in the proof of Lemma \ref{algorithmreverse}. 
\end{proof}

\begin{theorem} \label{theoremalgowork}
\begin{itemize}
    \item[a)] Each coefficient $\eqref{coefalg}$ constructed in the algorithm lies in $K^\times$ and is independent of the path taken. Thus, the algorithm is well-defined.
    \item[b)] Given a dominant regular generic monomial $m$, if the algorithm works (i.e never fails and ends in finitely many steps) then it gives an element $T(m)\in \WWqt$ which has a unique dominant monomial $m$ with coefficient $1$.
\end{itemize}
\end{theorem}

\begin{proof}
    \textit{a)} There are two things to verify. 
    Firstly, we have to check if the coefficient $\lambda_{i,k}$ defined above is in $K^\times$, and then we have to verify that it is independent of the path chosen.\\

    Let $M_1\rightarrow M_2$ be an arrow in the algorithm graph such that $M_2=M_1A_i(zaq^{-r_i}t)^{-1}$. We assume $$M_1=Y_i(za)\prod_{j=1}^kY_i(zb_j)\prod_{u=1}^lY_i(zc_u)^{-1}M',$$
    with $M'$ a monomial in $\{Y_j(zd)^{\pm1}\}_{j\neq i,d\in q^\Z t^\Z}$. The coefficient $\lambda_{M_2}$ is defined in $\eqref{coefalg}$ as a quotient of residues. 
    In section \ref{calculcoeffappendix}, we get the following explicit formula:
    $$\lambda_{M_2}=\lambda_{M_1}\prod_{j=1}^k\frac{(1-q^{-2r_i}ab_j^{-1})(1-t^2ab_j^{-1})}{(1-ab_j^{-1})(1-q^{-2r_i}t^2ab_j^{-1})}\prod_{u=1}^l\frac{(1-ac_u^{-1})(1-q^{-2r_i}t^2ac_u^{-1})}{(1-q^{-2r_i}ac_u^{-1})(1-t^2ac_u^{-1})},$$
    where $\lambda_{M_1}$ is the coefficient associated to the monomial $M_1$ (see equation \eqref{coeff}).\\
    The Proposition \ref{coeffwelldefprop} proves that by regularity and genericity the coefficient is well-defined.

    Now we have to prove that the coefficient is independent of the choice of the path.  Firstly, we prove this for two different paths, then we conclude that this implies the independence of the choice of the path.
    We will prove it by induction on the height of monomials (i.e on the length of the paths). \\
    For the first monomial the coefficient is well-defined and equals $1$. \\
    Then, we assume all the coefficients above a fixed monomial $M$ are well-defined.\\
    We assume that this monomial $M$ is obtained by a transformation $A_i(zaq^{-r_i}t)^{-1}$ from a monomial $N_0$ and also by a transformation $A_j(zbq^{-r_j}t)^{-1}$ from a monomial $M_0$. 
    Hence according to Corollary \ref{cycle}, there exists $k\in\N$, $b_1,\ldots, b_{k+1}\in q^\Z t^\Z$ and monomials $N$, $M_1,\ldots,M_k$ and $N_1,\ldots,N_k$ such that we have the two paths:
    \begin{center}
        \scalebox{0.7}
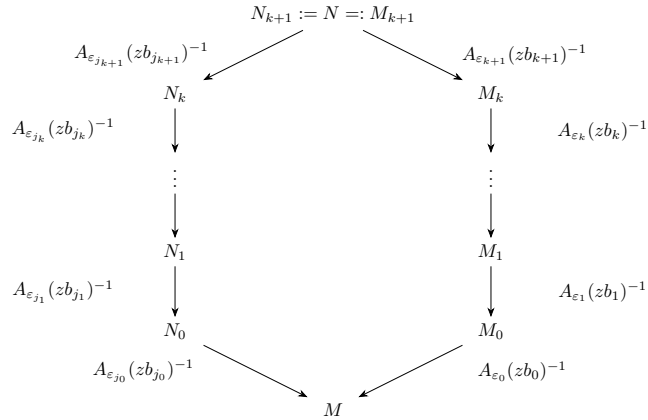
{\begin{tikzpicture}[
    >=Stealth,
    every node/.style={align=center, text width=4cm},
    every edge/.style={->, thick}
]
\node (N1) at (5, 0) {\shortstack{$N_{k+1}:=N=:M_{k+1}$}};
\node (N2) at (2, -1.5) { {$N_k$}};
\node (N3) at (8, -1.5) {\shortstack{$M_k$}};
\node (N4) at (2, -3) { {$\vdots$}};
\node (N5) at (8, -3) {\shortstack{$\vdots$}};
\node (N6) at (2, -4.5) { {$N_1$}};
\node (N7) at (8, -4.5) {\shortstack{$M_1$}};
\node (N8) at (2, -6) { {$N_0$}};
\node (N9) at (8, -6) {\shortstack{$M_0$}};
\node (N10) at (5, -7.5) { \shortstack{$M$}};
\draw[->] (N1) -- (N2) node[midway, left] {$A_{\varepsilon_{j_{k+1}}}(zb_{j_{k+1}})^{-1}$}; 
\draw[->] (N2) -- (N4) node[midway, left] {$A_{\varepsilon_{j_{k}}}(zb_{j_{k}})^{-1}$};
\draw[->] (N4) -- (N6) node[midway, left] {};
\draw[->] (N6) -- (N8) node[midway, left] {$A_{\varepsilon_{j_1}}(zb_{j_1})^{-1}$};
\draw[->] (N1) -- (N3) node[midway, right] {$A_{\varepsilon_{k+1}}(zb_{k+1})^{-1}$};
\draw[->] (N8) -- (N10) node[midway, left] {$A_{\varepsilon_{j_0}}(zb_{j_0})^{-1}$};
\draw[->] (N3) -- (N5) node[midway, right] {$A_{\varepsilon_{k}}(zb_{k})^{-1}$};
\draw[->] (N5) -- (N7) node[midway, right] {};
\draw[->] (N7) -- (N9) node[midway, right] {$A_{\varepsilon_1}(zb_1)^{-1}$};
\draw[->] (N9) -- (N10) node[midway, right] {$A_{\varepsilon_{0}}(zb_{0})^{-1}$};
\end{tikzpicture}}
\end{center}

Because of the algebraic independence of the variables $A_i(za)$, $u\mapsto j_u$ is a permutation of $\{0,\ldots,k+1\}$.
    To simplify the notations we write $a_u=b_u q^{r_{\varepsilon_u}}t^{-1}$. For clarity, we will omit all the $Y_u(zc)^\varepsilon$ with $u\not\in\{i,j\}$.\\
    Let $\varepsilon_r',a_r',\nu_r$ such that $N=\prod_{u=0}^{k+1}Y_{\varepsilon_u}(za_u)\prod_{v=1}^dY_{\varepsilon'_v}(za_v')^{\nu_v}$. 
    Here, $Y_{\varepsilon_u}(za_u)$ is admissible in $M_u$. We can have $a_u=a_v'$ and $\nu_v=-1$, it would mean that $Y_{\varepsilon_u}(za_u)$ does not appear in $N$ but appears after some transformations.\\
    The Corollary \ref{cycle} ensures that for all $u\in\{0,\ldots,k+1\}$, we have $\varepsilon_u\in\{i,j\}$.\\
    Thus, we can consider the Lie subalgebra induced by the nodes $i,j\in I$. This Lie algebra has rank at most $2$, then it can be of type $A_1$, $A_1\times A_1$, $A_2$, $B_2$, or $G_2$.\\
    \underline{\textbf{Type $A_2$:}}
    We have $$M_{k}=N A_{\varepsilon_{k+1}}(zb_{k+1})^{-1}=NY_{\varepsilon_{k+1}}(za_{k+1})^{-1}Y_{\varepsilon_{k+1}}(za_{k+1}q^{-2r_{\varepsilon_{k+1}}}t^2)^{-1}Y_{\overline{\varepsilon_{k+1}}}(za_{k+1}q^{-r_{\varepsilon_{k+1}}}t)$$ 
    with $\overline \varepsilon=i $ if $\varepsilon=j$ and $\overline \varepsilon=j $ if $\varepsilon=i$.\\
    Hence, for $u\in\llbracket-1,k+1\rrbracket$,
    $$M_u=\prod_{r=0}^u Y_{\varepsilon_r}(za_r)\prod_{v=u+1}^{k+1} Y_{\varepsilon_v}(za_vq^{-2r_{\varepsilon_v}}t^2)^{-1}\prod_{v=1}^dY_{\varepsilon'_v}(za_v')^{\nu_v}\prod_{v=u+1}^{k+1}Y_{\overline{\varepsilon_v}}(za_vq^{-r_{\varepsilon_v}}t)$$
    Since $M_u$ is obtained from $M_{u+1}$ by expanding $Y_{\varepsilon_{u+1}}(za_{u+1})$, we get
    {\footnotesize
    \begin{equation*}
        \begin{aligned}
            \lambda_{M_u} &= \lambda_{M_{u+1}}
            \prod_{\substack{v=0\\ \varepsilon_{u+1}=\varepsilon_v}}^{u}\frac{(1-a_{u+1}a_v^{-1}q^{-2r_{\varepsilon_v}})(1-a_{u+1}a_v^{-1}t^{2})}{(1-a_{u+1}a_v^{-1}q^{-2r_{\varepsilon_v}}t^2)(1-a_{u+1}a_v^{-1})}
            \prod_{\substack{v=u+2\\ \varepsilon_{u+1}=\varepsilon_v}}^{k+1}\frac{(1-a_{u+1}a_v^{-1})(1-a_{u+1}a_v^{-1}q^{2r_{\varepsilon_v}}t^{-2})}{(1-a_{u+1}a_v^{-1}t^{-2})(1-a_{u+1}a_v^{-1}q^{2r_{\varepsilon_v}})}\\[2mm]
             \times& \prod_{\substack{v=1\\ \varepsilon_{u+1}=\varepsilon_v'}}^{d}\left(\frac{(1-a_{u+1}a_v'^{-1}q^{-2r_{\varepsilon_v'}})(1-a_{u+1}a_v'^{-1}t^{2})}{(1-a_{u+1}a_v'^{-1}q^{-2r_{\varepsilon_v'}}t^2)(1-a_{u+1}a_v'^{-1})}\right)^{\nu_v}
            \prod_{\substack{v=u+2\\ \varepsilon_{u+1}\neq\varepsilon_v}}^{k+1}\frac{(1-a_{u+1}a_v^{-1}q^{-r_{\varepsilon_{u+1}}}t^{-1})(1-a_{u+1}a_v^{-1}q^{r_{\varepsilon_{u+1}}}t)}{(1-a_{u+1}a_v^{-1}q^{-r_{\varepsilon_{u+1}}}t)(1-a_{u+1}a_v^{-1}q^{r_{\varepsilon_{u+1}}}t^{-1})},
        \end{aligned}
    \end{equation*}}

    This product is well defined by the first point of this proof.\\
    To simplify notations we assume $\lambda_{M_{k+1}}=1$ and we put $\gamma_u$ to be this big product:
    $$\gamma_u=\lambda_{M_u}\lambda_{M_{u+1}}^{-1}.$$ Then by iterating this computation $k+2$ times, following the right path, we get 
    {\footnotesize
    \begin{align*}
        \lambda_M&=\prod_{u=-1}^k\gamma_{u}\\[4mm]
        &=\prod_{u=0}^{k+1}\left(\prod_{\substack{v=0\\ \varepsilon_{u}=\varepsilon_v}}^{u-1}\frac{(1-a_{u}a_v^{-1}q^{-2r_{\varepsilon_{u}}})(1-a_{u}a_v^{-1}t^{2})}{(1-a_{u}a_v^{-1}q^{-2r_{\varepsilon_{u}}}t^2)(1-a_{u}a_v^{-1})}\prod_{\substack{v=u+1\\ \varepsilon_{u}=\varepsilon_v}}^{k+1}\frac{(1-a_{u}a_v^{-1})(1-a_{u}a_v^{-1}q^{2r_{\varepsilon_{u}}}t^{-2})}{(1-a_{u}a_v^{-1}t^{-2})(1-a_{u}a_v^{-1}q^{2r_{\varepsilon_{u}}})}\right)\\[4mm]
        &\times\prod_{u=0}^{k+1}\Bigg(\prod_{\substack{v=1\\ \varepsilon_{u}=\varepsilon_v'}}^{d}\left(\frac{(1-a_{u}a_v'^{-1}q^{-2r_{\varepsilon_{u}}})(1-a_{u}a_v'^{-1}t^{2})}{(1-a_{u}a_v'^{-1}q^{-2r_{\varepsilon_{u}}}t^2)(1-a_{u}a_v'^{-1})}\right)^{\nu_v}\prod_{\substack{v=u+1\\ \varepsilon_{u}\neq\varepsilon_v}}^{k+1}\frac{(1-a_{u}a_v^{-1}q^{-r_{\varepsilon_{u}}}t^{-1})(1-a_{u}a_v^{-1}q^{r_{\varepsilon_{u}}}t^{})}{(1-a_{u}a_v^{-1}q^{-r_{\varepsilon_{u}}}t)(1-a_{u}a_v^{-1}q^{r_{\varepsilon_{u}}}t^{-1})}\Bigg)\\[4mm]
        &=\prod_{u=0}^{k+1}\left(\prod_{\substack{v=0\\ \varepsilon_{u}=\varepsilon_v}}^{u-1}\frac{(1-a_{u}a_v^{-1}q^{-2r_{\varepsilon_{u}}})(1-a_{u}a_v^{-1}t^{2})}{(1-a_{u}a_v^{-1}q^{-2r_{\varepsilon_{u}}}t^2)(1-a_{u}a_v^{-1})}\prod_{\substack{v=u+1\\ \varepsilon_{u}=\varepsilon_v}}^{k+1}\frac{(1-a_{v}a_u^{-1}q^{-2r_{\varepsilon_{u}}}t^2)(1-a_{v}a_u^{-1})}{(1-a_{v}a_u^{-1}q^{-2r_{\varepsilon_{u}}})(1-a_{v}a_u^{-1}t^{2})}\right)\\[4mm]
        &\times\prod_{u=0}^{k+1}\prod_{\substack{v=1\\ \varepsilon_{u}=\varepsilon_v'}}^{d}\left(\frac{(1-a_{u}a_v'^{-1}q^{-2r_{\varepsilon_{u}}})(1-a_{u}a_v'^{-1}t^{2})}{(1-a_{u}a_v'^{-1}q^{-2r_{\varepsilon_{u}}}t^2)(1-a_{u}a_v'^{-1})}\right)^{\nu_v}\prod_{\substack{v=u+1\\ \varepsilon_{u}\neq\varepsilon_v}}^{k+1}\frac{(1-a_{u}a_v^{-1}q^{-r_{\varepsilon_{u}}}t^{-1})(1-a_{u}a_v^{-1}q^{r_{\varepsilon_{u}}}t^{})}{(1-a_{u}a_v^{-1}q^{-r_{\varepsilon_{u}}}t)(1-a_{u}a_v^{-1}q^{r_{\varepsilon_{u}}}t^{-1})}\\[4mm]
        &=\prod_{u=0}^{k+1}\left(\prod_{\substack{v=0\\ \varepsilon_{u}=\varepsilon_v}}^{u-1}K_{u,v}\prod_{\substack{v=u+1\\ \varepsilon_{u}=\varepsilon_v}}^{k+1}{K_{v,u}}^{-1}\right)\left(\prod_{u=0}^{k+1}C_u\right)\left(\prod_{u=0}^{k+1}\prod_{\substack{v=u+1\\ \varepsilon_{u}\neq\varepsilon_v}}^{k+1}K'_{u,v}\right)\\[2mm]
        &=1\times \left(\prod_{u=0}^{k+1}C_u\right)\left(\prod_{\substack{u< v\\ \varepsilon_u\neq\varepsilon_v}}K'_{u,v}\right)
    \end{align*}}
    with for all $u,v\in\llbracket0,k+1\rrbracket$, 
    $$K_{u,v}:=\frac{(1-a_{u}a_v^{-1}q^{-2r_{\varepsilon_{u}}})(1-a_{u}a_v^{-1}t^{2})}{(1-a_{u}a_v^{-1}q^{-2r_{\varepsilon_{u}}}t^2)(1-a_{u}a_v^{-1})},\quad C_u:=\prod_{\substack{v=1\\ \varepsilon_{u}=\varepsilon_v'}}^{d}\left(\frac{(1-a_{u}a_v'^{-1}q^{-2r_{\varepsilon_{u}}})(1-a_{u}a_v'^{-1}t^{2})}{(1-a_{u}a_v'^{-1}q^{-2r_{\varepsilon_{u}}}t^2)(1-a_{u}a_v'^{-1})}\right)^{\nu_v},$$
    $$K_{u,v}':=\frac{(1-a_{u}a_v^{-1}q^{-r_{\varepsilon_{u}}}t^{-1})(1-a_{u}a_v^{-1}q^{r_{\varepsilon_{u}}}t^{})}{(1-a_{u}a_v^{-1}q^{-r_{\varepsilon_{u}}}t)(1-a_{u}a_v^{-1}q^{r_{\varepsilon_{u}}}t^{-1})}=K_{v,u}'.$$

    Similarly coefficient $\mu_M$ computed in the left path is the following:
    $$\mu_M=\left(\prod_{u=0}^{k+1}C_{j_u}\right)\left(\prod_{\substack{u< v\\ \varepsilon_{j_u}\neq\varepsilon_{j_v}}}K_{j_u,j_v}'\right).$$
    
    However, $u\mapsto j_u$ is a bijection. Hence $$\mu_M=\lambda_M$$

\underline{\textbf{Type $A_1$ or $A_1\times A_1$:}} The proof is exactly the same as for $A_2$ but without the $K'_{u,v}$.

\underline{\textbf{Type $B_2$:}} Let $$A_i(z)=Y_i(zq^{-1}t)Y_i(zqt^{-1})Y_j(z)^{-1}$$
$$A_j(z)=Y_j(zq^{-2}t)Y_j(zq^2t^{-1})Y_i(zq^{-1})^{-1}Y_i(zq)^{-1}$$
$$(r_i,r_j)=(1,2)$$
Thus, if $\varepsilon_{k+1}=i$, we get 
$$M_{k}=N A_{\varepsilon_{k+1}}(zb_{k+1})^{-1}=NY_{\varepsilon_{k+1}}(za_{k+1})^{-1}Y_{\varepsilon_{k+1}}(za_{k+1}q^{-2r_{\varepsilon_{k+1}}}t^2)^{-1}Y_{\overline{\varepsilon_{k+1}}}(za_{k+1}q^{-1}t)$$
and if $\varepsilon_{k+1}=j$, we get
\begin{align*}
    M_{k}&=N A_{\varepsilon_{k+1}}(zb_{k+1})^{-1}\\
    &=NY_{\varepsilon_{k+1}}(za_{k+1})^{-1}Y_{\varepsilon_{k+1}}(za_{k+1}q^{-2r_{\varepsilon_{k+1}}}t^2)^{-1}Y_{\overline{\varepsilon_{k+1}}}(za_{k+1}q^{-1}t)Y_{\overline{\varepsilon_{k+1}}}(za_{k+1}q^{-3}t).
\end{align*}
Hence, for $u\in\llbracket-1,k+1\rrbracket$,
{\scriptsize
    \begin{align*}
        M_u&=\prod_{v=0}^u Y_{\varepsilon_v}(za_v)\prod_{v=u+1}^{k+1} Y_{\varepsilon_v}(za_vq^{-2r_{\varepsilon_v}}t^2)^{-1}\prod_{v=1}^dY_{\varepsilon'_v}(za_v')^{\nu_v}\prod_{\substack{v=u+1\\ \varepsilon_v=i}}^{k+1}Y_{\overline{\varepsilon_v}}(za_vq^{-1}t)\prod_{\substack{v=u+1\\ \varepsilon_v=j}}^{k+1}Y_{\overline{\varepsilon_v}}(za_vq^{-1}t)Y_{\overline{\varepsilon_v}}(za_vq^{-3}t)
\end{align*}}

    By a straightforward computation we get:
{\scriptsize
    \begin{align*}
    \lambda_{M_u} &= \lambda_{M_{u+1}}\prod_{\substack{v=0\\ \varepsilon_{u+1}=\varepsilon_v}}^{u}\frac{(1-a_{u+1}a_v^{-1}q^{-2r_{\varepsilon_{u+1}}})(1-a_{u+1}a_v^{-1}t^{2})}{(1-a_{u+1}a_v^{-1}q^{-2r_{\varepsilon_{u+1}}}t^2)(1-a_{u+1}a_v^{-1})}
    \prod_{\substack{v=u+2\\ \varepsilon_{u+1}=\varepsilon_v}}^{k+1}\frac{(1-a_{u+1}a_v^{-1})(1-a_{u+1}a_v^{-1}q^{2r_{\varepsilon_{u+1}}}t^{-2})}{(1-a_{u+1}a_v^{-1}t^{-2})(1-a_{u+1}a_v^{-1}q^{2r_{\varepsilon_{u+1}}})}\\
    &\quad \times \prod_{\substack{v=1\\ \varepsilon_{u+1}=\varepsilon_v'}}^{d}\left(\frac{(1-a_{u+1}a_v'^{-1}q^{-2r_{\varepsilon_{u+1}}})(1-a_{u+1}a_v'^{-1}t^{2})}{(1-a_{u+1}a_v'^{-1}q^{-2r_{\varepsilon_{u+1}}}t^2)(1-a_{u+1}a_v'^{-1})}\right)^{\nu_v}
    \prod_{\substack{v=u+2\\ \varepsilon_{u+1}\neq\varepsilon_v\\ \varepsilon_v=i}}^{k+1}\frac{(1-a_{u+1}a_v^{-1}q^{-3}t^{-1})(1-a_{u+1}a_v^{-1}qt)}{(1-a_{u+1}a_v^{-1}q^{-3}t)(1-a_{u+1}a_v^{-1}qt^{-1})}\\
    &\quad \times \prod_{\substack{v=u+2\\ \varepsilon_{u+1}\neq\varepsilon_v\\ \varepsilon_v=j}}^{k+1}\frac{(1-a_{u+1}a_v^{-1}q^{3}t)(1-a_{u+1}a_v^{-1}q^{-1}t^{-1})}{(1-a_{u+1}a_v^{-1}q^{3}t^{-1})(1-a_{u+1}a_v^{-1}q^{-1}t)}
\end{align*}}
Hence, $$\lambda_M=\left(\prod_{\substack{v< u\\ \varepsilon_u=\varepsilon_v}}K_{u,v}\right)\left(\prod_{\substack{u< v\\ \varepsilon_u=\varepsilon_v}}K_{v,u}\right)^{-1}\left(\prod_{u=0}^{k+1}C_u\right)\left(\prod_{\substack{u,v\\ \varepsilon_v=i\\ \varepsilon_u=j}}K''_{u,v}\right)=\left(\prod_{u=0}^{k+1}C_u\right)\left(\prod_{\substack{u,v\\\varepsilon_v=i\\ \varepsilon_u=j}}K''_{u,v}\right)$$
with for all $u,v\in\llbracket0,k+1\rrbracket$, 
    $$K_{u,v}:=\frac{(1-a_{u}a_v^{-1}q^{-2r_{\varepsilon_{u}}})(1-a_{u}a_v^{-1}t^{2})}{(1-a_{u}a_v^{-1}q^{-2r_{\varepsilon_{u}}}t^2)(1-a_{u}a_v^{-1})},\quad C_u:=\prod_{\substack{v=1\\ \varepsilon_{u}=\varepsilon_v'}}^{d}\left(\frac{(1-a_{u}a_v'^{-1}q^{-2r_{\varepsilon_{u}}})(1-a_{u}a_v'^{-1}t^{2})}{(1-a_{u}a_v'^{-1}q^{-2r_{\varepsilon_{u}}}t^2)(1-a_{u}a_v'^{-1})}\right)^{\nu_v},$$
    $$K_{u,v}'':=\frac{(1-a_{u}a_v^{-1}q^{-3}t^{-1})(1-a_{u}a_v^{-1}q^{}t^{})}{(1-a_{u}a_v^{-1}q^{-3}t)(1-a_{u}a_v^{-1}qt^{-1})},$$
    where we used that the factors of the two last products in the formula for $\lambda_{M_u}$ are $K''_{u+1,v}$ and $K''_{v,u+1}$ respectively,
    and similarly
    $$\mu_M=\left(\prod_{u=0}^{k+1}C_{j_u}\right)\left(\prod_{\substack{\varepsilon_{j_v}=i\\ \varepsilon_{j_u}=j}}K''_{j_u,j_v}\right)=\lambda_M$$
    and the result follows in type $B_2$.

\underline{\textbf{Type $G_2$:}} Let $$A_i(z)=Y_i(zq^{-3}t)Y_i(zq^3t^{-1})Y_j(zq^{-2})^{-1}Y_j(z)^{-1}Y_j(zq^{2})^{-1}$$
$$A_j(z)=Y_j(zq^{-1}t)Y_j(zqt^{-1})Y_i(z)^{-1}$$
$$(r_i,r_j)=(3,1)$$
Thus, if $\varepsilon_{k+1}=i$, we get 
\begin{align*}
    M_{k}&=N A_{\varepsilon_{k+1}}(zb_{k+1})^{-1}\\
    &=NY_{\varepsilon_{k+1}}(za_{k+1})^{-1}Y_{\varepsilon_{k+1}}(za_{k+1}q^{-2r_{\varepsilon_{k+1}}}t^2)^{-1}Y_{\overline{\varepsilon_{k+1}}}(za_{k+1}q^{-5}t)\\[2mm]
    &\qquad\qquad\times Y_{\overline{\varepsilon_{k+1}}}(za_{k+1}q^{-3}t)Y_{\overline{\varepsilon_{k+1}}}(za_{k+1}q^{-1}t)
\end{align*}

and if $\varepsilon_{k+1}=j$, we get
$$M_{k}=N A_{\varepsilon_{k+1}}(zb_{k+1})^{-1}=NY_{\varepsilon_{k+1}}(za_{k+1})^{-1}Y_{\varepsilon_{k+1}}(za_{k+1}q^{-2r_{\varepsilon_{k+1}}}t^2)^{-1}Y_{\overline{\varepsilon_{k+1}}}(za_{k+1}q^{-1}t)$$
Hence, for $u\in\llbracket-1,k+1\rrbracket$,
    $$M_u=\prod_{v=0}^u Y_{\varepsilon_v}(za_v)\prod_{v=u+1}^{k+1} Y_{\varepsilon_v}(za_vq^{-2r_{\varepsilon_v}}t^2)^{-1}\prod_{v=1}^dY_{\varepsilon'_v}(za_v')^{\nu_v}\prod_{\substack{v=u+1\\ \varepsilon_v=j}}^{k+1}Y_{\overline{\varepsilon_v}}(za_vq^{-1}t)$$
    $$\times\prod_{\substack{v=u+1\\ \varepsilon_v=i}}^{k+1}Y_{\overline{\varepsilon_v}}(za_vq^{-1}t)Y_{\overline{\varepsilon_v}}(za_vq^{-3}t)Y_{\overline{\varepsilon_v}}(za_vq^{-5}t)$$
    Again by a straightforward computation we get:
{\scriptsize
\begin{equation*}
    \begin{aligned}
        \lambda_{M_u} &= \lambda_{M_{u+1}}
        \prod_{\substack{v=0\\ \varepsilon_{u+1}=\varepsilon_v}}^{u}\frac{(1-a_{u+1}a_v^{-1}q^{-2r_{\varepsilon_{u+1}}})(1-a_{u+1}a_v^{-1}t^{2})}{(1-a_{u+1}a_v^{-1}q^{-2r_{\varepsilon_{u+1}}}t^2)(1-a_{u+1}a_v^{-1})}
        \prod_{\substack{v=u+2\\ \varepsilon_{u+1}=\varepsilon_v}}^{k+1}\frac{(1-a_{u+1}a_v^{-1})(1-a_{u+1}a_v^{-1}q^{2r_{\varepsilon_{u+1}}}t^{-2})}{(1-a_{u+1}a_v^{-1}t^{-2})(1-a_{u+1}a_v^{-1}q^{2r_{\varepsilon_{u+1}}})}\\[2mm]
        &\quad\times\prod_{\substack{v=1\\ \varepsilon_{u+1}=\varepsilon_v'}}^{d}\left(\frac{(1-a_{u+1}a_v'^{-1}q^{-2r_{\varepsilon_{u+1}}})(1-a_{u+1}a_v'^{-1}t^{2})}{(1-a_{u+1}a_v'^{-1}q^{-2r_{\varepsilon_{u+1}}}t^2)(1-a_{u+1}a_v'^{-1})}\right)^{\nu_v}
        \prod_{\substack{v=u+2\\ \varepsilon_{u+1}\neq\varepsilon_v\\ \varepsilon_v=i}}^{k+1}\frac{(1-a_{u+1}a_v^{-1}q^{-1}t^{-1})(1-a_{u+1}a_v^{-1}q^{5}t)}{(1-a_{u+1}a_v^{-1}q^{-1}t)(1-a_{u+1}a_v^{-1}q^{5}t^{-1})}\\[2mm]
        &\quad\times\prod_{\substack{v=u+2\\ \varepsilon_{u+1}\neq\varepsilon_v\\ \varepsilon_v=j}}^{k+1}\frac{(1-a_{u+1}a_v^{-1}q^{-5}t^{-1})(1-a_{u+1}a_v^{-1}qt)}{(1-a_{u+1}a_v^{-1}q^{-5}t)(1-a_{u+1}a_v^{-1}qt^{-1})}
    \end{aligned}
\end{equation*}}

Hence, $$\lambda_M=1\times\left(\prod_{u=0}^{k+1}C_u\right)\left(\prod_{\substack{\varepsilon_v=i\\ \varepsilon_u=j}}K'''_{u,v}\right)$$
with for all $u,v\in\llbracket0,k+1\rrbracket$, 
    $$C_u:=\prod_{\substack{v=1\\ \varepsilon_{u}=\varepsilon_v'}}^{d}\left(\frac{(1-a_{u}a_v'^{-1}q^{-2r_{\varepsilon_{u}}})(1-a_{u}a_v'^{-1}t^{2})}{(1-a_{u}a_v'^{-1}q^{-2r_{\varepsilon_{u}}}t^2)(1-a_{u}a_v'^{-1})}\right)^{\nu_v},~ K_{u,v}''':=\frac{(1-a_{u}a_v^{-1}q^{5}t^{})(1-a_{u}a_v^{-1}q^{-1}t^{-1})}{(1-a_{u}a_v^{-1}q^{5}t^{-1})(1-a_{u}a_v^{-1}q^{-1}t)},$$
    where we used that the factors of the two last products in the formula for $\lambda_{M_u}$ are $K'''_{u+1,v}$ and $K'''_{v,u+1}$ respectively,
    and similarly
    $$\mu_M=\left(\prod_{u=0}^{k+1}C_{j_u}\right)\left(\prod_{\substack{\varepsilon_{j_v}=i\\ \varepsilon_{j_u}=j}}K'''_{j_u,j_v}\right)=\lambda_M,$$
    and the result follows in type $G_2$.
Finally, if $n$ different arrows target a monomial $M$ then the coefficient defined from each arrows are pairwise equals, hence the independence of the choice of the path. 
\vspace{5mm}

\textit{b)} Let $\Phi(z)=T(m)$ be a field obtained by the algorithm from a dominant monomial $m$. We want to prove that for all $i\in I$, $\Res_w[\Phi(z),S_i^-(w)]=0$.\\
To prove this, we will prove that all terms of all residues $\Res_w[M,S_i^-(w)]$ for all monomials $M$ appearing in the algorithm cancel each other.\\
In Section \ref{computation}, we computed the commutator $[M,S_i^-(w)]$ for all generic regular monomial $M$. Let us recall the result. 
Let $M$ be a regular generic monomial appearing in the algorithm and let $i\in I$.
Let $$R_M^{(i)}=\{a\in q^\Z t^\Z \mid d_{i,a}(M)=1, \quad d_{i,aq^{-2r_i}}(M)=0\},$$
and $$S_M^{(i)}=\{b\in q^\Z t^\Z \mid  d_{i,b}(M)=-1,\quad d_{i,bq^{2r_i}}(M)=0\}.$$
Then
\begin{equation}\label{commutatorhhh}
[M,S_i^-(w)]=\sum_{a\in R_M^{(i)}}\delta\left(\frac{w}{zaq^{-r_i}}\right)C_{a,M}:M S_i^-(w):+\sum_{b\in S_M^{(i)}}\delta\left(\frac{w}{zbq^{r_i}} \right)D_{b,M}:M S_i^-(w):
\end{equation}
Let $a\in R_M^{(i)}$. By definition, $d_{i,a}(M)=1$ and $d_{i,aq^{-2r_i}}(M)=0$. Let $s\geq0$ such that 
$$ d_{i,a}(M)=d_{i,at^2}(M)=\cdots=d_{i,at^{2s}}(M)=1, \quad d_{i,at^{2s+2}}(M)=0. $$
We remark that we do not assume anything on $d_{i,at^{-2}}(M)$.
Moreover, if for a $u\in\llbracket 1,s \rrbracket$,we have $d_{i,aq^{-2r_i}t^{2u}}(M)=-1$,
then $M$ is not regular as $d_{i,at^{2u}}(M)=1$. Additionally, if for a $u\in\llbracket 1,s \rrbracket$, $d_{i,aq^{-2r_i}t^{2u}}(M)=1$,
then $$d_{i,aq^{-2r_i}t^{2u}}(M)=d_{i,at^{2(u-1)}}(M)=1.$$
By regularity of $M$, $d_{i,aq^{-2r_i}t^{2(u-1)}}(M)=1$. Iterating this argument, we get $d_{i,aq^{-2r_i}}(M)=1$ which is a contradiction. Hence,
$$ d_{i,aq^{-2r_i}}(M)=d_{i,aq^{-2r_i}t^2}(M)=\cdots=d_{i,aq^{-2r_i}t^{2s}}(M)=0.$$
We set 
$$M=M'Y_i(za)Y_i(zat^2)\dots Y_i(zat^{2s})Y_i(za_1)\dots Y_i(za_k)Y_i(zb_1)^{-1}\dots Y_i(zb_n)^{-1},$$
where $M'$ is a monomial in the variables $Y_j(zc)$ for $c\in q^\Z t^\Z$, $j\neq i$.\\
Then, by a straightforward computation, we have
$$ C_{a,M}=q^{2(s+1+k-n)r_i}(1-q^{-2r_i})\prod_{v=1}^s\frac{1-q^{-2r_i}t^{-2v}}{1-t^{-2v}}\prod_{j=1}^k\frac{1-q^{-2r_i}aa_j^{-1}}{1-aa_j^{-1}}
\prod_{u=1}^n\frac{1-ab_u^{-1}}{1-q^{-2r_i}ab_u^{-1}} $$
Moreover, by definition of the algorithm and by construction of $s$, we get the following admissibility and path of transformations:

\begin{center}
        \scalebox{0.7}{
            \begin{tikzpicture}[
    >=Stealth,
    every node/.style={align=center, text width=8cm},
    every edge/.style={->, thick}
]
\node (N1) at (8, 0) {\shortstack{$M$}};
\node (N3) at (8, -2) {\shortstack{$MA_i(zaq^{-r_i}t^{2s+1})^{-1}$}};
\node (N5) at (8, -4) {\shortstack{$\vdots$}};
\node (N7) at (8, -7) {$N:=MA_i(zaq^{-r_i}t^{2s+1})^{-1}\cdots A_i(zaq^{-r_i}t)^{-1}$};
\draw[->] (N1) -- (N3) node[midway, right] {$A_i(zaq^{-r_i}t^{2s+1})^{-1}$};
\draw[->] (N3) -- (N5) node[midway, right] {$A_i(zaq^{-r_i}t^{2s-1})^{-1}$};
\draw[->] (N5) -- (N7) node[midway, right] {$A_i(zaq^{-r_i}t)^{-1}$};
\end{tikzpicture}}
\end{center}

Then, $$N=N'Y_i(zaq^{-2r_i}t^2)^{-1}\dots Y_i(zaq^{-2r_i}t^{2s+2})^{-1}Y_i(za_1)\dots Y_i(za_k)Y_i(zb_1)^{-1}\dots Y_i(zb_n)^{-1}$$
where $N'$ is a monomial in the variables $Y_j(zc)$ for $c\in q^\Z t^\Z$, $j\neq i$.\\
Let us prove that $aq^{-2r_i}t^{2s+2}\in S_N^{(i)}$. Indeed, by definition of the $A_i(zc)$ transformations and by regularity, we have
$$ d_{i,aq^{-2r_i}t^{2s+2}}(N)=d_{i,aq^{-2r_i}t^{2s+2}}(M)-1=0-1=-1 $$
and
$$ d_{i,at^{2s+2}}(N)=d_{i,at^{2s+2}}(M)=0. $$
Then the delta function $\delta\left(\frac{w}{zaq^{-r_i}t^{2s+2}}\right)$ appears in the commutator $[N,S_i^-(w)]$. 
Applying $s+1$ difference relations, we get
{\footnotesize
    \begin{align*}
    &\delta\left(\frac{w}{zaq^{-r_i}t^{2s+2}}\right)D_{b,N}:NS_i^-(w): \\
    &\quad = \delta\left(\frac{wt^{-(2s+2)}}{zaq^{-r_i}}\right) D_{b,N}:MA_i(zaq^{-r_i}t^{2s+1})^{-1}\cdots A_i(zaq^{-r_i}t)^{-1}S_i^-(w): \\[2mm]
    &\quad = \delta\left(\frac{wt^{-(2s+2)}}{zaq^{-r_i}}\right)D_{b,N}:MA_i(wt^{-1})^{-1}\cdots A_i(wt^{-2s-1})^{-1}S_i^-(w): \\[2mm]
    &\quad = q^{2(s+1)r_i}t^{-2(s+1)}\delta\left(\frac{wt^{-(2s+2)}}{zaq^{-r_i}}\right)D_{b,N}:MS_i^-(wt^{-2s-2}):,
\end{align*}}
with $b=aq^{-2r_i}t^{2s+2}$. Again, by a straightforward computation,
$$D_{b,N}=q^{2(k-n-s-1)r_i}(1-q^{2r_i})\prod_{v=0}^{s-1}\frac{1-q^{2r_i}t^{2(s-v)}}{1-t^{2(s-v)}}\prod_{j=1}^k\frac{1-q^{-2r_i}t^{2s+2}aa_j^{-1}}{1-t^{2s+2}aa_j^{-1}}\prod_{u=1}^n\frac{1-t^{2s+2}ab_u^{-1}}{1-q^{-2r_i}t^{2s+2}ab_u^{-1}}$$
To cancel the residue, we have to verify that 
{\footnotesize
    \begin{align*}
    0&=\Res_w\Bigg[\lambda_M \delta\left(\frac{w}{zaq^{-r_i}}\right)C_{a,M}:M S_i^-(w):\\[2mm]
    &\qquad\qquad+\lambda_Nq^{2(s+1)r_i}t^{-2(s+1)}\delta\left(\frac{wt^{-(2s+2)}}{zaq^{-r_i}}\right)D_{b,N}:MS_i^-(wt^{-2s-2}):\Bigg].
\end{align*} }
Thus, the coefficients have to satisfy:
{\footnotesize 
\begin{align*}
    \lambda_N&=-\lambda_Mq^{-2(s+1)r_i}\frac{C_{a,M}}{D_{b,N}}\\[4mm]
    &=-\lambda_Mq^{2(s+1)r_i}\frac{1-q^{-2r_i}}{1-q^{2r_i}}\prod_{v=1}^s\frac{(1-q^{-2r_i}t^{-2v})(1-t^{2v})}{(1-t^{-2v})(1-q^{2r_i}t^{2v})}\times \\[2mm]
    &\times \prod_{j=1}^k\frac{(1-q^{-2r_i}aa_j^{-1})(1-t^{2s+2}aa_j^{-1})}{(1-q^{-2r_i}t^{2s+2}aa_j^{-1})(1-aa_j^{-1})}\prod_{u=1}^n\frac{(1-q^{-2r_i}t^{2s+2}ab_u^{-1})(1-ab_u^{-1})}{(1-q^{-2r_i}ab_u^{-1})(1-t^{2s+2}ab_u^{-1})} \\[4mm]
    &=\lambda_M\prod_{j=1}^k\frac{(1-q^{-2r_i}aa_j^{-1})(1-t^{2s+2}aa_j^{-1})}{(1-q^{-2r_i}t^{2s+2}aa_j^{-1})(1-aa_j^{-1})}\prod_{u=1}^n\frac{(1-q^{-2r_i}t^{2s+2}ab_u^{-1})(1-ab_u^{-1})}{(1-q^{-2r_i}ab_u^{-1})(1-t^{2s+2}ab_u^{-1})}
\end{align*}}
Applying recursively the formula \eqref{coefalg} for each transformation we get:

{\footnotesize
    \begin{align*}
   \lambda_{N}&=\lambda_Mq^{-2(s+1)r_i}\prod_{j=0}^s\frac{(1-t^{2(j+1)})(1-q^{2r_i}t^{2(j-s-1)})}{(1-q^{-2r_i}t^{2(j+1)})(1-t^{2(j-s-1)})}\\[4mm]
   &\qquad\times\prod_{m=0}^s\left[\prod_{j=1}^k\frac{(1-q^{-2r_i}t^{2m}aa_j^{-1})(1-t^{2m+2}aa_j^{-1})}{(1-q^{-2r_i}t^{2m+2}aa_j^{-1})
   (1-t^{2m}aa_j^{-1})}\prod_{u=1}^n\frac{(1-q^{-2r_i}t^{2m+2}ab_u^{-1})(1-t^{2m}ab_u^{-1})}{(1-q^{-2r_i}t^{2m}ab_u^{-1})(1-t^{2m+2}ab_u^{-1})} \right]\\[4mm]
   &=\lambda_M \prod_{j=1}^k\frac{(1-q^{-2r_i}aa_j^{-1})(1-t^{2s+2}aa_j^{-1})}{(1-q^{-2r_i}t^{2s+2}aa_j^{-1})(1-aa_j^{-1})}\prod_{u=1}^n\frac{(1-q^{-2r_i}t^{2s+2}ab_u^{-1})(1-ab_u^{-1})}{(1-q^{-2r_i}ab_u^{-1})(1-t^{2s+2}ab_u^{-1})}.
\end{align*}}
Hence, both delta functions cancel each other.\\
To conclude, we will prove that there is a one to one correspondence between the elements of the two following sets: 
$$ R:=\bigcup_{M\in\mathbf{A}(m)}\left\{(M,i,a)\mid i\in I,~a\in R_M^{(i)}\right\}\overset{1:1}{\longleftrightarrow}\bigcup_{M\in\mathbf{A}(m)}\left\{(M,i,a)\mid i\in I,~a\in S_M^{(i)}\right\}=:S,$$
We recall that $\mathbf{A}(m)$ is the set of monomials appearing in the algorithm, or equivalently the set of vertices of the graph $G(m)$.\\
To lighten the notations, for all $(M,i,a)\in R$, $(N,i,b)\in S$, we denote by $f_{(M,i,a)}(z,w)$ and $g_{(N,i,b)}(z,w)$ the terms of the sum \eqref{commutatorhhh}. We get:
$$[T(m),:S_i^-(w):]=\sum_{(M,i,a)\in R}f_{(M,i,a)}(z,w)+\sum_{(N,i,b)\in S}g_{(N,i,b)}(z,w).$$
Let $(M,i,a)\in R$. There exists $s\geq0$ such that $d_{i,at^{2j}}(M)=1$, and $d_{i,aq^{-2r_i}t^{2j}}(M)=0$ for $0\le j\le s$ and $d_{i,at^{2(s+1)}}(M)=0$.\\ 
Then as above, let $$N=MA_i(zaq^{-r_i}t^{2s+1})^{-1}\cdots A_i(zaq^{-r_i}t)^{-1}.$$ 
We proved that $(N,i,aq^{-2r_i}t^{2s+2})\in S$, and
$$ \Res_w\left[f_{(M,i,a)}(z,w)+g_{(N,i,aq^{-2r_i}t^{2s+2})}(z,w)\right]=0.$$
To this $(M,i,a)$, we associate $(N,i,aq^{-2r_i}t^{2s+2})$.\\
Let $(N,i,b)\in S$. There exists $s\geq0$ such that $d_{i,bt^{-2j}}(N)=-1$ for $0\le j\le s$ and $d_{i,bt^{-2(s+1)}}(N)\geq0$. Again  as before, by regularity we can prove that $d_{i,bq^{2r_i}t^{-2j}}(N)=0$ for $0\le j\le s$.\\
If  $d_{i,bt^{-2(s+1)}}(N)=1$, then by regularity $d_{i,bt^{-2s}}(N)\ne -1$. It is impossible, thus $d_{i,bt^{-2(s+1)}}(N)=0$.\\
By  Lemma \ref{algorithmreverse}, there exists a path of transformations

\begin{center}\qquad\qquad\qquad\qquad
        \scalebox{0.7}{\begin{tikzpicture}[
    >=Stealth,
    every node/.style={align=center, text width=8cm},
    every edge/.style={->, thick}
]
\node (N1) at (8, 1) {\shortstack{$M:=NA_i(zbq^{r_i}t^{-2s-1})\cdots A_i(zbq^{r_i}t^{-1})$}};
\node (N3) at (8, -2) {\shortstack{$\vdots$}};
\node (N5) at (8, -4) {\shortstack{$NA_i(zbq^{r_i}t^{-2s-1})$}};
\node (N7) at (8, -6) {\shortstack{$N$}};
\draw[->] (N1) -- (N3) node[midway, right] {$A_i(zbq^{r_i}t^{-1})^{-1}$};
\draw[->] (N3) -- (N5) node[midway, right] {$A_i(zbq^{r_i}t^{-2s+1})^{-1}$};
\draw[->] (N5) -- (N7) node[midway, right] {$A_i(zbq^{r_i}t^{-2s-1})^{-1}$};
\end{tikzpicture}}
\end{center}

Then, by construction, $(M,i,bq^{2r_i}t^{-2s-2})\in R$. We read in the above graph that $$d_{i,bt^{-2s}}(M)=d_{i,bt^{-2s}}(N)+1=0. $$
Thus, by admissibility and from the above graph we read: $$d_{i,bq^{2r_i}t^{-2s-2}}(M)=d_{i,bq^{2r_i}t^{-2s}}(M)=\cdots=d_{i,bq^{2r_i}t^{-2}}(M)= 1,$$
and $$d_{i,bq^{2r_i}}(M)= 0.$$
Hence, the previous construction associates from the tuple $(M,i,bq^{2r_i}t^{-2s-2})\in R$ the tuple $(N,i,b)\in S$.\\
Finally, 
$$\forall 1\le i\le \ell,\qquad [T(m),:S_i^-:]=0.$$

Then, we also need to prove that $$[T(m),:S_i^+:]=0.$$
All the formulas are exactly the same for $S_i^+$ except that $q^{-r_i}$ and $t$ are exchanged. The definition of the coefficient \eqref{coefalg} in the algorithm is also symmetric in $q^{-r_i}\leftrightarrow t$. 
Hence the result.

\end{proof}

Hence, as seen in Example \ref{example1}, the conditions are not too restrictive and the algorithm constructs explicit fields. 
In the next section, we will apply these results to construct \textit{fundamental} fields of $\WWqt$ and prove Conjecture \ref{conj1} in some cases.

\section{Fundamental fields in $\WWqt$}\label{sect6}
In this section, we give explicit formulas for \textit{fundamental fields} in $\WWqt$, and we prove Conjecture \ref{conj1} in new cases. 
Firstly, we adapt a proof of Frenkel and Reshetikhin \cite{MR1646483} to our context to prove that the limit $t\to1$ of any field without derivative in $\WWqt$ can be identified with a linear combination of $q$-characters. 
Then, we apply the algorithm described in the previous section to construct fundamental fields in $\WWqt$.
There are two possible applications of this algorithm.
Firstly, we can compute explicitly fields in $\WWqt$ by applying the algorithm to some dominant monomials. That is what we do for exceptional types.
Secondly, following the approach of Theorem 6.1 in \cite{fjm1}, we can define intuitively a set of monomials with a unique dominant monomial, and then prove that the set of monomials we introduced is exactly the set of monomials produced by the algorithm starting from this unique dominant monomial. 
That proves that there exists a linear combination of these monomials lying in $\WWqt$.
Moreover, we know that they specialize to linear combination of $q$-characters. 
This allows us to prove Conjecture \ref{conj1} in some cases. That is what we do for classical types.\\

\begin{defi}
    A field $\Phi(z)\in \WWqt$ is called \textit{fundamental} if its expansion has a unique dominant monomial equal to $Y_{i}(z)$ for some $i\in I$.
\end{defi}

We recall the Conjecture \ref{conj1} in \cite{MR1646483} we want to prove in this section.

\begin{conj}\label{conj1}
    For each $i = 1, \ldots, \ell$, there exists a field $T_i(z)$ in $\mathbf{W}_{q,t}(\mathfrak{g})$, such that $T_i(z) = Y_i(z) +$ the sum of elementary terms of the form
\[
    c(q,t): Y_i(z) A_{i_1}(zq^{a_1}t^{b_1})^{-1} \cdots A_{i_k}(zq^{a_k}t^{b_k})^{-1}:
\]
(where $c(q,1)$ is a positive integer independent of $q$). Furthermore, the set of weights of these terms counted with multiplicity $c(q,1)$ is the set of weights of the finite-dimensional irreducible representation $V_{\omega_i}$ of $U_q(\widehat{\mathfrak{g}})$ with highest weight $\omega_i$, 
where the weight of $:\prod_{j}Y_{i_j}(za_j)^{\varepsilon_j}:$
is $\sum_{j} \varepsilon_j \omega_{i_j}.$
\end{conj}

For $i=1$ the result is proved in \cite{MR1646483} for all classical types  (see Section \ref{explicitform}). In type $A_\ell$, for all $(1\leq i\leq \ell)$, the explicit relation was given by H. Awata, H. Kubo, S. Odake and J. Shiraishi in \cite{awata1996quantum} and by B. Feigin and E. Frenkel in \cite{feigin1996quantum}.
For $\g$ of type $B_2,C_2,G_2$, the result was proved by Bouwknegt and Pilch in \cite{MR1633032}.\\
Here, we give explicitly the monomials in the fields but we do not compute explicitly the coefficients. We use the algorithm described above to prove that there exist coefficients
such that the fields lie in $\WWqt$.\\
We do it in types $A_\ell$ ($i\in\llbracket 1, \ell \rrbracket$), 
$B_\ell$ ($i\in\llbracket 1, \ell \rrbracket$), 
$C_\ell$ ($i\in\llbracket 1, \ell \rrbracket$), 
$D_\ell$ ($i\in\{1,\ell-1,\ell\}$), 
$E_6$ ($i\in\{1,5\}$), 
$E_7$ ($i=6$), 
$F_4$ ($i\in\{1,4\}$), 
$G_2$ ($i\in\{1,2\}$). \\
Firstly, we need to prove that the limit $t\to1$ of any field without derivative in $\WWqt$ is a linear combination of $q$-characters. 
\subsection{The limit $t \to 1$: the $q$-characters}\label{sectionlimiteqchar}

in this subsection , we study the limit $t \to 1$ of  $\mathbf{W}_{q,t}(\mathfrak{g})\bigcap \mathbf{H}_{q,t}^0(\mathfrak{g}) $. 
We recall that the parameter $t$ is defined as $t = e^{h\beta}$. 
The limit $t \to 1$ corresponds to $\beta \to 0$.\\

\subsubsection{The quantum affine algebra $U_q(\widehat{\g})$ and the $q$-characters}\label{sectqfa}

In 1985, Drinfeld \cite{Drinfeldx} and Jimbo \cite{jimbo1985q} introduced the quantum affine algebra as a deformation of the
universal enveloping algebra of $\widehat{\g}$. In \cite{Drinfeldx}, the deformation parameter is formal: the algebra is
defined over $\Q\llbracket h\rrbracket$, with $q=e^h$. In the presentation of \cite{jimbo1985q}, with generators $E_i$, $F_i$,
$K_i^{\pm1}$, the relations only involve Laurent polynomials in $q$ with integer coefficients and the inverses of the elements
$q^{r_i}-q^{-r_i}$. Let $\mathbb F$ be an algebraic closure of $\on{Frac}(\Q\llbracket h\rrbracket)=\Q((h))$, and
keep $q=e^h\in\mathbb F$. Since $q^{n}-1=nh+O(h^2)\neq0$ for $n\neq0$, $q$ is not a root of unity; in particular the elements
$q^{r_i}-q^{-r_i}$ are invertible in $\mathbb F$, and the same generators and relations define an algebra $U_q(\widehat{\g})$
over $\mathbb F$.Moreover, the results of \cite{chari1995guide,charipqfarep,frenkel2008qcharactersrepresentationsquantumaffine}, stated over
$\mathbb C$ for $q$ not a root of unity, hold over $\mathbb F$, since their proofs only use that the base field is
algebraically closed of characteristic $0$.

The quantum affine algebra $U_q\hat{\mathfrak{g}}$ has a structure of a Hopf algebra. 
Thus, the category of its finite-dimensional representations $Rep(U_q\hat{\mathfrak{g}})$ is a monoidal category.\\
In \cite{chari1995guide,charipqfarep}, Chari and Pressley study this category of representations. 
They prove that the simple representations are characterized by $I$-tuples of polynomials called \textit{Drinfeld polynomials} with constant coefficient $1$. 
In particular, they study the \textit{fundamental representations} $V_{\omega_i}(a)$, where $i \in I$ and $a \in \kkk$, are the simple representations associated to the $I$-tuple of polynomials $(P_j(u))_{j \in I}$ such that $P_j(u) = 1$ for $j \neq i$ and $P_i(u) = 1 - au$.

To study $Rep(U_q(\widehat{\g}))$, Frenkel and Reshetikhin introduced in \cite{frenkel2008qcharactersrepresentationsquantumaffine}, the $q$-character homomorphism $\chi_q$ which is an injective ring homomorphism from the Grothendieck ring of the category of finite-dimensional representations of $U_q(\widehat{\g})$ to a Laurent polynomial ring in variables $Y_{i,a}$, where $i \in I$ and $a \in \kkk$.
They prove that the image of $\chi_q$ is exactly the polynomial subring generated by the $q$-characters of the fundamental representations $V_{\omega_i}(a)$, where $i \in I$ and $a \in \kkk^*$.

We first recall the definition of the commutative algebra $\mathcal{Y}$ equipped with screening operators where the $q$-characters live. 
Then we set $\beta = 0$ in the fields
without derivatives, and we show that
the result lies in the kernels of the
screening operators.

\subsubsection{The algebra $\mathcal{Y}$ and the screening operators}

Let $\mathcal{Y} = \kkk[Y_{i,a}^{\pm 1}]_{i \in I, a \in \kkk^*}$ be the Laurent polynomial ring in variables $Y_{i,a}$. 
We define $$A_{i,a}=Y_{i,aq^{-r_i}}Y_{i,aq^{r_i}}\prod_{j|\mathcal{I}_{j,i}=1}Y_{j,a}^{-1}\prod_{j|\mathcal{I}_{j,i}=2}Y_{j,aq}^{-1}Y_{j,aq^{-1}}^{-1}\prod_{j|\mathcal{I}_{j,i}=3}Y_{j,aq^2}^{-1}Y_{j,a}^{-1}Y_{j,aq^{-2}}^{-1}\in\mathcal{Y}.$$
We define a structure of screening operators on this algebra following the construction in \cite{MR1646483}.

For each $i \in I$, consider the free $\mathcal{Y}$-module generated by symbols $S_{i,x}$ for $x \in\kkk^*$:
$$ \tilde{\mathcal{Y}}_i = \bigoplus_{x \in \kkk^*} \mathcal{Y} \cdot S_{i,x} $$
Let $\mathcal{Y}_i$ be the quotient of $\tilde{\mathcal{Y}}_i$ by the relations:
\begin{equation}
S_{i, x q^{-2r_i}} = A_{i, x q^{-r_i}} S_{i,x}, \label{eq:limit_relation}
\end{equation}

We define a linear operator $\tilde{S}_i: \mathcal{Y} \to \mathcal{Y}_i$ by the formula on the generators:
$$ \tilde{S}_i(Y_{j,a}) = \delta_{ij} Y_{i,a} S_{i,a} $$
and extended by the Leibniz rule $\tilde{S}_i(ab) =  \tilde{S}_i(a)b + a \tilde{S}_i(b)$. Finally, let $S_i: \mathcal{Y} \to \mathcal{Y}_i$ be the composition of $\tilde{S}_i$ and the projection $\tilde{\mathcal{Y}}_i \to \mathcal{Y}_i$. We call $S_i$ the $i$-th screening operator.
Let $\chi_q$ be the $q$-character homomorphism. Frenkel and Reshetikhin proved the following theorem in \cite{frenkel2008qcharactersrepresentationsquantumaffine}:
\begin{theorem}[Corollary 2, \cite{frenkel2008qcharactersrepresentationsquantumaffine}]
    The $q$-character homomorphism $\chi_q$ gives an isomorphism of rings:
    $$ \chi_q: K_0\left(\text{Rep}(U_q(\hat{\mathfrak{g}}))\right) \xrightarrow{\sim} \Z[T_{i,a}]_{i \in I, a \in \kkk^*}, $$
    where $T_{i,a}$ are the $q$-characters of the fundamental representations of $U_q(\hat{\mathfrak{g}})$
\end{theorem}
Moreover, E. Frenkel and E. Mukhin proved the following theorem:
\begin{theorem}[Theorem 5.1, \cite{FM1}]\label{imageqchar}
    The image of the $q$-character homomorphism is equal to the intersection of the kernels of the screening operators:
    $$ \kkk\otimes_\Z \text{Im}(\chi_q) = \bigcap_{i \in I} \text{Ker}(S_i). $$
\end{theorem}
This gives a generators-description of the intersection of the kernels of the screening operators. The generators are the $q$-characters of the fundamental representations $T_{i,a}$.

\subsubsection{The limit of $\WWqt$}

The algebra $\Hqt(\g)$ is over the ring $K=\Q\llbracket h,\beta\rrbracket $. This allows us to rigorously define the limit as $\beta \to 0$ by quotienting the coefficient ring by the non-trivial ideal generated by $\beta$. 
In contrast, $S_i^-$ involves negative powers of $\beta$, which makes its limit singular in this framework.
This asymmetry, however, does not impede our analysis: since the elements considered in the limit are known to commute with $S_i^+$, this commutativity provides sufficient conditions to show that the limit $\mathbf{W}_{q,1}^0(\g)$ defined belkow is contained in the intersection of the kernels of the screening operators defined just above.\\

Recall from Proposition~\ref{polynomial} that $\mathbf H^0_{q,t}(\g)\simeq K[Y_{i,a}^{\pm1}]_{i\in I,a\in q^\Z t^\Z}$. Let
$\overline{\mathcal Y}:=\Q\llbracket h\rrbracket[\overline Y_{i,c}^{\pm1}]_{i\in I,c\in q^\Z}$ and let
$\mathbf p:\mathbf H^0_{q,t}(\g)\to\overline{\mathcal Y}$ be the ring homomorphism ``$\beta=0$'': it reduces the
coefficients modulo $\beta$ (recall $K/\beta K=\Q\llbracket h\rrbracket$) and sends $Y_i(zq^mt^n)$ to $\overline Y_{i,q^m}$.
We define
$${\mathbf W}_{q,1}^0(\g):=\mathbf p\big(\WWqt\cap\mathbf H^0_{q,t}(\g)\big),$$
and we view $\overline{\mathcal Y}$ as a subring of $\mathcal Y$ via $\overline Y_{i,c}\mapsto Y_{i,c}$, since
$\Q\llbracket h\rrbracket\subset\kkk$ and $q^\Z\subset\kkk^\times$.

\begin{prop}\label{limt1}
The limit ${\mathbf{W}}_{q,1}^0(\mathfrak{g})$ is contained in the intersection of the kernels of the screening
operators $S_i$:
$$ {\mathbf{W}}_{q,1}^0(\mathfrak{g}) \subset \bigcap_{i \in I} \operatorname{Ker}(S_i). $$
\end{prop}

\begin{proof}
We follow the idea of \cite[8.4]{frenkel2008qcharactersrepresentationsquantumaffine}, but we work on $\pi_0$, where the
fields of $\WWqt$ commute with the screening operators. We only use $S_i^+$, since $S_i^-$ involves negative powers of
$\beta$. Fix $i\in I$ and set $\nu:=-\beta r\al_i/r_i$, so that $S_i^+$ maps $\pi_0$ to $\pi_\nu$.

Firstly we define the reduction modulo $\beta$ at the stage of coefficients of the deformed Heisenberg algebra. Identify $\pi_0$ and $\pi_\nu$ with $\mathcal S=K[x_{j,-n}]$ by
$fv_\mu\mapsto f$ (Proposition~\ref{fockrep}), and set $\overline{\mathcal S}:=\mathcal S/\beta\mathcal S
=\Q\llbracket h\rrbracket[x_{j,-n}]$. For $n>0$, $x_j[n]$ acts on $\mathcal S$ by
$\sum_k c_n\B_{jk}(q^n,t^n)\,\partial/\partial x_{k,-n}$ with $c_n\in\beta K$, and the shift operators act by the
identity. Let $M(z)=\prod_rY_{j_r}(za_r)^{m_r}$ be a monomial, with distinct pairs $(j_r,a_r)$. In $:M(z):$ acting on
$\pi_0$, and in $:M(z)S_i^+(w):$ acting on $\pi_0$, the zero modes act on the vector of $\pi_0$ we start from. For
$:M(z)S_i^+(w):$, this is our normal ordering convention: they act before the shift $e^{-Q_i/r_i}$. So they act by $1$.
Hence these two operators reduce modulo $\beta$ to the multiplication by
$$\overline M(z):=C_M\exp\Big(\sum_{n\ge1}\sum_rm_r\,\overline{y}_{j_r}[-n](z\overline a_r)^n\Big)
\quad\text{and}\quad\overline M(z)\,\overline S_i(w),\qquad
\overline S_i(w):=\exp\Big(\sum_{n\ge1}\overline s_i[-n]w^n\Big),$$
respectively, where $C_M:=\prod_rq^{-2r(\rho,\omega_{j_r})m_r}\in\Q\llbracket h\rrbracket^\times$ is the reduction of the
constant $\prod_r\big(t^{2(\rho^\vee,\omega_{j_r})}q^{-2r(\rho,\omega_{j_r})}\big)^{m_r}$. Here a bar denotes the
reduction modulo $\beta$: in particular $\overline a=a_{|t=1}\in q^\Z$, and the elements
$\overline y_j[-n],\overline s_i[-n]\in\overline{\mathcal S}$ are linear in the $x_{k,-n}$. Thus $\overline M(z)$ only depends on $\mathbf p(M)$, and $P\mapsto P(z)$ defines a ring
homomorphism $\overline{\mathcal Y}\to\overline{\mathcal S}\llbracket z\rrbracket$. For $c\in q^\Z$ we set
$\overline S_{i,c}(z):=zc\,\overline S_i(zc)$.\\

Let us now compute the commutator. Let $\Phi(z)=\sum_M\lambda_M:M(z):\in\WWqt\cap\mathbf H^0_{q,t}(\g)$. By
subsection~\ref{opesec}, on $\pi_0$,
$$[:M(z):,:S_i^+(w):]=\big(i_{z,w}-i_{w,z}\big)(F)\,:M(z)S_i^+(w):,\qquad F=\prod_{r:\,j_r=i}f\Big(\frac{za_r}{w}\Big)^{m_r},
\quad f(x)=\frac{1-t^{-1}x}{1-tx}.$$
Since $f(x)-1=(t-t^{-1})\frac{x}{1-tx}$ and $t-t^{-1}\in\beta K$, we have $F-1-\sum_{r:\,j_r=i}m_r\big(f(za_r/w)-1\big)\in\beta^2R$,
where $R$ is the set of rational functions with coefficients in $K$ whose denominators are products of the
$1-tza_r/w$. Moreover $(i_{z,w}-i_{w,z})\big(f(za/w)-1\big)=(t^{-2}-1)\,\delta\big(\frac{w}{zat}\big)$, and
$t^{-2}-1\in-2h\beta+\beta^2K$. Taking $\Res_w$, we get on $\pi_0$
$$0=\Res_w[\Phi(z),S_i^+(w)]=-2h\beta\sum_M\lambda_M\sum_{r:\,j_r=i}m_r\,za_rt\,:M(z)S_i^+(za_rt):\;+\;\beta^2X(z)$$
for a series $X(z)$ of operators $\pi_0\to\pi_\nu$ with coefficients in $K$. Since $\mathcal S$ is free over the domain
$K$, we can divide by $\beta$. Reducing modulo $\beta$ and dividing by $-2h$, we get
\begin{equation}\label{limeq}
\sum_M\overline\lambda_M\sum_{r:\,j_r=i}m_r\,\overline M(z)\,\overline S_{i,\overline a_r}(z)=0
\quad\text{in }\overline{\mathcal S}\llbracket z\rrbracket .
\end{equation}

Now, let $\overline{\mathcal Y}_i$ be the quotient of $\bigoplus_{c\in q^\Z}\overline{\mathcal Y}\cdot S_{i,c}$
by the relations $S_{i,cq^{-2r_i}}=A_{i,cq^{-r_i}}S_{i,c}$, and let $\overline S_i:\overline{\mathcal Y}\to\overline{\mathcal Y}_i$
be the derivation defined by $\overline S_i(\overline Y_{j,c})=\delta_{ij}\overline Y_{j,c}S_{i,c}$. Since the
$A_{i,c}$ are invertible, $\overline{\mathcal Y}_i$ is a free $\overline{\mathcal Y}$-module with basis
$(S_{i,q^k})_{0\le k<2r_i}$. Restricting \eqref{scr1} to $\pi_0$ and reducing modulo $\beta$ gives
$\overline S_{i,cq^{-2r_i}}(z)=A_{i,cq^{-r_i}}(z)\,\overline S_{i,c}(z)$. Hence
$$\psi:\overline{\mathcal Y}_i\to\overline{\mathcal S}\llbracket z\rrbracket,\qquad P\,S_{i,c}\mapsto P(z)\,\overline S_{i,c}(z)$$
is well defined, and \eqref{limeq} reads $\psi\big(\overline S_i(\mathbf p(\Phi))\big)=0$.

Let us prove that $\psi$ is injective. For a monomial $m=\prod_{j,c}\overline Y_{j,c}^{\,d_{j,c}}$ of
$\overline{\mathcal Y}$ and $0\le k<2r_i$, we have
$$\psi(m\,S_{i,q^k})=C_mq^kz\exp\Big(\sum_{n\ge1}\underbrace{\Big(\sum_{j,c}d_{j,c}\,c^n\,\overline y_j[-n]+q^{kn}\,\overline s_i[-n]\Big)}_{=:L_n(m,k)}z^n\Big).$$
The elements $m\,S_{i,q^k}$ form a basis of $\overline{\mathcal Y}_i$ over $\Q\llbracket h\rrbracket$, so it is enough to show
that their images are linearly independent.

The derivation $\partial/\partial x_{j,-n}$ multiplies $\exp\big(\sum_nL_n(m,k)z^n\big)$ by $az^n$, where $a$ is the
coefficient of $x_{j,-n}$ in $L_n(m,k)$. Hence these exponentials are linearly independent as soon as the sequences
$(L_n(m,k))_n$ are pairwise distinct. Indeed, in a relation of minimal length, choose $(j,n)$ such that this coefficient
differs for the first two terms, and apply $\partial/\partial x_{j,-n}-az^n$, with $a$ the coefficient for the first term.
Dividing by $z^n$ gives a shorter relation.

It remains to show that $L_n(m,k)=L_n(m',k')$ for all $n\ge1$ implies $(m,k)=(m',k')$. Let
$m'=\prod_{j,c}\overline Y_{j,c}^{\,d'_{j,c}}$ and $0\le k'<2r_i$, and assume that $L_n(m,k)=L_n(m',k')$ for all $n\ge1$,
that is,
\begin{equation}\label{sameexp}
\sum_{j,c}(d_{j,c}-d'_{j,c})\,c^n\,\overline y_j[-n]=(q^{k'n}-q^{kn})\,\overline s_i[-n].
\end{equation}
By \eqref{sm+}, $\overline a_j[-n]$ is a nonzero multiple of $x_{j,-n}$. Moreover
$\overline a_i[-n]=\sum_jC_{ji}(q^{-n},1)\,\overline y_j[-n]$, and $\det C(q^{-n},1)\equiv\det C\not\equiv0$ modulo $h$.
Hence the $\overline y_j[-n]$, $j\in I$, are linearly independent. Multiply \eqref{sameexp} by $q^{-nr_i}-q^{nr_i}$, so
that its right-hand side becomes $(q^{k'n}-q^{kn})\,\overline a_i[-n]$, and take the coefficient of $\overline y_i[-n]$.
Since $C_{ii}(q^{-n},1)=q^{nr_i}+q^{-nr_i}$, we get:
$$\Big(\sum_{s\in\Z}(d_{i,q^s}-d'_{i,q^s})\,q^{ns}\Big)(q^{-nr_i}-q^{nr_i})=(q^{nr_i}+q^{-nr_i})(q^{nk'}-q^{nk}).$$
This is an identity between Laurent polynomials with coefficients in $\Q$ in $q^n$ which holds for infinitely many values, hence
identically and we can replace $q^n$ by any other variable in $\C$. Replacing $q^n$ by a primitive $2r_i$-th root of unity $\zeta$, the left-hand side vanishes and
$\zeta^{r_i}+\zeta^{-r_i}=2\zeta^{r_i}\neq0$. So $\zeta^{k'}=\zeta^k$, that is, $k=k'$ since $0\le k,k'<2r_i$. Then the
right-hand side of \eqref{sameexp} vanishes, and for every $j\in I$ the coefficient of $\overline y_j[-n]$ gives
$\sum_{s\in\Z}(d_{j,q^s}-d'_{j,q^s})\,q^{sn}=0$ for all $n\ge1$. Hence the Laurent polynomial
$\sum_{s\in\Z}(d_{j,q^s}-d'_{j,q^s})\,u^s$ vanishes at infinitely many points, so it is zero, and $m=m'$.

Thus, $\overline S_i(\mathbf p(\Phi))=0$. Since $\mathcal Y_i$ is free over
$\mathcal Y$ with basis $S_{i,x}$, where $x$ runs over therepresentatives of $\kkk^\times/q^{2r_i\Z}$ containing
$1,q,\dots,q^{2r_i-1}$, the module $\overline{\mathcal Y}_i$ embeds into $\mathcal Y_i$ compatibly with $S_i$. Hence
$\mathbf p(\Phi)\in\bigcap_{i\in I}\operatorname{Ker}(S_i)$.
\end{proof}

\begin{rem}
    It is clear that the classical limit $t\to1$ is not equal to the image of the $q$-character homomorphism over the whole finite dimensional representation ring as we obtain only $q$-characters of representations with spectral parameters in $q^\Z$. Then one may ask what is the image of this embedding. This question is not obvious as it is not clear whether a product of two limit of fields in $\WWqt$ is a limit of a field in $\WWqt$. This will be studied in an upcoming paper.
\end{rem}

\begin{conj}
Let $\mathscr{C}$ be the full subcategory of the category of finite dimensional $U_q(\widehat{\g})$-modules whose objects $M$ satisfy:
$$\text{Every composition factor of $M$ is a simple module with highest $\el$-weight in $\sum_{(i,r)\in I\times\Z}\N\omega_{i,q^r}$. }$$
These are exactly the modules whose $q$-character lies in $\Z[Y_{i,q^r}^{\pm1}]_{i\in I, r\in \Z}$. Then 
$${\mathbf{W}}_{q,1}^0(\g)\simeq K_0(\mathscr{C}).$$
\end{conj}

It is proved in \cite{FM1} (Theorem 5.1) and \cite{FHR} (Theorem 3.1) that the image of the $q$-character homomorphism is actually isomorphic to the Grothendieck ring of the category of finite-dimensional representations of $U_q(\hat{\mathfrak{g}})$. Under this isomorphism, the unique fields with only one dominant monomial equal to $Y_{i,a}$ correspond to the $q$-characters of the fundamental representations with dominant $\ell$-weight $Y_{i,a}$.
In order to know what is the image of the embedding above, we need to study the existence of elements in $\mathbf{W}_{q,t}(\mathfrak{g})$ corresponding to the fundamental $q$-characters. This is what we do in the next sections.

\subsection{Explicit formulas for fundamental fields}\label{explicitform}
To simplify the notations, we introduce the definition of the $k$-th canonical projection $\Pi_k$ for all $1\le k\le \el$ as follows:
Let $\varphi:\mathbf{H}_{q,t}^0(\g)\overset{\sim}{\longrightarrow}K[Y_{i,a}^{\pm1}]_{i\in I,a \in q^\Z t^\Z}$ be the isomorphism defined in Proposition \ref{polynomial}.
Let $$\varpi_k:K[Y_{i,a}^{\pm1}]_{i\in I,a \in q^\Z t^\Z}\longrightarrow K[Y_{k,a}^{\pm1}]_{a \in q^\Z t^\Z}\simeq K[Y_{i,a}^{\pm1}]_{i\in I,a \in q^\Z t^\Z}/(Y_{j,a}-1)_{j\neq k, a \in q^\Z t^\Z}$$
be the canonical projection. We define
$$\Pi_k:\varphi^{-1}\circ \varpi_k\circ\varphi:\mathbf{H}_{q,t}^0(\g)\longrightarrow\mathbf{H}_{q,t}^0(\g).$$

In all this section, we will apply the algorithm described in the previous section to the monomial $Y_i(z)$ for some $i\in I$. Here is the second main theorem of this article:

\begin{theorem}\label{theoremefinal}
    Conjecture \ref{conj1} holds in types $A_\ell$ ( for all $i\in I$), $B_\ell$ ( for all $i\in I$), $C_\ell$ ( for all $i\in I$), $D_\ell$ ( for $i=1,\ell-1,\ell$), $E_6$ ( for $i=1,5$), $E_7$ ( for $i=6$), $F_4$ ( for $i=1,4$) and $G_2$ ( for $i=1,2$). 
\end{theorem}
It is already known in types $A_\el$ and $G_2$. We give a proof in all other cases.
\begin{proof}
    In the next sections, we prove that there exist fundamental fields $T_i(z)\in\WWqt$ whose unique dominant monomial is $Y_i(z)$, appearing with coefficient $1$ for all the cases considered in this theorem.
    Moreover, each of these fields satisfies that their limit $t\to1$ also has a unique dominant monomial equal to $Y_i(z)$, appearing with coefficient $1$ (see the explicit formulas which specialize to the fundamental $q$-characters listed in  \cite[A.3]{MR1646483}). 
    However, Theorem \ref{imageqchar} and Proposition \ref{limt1} imply that the limit $t\to1$ of these fields are $\Q\llbracket h \rrbracket$-linear combination of $q$-characters of finite-dimensional representations of $U_q(\hat{\mathfrak{g}})$.
    Since the $q$-characters of the fundamental representations are the only ones with a unique dominant monomial equal to $Y_i(z)$, we get that the limit $t\to1$ of $T_i(z)$ is the $q$-character of the $i$-th fundamental representation. 
    Consequently, the monomials of $T_i(z)$ specialize to the monomials of $\chi_q(V_{\omega_i}(1))$, so that $T_i(z)$ has the
same weights as $V_{\omega_i}$, and the coefficients of the limit are the dimensions of the $\ell$-weight spaces of
$V_{\omega_i}(1)$, hence positive integers independent of $q$. This proves the theorem.
\end{proof}

\begin{rem}\label{thinh}
    These cases are exactly the cases where the fundamental representations are thin (that is, all their $\el$-weight spaces are of dimension $1$). 
    We think there may be a link between the multiplicity of the monomials computed in the Frenkel-Mukhin algorithm and the fact that our algorithm works or not.
\end{rem}

\begin{notation}
    To simplify the notations in the next sections, for all $i\in I$, we define the \textit{height} of a tuple $\mathbf{j}=(j_1,\dots,j_i)\in \N^i$ as the integer $$ht_i(\mathbf{j}):=\sum_{\al=1}^i (j_\al-\al).$$
\end{notation}

\subsubsection{Type $A_\el$:}
In type $A_\el$ the algorithm gives all the fundamental fields for $1\leq i\leq \el$. Indeed, it is clear that the fundamental fields $T_i(z)$ are the $q$-characters of the $i$-th fundamental representation, with replacing $q$ by $(q^{-1}t)$. The $q$-characters in type $A_n$ are computed for example in \cite{frenkel1996quantum}. We get:
$$\La_i(z):= Y_i(zq^{-i+1}t^{i-1}) Y_{i-1}(zq^{-i}t^i)^{-1}, \quad \quad
i=1,\ldots,\el+1,$$
where $Y_0=Y_{\el+1}=1$.

The fundamental fields of $\mathbf{W}_{q,t}(A_\el)$ are
$$\s_i(z) = \sum_{1\leq j_1 < \ldots < j_i\leq \el+1} \La_{j_1}(zq^{-(i-1)}t^{i-1})
\La_{j_2}(zq^{-(i-3)}t^{i-3}) \ldots  \La_{j_i}(zq^{i-1}t^{-(i-1)}),$$
with $\quad i=1,\ldots,\el+1.$ We have: $\s_{\el+1}(z) = 1$.
\begin{prop}{\cite{awata1996quantum,feigin1996quantum}}
    For all $1\le i\le\el$, the field $T_i(z)$ belongs to $\WWqt$.
\end{prop}

We do the type $C_\el$ before the type $B_\el$ because there are more intricate conditions on coefficients in type $C_\el$ than in type $B_\el$.

\subsubsection{Type $C_\el$:} 
Let 
\[D=
\begin{pmatrix}
1 & 0 &  \cdots  & 0 &0\\
0 & 1 &  \cdots  & 0&0 \\
\vdots & \ddots & \ddots & \ddots &\vdots\\
0 & \cdots & 0 & 1&0 \\
0 & \cdots & 0 & 0&2 
\end{pmatrix}
\]
the diagonal matrix such that $DC$ is symmetric. To define the following fields, we deform the fields of the one-parameter deformation given by Frenkel and Reshetikhin~\cite[A.3]{MR1646483}.

In type $C_\el$, the algorithm works for all the fundamental fields.
Let 
$$\La_i(z):= Y_i(zq^{-i+1}t^{i-1}) Y_{i-1}(zq^{-i}t^i)^{-1}, \quad \quad
i=1,\ldots,\el.$$
$$\La_{\bar i}(z):= Y_{i-1}(zq^{-2\el+i-2}t^{2\el-i}) Y_{i}(zq^{-2\el+i-3}t^{2\el-i+1})^{-1}, \quad \quad
i=1,\ldots,\el.$$

with $\bar i=2\el-i+1$. Then, the fundamental fields of $\mathbf{W}_{q,t}(C_\el)$ are
\begin{align}\label{qtcharacterBi}\s_i(z) = \sum_{\textbf{j}=(j_1,\dots,j_i)\in S} c_{\textbf{j}}(q,t)\La_{\textbf{j}},
\end{align}
where $$\La_{\textbf{j}}=\La_{j_1}(zq^{-(i-1)}t^{i-1})
\La_{j_2}(zq^{-(i-3)}t^{i-3}) \ldots  \La_{j_i}(zq^{i-1}t^{-(i-1)}),$$
with $ i=1,\ldots,\el$, the coefficients $c_\textbf{j}(q,t)\in K$ will be defined below, and $S$ is the set of $1\le j_1<\dots<j_i\le2\el $ such that if $j_\al=k$ and $j_\beta=2\el-k+1$ then $k\ne\el+\al-\beta+1$ for all $1\le k\le \el$.

\begin{prop}\label{fundamb}
    For all $1\le i\le \el$, there exists $c_{j_1,\dots,j_i}(q,t)\in K^S$ such that $T_i(z)$ belongs to $\WWqt$. 
\end{prop}

\begin{proof}
Let $i\in I$. Firstly, we list all the cases where $d_{k,a}(\La_\textbf{j})\neq0$ for a $\textbf{j}\in S$ and $k\in I$. We list the possible cases for $\Pi_k(\La_\textbf{j})$. For each case we construct all the monomials produced by the algorithm from $\La_\textbf{j}$ in direction $k$, and we prove that they are a $\La_{\tilde{\textbf{j}}}$ for a $\tilde{\textbf{j}}\in S$ such that  $ht_i(\tilde{\textbf{j}})=ht_i(\textbf{j})+1$. Moreover, we also exhibit a monomial $\La_{\textbf{j}'}$ verifying $ht_i(\textbf{j}')=ht_i(\textbf{j})-1$ such that an admissible transformation in $\La_{\textbf{j}'}$ gives exactly $\La_\textbf{j}$.
\begin{itemize}
    \item[a)] There exists $1\le\al\le i$ such that $j_\al=k$ and for all other $1\le\beta\le i$, $j_\beta\not\in\{k+1,\bar k, \overline{k+1}\}$. In this case, $\Pi_k(\La_\textbf{j})=Y_k(zq^{-(i+k-2\al)}t^{i+k-2\al})$. 
    Here, the variable $Y_k$ is admissible and the algorithm constructs exactly the monomial $\La_{\tilde{\textbf{j}}}$ where $\tilde{\textbf{j}}=(j_1,\dots,j_{\al-1}, j_\al+1, j_{\al+1},\dots,j_i)$, which is in $S$ since $\textbf{j}$ is and $j_\beta\neq \overline{k+1}$ for all $\beta\neq \al$.\\
    We have $ht_i(\tilde{\textbf{j}})=ht_i(\textbf{j})+1$. 
    
    \item[b)] There exists $1\le\al\le i$ such that $j_\al=k+1$ and for all other $1\le\beta\le i$, $j_\beta\not\in\{k,\bar k, \overline{k+1}\}$. In this case, $\Pi_k(\La_\textbf{j})=Y_k(zq^{-(i+k-2\al)-2}t^{i+k-2\al+2})^{-1}$. 
    Here, the monomial $\La_{\textbf{j}'}$ is in the case (a), where $\textbf{j}'=(j_1,\dots,j_{\al-1}, j_\al-1, j_{\al+1},\dots,j_i)$ is in $S$ since $\textbf{j}$ is and $j_\beta\neq \overline{k}$ for all $\beta\neq \al$. Moreover, the transformation $A_k(zq^{-(i+k-2\al)-r_k}t^{i+k-2\al+1})$ is admissible in $\La_{\textbf{j}'}$ and gives exactly $\La_\textbf{j}$.
    We have $ht_i(\textbf{j}')=ht_i(\textbf{j})-1$.

    \item[c)] There exists $1\le\al\le i$ such that $j_\al=\overline{k+1}$, and for all other $1\le\beta\le i$, $j_\beta\not\in\{ k, {k+1}, \overline k\}$. In this case, $\Pi_k(\La_\textbf{j})=Y_k(zq^{-2\el+k-i+2\al-2}t^{2\el-k+i-2\al})$.
    Here, the variable $Y_k$ is admissible and the algorithm constructs exactly the monomial $\La_{\tilde{\textbf{j}}}$ where $\tilde{\textbf{j}}=(j_1,\dots,j_{\al-1}, j_\al+1, j_{\al+1},\dots,j_i)$, which is in $S$ since $\textbf{j}$ is and $j_\beta\neq k$ for all $\beta\neq \al$. 
    We have $ht_i(\tilde{\textbf{j}})=ht_i(\textbf{j})+1$. 
    
    \item[d)] There exists $1\le\al\le i$ such that $j_\al=\overline{k}$ and for all other $1\le\beta\le i$, $j_\beta\not\in\{\overline{k+1},  k, {k+1}\}$. In this case, $\Pi_k(\La_\textbf{j})=Y_k(zq^{-2\el+k-i+2\al-4}t^{2\el-k+i-2\al+2})^{-1}$.
    Here, the monomial $\La_{\textbf{j}'}$ is in the case (c), where $\textbf{j}'=(j_1,\dots,j_{\al-1}, j_\al-1, j_{\al+1},\dots,j_i)$ is in $S$ since $\textbf{j}$ is and $j_\beta\neq k+1$ for all $\beta\neq \al$. Moreover, the transformation $A_k(zq^{-2\el+k-i+2\al-2-r_k}t^{2\el-k+i-2\al+1})$ is admissible in $\La_{\textbf{j}'}$ and gives exactly $\La_\textbf{j}$.
    We have $ht_i(\textbf{j}')=ht_i(\textbf{j})-1$. 
    
    \item[e)] There exists $1\le\al<\beta\le i$ such that $j_\al=k\neq\el$, $j_\beta=\overline{k+1}$, and for all other $1\le\gamma\le i$, $j_\gamma\not\in\{ \overline k, {k+1}\}$. In this case, $$\Pi_k(\La_\textbf{j})=Y_k(zq^{-(i+k-2\al)}t^{i+k-2\al})Y_k(zq^{-2\el+k-i+2\beta-2}t^{2\el-k+i-2\beta}).$$ 
    In the first spectral parameter $q^{-1}$ and $t$ play symmetric roles and it is not the case in the second term so we are not in the case of a $Y_k(za)Y_k(zaq^{-2}t^2)$. 
    To know which variable $Y_k$ is admissible, we need to compare the spectral parameters. This is a straightforward computation:\\
    To simplify notations let $a=q^{-(i+k-2\al)}t^{i+k-2\al}$. We recall $q,t$ are generic variables.
    \begin{align*}
        \Pi_k(\La_\textbf{j})=Y_k(za)Y_k(zaq^{\pm2r_k})&\Longleftrightarrow aq^{\pm2r_k}=q^{-2\el+k-i+2\beta-2}t^{2\el-k+i-2\beta}\\[4mm]
        &\Longleftrightarrow\left\{\begin{matrix} -(i+k-2\al)\pm2r_k&=&-2\el+k-i+2\beta-2\\
        i+k-2\al&=&2\el-k+i-2\beta
        \end{matrix}  \right.\\[5mm]
        &\Longleftrightarrow\left\{\begin{matrix} \pm2r_k&=&-2\\
        i+k-2\al&=&2\el-k+i-2\beta
        \end{matrix}  \right.\\[5mm]
        &\Longleftrightarrow r_k=1,\quad \Pi_k(\La_\textbf{j})=Y_k(za)Y_k(zaq^{-2}),\quad k=\el+\al-\beta
    \end{align*}
    and in the same way, 
    \begin{align*}
    \Pi_k(\La_\textbf{j})=Y_k(za)Y_k(zat^{\pm2})&\Longleftrightarrow at^{\pm2}=q^{-2\el+k-i+2\beta-2}t^{2\el-k+i-2\beta}\\[4mm]
    &\Longleftrightarrow\left\{\begin{matrix} -(i+k-2\al)&=&-2\el+k-i+2\beta-2\\
    i+k-2\al\pm2 &=&2\el-k+i-2\beta
    \end{matrix}  \right.\\[5mm]
    &\Longleftrightarrow\left\{\begin{matrix}  -(i+k-2\al)&=&-2\el+k-i+2\beta-2\\
    \pm2&=&-2
    \end{matrix}  \right.\\[5mm]
    &\Longleftrightarrow  \Pi_k(\La_\textbf{j})=Y_k(za)Y_k(zat^{-2}),\quad k=\el+\al-\beta+1
    \end{align*}
    Finally, if $k=\el+\al-\beta$, then the second variable is the only admissible variable and the algorithm constructs the monomial $\La_{\tilde{\textbf{j}}}$ with $\tilde{\textbf{j}}=(\tilde{{{j}}}_\gamma)_\gamma$ such that $\tilde{{j}}_\gamma=j_\gamma$ for $\gamma\not\in\{\al,\beta\}$ and $(\tilde{{j}}_\al,\tilde{{j}}_\beta)=(k,\overline{k})$. Indeed, $k=\el+\al-\beta\neq\el+\al-\beta+1$ then $\tilde{\textbf{j}}\in S$.\\
    If $k=\el+\al-\beta+1$, then the first variable is the only admissible variable and the algorithm constructs the monomial $\La_{\tilde{\tilde{\textbf{j}}}}$ with $\tilde{\tilde{\textbf{j}}}=(\tilde{\tilde{{j}}}_\gamma)_\gamma$ such that $\tilde{\tilde{{j}}}_\gamma=j_\gamma$ for $\gamma\not\in\{\al,\beta\}$ and $(\tilde{\tilde{{j}}}_\al,\tilde{\tilde{{j}}}_\beta)=(k+1,\overline{k+1})$. Indeed, $k+1=\el+\al-\beta+2\neq\el+\al-\beta+1$ then $\tilde{\tilde{\textbf{j}}} \in S$. \\
    Else, the algorithm constructs both monomials, and both $\tilde{{\textbf{j}}}$ and $\tilde{\tilde{\textbf{j}}}$ are in $S$ since $k,k+1\neq\el+\al-\beta+1$. 
    We have $ht_i(\tilde{\tilde{\textbf{j}}})=ht_i(\tilde{{\textbf{j}}})=ht_i(\textbf{j})+1$. 
    
    \item[f)] We have $k\neq \el$ and there exists $1\le\al<\beta\le i$ such that $j_\al=k+1$, $j_\beta=\overline{k+1}$, and for all other $1\le\gamma\le i$, $j_\gamma\not\in\{k, \overline k\}$. In this case, $$\Pi_k(\La_\textbf{j})=Y_k(zq^{-(i+k-2\al)-2}t^{i+k-2\al+2})^{-1}Y_k(zq^{-2\el+k-i+2\beta-2}t^{2\el-k+i-2\beta}).$$ 
    This contradicts the regularity condition if and only if $\Pi_k(\Lambda_\textbf{j}) = Y_k(za)Y_k(zat^2)^{-1}$ (the other cases implying the non-regularity can not happen by a straightforward verification on the above spectral parameters). This equality holds if and only if $k+1=\ell+\alpha-\beta+1$, a case that does not appear in $S$. \\  
    The second variable is clearly admissible and the algorithm produces a monomial $\La_{\tilde{\textbf{j}}}$ with $\tilde{\textbf{j}}=(\tilde{{j}}_\gamma)_\gamma$ such that $\tilde{{j}}_\gamma=j_\gamma$ for $\gamma\not\in\{\al,\beta\}$ and $(\tilde{{j}}_\al,\tilde{{j}}_\beta)=(k+1,\bar k)$. Since $\textbf{j} \in S$, we have $\tilde{\textbf{j}}\in S$.\\
    Moreover, let $\textbf{j}'=(j_\gamma')_\gamma$ such that $j'_\gamma=j_\gamma$ for $\gamma\neq\al$ and $j'_\al=k$. Since $\textbf{j}\in S$, we have $\textbf{j}'\in S$. Moreover, the transformation $A_k(zq^{-(i+k-2\al)-1}t^{i+k-2\al+1})^{-1}$ is admissible in $\La_{\textbf{j}'}$ (which is in case (e)) and the algorithm constructs exactly $\La_\textbf{j}$ (there can not be any obstruction to the admissibility of this transformation as there is a unique term with power strictly non-negative).
    We have $ht_i(\tilde{\textbf{j}})=ht_i(\textbf{j})+1$, and $ht_i(\textbf{j}')=ht_i(\textbf{j})-1$.

    \item[g)] There exists $1\le\al<\beta\le i$ such that $j_\al=k$, $j_\beta=\overline{k}$, and for all other $1\le\gamma\le i$, $j_\gamma\not\in\{ \overline{k+1}, {k+1}\}$. In this case, $$\Pi_k(\La_\textbf{j})=Y_k(zq^{-(i+k-2\al)}t^{i+k-2\al})Y_k(zq^{-2\el+k-i+2\beta-4}t^{2\el-k+i-2\beta+2})^{-1}.$$ 
    This contradicts the regularity condition if and only if $\Pi_k(\Lambda_\textbf{j}) = Y_k(za)Y_k(zaq^{-2r_k})^{-1}$ (the other cases implying the non-regularity can not happen by a straightforward verification on the above spectral parameters). This equality holds if and only if $k=\ell+\alpha-\beta+1$, a case that does not appear in $S$.    
    The first variable is clearly admissible and the algorithm produces a monomial $\La_{\tilde{\textbf{j}}}$ with $\tilde{\textbf{j}}=(\tilde{{j}}_\gamma)_\gamma$ such that $\tilde{{j}}_\gamma=j_\gamma$ for $\gamma\not\in\{\al,\beta\}$ and $(\tilde{{j}}_\al,\tilde{{j}}_\beta)=(k,\overline{k+1})$. Since $\textbf{j} \in S$, we have $\tilde{\textbf{j}} \in S$.\\
    Moreover, let $\textbf{j}'=(j_\gamma')_\gamma$ such that $j'_\gamma=j_\gamma$ for $\gamma\neq\beta$ and $j'_\beta=\overline{k+1}$. Since $\textbf{j}\in S$, we have $\textbf{j}'\in S$. Moreover, the transformation $A_k(zq^{-2\el+k-i+2\beta-3}t^{2\el-k+i-2\beta+1})^{-1}$ is admissible in $\La_{\textbf{j}'}$ (which is in case (e)) and the algorithm constructs exactly $\La_\textbf{j}$ (there can not be any obstruction to the admissibility of this transformation as there is a unique term with power strictly non-negative).
    We have $ht_i(\tilde{\textbf{j}})=ht_i(\textbf{j})+1$, and $ht_i(\textbf{j}')=ht_i(\textbf{j})-1$.

    \item[h)] There exists $1\le\al<\beta\le i$ such that $j_\al=k+1$, $j_\beta=\overline{k}$, and for all other $1\le\gamma\le i$, $j_\gamma\not\in\{ \overline{k+1}, {k}\}$. In this case, $$\Pi_k(\La_\textbf{j})=Y_k(zq^{-(i+k-2\al)-2}t^{i+k-2\al+2})^{-1}Y_k(zq^{-2\el+k-i+2\beta-4}t^{2\el-k+i-2\beta+2})^{-1}.$$
    This contradicts the regularity condition if and only if $r_k=2$ and $k=\el$, which is impossible as we would have $j_\al=j_\beta$.\\
    Let $\textbf{j}',\textbf{j}''$ such that $j_\gamma'=j_\gamma$ for all $\gamma\neq\al$, $j_\al'=k$ and $j_\gamma''=j_\gamma$ for all $\gamma\neq\beta$, $j_\beta''=\overline{k+1}$. 
    As in the case (e), a straightforward computation shows that if $k\neq\el+\al-\beta$, then $k+1\neq\el+\al-\beta+1$ thus $\textbf{j}''\in S$. Moreover, the transformation $A_k(zq^{-2\el+k-i+2\beta-3}t^{2\el-k+i-2\beta+1})^{-1}$ is admissible in $\La_{\textbf{j}''}$ (which is in case (f)) and the algorithm constructs exactly $\La_\textbf{j}$.\\
    If $k=\el+\al-\beta$, then $\textbf{j}'\in S$ and the transformation $A_k(zq^{-(i+k-2\al)-1}t^{i+k-2\al+1})^{-1}$ is admissible in $\La_{\textbf{j}'}$ (which is in case (g)) and the algorithm constructs exactly $\La_\textbf{j}$.\\
    We have $ht_i(\textbf{j}'')=ht_i(\textbf{j}')=ht_i(\textbf{j})-1$. 
\end{itemize}
If $\textbf{j}$ contains $k$ and $k+1$ (resp. $\overline{k+1}$ and $\bar k$), they are at consecutive positions and their variables $Y_k$ cancel. Hence the tuples not covered by (a)--(h) reduce to the cases (a)--(d), or satisfy $\Pi_k(\La_\textbf{j})=1$.

We prove by induction that for all $n\in \llbracket 0, i(2\el-i)\rrbracket$, the property $\mathcal{P}_n$: " At the $n$-th step of the algorithm starting from the monomial $Y_i(z)$,
we obtain exactly the monomials 
$$\La_{j_1}(zq^{-(i-1)}t^{i-1})\La_{j_2}(zq^{-(i-3)}t^{i-3}) \ldots \La_{j_{i-1}}(zq^{(i-3)}t^{-(i-3)}) \La_{j_i}(zq^{(i-1)}t^{-(i-1)}),$$
for all $\mathbf{j}=(j_\al)_\al\in S$ such that $ht_i(\mathbf{j})=\sum_{\alpha=1}^i(j_\alpha-\alpha)=n$" is true.\\
If $n=0$ then we get $$\La_{j_1}(zq^{-(i-1)}t^{i-1})
\La_{j_2}(zq^{-(i-3)}t^{i-3}) \ldots \La_{j_{i-1}}(zq^{(i-3)}t^{-(i-3)}) \La_{j_i}(zq^{(i-1)}t^{-(i-1)})=Y_i(z)$$ 
with $j_\al=\al$ for $\al=1,2,\ldots,i$. Then the only monomial at the $0$-th step is $Y_i(z)$.\\
We assume $\mathcal{P}_n$ is true for an integer $n\in  \llbracket 0, i(2\el-i)-1\rrbracket$. \\
Let us prove that all monomials $\La_\textbf{j}$ such that $ht_i(\textbf{j})=n+1$ are obtained by the algorithm by one transformation from a monomial $\La_{\textbf{j}'}$ such that $ht_i(\textbf{j}')=n$.
Let $\textbf{j}$ be such that $ht_i(\textbf{j})=n+1>0$. The only dominant monomial is for $\textbf{j}=(1,2,\dots,i)$, which is obtained at the $0$-th step. Hence, there exists a $k\in I$ and $a\in q^\Z t^\Z$ such that $d_{k,a}(\La_\textbf{j})<0$.
Hence, we are in one of the cases (b), (d), (f), (g), or (h). In all of these cases, we have exhibited a monomial $\La_{\textbf{j}'}$ such that $ht_i(\textbf{j}')=ht_i(\textbf{j})-1=n$ and an admissible transformation in $\La_{\textbf{j}'}$ giving exactly $\La_\textbf{j}$, hence $\La_\textbf{j}$ is obtained at the $(n+1)$-th step of the algorithm.\\
Conversely, a monomial $X\in\mathbf{M}$ obtained at the $(n+1)$-th step of the algorithm is obtained from a monomial $X'\in\mathbf{M}$ obtained at the $n$-th step by an admissible transformation. By induction hypothesis, we have $X'=\La_{\textbf{j}'}$ with $ht_i(\textbf{j}')=n$. 
Moreover, by the cases (a), (c), (e), (f), and (g) we have proved that all the monomials produced by the algorithm are of the form $\La_\textbf{j}$ such that $ht_i(\textbf{j})=ht_i(\textbf{j}')+1=n+1$. 
Hence the result by induction.

\end{proof}

\subsubsection{Type $B_\el$:}
Let 
\[D=
\begin{pmatrix}
2 & 0 &  \cdots  & 0 &0\\
0 & 2 &  \cdots  & 0&0 \\
\vdots & \ddots & \ddots & \ddots &\vdots\\
0 & \cdots & 0 & 2&0 \\
0 & \cdots & 0 & 0&1 
\end{pmatrix}
\]
the diagonal matrix such that $DC$ is symmetric. Again, to define the following fields, we deformed the fields in the one-parameter deformation given by Frenkel and Reshetikhin~\cite[A.2]{MR1646483}.

Let 
\begin{align*}
\La_i(z) &:= Y_i(z q^{-2i+2} t^{i-1})\, Y_{i-1}(z q^{-2i} t^{i})^{-1}, \quad\quad i = 1, \ldots, \el - 1, \\
\La_\ell(z) &:= Y_\ell(z q^{-2\ell+3} t^{\ell-1})\, Y_\ell(z q^{-2\ell+1} t^{\ell-1})\, Y_{\ell-1}(z q^{-2\ell} t^{\ell})^{-1}, \\
\La_{\el+1}(z) &:=  Y_\ell(z q^{-2\ell+3} t^{\ell-1})\, Y_\ell(z q^{-2\ell-1} t^{\ell+1})^{-1}, \\
\La_{\bar{\ell}}(z) &:= Y_{\ell-1}(z q^{-2\ell+2} t^{\ell})\, Y_\ell(z q^{-2\ell+1} t^{\ell+1})^{-1}\, Y_\ell(z q^{-2\ell-1} t^{\ell+1})^{-1}, \\
\La_{\bar{i}}(z) &:= Y_{i-1}(z q^{-4\ell+2i+2} t^{2\ell-i})\, Y_i(z q^{-4\ell+2i} t^{2\ell-i+1})^{-1},
  \quad\quad i = 1, \ldots, \ell - 1,
\end{align*}
with $\bar i=2\el+2-i$.  Then, the fundamental fields of $\mathbf{W}_{q,t}(B_\el)$ are
\begin{align}\label{qtcharacterCl}
\s_i(z) = \sum_{\textbf{j}=(j_1,\dots,j_i)\in S} c_{\textbf{j}}(q,t)\La_{\textbf{j}},
\end{align}
where $$\La_{\textbf{j}}=\La_{j_1}(zq^{-2(i-1)}t^{i-1})
\La_{j_2}(zq^{-2(i-3)}t^{i-3}) \ldots  \La_{j_i}(zq^{2(i-1)}t^{-(i-1)}),$$
with $i=1,\ldots,\el-1$, the coefficients $c_\textbf{j}(q,t)\in K$ will be defined below, and $S$ is the set of $1\le j_1\le \dots\le j_i\le 2\el+1$ such that for all $\al$, $j_\al<j_{\al+1}$ or $j_\al=j_{\al+1}=\el+1$.

The formula for $T_{\el}(z)$, which corresponds to the spinor representation, is inspired by the formula of the $\el$-th fundamental $q$-character of $\widehat{\mathfrak{so}_{2\el+1}}$ in \cite{frenkel1996quantum}:
\begin{align} \label{spinorCl}
T_{\el}(z)
=
\sum_{\sigma_1, \ldots, \sigma_{\el} = \pm 1}c_{\sigma_1,\ldots,\sigma_\el}(q,t)
b_{\sigma_1}(z|1)\,
b_{\sigma_2}(z q^{-2\sigma_1}t^{\sigma_1}|2)
\cdots
b_{\sigma_\el}(z q^{-2\sigma_1-\cdots-2\sigma_{\el-1}}t^{\sigma_1+\ldots+\sigma_{\el-1}}|\el),
\end{align}
where
\[
\begin{aligned}
b_{1}(z|\el) &= Y_{\el}(z q^{-2\el }t^{\ell+1})^{-1}, \\[4pt]
b_{1}(z|k) &= 1, \qquad k = 1, \ldots, \el-1, \\[4pt]
b_{-1}(z|\el) &= Y_{\el-1}(z q^{-2\el+1}t^\el)^{-1} Y_{\el}(z q^{-2\el + 2}t^{\el-1}) , \\[4pt]
b_{-1}(z|k) &= Y_{k-1}(z q^{-2\el + 1}t^{\el})^{-1} Y_{k}(z q^{-2\el +3}t^{\el-1}),
\qquad k = 1, \ldots, \el-1.
\end{aligned}
\]

\begin{prop}
    For all $1\le i\le \el$, there exists $c_{\textbf{j}}(q,t)\in K^S$ such that $T_i(z)$ belongs to $\WWqt$.
\end{prop}

\begin{proof}
Let $i\in\llbracket 1,\el-1\rrbracket$. Firstly, we list all the cases where $d_{k,a}(\La_\textbf{j})\neq0$ for a $\textbf{j}\in S$ and $k\in I$. For each case we construct all the monomials produced by the algorithm from $\La_\textbf{j}$ in direction $k$ and, when possible, we exhibit a monomial $\La_{\textbf{j}'}$ verifying $ht_i(\textbf{j}')<ht_i(\textbf{j})$ such that an admissible transformation in $\La_{\textbf{j}'}$ gives exactly $\La_\textbf{j}$. However, we need to distinguish the cases $k=\el$ and $k\neq \el$.
If $k<\el$:
\begin{itemize}
    \item[a)] There exists $1\le\al\le i$ such that $j_\al=k$ and for all other $1\le\beta\le i$, $j_\beta\not\in\{k+1,\bar k, \overline{k+1}\}$. In this case, $\Pi_k(\La_\textbf{j})=Y_k(zq^{-2(i+k-2\al)}t^{i+k-2\al})$.\\ 
    Here the algorithm produces a monomial $\La_{\textbf{j}'}$ with $\textbf{j}'=(j_1,\dots,j_{\al-1}, j_\al+1, j_{\al+1},\dots,j_i)$, which is in $S$. Moreover, $ht_i(\textbf{j}')=ht_i(\textbf{j})+1$.
    
    \item[b)] There exists $1\le\al\le i$ such that $j_\al=k+1$ and for all other $1\le\beta\le i$, $j_\beta\not\in\{k,\bar k, \overline{k+1}\}$. In this case, $\Pi_k(\La_\textbf{j})=Y_k(zq^{-2(i+k-2\al)-4}t^{i+k-2\al+2})^{-1}$.
    Here, the monomial $\La_{\textbf{j}'}$ is in the case (a), where $\textbf{j}'=(j_1,\dots,j_{\al-1}, j_\al-1, j_{\al+1},\dots,j_i)$ is in $S$. \\
    Moreover, the transformation $A_k(zq^{-2(i+k-2\al)-2}t^{i+k-2\al+1})^{-1}$ is admissible in $\La_{\textbf{j}'}$ and gives exactly $\La_\textbf{j}$.
    We have $ht_i(\textbf{j}')=ht_i(\textbf{j})-1$.

    \item[c)] There exists $1\le\al\le i$ such that $j_\al=\overline{k+1}$, and for all other $1\le\beta\le i$, $j_\beta\not\in\{ k, {k+1}, \overline k\}$. In this case, $\Pi_k(\La_\textbf{j})=Y_k(zq^{-4\el+2k-2i+4\al+2}t^{2\el-k+i-2\al})$.\\
    Here, the variable $Y_k$ is admissible and the algorithm constructs the monomial $\La_{\textbf{j}'}$ with $\textbf{j}'=(j_1,\dots,j_{\al-1}, j_\al+1, j_{\al+1},\dots,j_i)$, which is in $S$. 
    We have $ht_i(\textbf{j}')=ht_i(\textbf{j})+1$.  
    
    \item[d)] There exists $1\le\al\le i$ such that $j_\al=\overline{k}$ and for all other $1\le\beta\le i$, $j_\beta\not\in\{\overline{k+1},  k, {k+1}\}$. In this case, $\Pi_k(\La_\textbf{j})=Y_k(zq^{-4\el+2k-2i+4\al-2}t^{2\el-k+i-2\al+2})^{-1}$.
    Here, the monomial $\La_{\textbf{j}'}$ is in the case (c), where $\textbf{j}'=(j_1,\dots,j_{\al-1}, j_\al-1, j_{\al+1},\dots,j_i)$ is in $S$. \\
    Moreover, the transformation $A_k(zq^{-4\el+2k-2i+4\al}t^{2\el-k+i-2\al+1})^{-1}$ is admissible in $\La_{\textbf{j}'}$ and gives exactly $\La_\textbf{j}$.
    We have $ht_i(\textbf{j}')=ht_i(\textbf{j})-1$.

    \item[e)] There exists $1\le\al<\beta\le i$ such that $j_\al=k$, $j_\beta=\overline{k+1}$, and for all other $1\le\gamma\le i$, $j_\gamma\not\in\{ \overline k, {k+1}\}$. In this case, $$\Pi_k(\La_\textbf{j})=Y_k(zq^{-2(i+k-2\al)}t^{i+k-2\al})Y_k(zq^{-4\el+2k-2i+4\beta+2}t^{2\el-k+i-2\beta}).$$
    In the first spectral parameter $q^{-2}=q^{-r_k}$ and $t$ play symmetric roles and it is not the case in the second term so we are not in the case of a $Y_k(za)Y_k(zaq^{-2r_k}t^2)$.
    As in type $C_\el$, a straightforward computation is necessary to determine which variable is admissible. However we get that both are admissible. Hence, the algorithm produces two monomials $\La_{\textbf{j}'}$ and $\La_{\textbf{j}''}$ with $\textbf{j}'=(j_\gamma')_\gamma$ such that $j'_\gamma=j_\gamma$ for $\gamma\not\in\{\al,\beta\}$ and $(j_\al',j'_\beta)=(k,\overline{k})$, and $\textbf{j}''=(j_\gamma'')_\gamma$ such that $j''_\gamma=j_\gamma$ for $\gamma\not\in\{\al,\beta\}$ and $(j_\al'',j''_\beta)=(k+1,\overline{k+1})$. Since $\textbf{j} \in S$, we have $\textbf{j}' \in S$ and $\textbf{j}'' \in S$.\\
    We have $ht_i(\textbf{j}'')=ht_i(\textbf{j}')=ht_i(\textbf{j})+1$.

    \item[f)] There exists $1\le\al<\beta\le i$ such that $j_\al=k+1$, $j_\beta=\overline{k+1}$, and for all other $1\le\gamma\le i$, $j_\gamma\not\in\{k, \overline k\}$. In this case, $$\Pi_k(\La_\textbf{j})=Y_k(zq^{-2(i+k-2\al)-4}t^{i+k-2\al+2})^{-1}Y_k(zq^{-4\el+2k-2i+4\beta+2}t^{2\el-k+i-2\beta}).$$
    A straightforward computation shows that this never contradicts the regularity condition (we never have $\Pi_k(\La_\textbf{j})=Y_k(za)Y_k(zaq^{-2r_k})^{-1}$ nor $\Pi_k(\La_\textbf{j})=Y_k(za)Y_k(zat^2)^{-1}$).
    Moreover, the algorithm produces a monomial $\La_{\textbf{j}'}$ with $\textbf{j}'=(j_1,\dots,j_{\beta-1}, j_\beta+1, j_{\beta+1},\dots,j_i)$, which is in $S$. \\
    Let $\textbf{j}''=(j_1,\dots,j_{\al-1}, j_\al-1, j_{\al+1},\dots,j_i)$, which is in $S$. The monomial $\La_{\textbf{j}''}$ is in the case (e) and the transformation $A_k(zq^{-2(i+k-2\al)-2}t^{i+k-2\al+1})^{-1}$ is admissible in $\La_{\textbf{j}''}$ and gives exactly $\La_\textbf{j}$. 
    We have $ht_i(\textbf{j}')=ht_i(\textbf{j})+1$, and $ht_i(\textbf{j}'')=ht_i(\textbf{j})-1$.

    \item[g)] There exists $1\le\al<\beta\le i$ such that $j_\al=k$, $j_\beta=\overline{k}$, and for all other $1\le\gamma\le i$, $j_\gamma\not\in\{ \overline{k+1}, {k+1}\}$. In this case, $$\Pi_k(\La_\textbf{j})=Y_k(zq^{-2(i+k-2\al)}t^{i+k-2\al})Y_k(zq^{-4\el+2k-2i+4\beta-2}t^{2\el-k+i-2\beta+2})^{-1}.$$
    A straightforward computation shows that this never contradicts the regularity condition (we never have $\Pi_k(\La_\textbf{j})=Y_k(za)Y_k(zaq^{-2r_k})^{-1}$ nor $\Pi_k(\La_\textbf{j})=Y_k(za)Y_k(zat^2)^{-1}$).
    Moreover, the algorithm produces a monomial $\La_{\textbf{j}'}$ with $\textbf{j}'=(j_1,\dots,j_{\al-1}, j_\al+1, j_{\al+1},\dots,j_i)$, which is in $S$. \\
    Let $\textbf{j}''=(j_1,\dots,j_{\beta-1}, j_\beta-1, j_{\beta+1},\dots,j_i)$, which is in $S$. The monomial $\La_{\textbf{j}''}$ is in the case (e) and the transformation $A_k(zq^{-4\el+2k-2i+4\beta}t^{2\el-k+i-2\beta+1})^{-1}$ is admissible in $\La_{\textbf{j}''}$ and gives exactly $\La_\textbf{j}$. 
    We have $ht_i(\textbf{j}')=ht_i(\textbf{j})+1$, and $ht_i(\textbf{j}'')=ht_i(\textbf{j})-1$.

    \item[h)] There exists $1\le\al<\beta\le i$ such that $j_\al=k+1$, $j_\beta=\overline{k}$, and for all other $1\le\gamma\le i$, $j_\gamma\not\in\{ \overline{k+1}, {k}\}$. In this case, $$\Pi_k(\La_\textbf{j})=Y_k(zq^{-2(i+k-2\al)-4}t^{i+k-2\al+2})^{-1}Y_k(zq^{-4\el+2k-2i+4\beta-2}t^{2\el-k+i-2\beta+2})^{-1}.$$
    In the first spectral parameter $q^{-2}=q^{-r_k}$ and $t$ play symmetric roles and it is not the case in the second term so we are not in the case of a $Y_k(za)^{-1}Y_k(zaq^{-2r_k}t^2)^{-1}$.\\
    Let $\textbf{j}'=(j_1,\dots,j_{\al-1}, j_\al-1, j_{\al+1},\dots,j_i)$, which is in $S$. The monomial $\La_{\textbf{j}'}$ is in the case (g) and the transformation $A_k(zq^{-2(i+k-2\al)-2}t^{i+k-2\al+1})^{-1}$ is admissible in $\La_{\textbf{j}'}$ and gives exactly $\La_\textbf{j}$. 
    We have $ht_i(\textbf{j}')=ht_i(\textbf{j})-1$. We could exhibit another monomial such that an admissible transformation gives exactly $\La_\textbf{j}$, but it is not necessary for the proof.
\end{itemize}
As in type $C_\el$, if $\textbf{j}$ contains $k$ and $k+1$ (resp. $\overline{k+1}$ and $\bar k$), they are at consecutive positions and their variables $Y_k$ cancel. Hence the tuples not covered by (a)--(h) reduce to the cases (a)--(d), or satisfy $\Pi_k(\La_\textbf{j})=1$.

If $k=\el$, we are in one of the following cases:
\begin{itemize}
    \item[i)] There exists $1\le\al\le i$ and $s\geq0$ such that $j_\al=\el$, $j_{\al+1}=\ldots=j_{\al+s}=\el+1$ and $j_{\al+s+1}\not\in\{\el+1, \bar\el\}$.
    In this case, $$\Pi_\el(\La_\textbf{j})=Y_\el(zq^{-2\el-2i+4\al-1}t^{\el+i-2\al})Y_\el(zq^{-2\el-2i+4\al+4s+1}t^{\el+i-2\al-2s}).$$ 
    Hence if $s=0$ only the first $Y_\el$ is admissible. If $s>0$, both are admissible in the monomial $\La_\textbf{j}$ in the direction $\el$. the algorithm applies these transformations to get $\La_{\textbf{j}'}$ (resp. $\La_{\textbf{j}''}$) with $j_\al'=j_\al+1=\el+1$ (resp. $j_{\al+s}''=j_{\al+s}+1=\bar\el$), and for all $\gamma\ne\al$, (resp. $\gamma\ne\al+s$) $j_\gamma'=j_\gamma$ (resp. $j_\gamma''=j_\gamma$). 
    Both $\textbf{j}'$ and $\textbf{j}''$ are in $S$. Moreover, we have $ht_i(\textbf{j}'')=ht_i(\textbf{j}')=ht_i(\textbf{j})+1$.

    \item[ii)] There exists $1\le\al\le i$ and $s\geq0$ such that $j_\al=\el$, $j_{\al+1}=\ldots=j_{\al+s}=\el+1$ and $j_{\al+s+1}= \bar\el$. 
    Here, $$\Pi_\el(\La_\textbf{j})=Y_\el(zq^{-2\el-2i+4\al-1}t^{\el+i-2\al})Y_\el(zq^{-2\el-2i+4\al+4s+3}t^{\el+i-2\al-2s})^{-1}.$$ It is clear that it does not contradict the regularity conditoin.
    Then, there is a unique admissible transformation in the monomial $\La_\textbf{j}$ in the direction $k$. the algorithm applies this transformation to get $\La_{\textbf{j}'}$ with $j_\al'=j_\al+1=\el+1$ and for all $\gamma\ne\al$, $j_\gamma'=j_\gamma$. 
    Let $\textbf{j}''=(j_\gamma'')_\gamma\in S$ such that $j''_\gamma=j_\gamma$ for $\gamma\ne\al+s+1$ and $j''_{\al+s+1}=j_{\al+s+1}-1=\el+1$. The monomial $\La_{\textbf{j}''}$ is in the case (i) and the transformation $A_\el(zq^{-2\el-2i+4\al+4s+4}t^{\el+i-2\al-2s-1})^{-1}$ is admissible in $\La_{\textbf{j}''}$ and gives exactly $\La_\textbf{j}$.
    We have $ht_i(\textbf{j}')=ht_i(\textbf{j})+1$, and $ht_i(\textbf{j}'')=ht_i(\textbf{j})-1$.

    \item[iii)] There exists $1\le\al\le i$ and $s\geq0$ such that $j_{\al-1}\neq \el$, $j_\al=j_{\al+1}=\ldots=j_{\al+s}=\el+1$ and $j_{\al+s+1}\neq \el+1,\bar\el$. 
    In this case, $$\Pi_\el(\La_\textbf{j})=Y_\el(zq^{-2\el-2i+4\al+4s+1}t^{\el+i-2\al-2s})Y_\el(zq^{-2\el-2i+4\al-3}t^{\el+i-2\al+2})^{-1}.$$ 
    Again, it is clear that it does not contradict the regularity condition. 
    There is still a unique admissible transformation in the monomial $\La_\textbf{j}$ in direction $k$. 
    Hence, the algorithm applies this transformation to get $\La_{\textbf{j}'}$ with $j_{\al+s}'=j_{\al+s}+1=\bar\el$, and for all $\gamma\ne\al+s$, $j_\gamma'=j_\gamma$. 
    Let $\textbf{j}''=(j_\gamma'')_\gamma\in S$ such that $j''_\gamma=j_\gamma$ for $\gamma\ne\al$ and $j''_{\al}=j_{\al}-1=\el$. The monomial $\La_{\textbf{j}''}$ is in the case (i) and the transformation $A_\el(zq^{-2\el-2i+4\al-2}t^{\el+i-2\al+1})^{-1}$ is admissible in $\La_{\textbf{j}''}$ and gives exactly $\La_\textbf{j}$.
    We have $ht_i(\textbf{j}')=ht_i(\textbf{j})+1$, and $ht_i(\textbf{j}'')=ht_i(\textbf{j})-1$.
    
    \item[iv)] There exists $1\le\al\le i$ and $s\geq0$ such that $j_{\al}=\bar\el$, $j_{\al-1}=\ldots=j_{\al-s}=\el+1$ and $j_{\al-s-1}< \el$. 
    Here, $$\Pi_\el(\La_\textbf{j})=Y_\el(zq^{-2\el-2i+4\al-1}t^{\el+i-2\al+2})^{-1}Y_\el(zq^{-2\el-2i+4\al-4s-3}t^{\el+i-2\al+2s+2})^{-1}.$$ It is clear that it does not contradict the regularity conditoin.
    Let $\textbf{j}'=(j_\gamma')_\gamma\in S$ such that $j'_\gamma=j_\gamma$ for $\gamma\ne\al$ and $j'_{\al}=j_{\al}-1=\el+1$. The monomial $\La_{\textbf{j}'}$ is in the case (iii) and the transformation $A_\el(zq^{-2\el-2i+4\al}t^{\el+i-2\al+1})^{-1}$ is admissible in $\La_{\textbf{j}'}$ and gives exactly $\La_\textbf{j}$.
    We have $ht_i(\textbf{j}')=ht_i(\textbf{j})-1$. We could exhibit another monomial such that an admissible transformation gives exactly $\La_\textbf{j}$, but it is not necessary for the proof.
\end{itemize}

Thus, similarly, we can apply the same proof as Proposition \ref{fundamb} and prove by induction that all the monomials obtained by the algorithm at step $n$ are exactly the monomials $\La_\textbf{j}$ such that $ht_i(\textbf{j})=n$. Hence, $T_i(z)$ belongs to $\WWqt$ for all $1\le i\le \el-1$.\\

Now we consider $T_\el(z)$. Let us prove by induction that for all $n$, all the monomials obtained at step $n$ of the algorithm are in $T_\el(z)$. We have $Y_\el(z)=\prod_k b_{-1}(zq^{2(k-1)}t^{-(k-1)}|k)$.  For $\sigma=(\sigma_i)_i\in\{\pm1\}^\el$, $b_\sigma$ will denote the following monomial:
$$b_\sigma:=b_{\sigma_1}(z|1)\,
b_{\sigma_2}(z q^{-2\sigma_1}t^{\sigma_1}|2)
\cdots
b_{\sigma_\el}(z q^{-2\sigma_1-\cdots-2\sigma_{\el-1}}t^{\sigma_1+\ldots+\sigma_{\el-1}}|\el)$$
If $M$ is a monomial obtained at step $n$, if $M$ is anti-dominant then the algorithm stops and there is no step $n+1$. Else, there exists $k$ such that for some $a\in  q^\Z t^\Z$, $d_{k,a}(M)>0$.
It implies $\sigma_k=-1$ and $\sigma_{k+1}\neq-1$, and $Y_k(za)$ is admissible in $b_\sigma$. Let $k$ be such an integer. Then, we define a monomial $b_{\sigma'}$ such that $\sigma'_i=\sigma_i$ if $i\not\in\{k,k+1\}$ and $\sigma_i'=-\sigma_i $ for $i\in\{k,k+1\}$ (we take the convention $\sigma_{\el+1}=0$). It is easy to see that $b_{\sigma'}$ is the monomial given by the algorithm by taking the expansion in the $k^\text{th}$ $\sld_2$ direction in $b_\sigma$. Hence all the monomials in the algorithm appear in the sum $\eqref{spinorCl}$.\\
Conversely, let $\sigma^{(1)}<\ldots <\sigma^{(2^\el)} $ be the non-decreasing set of elements of $\{\pm1\}^\el$ with respect to the lexicographic order. \\
We have $\sigma^{(1)}_i=-1$ for all $i$. Hence $b_{\sigma^{(1)}}=Y_\el(z)$ appears in the algorithm. Now we assume the monomial $b_{\sigma^{(i)}}$ appears in the algorithm for all $i\le n$ for a $1\le n\le2^\el-1$. Let us prove that $b_{\sigma^{(n+1)}}$ appears in the algorithm. $\sigma^{(n+1)}>\sigma^{(1)} $ then there exists $j$ such that $\sigma^{(n+1)}_j=1$. Let $k=\max\{j,\sigma^{(n+1)}_j=1\}$. By definition of $k$, $\sigma^{(n+1)}_{k+1}\neq1$ Let $\sigma\in\{\pm1\}^\el$ such that $\sigma_i=\sigma^{(n+1)}_i$ if $i\not\in\{k,k+1\}$ and $\sigma_i=-\sigma^{(n+1)}_i$ if $i\in\{k,k+1\}$. Hence there exists $m<n+1$ such that $\sigma=\sigma^{(m)}$. By the induction assumption $b_\sigma$ appears in the algorithm. Again, it is easy to see that $b_{\sigma^{(n+1)}}$ is the monomial obtained in the algorithm by taking the expansion in the $k^\text{th}$ $\sld_2$ direction in $b_\sigma$. Hence the result.
\end{proof}

\subsubsection{Type $D_\el$:} 
The Cartan matrix is 
\[
C=
\begin{pmatrix}
2 & -1 & \cdots & 0 & 0 & 0\\
  -1& 2 & -1 & 0 & 0 & 0\\
\vdots  & \ddots  & \ddots & \ddots & 0 & 0\\
0 & 0 &  -1 & 2 & -1 & -1\\
0  & 0 & 0 & -1 & 2 & 0\\
0 & 0 & 0 & -1 & 0 & 2
\end{pmatrix},
\]
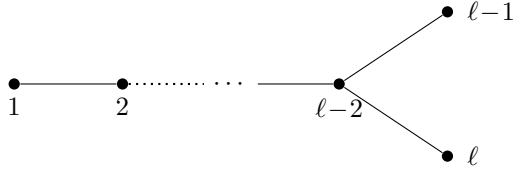
\begin{center}
\begin{tikzpicture}[scale=1.2, baseline=(current bounding box.center)]

\tikzstyle{point}=[circle,fill,inner sep=1.5pt]

\node[point] (v1) at (0,0) {};
\node[point] (v2) at (1.2,0) {};
\node (vdots) at (2.4,0) {$\cdots$};
\node[point] (v4) at (3.6,0) {};
\node[point] (v5) at (4.8,0.8) {};
\node[point] (v6) at (4.8,-0.8) {};

\node[below=2pt] at (v1) {1};
\node[below=2pt] at (v2) {2};
\node[below=2pt] at (v4) {$\el\!-\!2$};
\node[right=4pt] at (v5) {$\el\!-\!1$};
\node[right=4pt] at (v6) {$\el$};

\draw (v1) -- (v2);
\draw[dotted,thick] (v2) -- (vdots);
\draw (vdots) -- (v4);
\draw (v4) -- (v5);
\draw (v4) -- (v6);
\end{tikzpicture}

\end{center}

The fundamental field $T_1(z)$ is given in \cite{MR1646483} and $T_{\el-1}(z)$, $T_\el(z)$ are deduced from the formulas of the associated fundamental $q$-characters (corresponding to the spinor representations of $\widehat{\mathfrak{so}_{2\el}}$) in \cite{frenkel1996quantum}.
We define the set $S=\{1,\ldots,\el,\overline\el,\dots,\overline1\}$:
\begin{align*}
\Lambda_i(z) &:= Y_i(zq^{-i+1}t^{i-1})\,Y_{i-1}(zq^{-i}t^{i})^{-1}, && i = 1, \ldots, \ell - 2, \\[4pt]
\Lambda_{\ell-1}(z) &:= Y_{\ell}(zq^{-\ell+2}t^{\ell-2})\,Y_{\ell-1}(zq^{-\ell+2}t^{\ell-2})\,Y_{\ell-2}(zq^{-\ell+1}t^{\ell-1})^{-1}, \\[4pt]
\Lambda_{\ell}(z) &:= Y_{\ell}(zq^{-\ell+2}t^{\ell-2})\,Y_{\ell-1}(zq^{-\ell}t^{\ell})^{-1}, \\[4pt]
\Lambda_{\bar{\ell}}(z) &:= Y_{\ell-1}(zq^{-\ell+2}t^{\ell-2})\,Y_{\ell}(zq^{-\ell}t^{\ell})^{-1}, \\[4pt]
\Lambda_{\overline{\ell-1}}(z) &:= Y_{\ell-2}(zq^{-\ell+1}t^{\ell-1})\,Y_{\ell-1}(zq^{-\ell}t^{\ell})^{-1}\,Y_{\ell}(zq^{-\ell}t^{\ell})^{-1}, \\[4pt]
\Lambda_{\bar{i}}(z) &:= Y_{i-1}(zq^{-2\ell+i+2}t^{2\ell-i-2})\,Y_i(zq^{-2\ell+i+1}t^{2\ell-i-1})^{-1}, && i = 1, \ldots, \ell - 2.
\end{align*}
 Then, we define the first fundamental field as follows:
 $$\s_1(z) = \sum_{j=1}^{2\el} \La_{j}(z),$$
with $\quad i=1,\ldots,\el$, and for $k\in\llbracket \el+1,2\el\rrbracket$, $\overline{2\el-k+1}=k$. \\

We also define the fundamental fields $T_{\el-1},T_\el$. For $\varepsilon\in\{+,-\}$, the subscript $\varepsilon$
means $\el$ if $\varepsilon=+$ and $\el-1$ if $\varepsilon=-$; thus $T_+=T_\el$, $T_-=T_{\el-1}$, $Y_+=Y_\el$ and
$Y_-=Y_{\el-1}$. Set $Y_0:=1$ and
\begin{equation}\label{spinorDl}
T_{\varepsilon}(z)=\sum_{\sigma_1,\ldots,\sigma_{\el-1}=\pm1}
\Bigg(\prod_{k=1}^{\el-2}b_{\sigma_k}\big(zq^{-\sigma_1-\cdots-\sigma_{k-1}}t^{\sigma_1+\cdots+\sigma_{k-1}}\big|k\big)\Bigg)\,
b^{\,\varepsilon(-\sigma_1)\cdots(-\sigma_{\el-2})}_{\sigma_{\el-1}}
\big(zq^{-\sigma_1-\cdots-\sigma_{\el-2}}t^{\sigma_1+\cdots+\sigma_{\el-2}}\big|\el-1\big),
\end{equation}
where, for $k=1,\ldots,\el-2$ and $\varepsilon'\in\{+,-\}$,
\[
\begin{aligned}
b_{1}(z|k) &= 1, &
b_{-1}(z|k) &= Y_{k-1}(z q^{-\el+1}t^{\el-1})^{-1}\,Y_{k}(z q^{-\el+2}t^{\el-2}),\\[4pt]
b^{\varepsilon'}_{1}(z|\el-1) &= Y_{\varepsilon'}(z q^{-\el}t^\el)^{-1}, &
b^{\varepsilon'}_{-1}(z|\el-1) &= Y_{\el-2}(z q^{-\el+1}t^{\el-1})^{-1}\,Y_{\varepsilon'}(z q^{-\el+2}t^{\el-2}).
\end{aligned}
\]

\begin{prop}
    For all $i\in\{1,\el-1,\el\}$, $T_i(z)$ belongs to $\WWqt$.
\end{prop}
\begin{proof}
In \cite{MR1646483}, the authors proved that $T_1(z)\in\WWqt$. The nodes $\el-1$ and $\el$ play symmetric roles, so it is
sufficient to treat $\varepsilon=+$.

For $\sigma\in\{\pm1\}^{\el-1}$, let $b_\sigma$ be the term of \eqref{spinorDl} indexed by $\sigma$, set
$c_k(\sigma):=\sigma_1+\cdots+\sigma_{k-1}$ and $e(\sigma):=(-\sigma_1)\cdots(-\sigma_{\el-2})\in\{+,-\}$, so that the last
factor of $b_\sigma$ is $b^{e(\sigma)}_{\sigma_{\el-1}}$.\\

Let $1\le k\le\el-2$. The variable $Y_k$ appears with a positive exponent only in the
$k$-th factor, as $Y_k(zq^{-\el+2-c_k}t^{\el-2+c_k})$ if $\sigma_k=-1$, and with a negative exponent only in the $(k+1)$-th
factor, as $Y_k(zq^{-\el+1-c_{k+1}}t^{\el-1+c_{k+1}})^{-1}$ if $\sigma_{k+1}=-1$. If $\sigma_k=-1$, then $c_{k+1}=c_k-1$, and
these two variables cancel. Moreover $Y_\el,Y_{\el-1}$ only appear in the last factor. Hence:
\begin{itemize}
\item for $1\le k\le\el-2$, $b_\sigma$ contains a variable $Y_k(za)$ with positive exponent if and only if $\sigma_k=-1$
and $\sigma_{k+1}=+1$; it is then unique, with exponent $1$;
\item $b_\sigma$ contains a variable $Y_\el(za)$ or $Y_{\el-1}(za)$ with positive exponent if and only if $\sigma_{\el-1}=-1$;
it is then unique, namely $Y_{e(\sigma)}(zq^{-\el+2-c_{\el-1}}t^{\el-2+c_{\el-1}})$, with exponent $1$.
\end{itemize}
In particular, in each $\sld_2$ direction $j$, $b_\sigma$ has at most one positive variable $Y_j(za)$, and the algorithm only produces in
this direction the monomial $b_\sigma A_j(zaq^{-1}t)^{-1}$ with coefficient $1$. Define the moves
\begin{itemize}
\item for $1\le k\le\el-2$ and $(\sigma_k,\sigma_{k+1})=(-1,+1)$: $s_k(\sigma)$ is obtained from $\sigma$ by replacing
$(\sigma_k,\sigma_{k+1})$ by $(+1,-1)$;
\item for $\sigma_{\el-1}=-1$: $s_{e(\sigma)}(\sigma)$ is obtained from $\sigma$ by replacing $\sigma_{\el-1}$ by $+1$
(here $s_+:=s_\el$ and $s_-:=s_{\el-1}$).
\end{itemize}
A direct computation shows that $b_\sigma A_j(zaq^{-1}t)^{-1}=b_{s_j(\sigma)}$ in each case. Note that $s_{\el-2}$ changes the
sign of $e(\sigma)$, because of the variables $Y_{\el-1}$ and $Y_\el$ in $A_{\el-2}$.\\

Let us prove that the algorithm only produces terms of \eqref{spinorDl}. The algorithm starts at $Y_\el(z)=b_{(-1,\ldots,-1)}$: all the
factors telescope, and $e(-1,\ldots,-1)=+$. If $b_\sigma$ is obtained at some step and is not anti-dominant, then by the above reasoning,
the monomials obtained from $b_\sigma$ at the next step are the $b_{s_j(\sigma)}$. By induction on the step, every monomial
produced by the algorithm is a $b_\sigma$.

Let us prove that every term of \eqref{spinorDl} is produced. Set $\on{ht}(\sigma):=\sum_{j\,:\,\sigma_j=+1}(\el-j)$. Each move raises
$\on{ht}$ by exactly $1$, and $\on{ht}(\sigma)=0$ if and only if $\sigma=(-1,\ldots,-1)$. We prove by induction on $\on{ht}(\sigma)$
that $b_\sigma$ is produced by the algorithm. Let $\on{ht}(\sigma)>0$.
If $\sigma_{\el-1}=+1$, let $\sigma'$ be obtained from $\sigma$ by replacing $\sigma_{\el-1}$ by $-1$; then
$\sigma=s_{e(\sigma')}(\sigma')$.
Otherwise $\sigma_{\el-1}=-1$ and $k:=\max\{j\mid\sigma_j=+1\}\le\el-2$ satisfies $\sigma_{k+1}=-1$; let $\sigma'$ be obtained
from $\sigma$ by replacing $(\sigma_k,\sigma_{k+1})$ by $(-1,+1)$; then $\sigma=s_k(\sigma')$.
In both cases $\on{ht}(\sigma')=\on{ht}(\sigma)-1$, so $b_{\sigma'}$ is produced by the induction hypothesis, and $b_\sigma$ is
obtained from it at the next step.

Hence the monomials produced by the algorithm are exactly the $2^{\el-1}$ terms of \eqref{spinorDl}, which proves the result.
\end{proof}

\subsubsection{Fundamental fields in exceptional types} 
Our algorithm works in type $E_6$ for $i=1,5$, in type $E_7$ for $i=6$, in type $F_4$ for $i=1,4$ and in type $G_2$ for $i=1,2$.
Moreover, the expressions of the fundamental fields in these cases validate Conjecture \ref{conj1}. Our algorithm fails in all other cases. 
The reader can see the expressions of the fundamental fields in these cases in Appendix \ref{appendixefg}.\\

\section{Concluding remarks and perspectives}\label{sect7}

In \cite{MR1646483,MR1633032} or in \cite{frenkelhernandez2011langlands,FHR}, 
it is observed that the deformed $W$-algebra $\WWqt$ interpolates $q$-characters and $t$-(twisted) characters of 
finite-dimensional representations of $U_q(\widehat{\g})$ and 
$U_t({}^L\widehat{\g})$ in the classical limits. 
In this paper, we proved that in types $A_\el,B_\el,C_\el,G_2$, the limit $t\to1$ of the deformed $W$-algebra $\WWqt$ contains exactly all the $q$-characters of fundamental representations of $U_q(\widehat{\g})$.
Moreover, the proof of the well-definedness of our algorithm gives a strong tool to exhibit fields of $\WWqt$ from a given dominant regular generic monomial as shown in Theorem \ref{theoremefinal}.
Hence, this framework opens the way to the study of the classical limits of $\WWqt$. 
This will be studied in an upcoming paper. \\

We can also remark that our algorithm requires the condition of genericity. 
However, non-generic fields appear for example in the work of Kimura and Noshita (see Appendix C.2 in \cite{kimuranoshita}). 
Our proofs of the well-definedness of the algorithm and of the fact that it produces fields of $\WWqt$ fail in the non-generic case. 
Indeed, the one-to-one correspondence between the delta functions explained in the proof of Theorem~\ref{theoremalgowork}(b) requires the genericity condition.
Hence, in this article we restricted ourselves to the generic and without derivatives case, which allowed us to rigorously prove Conjecture \ref{conj1} in various new cases.
We also formulate the following conjecture:
\begin{conj}\label{conjjj}
    \begin{enumerate}
        \item Let $m\in\mathbf{M}$ be a dominant monomial. 
        Then the algorithm starting from $m$ terminates in finitely many steps (with or without a failure).
        \item Let $m \in \mathbf{M}$ be a dominant regular generic monomial. 
        If the algorithm starting from $m$ creates a non-regular monomial, then there exists no field without derivative in $\mathbf{W}_{q,t}(\mathfrak{g})$ with $m$ as its unique dominant monomial.
    \end{enumerate}
\end{conj}
The first point is verified algorithmically for a large number of dominant monomials $m$ in all types of simple Lie algebras $\g$.\\
The second point would imply that Conjecture \ref{conj1} fails in some cases. 
For example in type $D_4$, for $i=2$, the Remark \ref{faild4} would prove that a field with $Y_2(z)$ as its unique dominant monomial does not exist in $\WWqt$ without derivatives. 
We believe this second point is true because of a technical argument. In the computation of the residue of $[\Phi(z),S_i^\pm]$, 
we can prove that the only way to cancel the residue of a term $[M,S_i^\pm]$ is with another term of the form 
$[M\prod_rA_i(zc_r)^{\varepsilon_r},S_i^\pm]$. 
However, it seems that the only way for all terms to cancel each other is by following the algorithm described in Section \ref{sect5}.\\
For instance, we saw that the algorithm fails for all fields for $\g$ of type $E_8$. 
Hence the conjecture would imply that the fields in $\mathbf{W}_{q,t}(E_8)$ without derivative are exactly the constant fields $K$.\\

Finally, Remark \ref{thinh} highlights a link between the deformed $W$-algebras (or their subsets of fields produced by our algorithm) and the thin fundamental representations of quantum affine algebras.
This link is still mysterious and deserves to be studied in depth. However, we observe that our algorithm works only for fundamental fields corresponding to fundamental representations which are thin. 
Thus, we formulate the following conjecture, refining a conjecture of Bouwknegt and Pilch (see Assumption 2.4 in \cite{MR1633032}):
\begin{conj}\label{conjjjjj}
Let $m$ be a dominant monomial whose $t=1$ specialization is the unique dominant monomial of a special thin representation of $U_q(\widehat{\g})$, that is a representation whose $q$-character has a single dominant monomial and all coefficients equal to 1. 
Then, the algorithm successfully generates an element of $\WWqt$ from $m$.
\end{conj}
\begin{rem}
    This article proves this conjecture for fundamental fields (Theorem \ref{theoremefinal}, and see Section 7.3 in \cite{frenkel2015baxter} for the thinness of fundamental representations).
\end{rem}
However, the algorithm also works from monomials which, at the specialization $t=1$, are the dominant monomials of non-thin representations and the thinness is not a necessary condition for the algorithm to work.
For example, in type $A_1$, the algorithm works from $Y(z)Y(zq^{-2})$ which is the dominant monomial of the $q$-character of a thin Kirillov-Reshetikhin module. By symmetry, the algorithm also works from $Y(z)Y(zt^2)$ which specializes at $t=1$ to the dominant monomial of a non-thin representation. 
Another example is given by the dominant monomial $Y(z)Y(zq^{-2})Y(zt^2)Y(zq^{-2}t^2)$ in type $A_1$. The algorithm works from this monomial. However, this monomial specializes at $t=1$ (resp. $q=1$) to the unique dominant monomial of the $q$-character (resp. $t$-character) of a non-thin representation.
\vspace{20 pt}
\begin{center}
\textbf{AI disclosure}
\end{center}
This article was written by the author, and all the ideas and results it contains are the author's own. An AI assistant
was used for proofreading purposes and to correct typographical errors.
\vspace{20 pt}

\begin{appendices}
\section{Explicit expressions for fundamental fields in exceptional types}\label{appendixefg}
    In this section, we list the results of the algorithm computing the monomials and the associated coefficients. All these monomials and coefficients have been computed with SageMath.
\subsection{Type $E$}
\subsubsection{Type $E_6$}
In type $E_6$, the algorithm works and gives fundamental fields associated to $i\in\{1,5\}$. $T_1,T_5$ both have 27 monomials. The algorithm fails for computing $T_i(z)$, $i\neq1,5$.\\
\textbf{Here is the list  of monomials in $T_1(z)$}:
\begin{itemize}[leftmargin=3cm, labelsep=0.5cm]
  \item[Monomial~1:] $ 1 \cdot:Y_{1}(z): $
  \item[Monomial~2:] $ 1 \cdot:Y_{1}(zq^{-2}t^{2})^{-1}  Y_{2}(zq^{-1}t): $
  \item[Monomial~3:] $ 1 \cdot:Y_{2}(zq^{-3}t^{3})^{-1}  Y_{3}(zq^{-2}t^{2}): $
  \item[Monomial~4:] $ 1 \cdot:Y_{3}(zq^{-4}t^{4})^{-1}  Y_{4}(zq^{-3}t^{3})  Y_{6}(zq^{-3}t^{3}): $
  \item[Monomial~5:] $ 1 \cdot:Y_{4}(zq^{-3}t^{3})  Y_{6}(zq^{-5}t^{5})^{-1}: $
  \item[Monomial~6:] $ 1 \cdot:Y_{4}(zq^{-5}t^{5})^{-1}  Y_{5}(zq^{-4}t^{4})  Y_{6}(zq^{-3}t^{3}): $
  \item[Monomial~7:] $ 1 \cdot:Y_{3}(zq^{-4}t^{4})  Y_{4}(zq^{-5}t^{5})^{-1}  Y_{5}(zq^{-4}t^{4})  Y_{6}(zq^{-5}t^{5})^{-1}: $
  \item[Monomial~8:] $ 1 \cdot:Y_{5}(zq^{-6}t^{6})^{-1}  Y_{6}(zq^{-3}t^{3}): $
  \item[Monomial~9:] $ 1 \cdot:Y_{2}(zq^{-5}t^{5})  Y_{3}(zq^{-6}t^{6})^{-1}  Y_{5}(zq^{-4}t^{4}): $
  \item[Monomial~10:] $ 1 \cdot:Y_{3}(zq^{-4}t^{4})  Y_{5}(zq^{-6}t^{6})^{-1}  Y_{6}(zq^{-5}t^{5})^{-1}: $
  \item[Monomial~11:] $ 1 \cdot:Y_{2}(zq^{-5}t^{5})  Y_{3}(zq^{-6}t^{6})^{-1}  Y_{4}(zq^{-5}t^{5})  Y_{5}(zq^{-6}t^{6})^{-1}: $
  \item[Monomial~12:] $ 1 \cdot:Y_{1}(zq^{-6}t^{6})  Y_{2}(zq^{-7}t^{7})^{-1}  Y_{5}(zq^{-4}t^{4}): $
  \item[Monomial~13:] $ 1 \cdot:Y_{1}(zq^{-6}t^{6})  Y_{2}(zq^{-7}t^{7})^{-1}  Y_{4}(zq^{-5}t^{5})  Y_{5}(zq^{-6}t^{6})^{-1}: $
  \item[Monomial~14:] $ 1 \cdot:Y_{2}(zq^{-5}t^{5})  Y_{4}(zq^{-7}t^{7})^{-1}: $
  \item[Monomial~15:] $ 1 \cdot:Y_{1}(zq^{-8}t^{8})^{-1}  Y_{5}(zq^{-4}t^{4}): $
  \item[Monomial~16:] $ 1 \cdot:Y_{1}(zq^{-6}t^{6})  Y_{2}(zq^{-7}t^{7})^{-1}  Y_{3}(zq^{-6}t^{6})  Y_{4}(zq^{-7}t^{7})^{-1}: $
  \item[Monomial~17:] $ 1 \cdot:Y_{1}(zq^{-8}t^{8})^{-1}  Y_{4}(zq^{-5}t^{5})  Y_{5}(zq^{-6}t^{6})^{-1}: $
  \item[Monomial~18:] $ 1 \cdot:Y_{1}(zq^{-8}t^{8})^{-1}  Y_{3}(zq^{-6}t^{6})  Y_{4}(zq^{-7}t^{7})^{-1}: $
  \item[Monomial~19:] $ 1 \cdot:Y_{1}(zq^{-6}t^{6})  Y_{3}(zq^{-8}t^{8})^{-1}  Y_{6}(zq^{-7}t^{7}): $
  \item[Monomial~20:] $ 1 \cdot:Y_{1}(zq^{-8}t^{8})^{-1}  Y_{2}(zq^{-7}t^{7})  Y_{3}(zq^{-8}t^{8})^{-1}  Y_{6}(zq^{-7}t^{7}): $
  \item[Monomial~21:] $ 1 \cdot:Y_{1}(zq^{-6}t^{6})  Y_{6}(zq^{-9}t^{9})^{-1}: $
  \item[Monomial~22:] $ 1 \cdot:Y_{1}(zq^{-8}t^{8})^{-1}  Y_{2}(zq^{-7}t^{7})  Y_{6}(zq^{-9}t^{9})^{-1}: $
  \item[Monomial~23:] $ 1 \cdot:Y_{2}(zq^{-9}t^{9})^{-1}  Y_{6}(zq^{-7}t^{7}): $
  \item[Monomial~24:] $ 1 \cdot:Y_{2}(zq^{-9}t^{9})^{-1}  Y_{3}(zq^{-8}t^{8})  Y_{6}(zq^{-9}t^{9})^{-1}: $
  \item[Monomial~25:] $ 1 \cdot:Y_{3}(zq^{-10}t^{10})^{-1}  Y_{4}(zq^{-9}t^{9}): $
  \item[Monomial~26:] $ 1 \cdot:Y_{4}(zq^{-11}t^{11})^{-1}  Y_{5}(zq^{-10}t^{10}): $
  \item[Monomial~27:] $ 1 \cdot:Y_{5}(zq^{-12}t^{12})^{-1}: $
\end{itemize}
\textbf{Here is the list  of monomials in $T_5(z)$}:
\begin{itemize}[leftmargin=3cm, labelsep=0.5cm]
  \item[Monomial~1:] $ 1 \cdot:Y_{5}(z): $
  \item[Monomial~2:] $ 1 \cdot:Y_{4}(zq^{-1}t)  Y_{5}(zq^{-2}t^{2})^{-1}: $
  \item[Monomial~3:] $ 1 \cdot:Y_{3}(zq^{-2}t^{2})  Y_{4}(zq^{-3}t^{3})^{-1}: $
  \item[Monomial~4:] $ 1 \cdot:Y_{2}(zq^{-3}t^{3})  Y_{3}(zq^{-4}t^{4})^{-1}  Y_{6}(zq^{-3}t^{3}): $
  \item[Monomial~5:] $ 1 \cdot:Y_{2}(zq^{-3}t^{3})  Y_{6}(zq^{-5}t^{5})^{-1}: $
  \item[Monomial~6:] $ 1 \cdot:Y_{1}(zq^{-4}t^{4})  Y_{2}(zq^{-5}t^{5})^{-1}  Y_{6}(zq^{-3}t^{3}): $
  \item[Monomial~7:] $ 1 \cdot:Y_{1}(zq^{-4}t^{4})  Y_{2}(zq^{-5}t^{5})^{-1}  Y_{3}(zq^{-4}t^{4})  Y_{6}(zq^{-5}t^{5})^{-1}: $
  \item[Monomial~8:] $ 1 \cdot:Y_{1}(zq^{-6}t^{6})^{-1}  Y_{6}(zq^{-3}t^{3}): $
  \item[Monomial~9:] $ 1 \cdot:Y_{1}(zq^{-6}t^{6})^{-1}  Y_{3}(zq^{-4}t^{4})  Y_{6}(zq^{-5}t^{5})^{-1}: $
  \item[Monomial~10:] $ 1 \cdot:Y_{1}(zq^{-4}t^{4})  Y_{3}(zq^{-6}t^{6})^{-1}  Y_{4}(zq^{-5}t^{5}): $
  \item[Monomial~11:] $ 1 \cdot:Y_{1}(zq^{-6}t^{6})^{-1}  Y_{2}(zq^{-5}t^{5})  Y_{3}(zq^{-6}t^{6})^{-1}  Y_{4}(zq^{-5}t^{5}): $
  \item[Monomial~12:] $ 1 \cdot:Y_{1}(zq^{-4}t^{4})  Y_{4}(zq^{-7}t^{7})^{-1}  Y_{5}(zq^{-6}t^{6}): $
  \item[Monomial~13:] $ 1 \cdot:Y_{1}(zq^{-6}t^{6})^{-1}  Y_{2}(zq^{-5}t^{5})  Y_{4}(zq^{-7}t^{7})^{-1}  Y_{5}(zq^{-6}t^{6}): $
  \item[Monomial~14:] $ 1 \cdot:Y_{2}(zq^{-7}t^{7})^{-1}  Y_{4}(zq^{-5}t^{5}): $
  \item[Monomial~15:] $ 1 \cdot:Y_{1}(zq^{-4}t^{4})  Y_{5}(zq^{-8}t^{8})^{-1}: $
  \item[Monomial~16:] $ 1 \cdot:Y_{2}(zq^{-7}t^{7})^{-1}  Y_{3}(zq^{-6}t^{6})  Y_{4}(zq^{-7}t^{7})^{-1}  Y_{5}(zq^{-6}t^{6}): $
  \item[Monomial~17:] $ 1 \cdot:Y_{1}(zq^{-6}t^{6})^{-1}  Y_{2}(zq^{-5}t^{5})  Y_{5}(zq^{-8}t^{8})^{-1}: $
  \item[Monomial~18:] $ 1 \cdot:Y_{2}(zq^{-7}t^{7})^{-1}  Y_{3}(zq^{-6}t^{6})  Y_{5}(zq^{-8}t^{8})^{-1}: $
  \item[Monomial~19:] $ 1 \cdot:Y_{3}(zq^{-8}t^{8})^{-1}  Y_{5}(zq^{-6}t^{6})  Y_{6}(zq^{-7}t^{7}): $
  \item[Monomial~20:] $ 1 \cdot:Y_{3}(zq^{-8}t^{8})^{-1}  Y_{4}(zq^{-7}t^{7})  Y_{5}(zq^{-8}t^{8})^{-1}  Y_{6}(zq^{-7}t^{7}): $
  \item[Monomial~21:] $ 1 \cdot:Y_{5}(zq^{-6}t^{6})  Y_{6}(zq^{-9}t^{9})^{-1}: $
  \item[Monomial~22:] $ 1 \cdot:Y_{4}(zq^{-7}t^{7})  Y_{5}(zq^{-8}t^{8})^{-1}  Y_{6}(zq^{-9}t^{9})^{-1}: $
  \item[Monomial~23:] $ 1 \cdot:Y_{4}(zq^{-9}t^{9})^{-1}  Y_{6}(zq^{-7}t^{7}): $
  \item[Monomial~24:] $ 1 \cdot:Y_{3}(zq^{-8}t^{8})  Y_{4}(zq^{-9}t^{9})^{-1}  Y_{6}(zq^{-9}t^{9})^{-1}: $
  \item[Monomial~25:] $ 1 \cdot:Y_{2}(zq^{-9}t^{9})  Y_{3}(zq^{-10}t^{10})^{-1}: $
  \item[Monomial~26:] $ 1 \cdot:Y_{1}(zq^{-10}t^{10})  Y_{2}(zq^{-11}t^{11})^{-1}: $
  \item[Monomial~27:] $ 1 \cdot:Y_{1}(zq^{-12}t^{12})^{-1}: $
\end{itemize}

\subsubsection{Type $E_7$}
In type $E_7$, the algorithm works and gives a fundamental field associated to $i=6$ which contains 56 monomials. The algorithm fails for computing $T_i(z)$, $i\neq6$.\\
\textbf{Here is the list  of monomials in $T_6(z)$}:
\begin{itemize}[leftmargin=3cm, labelsep=0.5cm]
  \item[Monomial~1:] $ 1 \cdot:Y_{6}(z): $
  \item[Monomial~2:] $ 1 \cdot:Y_{5}(zq^{-1}t)  Y_{6}(zq^{-2}t^{2})^{-1}: $
  \item[Monomial~3:] $ 1 \cdot:Y_{4}(zq^{-2}t^{2})  Y_{5}(zq^{-3}t^{3})^{-1}: $
  \item[Monomial~4:] $ 1 \cdot:Y_{3}(zq^{-3}t^{3})  Y_{4}(zq^{-4}t^{4})^{-1}: $
  \item[Monomial~5:] $ 1 \cdot:Y_{2}(zq^{-4}t^{4})  Y_{3}(zq^{-5}t^{5})^{-1}  Y_{7}(zq^{-4}t^{4}): $
  \item[Monomial~6:] $ 1 \cdot:Y_{2}(zq^{-4}t^{4})  Y_{7}(zq^{-6}t^{6})^{-1}: $
  \item[Monomial~7:] $ 1 \cdot:Y_{1}(zq^{-5}t^{5})  Y_{2}(zq^{-6}t^{6})^{-1}  Y_{7}(zq^{-4}t^{4}): $
  \item[Monomial~8:] $ 1 \cdot:Y_{1}(zq^{-5}t^{5})  Y_{2}(zq^{-6}t^{6})^{-1}  Y_{3}(zq^{-5}t^{5})  Y_{7}(zq^{-6}t^{6})^{-1}: $
  \item[Monomial~9:] $ 1 \cdot:Y_{1}(zq^{-7}t^{7})^{-1}  Y_{7}(zq^{-4}t^{4}): $
  \item[Monomial~10:] $ 1 \cdot:Y_{1}(zq^{-7}t^{7})^{-1}  Y_{3}(zq^{-5}t^{5})  Y_{7}(zq^{-6}t^{6})^{-1}: $
  \item[Monomial~11:] $ 1 \cdot:Y_{1}(zq^{-5}t^{5})  Y_{3}(zq^{-7}t^{7})^{-1}  Y_{4}(zq^{-6}t^{6}): $
  \item[Monomial~12:] $ 1 \cdot:Y_{1}(zq^{-7}t^{7})^{-1}  Y_{2}(zq^{-6}t^{6})  Y_{3}(zq^{-7}t^{7})^{-1}  Y_{4}(zq^{-6}t^{6}): $
  \item[Monomial~13:] $ 1 \cdot:Y_{1}(zq^{-5}t^{5})  Y_{4}(zq^{-8}t^{8})^{-1}  Y_{5}(zq^{-7}t^{7}): $
  \item[Monomial~14:] $ 1 \cdot:Y_{1}(zq^{-7}t^{7})^{-1}  Y_{2}(zq^{-6}t^{6})  Y_{4}(zq^{-8}t^{8})^{-1}  Y_{5}(zq^{-7}t^{7}): $
  \item[Monomial~15:] $ 1 \cdot:Y_{2}(zq^{-8}t^{8})^{-1}  Y_{4}(zq^{-6}t^{6}): $
  \item[Monomial~16:] $ 1 \cdot:Y_{1}(zq^{-5}t^{5})  Y_{5}(zq^{-9}t^{9})^{-1}  Y_{6}(zq^{-8}t^{8}): $
  \item[Monomial~17:] $ 1 \cdot:Y_{2}(zq^{-8}t^{8})^{-1}  Y_{3}(zq^{-7}t^{7})  Y_{4}(zq^{-8}t^{8})^{-1}  Y_{5}(zq^{-7}t^{7}): $
  \item[Monomial~18:] $ 1 \cdot:Y_{1}(zq^{-7}t^{7})^{-1}  Y_{2}(zq^{-6}t^{6})  Y_{5}(zq^{-9}t^{9})^{-1}  Y_{6}(zq^{-8}t^{8}): $
  \item[Monomial~19:] $ 1 \cdot:Y_{1}(zq^{-5}t^{5})  Y_{6}(zq^{-10}t^{10})^{-1}: $
  \item[Monomial~20:] $ 1 \cdot:Y_{2}(zq^{-8}t^{8})^{-1}  Y_{3}(zq^{-7}t^{7})  Y_{5}(zq^{-9}t^{9})^{-1}  Y_{6}(zq^{-8}t^{8}): $
  \item[Monomial~21:] $ 1 \cdot:Y_{3}(zq^{-9}t^{9})^{-1}  Y_{5}(zq^{-7}t^{7})  Y_{7}(zq^{-8}t^{8}): $
  \item[Monomial~22:] $ 1 \cdot:Y_{1}(zq^{-7}t^{7})^{-1}  Y_{2}(zq^{-6}t^{6})  Y_{6}(zq^{-10}t^{10})^{-1}: $
  \item[Monomial~23:] $ 1 \cdot:Y_{3}(zq^{-9}t^{9})^{-1}  Y_{4}(zq^{-8}t^{8})  Y_{5}(zq^{-9}t^{9})^{-1}  Y_{6}(zq^{-8}t^{8})  Y_{7}(zq^{-8}t^{8}): $
  \item[Monomial~24:] $ 1 \cdot:Y_{2}(zq^{-8}t^{8})^{-1}  Y_{3}(zq^{-7}t^{7})  Y_{6}(zq^{-10}t^{10})^{-1}: $
  \item[Monomial~25:] $ 1 \cdot:Y_{5}(zq^{-7}t^{7})  Y_{7}(zq^{-10}t^{10})^{-1}: $
  \item[Monomial~26:] $ 1 \cdot:Y_{3}(zq^{-9}t^{9})^{-1}  Y_{4}(zq^{-8}t^{8})  Y_{6}(zq^{-10}t^{10})^{-1}  Y_{7}(zq^{-8}t^{8}): $
  \item[Monomial~27:] $ 1 \cdot:Y_{4}(zq^{-8}t^{8})  Y_{5}(zq^{-9}t^{9})^{-1}  Y_{6}(zq^{-8}t^{8})  Y_{7}(zq^{-10}t^{10})^{-1}: $
  \item[Monomial~28:] $ 1 \cdot:Y_{4}(zq^{-10}t^{10})^{-1}  Y_{6}(zq^{-8}t^{8})  Y_{7}(zq^{-8}t^{8}): $
  \item[Monomial~29:] $ 1 \cdot:Y_{4}(zq^{-8}t^{8})  Y_{6}(zq^{-10}t^{10})^{-1}  Y_{7}(zq^{-10}t^{10})^{-1}: $
  \item[Monomial~30:] $ 1 \cdot:Y_{4}(zq^{-10}t^{10})^{-1}  Y_{5}(zq^{-9}t^{9})  Y_{6}(zq^{-10}t^{10})^{-1}  Y_{7}(zq^{-8}t^{8}): $
  \item[Monomial~31:] $ 1 \cdot:Y_{3}(zq^{-9}t^{9})  Y_{4}(zq^{-10}t^{10})^{-1}  Y_{6}(zq^{-8}t^{8})  Y_{7}(zq^{-10}t^{10})^{-1}: $
  \item[Monomial~32:] $ 1 \cdot:Y_{3}(zq^{-9}t^{9})  Y_{4}(zq^{-10}t^{10})^{-1}  Y_{5}(zq^{-9}t^{9})  Y_{6}(zq^{-10}t^{10})^{-1}  Y_{7}(zq^{-10}t^{10})^{-1}: $
  \item[Monomial~33:] $ 1 \cdot:Y_{5}(zq^{-11}t^{11})^{-1}  Y_{7}(zq^{-8}t^{8}): $
  \item[Monomial~34:] $ 1 \cdot:Y_{2}(zq^{-10}t^{10})  Y_{3}(zq^{-11}t^{11})^{-1}  Y_{6}(zq^{-8}t^{8}): $
  \item[Monomial~35:] $ 1 \cdot:Y_{2}(zq^{-10}t^{10})  Y_{3}(zq^{-11}t^{11})^{-1}  Y_{5}(zq^{-9}t^{9})  Y_{6}(zq^{-10}t^{10})^{-1}: $
  \item[Monomial~36:] $ 1 \cdot:Y_{3}(zq^{-9}t^{9})  Y_{5}(zq^{-11}t^{11})^{-1}  Y_{7}(zq^{-10}t^{10})^{-1}: $
  \item[Monomial~37:] $ 1 \cdot:Y_{1}(zq^{-11}t^{11})  Y_{2}(zq^{-12}t^{12})^{-1}  Y_{6}(zq^{-8}t^{8}): $
  \item[Monomial~38:] $ 1 \cdot:Y_{2}(zq^{-10}t^{10})  Y_{3}(zq^{-11}t^{11})^{-1}  Y_{4}(zq^{-10}t^{10})  Y_{5}(zq^{-11}t^{11})^{-1}: $
  \item[Monomial~39:] $ 1 \cdot:Y_{1}(zq^{-11}t^{11})  Y_{2}(zq^{-12}t^{12})^{-1}  Y_{5}(zq^{-9}t^{9})  Y_{6}(zq^{-10}t^{10})^{-1}: $
  \item[Monomial~40:] $ 1 \cdot:Y_{1}(zq^{-13}t^{13})^{-1}  Y_{6}(zq^{-8}t^{8}): $
  \item[Monomial~41:] $ 1 \cdot:Y_{1}(zq^{-11}t^{11})  Y_{2}(zq^{-12}t^{12})^{-1}  Y_{4}(zq^{-10}t^{10})  Y_{5}(zq^{-11}t^{11})^{-1}: $
  \item[Monomial~42:] $ 1 \cdot:Y_{2}(zq^{-10}t^{10})  Y_{4}(zq^{-12}t^{12})^{-1}: $
  \item[Monomial~43:] $ 1 \cdot:Y_{1}(zq^{-13}t^{13})^{-1}  Y_{5}(zq^{-9}t^{9})  Y_{6}(zq^{-10}t^{10})^{-1}: $
  \item[Monomial~44:] $ 1 \cdot:Y_{1}(zq^{-11}t^{11})  Y_{2}(zq^{-12}t^{12})^{-1}  Y_{3}(zq^{-11}t^{11})  Y_{4}(zq^{-12}t^{12})^{-1}: $
  \item[Monomial~45:] $ 1 \cdot:Y_{1}(zq^{-13}t^{13})^{-1}  Y_{4}(zq^{-10}t^{10})  Y_{5}(zq^{-11}t^{11})^{-1}: $
  \item[Monomial~46:] $ 1 \cdot:Y_{1}(zq^{-13}t^{13})^{-1}  Y_{3}(zq^{-11}t^{11})  Y_{4}(zq^{-12}t^{12})^{-1}: $
  \item[Monomial~47:] $ 1 \cdot:Y_{1}(zq^{-11}t^{11})  Y_{3}(zq^{-13}t^{13})^{-1}  Y_{7}(zq^{-12}t^{12}): $
  \item[Monomial~48:] $ 1 \cdot:Y_{1}(zq^{-13}t^{13})^{-1}  Y_{2}(zq^{-12}t^{12})  Y_{3}(zq^{-13}t^{13})^{-1}  Y_{7}(zq^{-12}t^{12}): $
  \item[Monomial~49:] $ 1 \cdot:Y_{1}(zq^{-11}t^{11})  Y_{7}(zq^{-14}t^{14})^{-1}: $
  \item[Monomial~50:] $ 1 \cdot:Y_{1}(zq^{-13}t^{13})^{-1}  Y_{2}(zq^{-12}t^{12})  Y_{7}(zq^{-14}t^{14})^{-1}: $
  \item[Monomial~51:] $ 1 \cdot:Y_{2}(zq^{-14}t^{14})^{-1}  Y_{7}(zq^{-12}t^{12}): $
  \item[Monomial~52:] $ 1 \cdot:Y_{2}(zq^{-14}t^{14})^{-1}  Y_{3}(zq^{-13}t^{13})  Y_{7}(zq^{-14}t^{14})^{-1}: $
  \item[Monomial~53:] $ 1 \cdot:Y_{3}(zq^{-15}t^{15})^{-1}  Y_{4}(zq^{-14}t^{14}): $
  \item[Monomial~54:] $ 1 \cdot:Y_{4}(zq^{-16}t^{16})^{-1}  Y_{5}(zq^{-15}t^{15}): $
  \item[Monomial~55:] $ 1 \cdot:Y_{5}(zq^{-17}t^{17})^{-1}  Y_{6}(zq^{-16}t^{16}): $
  \item[Monomial~56:] $ 1 \cdot:Y_{6}(zq^{-18}t^{18})^{-1}: $
\end{itemize}
\subsubsection{Type $E_8$}
In type $E_8$, the algorithm fails for all $1\leq i\leq8$.\\

\subsection{Type $F_4$}

In type $F_4$, there exist fundamental fields associated to the first and fourth fundamental modules of $U_q(\widehat{F_4})$.
The algorithm fails to compute $T_2(z)$, $T_3(z)$.

\textbf{Here is the list  of monomials in $T_1(z)$}:
\bigskip
\begin{itemize}[leftmargin=3cm, labelsep=0.5cm]
  \item[Monomial~1:] $ 1 \cdot:Y_{1}(z): $
  \item[Monomial~2:] $ 1 \cdot:Y_{1}(zq^{-2}t^{2})^{-1}  Y_{2}(zq^{-1}t): $
  \item[Monomial~3:] $ 1 \cdot:Y_{2}(zq^{-3}t^{3})^{-1}  Y_{3}(zq^{-2}t^{2}): $
  \item[Monomial~4:] $ 1 \cdot:Y_{2}(zq^{-5}t^{3})  Y_{3}(zq^{-6}t^{4})^{-1}  Y_{4}(zq^{-4}t^{3}): $
  \item[Monomial~5:] $ 1 \cdot:Y_{1}(zq^{-6}t^{4})  Y_{2}(zq^{-7}t^{5})^{-1}  Y_{4}(zq^{-4}t^{3}): $
  \item[Monomial~6:] $ 1 \cdot:Y_{2}(zq^{-5}t^{3})  Y_{4}(zq^{-8}t^{5})^{-1}: $
  \item[Monomial~7:] $ 1 \cdot:Y_{1}(zq^{-6}t^{4})  Y_{2}(zq^{-7}t^{5})^{-1}  Y_{3}(zq^{-6}t^{4})  Y_{4}(zq^{-8}t^{5})^{-1}: $
  \item[Monomial~8:] $ 1 \cdot:Y_{1}(zq^{-8}t^{6})^{-1}  Y_{4}(zq^{-4}t^{3}): $
  \item[Monomial~9:] $ 1 \cdot:Y_{1}(zq^{-8}t^{6})^{-1}  Y_{3}(zq^{-6}t^{4})  Y_{4}(zq^{-8}t^{5})^{-1}: $
  \item[Monomial~10:] $ 1 \cdot:Y_{1}(zq^{-6}t^{4})  Y_{2}(zq^{-9}t^{5})  Y_{3}(zq^{-10}t^{6})^{-1}: $
  \item[Monomial~11:] $ 1 \cdot:Y_{1}(zq^{-8}t^{6})^{-1}  Y_{2}(zq^{-9}t^{5})  Y_{2}(zq^{-7}t^{5})  Y_{3}(zq^{-10}t^{6})^{-1}: $
  \item[Monomial~12:] $ 1 \cdot:Y_{1}(zq^{-10}t^{6})  Y_{1}(zq^{-6}t^{4})  Y_{2}(zq^{-11}t^{7})^{-1}: $
  \item[Monomial~13:] $ \frac{{\left(q^{2} + 1\right)} {\left(q + t\right)} {\left(q - t\right)}}{{\left(q^{2} + t\right)} {\left(q^{2} - t\right)}} \cdot:Y_{1}(zq^{-10}t^{6})  Y_{1}(zq^{-8}t^{6})^{-1}  Y_{2}(zq^{-11}t^{7})^{-1}  Y_{2}(zq^{-7}t^{5}): $
  \item[Monomial~14:] $ \frac{{\left(q^{3} + t\right)} {\left(q^{3} - t\right)} {\left(q^{2} + t^{2}\right)} {\left(q + t\right)} {\left(q - t\right)}}{{\left(q^{3} + t^{2}\right)} {\left(q^{3} - t^{2}\right)} {\left(q^{2} + t\right)} {\left(q^{2} - t\right)}} \cdot:Y_{1}(zq^{-12}t^{8})^{-1}  Y_{1}(zq^{-6}t^{4}): $
  \item[Monomial~15:] $ 1 \cdot:Y_{1}(zq^{-10}t^{6})  Y_{2}(zq^{-11}t^{7})^{-1}  Y_{2}(zq^{-9}t^{7})^{-1}  Y_{3}(zq^{-8}t^{6}): $
  \item[Monomial~16:] $ 1 \cdot:Y_{1}(zq^{-12}t^{8})^{-1}  Y_{1}(zq^{-8}t^{6})^{-1}  Y_{2}(zq^{-7}t^{5}): $
  \item[Monomial~17:] $ 1 \cdot:Y_{1}(zq^{-12}t^{8})^{-1}  Y_{2}(zq^{-9}t^{7})^{-1}  Y_{3}(zq^{-8}t^{6}): $
  \item[Monomial~18:] $ 1 \cdot:Y_{1}(zq^{-10}t^{6})  Y_{3}(zq^{-12}t^{8})^{-1}  Y_{4}(zq^{-10}t^{7}): $
  \item[Monomial~19:] $ 1 \cdot:Y_{1}(zq^{-12}t^{8})^{-1}  Y_{2}(zq^{-11}t^{7})  Y_{3}(zq^{-12}t^{8})^{-1}  Y_{4}(zq^{-10}t^{7}): $
  \item[Monomial~20:] $ 1 \cdot:Y_{1}(zq^{-10}t^{6})  Y_{4}(zq^{-14}t^{9})^{-1}: $
  \item[Monomial~21:] $ 1 \cdot:Y_{2}(zq^{-13}t^{9})^{-1}  Y_{4}(zq^{-10}t^{7}): $
  \item[Monomial~22:] $ 1 \cdot:Y_{1}(zq^{-12}t^{8})^{-1}  Y_{2}(zq^{-11}t^{7})  Y_{4}(zq^{-14}t^{9})^{-1}: $
  \item[Monomial~23:] $ 1 \cdot:Y_{2}(zq^{-13}t^{9})^{-1}  Y_{3}(zq^{-12}t^{8})  Y_{4}(zq^{-14}t^{9})^{-1}: $
  \item[Monomial~24:] $ 1 \cdot:Y_{2}(zq^{-15}t^{9})  Y_{3}(zq^{-16}t^{10})^{-1}: $
  \item[Monomial~25:] $ 1 \cdot:Y_{1}(zq^{-16}t^{10})  Y_{2}(zq^{-17}t^{11})^{-1}: $
  \item[Monomial~26:] $ 1 \cdot:Y_{1}(zq^{-18}t^{12})^{-1}: $
\end{itemize}
\bigskip
\textbf{Here is the list  of monomials in $T_4(z)$}:
\begin{itemize}[leftmargin=3cm, labelsep=0.5cm]
  \item[Monomial~1:] $ 1 \cdot:Y_{4}(z): $
  \item[Monomial~2:] $ 1 \cdot:Y_{3}(zq^{-2}t)  Y_{4}(zq^{-4}t^{2})^{-1}: $
  \item[Monomial~3:] $ 1 \cdot:Y_{2}(zq^{-5}t^{2})  Y_{2}(zq^{-3}t^{2})  Y_{3}(zq^{-6}t^{3})^{-1}: $
  \item[Monomial~4:] $ \frac{{\left(q^{2} + 1\right)} {\left(q + t\right)} {\left(q - t\right)}}{{\left(q^{2} + t\right)} {\left(q^{2} - t\right)}} \cdot:Y_{1}(zq^{-6}t^{3})  Y_{2}(zq^{-7}t^{4})^{-1}  Y_{2}(zq^{-3}t^{2}): $
  \item[Monomial~5:] $ 1 \cdot:Y_{1}(zq^{-6}t^{3})  Y_{1}(zq^{-4}t^{3})  Y_{2}(zq^{-7}t^{4})^{-1}  Y_{2}(zq^{-5}t^{4})^{-1}  Y_{3}(zq^{-4}t^{3}): $
  \item[Monomial~6:] $ \frac{{\left(q^{2} + 1\right)} {\left(q + t\right)} {\left(q - t\right)}}{{\left(q^{2} + t\right)} {\left(q^{2} - t\right)}} \cdot:Y_{1}(zq^{-8}t^{5})^{-1}  Y_{2}(zq^{-3}t^{2}): $
  \item[Monomial~7:] $ \frac{{\left(q^{2} + 1\right)} {\left(q + t\right)} {\left(q - t\right)}}{{\left(q^{2} + t\right)} {\left(q^{2} - t\right)}} \cdot:Y_{1}(zq^{-8}t^{5})^{-1}  Y_{1}(zq^{-4}t^{3})  Y_{2}(zq^{-5}t^{4})^{-1}  Y_{3}(zq^{-4}t^{3}): $
  \item[Monomial~8:] $ 1 \cdot:Y_{1}(zq^{-6}t^{3})  Y_{1}(zq^{-4}t^{3})  Y_{3}(zq^{-8}t^{5})^{-1}  Y_{4}(zq^{-6}t^{4}): $
  \item[Monomial~9:] $ 1 \cdot:Y_{1}(zq^{-8}t^{5})^{-1}  Y_{1}(zq^{-6}t^{5})^{-1}  Y_{3}(zq^{-4}t^{3}): $
  \item[Monomial~10:] $ \frac{{\left(q^{2} + 1\right)} {\left(q + t\right)} {\left(q - t\right)}}{{\left(q^{2} + t\right)} {\left(q^{2} - t\right)}} \cdot:Y_{1}(zq^{-8}t^{5})^{-1}  Y_{1}(zq^{-4}t^{3})  Y_{2}(zq^{-7}t^{4})  Y_{3}(zq^{-8}t^{5})^{-1}  Y_{4}(zq^{-6}t^{4}): $
  \item[Monomial~11:] $ 1 \cdot:Y_{1}(zq^{-6}t^{3})  Y_{1}(zq^{-4}t^{3})  Y_{4}(zq^{-10}t^{6})^{-1}: $
  \item[Monomial~12:] $ 1 \cdot:Y_{1}(zq^{-8}t^{5})^{-1}  Y_{1}(zq^{-6}t^{5})^{-1}  Y_{2}(zq^{-7}t^{4})  Y_{2}(zq^{-5}t^{4})  Y_{3}(zq^{-8}t^{5})^{-1}  Y_{4}(zq^{-6}t^{4}): $
  \item[Monomial~13:] $ \frac{{\left(q^{2} + 1\right)} {\left(q + t\right)} {\left(q - t\right)}}{{\left(q^{2} + t\right)} {\left(q^{2} - t\right)}} \cdot:Y_{1}(zq^{-4}t^{3})  Y_{2}(zq^{-9}t^{6})^{-1}  Y_{4}(zq^{-6}t^{4}): $
  \item[Monomial~14:] $ \frac{{\left(q^{2} + 1\right)} {\left(q + t\right)} {\left(q - t\right)}}{{\left(q^{2} + t\right)} {\left(q^{2} - t\right)}} \cdot:Y_{1}(zq^{-8}t^{5})^{-1}  Y_{1}(zq^{-4}t^{3})  Y_{2}(zq^{-7}t^{4})  Y_{4}(zq^{-10}t^{6})^{-1}: $
  \item[Monomial~15:] $ \frac{{\left(q^{2} + 1\right)} {\left(q + t\right)} {\left(q - t\right)}}{{\left(q^{2} + t\right)} {\left(q^{2} - t\right)}} \cdot:Y_{1}(zq^{-6}t^{5})^{-1}  Y_{2}(zq^{-9}t^{6})^{-1}  Y_{2}(zq^{-5}t^{4})  Y_{4}(zq^{-6}t^{4}): $
  \item[Monomial~16:] $ 1 \cdot:Y_{1}(zq^{-8}t^{5})^{-1}  Y_{1}(zq^{-6}t^{5})^{-1}  Y_{2}(zq^{-7}t^{4})  Y_{2}(zq^{-5}t^{4})  Y_{4}(zq^{-10}t^{6})^{-1}: $
  \item[Monomial~17:] $ \frac{{\left(q^{2} + 1\right)} {\left(q + t\right)} {\left(q - t\right)}}{{\left(q^{2} + t\right)} {\left(q^{2} - t\right)}} \cdot:Y_{1}(zq^{-4}t^{3})  Y_{2}(zq^{-9}t^{6})^{-1}  Y_{3}(zq^{-8}t^{5})  Y_{4}(zq^{-10}t^{6})^{-1}: $
  \item[Monomial~18:] $ 1 \cdot:Y_{2}(zq^{-9}t^{6})^{-1}  Y_{2}(zq^{-7}t^{6})^{-1}  Y_{3}(zq^{-6}t^{5})  Y_{4}(zq^{-6}t^{4}): $
  \item[Monomial~19:] $ \frac{{\left(q^{2} + 1\right)} {\left(q + t\right)} {\left(q - t\right)}}{{\left(q^{2} + t\right)} {\left(q^{2} - t\right)}} \cdot:Y_{1}(zq^{-6}t^{5})^{-1}  Y_{2}(zq^{-9}t^{6})^{-1}  Y_{2}(zq^{-5}t^{4})  Y_{3}(zq^{-8}t^{5})  Y_{4}(zq^{-10}t^{6})^{-1}: $
  \item[Monomial~20:] $ \frac{{\left(q^{2} + 1\right)} {\left(q + t\right)} {\left(q - t\right)}}{{\left(q^{2} + t\right)} {\left(q^{2} - t\right)}} \cdot:Y_{1}(zq^{-4}t^{3})  Y_{2}(zq^{-11}t^{6})  Y_{3}(zq^{-12}t^{7})^{-1}: $
  \item[Monomial~21:] $ 1 \cdot:Y_{2}(zq^{-9}t^{6})^{-1}  Y_{2}(zq^{-7}t^{6})^{-1}  Y_{3}(zq^{-8}t^{5})  Y_{3}(zq^{-6}t^{5})  Y_{4}(zq^{-10}t^{6})^{-1}: $
  \item[Monomial~22:] $ 1 \cdot:Y_{3}(zq^{-10}t^{7})^{-1}  Y_{4}(zq^{-8}t^{6})  Y_{4}(zq^{-6}t^{4}): $
  \item[Monomial~23:] $ \frac{{\left(q^{2} + 1\right)} {\left(q + t\right)} {\left(q - t\right)}}{{\left(q^{2} + t\right)} {\left(q^{2} - t\right)}} \cdot:Y_{1}(zq^{-6}t^{5})^{-1}  Y_{2}(zq^{-11}t^{6})  Y_{2}(zq^{-5}t^{4})  Y_{3}(zq^{-12}t^{7})^{-1}: $
  \item[Monomial~24:] $ \frac{{\left(q^{2} + 1\right)} {\left(q + t\right)} {\left(q - t\right)}}{{\left(q^{2} + t\right)} {\left(q^{2} - t\right)}} \cdot:Y_{1}(zq^{-12}t^{7})  Y_{1}(zq^{-4}t^{3})  Y_{2}(zq^{-13}t^{8})^{-1}: $
  \item[Monomial~25:] $ \frac{{\left(q t + 1\right)} {\left(q t - 1\right)}}{{\left(q + t\right)} {\left(q - t\right)}} \cdot:Y_{3}(zq^{-10}t^{7})^{-1}  Y_{3}(zq^{-8}t^{5})  Y_{4}(zq^{-10}t^{6})^{-1}  Y_{4}(zq^{-8}t^{6}): $
  \item[Monomial~26:] $ \frac{{\left(q^{2} + q + 1\right)} {\left(q^{2} - q + 1\right)} {\left(q + t\right)} {\left(q - t\right)}}{{\left(q^{3} + t\right)} {\left(q^{3} - t\right)}} \cdot:Y_{2}(zq^{-11}t^{6})  Y_{2}(zq^{-7}t^{6})^{-1}  Y_{3}(zq^{-12}t^{7})^{-1}  Y_{3}(zq^{-6}t^{5}): $
  \item[Monomial~27:] $ -\frac{{\left(q^{3} + t\right)} {\left(q^{3} - t\right)} {\left(t^{2} + q\right)} {\left(t^{2} - q\right)}}{{\left(q^{3} + t^{2}\right)} {\left(q^{3} - t^{2}\right)} {\left(q + t\right)} {\left(q - t\right)}} \cdot:Y_{4}(zq^{-12}t^{8})^{-1}  Y_{4}(zq^{-6}t^{4}): $
  \item[Monomial~28:] $ \frac{{\left(q^{4} + t\right)} {\left(q^{4} - t\right)} {\left(q^{3} + t^{2}\right)} {\left(q^{3} - t^{2}\right)} {\left(q^{2} + 1\right)} {\left(q + t\right)} {\left(q - t\right)}}{{\left(q^{4} + t^{2}\right)} {\left(q^{3} + t\right)} {\left(q^{3} - t\right)} {\left(q^{2} + t\right)}^{2} {\left(q^{2} - t\right)}^{2}} \cdot:Y_{1}(zq^{-12}t^{7})  Y_{1}(zq^{-6}t^{5})^{-1}  Y_{2}(zq^{-13}t^{8})^{-1}  Y_{2}(zq^{-5}t^{4}): $
  \item[Monomial~29:] $ \frac{{\left(q^{5} + t^{2}\right)} {\left(q^{5} - t^{2}\right)} {\left(q^{4} + t^{3}\right)} {\left(q^{4} - t^{3}\right)} {\left(q^{2} + 1\right)} {\left(q + t\right)} {\left(q - t\right)}}{{\left(q^{5} + t^{3}\right)} {\left(q^{5} - t^{3}\right)} {\left(q^{4} + t^{2}\right)} {\left(q^{2} + t\right)}^{2} {\left(q^{2} - t\right)}^{2}} \cdot:Y_{1}(zq^{-14}t^{9})^{-1}  Y_{1}(zq^{-4}t^{3}): $
  \item[Monomial~30:] $ 1 \cdot:Y_{2}(zq^{-11}t^{6})  Y_{2}(zq^{-9}t^{6})  Y_{3}(zq^{-12}t^{7})^{-1}  Y_{3}(zq^{-10}t^{7})^{-1}  Y_{4}(zq^{-8}t^{6}): $
  \item[Monomial~31:] $ 1 \cdot:Y_{3}(zq^{-8}t^{5})  Y_{4}(zq^{-12}t^{8})^{-1}  Y_{4}(zq^{-10}t^{6})^{-1}: $
  \item[Monomial~32:] $ \frac{{\left(q^{2} + 1\right)} {\left(q + t\right)} {\left(q - t\right)}}{{\left(q^{2} + t\right)} {\left(q^{2} - t\right)}} \cdot:Y_{1}(zq^{-12}t^{7})  Y_{2}(zq^{-13}t^{8})^{-1}  Y_{2}(zq^{-7}t^{6})^{-1}  Y_{3}(zq^{-6}t^{5}): $
  \item[Monomial~33:] $ \frac{{\left(q^{2} + 1\right)} {\left(q + t\right)} {\left(q - t\right)}}{{\left(q^{2} + t\right)} {\left(q^{2} - t\right)}} \cdot:Y_{1}(zq^{-14}t^{9})^{-1}  Y_{1}(zq^{-6}t^{5})^{-1}  Y_{2}(zq^{-5}t^{4}): $
  \item[Monomial~34:] $ 1 \cdot:Y_{2}(zq^{-11}t^{6})  Y_{2}(zq^{-9}t^{6})  Y_{3}(zq^{-12}t^{7})^{-1}  Y_{4}(zq^{-12}t^{8})^{-1}: $
  \item[Monomial~35:] $ \frac{{\left(q^{2} + 1\right)} {\left(q + t\right)} {\left(q - t\right)}}{{\left(q^{2} + t\right)} {\left(q^{2} - t\right)}} \cdot:Y_{1}(zq^{-12}t^{7})  Y_{2}(zq^{-13}t^{8})^{-1}  Y_{2}(zq^{-9}t^{6})  Y_{3}(zq^{-10}t^{7})^{-1}  Y_{4}(zq^{-8}t^{6}): $
  \item[Monomial~36:] $ \frac{{\left(q^{2} + 1\right)} {\left(q + t\right)} {\left(q - t\right)}}{{\left(q^{2} + t\right)} {\left(q^{2} - t\right)}} \cdot:Y_{1}(zq^{-14}t^{9})^{-1}  Y_{2}(zq^{-7}t^{6})^{-1}  Y_{3}(zq^{-6}t^{5}): $
  \item[Monomial~37:] $ \frac{{\left(q^{2} + 1\right)} {\left(q + t\right)} {\left(q - t\right)}}{{\left(q^{2} + t\right)} {\left(q^{2} - t\right)}} \cdot:Y_{1}(zq^{-12}t^{7})  Y_{2}(zq^{-13}t^{8})^{-1}  Y_{2}(zq^{-9}t^{6})  Y_{4}(zq^{-12}t^{8})^{-1}: $
  \item[Monomial~38:] $ 1 \cdot:Y_{1}(zq^{-12}t^{7})  Y_{1}(zq^{-10}t^{7})  Y_{2}(zq^{-13}t^{8})^{-1}  Y_{2}(zq^{-11}t^{8})^{-1}  Y_{4}(zq^{-8}t^{6}): $
  \item[Monomial~39:] $ \frac{{\left(q^{2} + 1\right)} {\left(q + t\right)} {\left(q - t\right)}}{{\left(q^{2} + t\right)} {\left(q^{2} - t\right)}} \cdot:Y_{1}(zq^{-14}t^{9})^{-1}  Y_{2}(zq^{-9}t^{6})  Y_{3}(zq^{-10}t^{7})^{-1}  Y_{4}(zq^{-8}t^{6}): $
  \item[Monomial~40:] $ 1 \cdot:Y_{1}(zq^{-12}t^{7})  Y_{1}(zq^{-10}t^{7})  Y_{2}(zq^{-13}t^{8})^{-1}  Y_{2}(zq^{-11}t^{8})^{-1}  Y_{3}(zq^{-10}t^{7})  Y_{4}(zq^{-12}t^{8})^{-1}: $
  \item[Monomial~41:] $ \frac{{\left(q^{2} + 1\right)} {\left(q + t\right)} {\left(q - t\right)}}{{\left(q^{2} + t\right)} {\left(q^{2} - t\right)}} \cdot:Y_{1}(zq^{-14}t^{9})^{-1}  Y_{2}(zq^{-9}t^{6})  Y_{4}(zq^{-12}t^{8})^{-1}: $
  \item[Monomial~42:] $ \frac{{\left(q^{2} + 1\right)} {\left(q + t\right)} {\left(q - t\right)}}{{\left(q^{2} + t\right)} {\left(q^{2} - t\right)}} \cdot:Y_{1}(zq^{-14}t^{9})^{-1}  Y_{1}(zq^{-10}t^{7})  Y_{2}(zq^{-11}t^{8})^{-1}  Y_{4}(zq^{-8}t^{6}): $
  \item[Monomial~43:] $ \frac{{\left(q^{2} + 1\right)} {\left(q + t\right)} {\left(q - t\right)}}{{\left(q^{2} + t\right)} {\left(q^{2} - t\right)}} \cdot:Y_{1}(zq^{-14}t^{9})^{-1}  Y_{1}(zq^{-10}t^{7})  Y_{2}(zq^{-11}t^{8})^{-1}  Y_{3}(zq^{-10}t^{7})  Y_{4}(zq^{-12}t^{8})^{-1}: $
  \item[Monomial~44:] $ 1 \cdot:Y_{1}(zq^{-12}t^{7})  Y_{1}(zq^{-10}t^{7})  Y_{3}(zq^{-14}t^{9})^{-1}: $
  \item[Monomial~45:] $ 1 \cdot:Y_{1}(zq^{-14}t^{9})^{-1}  Y_{1}(zq^{-12}t^{9})^{-1}  Y_{4}(zq^{-8}t^{6}): $
  \item[Monomial~46:] $ 1 \cdot:Y_{1}(zq^{-14}t^{9})^{-1}  Y_{1}(zq^{-12}t^{9})^{-1}  Y_{3}(zq^{-10}t^{7})  Y_{4}(zq^{-12}t^{8})^{-1}: $
  \item[Monomial~47:] $ \frac{{\left(q^{2} + 1\right)} {\left(q + t\right)} {\left(q - t\right)}}{{\left(q^{2} + t\right)} {\left(q^{2} - t\right)}} \cdot:Y_{1}(zq^{-14}t^{9})^{-1}  Y_{1}(zq^{-10}t^{7})  Y_{2}(zq^{-13}t^{8})  Y_{3}(zq^{-14}t^{9})^{-1}: $
  \item[Monomial~48:] $ 1 \cdot:Y_{1}(zq^{-14}t^{9})^{-1}  Y_{1}(zq^{-12}t^{9})^{-1}  Y_{2}(zq^{-13}t^{8})  Y_{2}(zq^{-11}t^{8})  Y_{3}(zq^{-14}t^{9})^{-1}: $
  \item[Monomial~49:] $ \frac{{\left(q^{2} + 1\right)} {\left(q + t\right)} {\left(q - t\right)}}{{\left(q^{2} + t\right)} {\left(q^{2} - t\right)}} \cdot:Y_{1}(zq^{-10}t^{7})  Y_{2}(zq^{-15}t^{10})^{-1}: $
  \item[Monomial~50:] $ \frac{{\left(q^{2} + 1\right)} {\left(q + t\right)} {\left(q - t\right)}}{{\left(q^{2} + t\right)} {\left(q^{2} - t\right)}} \cdot:Y_{1}(zq^{-12}t^{9})^{-1}  Y_{2}(zq^{-15}t^{10})^{-1}  Y_{2}(zq^{-11}t^{8}): $
  \item[Monomial~51:] $ 1 \cdot:Y_{2}(zq^{-15}t^{10})^{-1}  Y_{2}(zq^{-13}t^{10})^{-1}  Y_{3}(zq^{-12}t^{9}): $
  \item[Monomial~52:] $ 1 \cdot:Y_{3}(zq^{-16}t^{11})^{-1}  Y_{4}(zq^{-14}t^{10}): $
  \item[Monomial~53:] $ 1 \cdot:Y_{4}(zq^{-18}t^{12})^{-1}: $
\end{itemize}

\subsection{Type $G_2$}\label{fundg2}
The algorithm works for $i=1,2$. These formulas in type $G_2$ are also computed by P. Bouwknegt and K. Pilch in \cite{MR1633032}.
Here is the $G_2$-type Cartan matrix:
$$\begin{pmatrix}
    2&-1\\
    -3&2
\end{pmatrix}$$
And $$D=\begin{pmatrix}
    3&0\\
    0&1
\end{pmatrix}$$
\textbf{Here is the list  of monomials in $T_1(z)$}:\\
\begin{itemize}[leftmargin=3cm, labelsep=0.5cm]
  \item[Monomial~1:] $ 1 \cdot:Y_{1}(z): $
  \item[Monomial~2:] $ 1 \cdot:Y_{1}(zq^{-6}t^{2})^{-1}  Y_{2}(zq^{-5}t)  Y_{2}(zq^{-3}t)  Y_{2}(zq^{-1}t): $
  \item[Monomial~3:] $ \frac{{\left(q^{2} + q + 1\right)} {\left(q^{2} - q + 1\right)} {\left(q + t\right)} {\left(q - t\right)}}{{\left(q^{3} + t\right)} {\left(q^{3} - t\right)}} \cdot:Y_{2}(zq^{-7}t^{3})^{-1}  Y_{2}(zq^{-3}t)  Y_{2}(zq^{-1}t): $
  \item[Monomial~4:] $ \frac{{\left(q^{2} + q + 1\right)} {\left(q^{2} - q + 1\right)} {\left(q + t\right)} {\left(q - t\right)}}{{\left(q^{3} + t\right)} {\left(q^{3} - t\right)}} \cdot:Y_{1}(zq^{-4}t^{2})  Y_{2}(zq^{-7}t^{3})^{-1}  Y_{2}(zq^{-5}t^{3})^{-1}  Y_{2}(zq^{-1}t): $
  \item[Monomial~5:] $ 1 \cdot:Y_{1}(zq^{-4}t^{2})  Y_{1}(zq^{-2}t^{2})  Y_{2}(zq^{-7}t^{3})^{-1}  Y_{2}(zq^{-5}t^{3})^{-1}  Y_{2}(zq^{-3}t^{3})^{-1}: $
  \item[Monomial~6:] $ \frac{{\left(q^{2} + q + 1\right)} {\left(q^{2} - q + 1\right)} {\left(q + t\right)} {\left(q - t\right)}}{{\left(q^{3} + t\right)} {\left(q^{3} - t\right)}} \cdot:Y_{1}(zq^{-10}t^{4})^{-1}  Y_{2}(zq^{-9}t^{3})  Y_{2}(zq^{-1}t): $
  \item[Monomial~7:] $ \frac{{\left(q^{4} + 1\right)} {\left(q^{2} + 1\right)} {\left(q + t\right)} {\left(q - t\right)}}{{\left(q^{4} + t\right)} {\left(q^{4} - t\right)}} \cdot:Y_{1}(zq^{-10}t^{4})^{-1}  Y_{1}(zq^{-2}t^{2})  Y_{2}(zq^{-9}t^{3})  Y_{2}(zq^{-3}t^{3})^{-1}: $
  \item[Monomial~8:] $ \frac{{\left(q^{2} + 1\right)} {\left(q t + 1\right)} {\left(q t - 1\right)}}{{\left(q^{2} + t\right)} {\left(q^{2} - t\right)}} \cdot:Y_{1}(zq^{-8}t^{4})^{-1}  Y_{1}(zq^{-4}t^{2}): $
  \item[Monomial~9:] $ \frac{{\left(q^{5} + t\right)} {\left(q^{5} - t\right)} {\left(q^{4} + t^{2}\right)} {\left(q^{2} + q + 1\right)} {\left(q^{2} - q + 1\right)} {\left(q^{2} + t\right)} {\left(q^{2} - t\right)} {\left(q + t\right)} {\left(q - t\right)}}{{\left(q^{5} + t^{2}\right)} {\left(q^{5} - t^{2}\right)} {\left(q^{4} + t\right)} {\left(q^{4} - t\right)} {\left(q^{3} + t\right)} {\left(q^{3} - t\right)}} \cdot:Y_{2}(zq^{-11}t^{5})^{-1}  Y_{2}(zq^{-1}t): $
  \item[Monomial~10:] $ 1 \cdot:Y_{1}(zq^{-10}t^{4})^{-1}  Y_{1}(zq^{-8}t^{4})^{-1}  Y_{2}(zq^{-9}t^{3})  Y_{2}(zq^{-7}t^{3})  Y_{2}(zq^{-5}t^{3}): $
  \item[Monomial~11:] $ \frac{{\left(q^{2} + q + 1\right)} {\left(q^{2} - q + 1\right)} {\left(q + t\right)} {\left(q - t\right)}}{{\left(q^{3} + t\right)} {\left(q^{3} - t\right)}} \cdot:Y_{1}(zq^{-2}t^{2})  Y_{2}(zq^{-11}t^{5})^{-1}  Y_{2}(zq^{-3}t^{3})^{-1}: $
  \item[Monomial~12:] $ \frac{{\left(q^{2} + q + 1\right)} {\left(q^{2} - q + 1\right)} {\left(q + t\right)} {\left(q - t\right)}}{{\left(q^{3} + t\right)} {\left(q^{3} - t\right)}} \cdot:Y_{1}(zq^{-8}t^{4})^{-1}  Y_{2}(zq^{-11}t^{5})^{-1}  Y_{2}(zq^{-7}t^{3})  Y_{2}(zq^{-5}t^{3}): $
  \item[Monomial~13:] $ \frac{{\left(q^{2} + q + 1\right)} {\left(q^{2} - q + 1\right)} {\left(q + t\right)} {\left(q - t\right)}}{{\left(q^{3} + t\right)} {\left(q^{3} - t\right)}} \cdot:Y_{2}(zq^{-11}t^{5})^{-1}  Y_{2}(zq^{-9}t^{5})^{-1}  Y_{2}(zq^{-5}t^{3}): $
  \item[Monomial~14:] $ 1 \cdot:Y_{1}(zq^{-6}t^{4})  Y_{2}(zq^{-11}t^{5})^{-1}  Y_{2}(zq^{-9}t^{5})^{-1}  Y_{2}(zq^{-7}t^{5})^{-1}: $
  \item[Monomial~15:] $ 1 \cdot:Y_{1}(zq^{-12}t^{6})^{-1}: $
\end{itemize}

\textbf{Here is the list  of monomials in $T_2(z)$}:
\begin{itemize}[leftmargin=3cm, labelsep=0.5cm]
  \item[Monomial~1:] $ 1 \cdot:Y_{2}(z): $
  \item[Monomial~2:] $ 1 \cdot:Y_{1}(zq^{-1}t)  Y_{2}(zq^{-2}t^{2})^{-1}: $
  \item[Monomial~3:] $ 1 \cdot:Y_{1}(zq^{-7}t^{3})^{-1}  Y_{2}(zq^{-6}t^{2})  Y_{2}(zq^{-4}t^{2}): $
  \item[Monomial~4:] $ \frac{{\left(q^{2} + 1\right)} {\left(q + t\right)} {\left(q - t\right)}}{{\left(q^{2} + t\right)} {\left(q^{2} - t\right)}} \cdot:Y_{2}(zq^{-8}t^{4})^{-1}  Y_{2}(zq^{-4}t^{2}): $
  \item[Monomial~5:] $ 1 \cdot:Y_{1}(zq^{-5}t^{3})  Y_{2}(zq^{-8}t^{4})^{-1}  Y_{2}(zq^{-6}t^{4})^{-1}: $
  \item[Monomial~6:] $ 1 \cdot:Y_{1}(zq^{-11}t^{5})^{-1}  Y_{2}(zq^{-10}t^{4}): $
  \item[Monomial~7:] $ 1 \cdot:Y_{2}(zq^{-12}t^{6})^{-1}: $
\end{itemize}

\section{Explicit examples of new fundamental fields in classical type:}
In this section, we provide some examples of fundamental fields in type $B_3,C_4,D_4$. All these monomials and coefficients have been computed with SageMath.
\subsection{The fundamental field $T_2(z)\in\mathbf{W}_{q,t}(B_3)$:}

\begin{itemize}[leftmargin=3cm, labelsep=0.5cm]
  \item[Monomial~1:] $ 1 \cdot:Y_{2}(z): $
  \item[Monomial~2:] $ 1 \cdot:Y_{1}(zq^{-2}t)  Y_{2}(zq^{-4}t^{2})^{-1}  Y_{3}(zq^{-3}t)  Y_{3}(zq^{-1}t): $
  \item[Monomial~3:] $ 1 \cdot:Y_{1}(zq^{-6}t^{3})^{-1}  Y_{3}(zq^{-3}t)  Y_{3}(zq^{-1}t): $
  \item[Monomial~4:] $ \frac{{\left(q^{2} + 1\right)} {\left(q + t\right)} {\left(q - t\right)}}{{\left(q^{2} + t\right)} {\left(q^{2} - t\right)}} \cdot:Y_{1}(zq^{-2}t)  Y_{3}(zq^{-5}t^{3})^{-1}  Y_{3}(zq^{-1}t): $
  \item[Monomial~5:] $ \frac{{\left(q^{2} + 1\right)} {\left(q + t\right)} {\left(q - t\right)}}{{\left(q^{2} + t\right)} {\left(q^{2} - t\right)}} \cdot:Y_{1}(zq^{-6}t^{3})^{-1}  Y_{2}(zq^{-4}t^{2})  Y_{3}(zq^{-5}t^{3})^{-1}  Y_{3}(zq^{-1}t): $
  \item[Monomial~6:] $ 1 \cdot:Y_{1}(zq^{-2}t)  Y_{2}(zq^{-2}t^{2})  Y_{3}(zq^{-5}t^{3})^{-1}  Y_{3}(zq^{-3}t^{3})^{-1}: $
  \item[Monomial~7:] $ 1 \cdot:Y_{1}(zq^{-6}t^{3})^{-1}  Y_{2}(zq^{-4}t^{2})  Y_{2}(zq^{-2}t^{2})  Y_{3}(zq^{-5}t^{3})^{-1}  Y_{3}(zq^{-3}t^{3})^{-1}: $
  \item[Monomial~8:] $ \frac{{\left(q^{2} + 1\right)} {\left(q + t\right)} {\left(q - t\right)}}{{\left(q^{2} + t\right)} {\left(q^{2} - t\right)}} \cdot:Y_{2}(zq^{-8}t^{4})^{-1}  Y_{3}(zq^{-7}t^{3})  Y_{3}(zq^{-1}t): $
  \item[Monomial~9:] $ 1 \cdot:Y_{1}(zq^{-4}t^{3})  Y_{1}(zq^{-2}t)  Y_{2}(zq^{-6}t^{4})^{-1}: $
  \item[Monomial~10:] $ \frac{{\left(q^{2} + q + 1\right)} {\left(q^{2} - q + 1\right)} {\left(q + t\right)} {\left(q - t\right)}}{{\left(q^{3} + t\right)} {\left(q^{3} - t\right)}} \cdot:Y_{2}(zq^{-8}t^{4})^{-1}  Y_{2}(zq^{-2}t^{2})  Y_{3}(zq^{-7}t^{3})  Y_{3}(zq^{-3}t^{3})^{-1}: $
  \item[Monomial~11:] $ \frac{{\left(q t + 1\right)} {\left(q t - 1\right)}}{{\left(q + t\right)} {\left(q - t\right)}} \cdot:Y_{1}(zq^{-6}t^{3})^{-1}  Y_{1}(zq^{-4}t^{3})  Y_{2}(zq^{-6}t^{4})^{-1}  Y_{2}(zq^{-4}t^{2}): $
  \item[Monomial~12:] $ \frac{{\left(q^{4} + t\right)} {\left(q^{4} - t\right)} {\left(q^{3} + t^{2}\right)} {\left(q^{3} - t^{2}\right)} {\left(q^{2} + 1\right)} {\left(q + t\right)} {\left(q - t\right)}}{{\left(q^{4} + t^{2}\right)} {\left(q^{3} + t\right)} {\left(q^{3} - t\right)} {\left(q^{2} + t\right)}^{2} {\left(q^{2} - t\right)}^{2}} \cdot:Y_{3}(zq^{-9}t^{5})^{-1}  Y_{3}(zq^{-1}t): $
  \item[Monomial~13:] $ -\frac{{\left(q^{3} + t\right)} {\left(q^{3} - t\right)} {\left(t^{2} + q\right)} {\left(t^{2} - q\right)}}{{\left(q^{3} + t^{2}\right)} {\left(q^{3} - t^{2}\right)} {\left(q + t\right)} {\left(q - t\right)}} \cdot:Y_{1}(zq^{-8}t^{5})^{-1}  Y_{1}(zq^{-2}t): $
  \item[Monomial~14:] $ 1 \cdot:Y_{1}(zq^{-4}t^{3})  Y_{2}(zq^{-8}t^{4})^{-1}  Y_{2}(zq^{-6}t^{4})^{-1}  Y_{3}(zq^{-7}t^{3})  Y_{3}(zq^{-5}t^{3}): $
  \item[Monomial~15:] $ \frac{{\left(q^{2} + 1\right)} {\left(q + t\right)} {\left(q - t\right)}}{{\left(q^{2} + t\right)} {\left(q^{2} - t\right)}} \cdot:Y_{2}(zq^{-2}t^{2})  Y_{3}(zq^{-9}t^{5})^{-1}  Y_{3}(zq^{-3}t^{3})^{-1}: $
  \item[Monomial~16:] $ 1 \cdot:Y_{1}(zq^{-8}t^{5})^{-1}  Y_{1}(zq^{-6}t^{3})^{-1}  Y_{2}(zq^{-4}t^{2}): $
  \item[Monomial~17:] $ \frac{{\left(q^{2} + 1\right)} {\left(q + t\right)} {\left(q - t\right)}}{{\left(q^{2} + t\right)} {\left(q^{2} - t\right)}} \cdot:Y_{1}(zq^{-4}t^{3})  Y_{2}(zq^{-6}t^{4})^{-1}  Y_{3}(zq^{-9}t^{5})^{-1}  Y_{3}(zq^{-5}t^{3}): $
  \item[Monomial~18:] $ 1 \cdot:Y_{1}(zq^{-8}t^{5})^{-1}  Y_{2}(zq^{-8}t^{4})^{-1}  Y_{3}(zq^{-7}t^{3})  Y_{3}(zq^{-5}t^{3}): $
  \item[Monomial~19:] $ \frac{{\left(q^{2} + 1\right)} {\left(q + t\right)} {\left(q - t\right)}}{{\left(q^{2} + t\right)} {\left(q^{2} - t\right)}} \cdot:Y_{1}(zq^{-8}t^{5})^{-1}  Y_{3}(zq^{-9}t^{5})^{-1}  Y_{3}(zq^{-5}t^{3}): $
  \item[Monomial~20:] $ 1 \cdot:Y_{1}(zq^{-4}t^{3})  Y_{3}(zq^{-9}t^{5})^{-1}  Y_{3}(zq^{-7}t^{5})^{-1}: $
  \item[Monomial~21:] $ 1 \cdot:Y_{1}(zq^{-8}t^{5})^{-1}  Y_{2}(zq^{-6}t^{4})  Y_{3}(zq^{-9}t^{5})^{-1}  Y_{3}(zq^{-7}t^{5})^{-1}: $
  \item[Monomial~22:] $ 1 \cdot:Y_{2}(zq^{-10}t^{6})^{-1}: $
\end{itemize}

\subsection{The fundamental field $T_2(z)\in\mathbf{W}_{q,t}(C_4)$}

\begin{itemize}[leftmargin=3cm, labelsep=0.5cm]
  \item[Monomial~1:] $ 1 \cdot:Y_{2}(z): $
  \item[Monomial~2:] $ 1 \cdot:Y_{1}(zq^{-1}t)  Y_{2}(zq^{-2}t^{2})^{-1}  Y_{3}(zq^{-1}t): $
  \item[Monomial~3:] $ 1 \cdot:Y_{1}(zq^{-1}t)  Y_{3}(zq^{-3}t^{3})^{-1}  Y_{4}(zq^{-2}t^{2}): $
  \item[Monomial~4:] $ 1 \cdot:Y_{1}(zq^{-3}t^{3})^{-1}  Y_{3}(zq^{-1}t): $
  \item[Monomial~5:] $ 1 \cdot:Y_{1}(zq^{-3}t^{3})^{-1}  Y_{2}(zq^{-2}t^{2})  Y_{3}(zq^{-3}t^{3})^{-1}  Y_{4}(zq^{-2}t^{2}): $
  \item[Monomial~6:] $ 1 \cdot:Y_{1}(zq^{-1}t)  Y_{3}(zq^{-5}t^{3})  Y_{4}(zq^{-6}t^{4})^{-1}: $
  \item[Monomial~7:] $ 1 \cdot:Y_{1}(zq^{-3}t^{3})^{-1}  Y_{2}(zq^{-2}t^{2})  Y_{3}(zq^{-5}t^{3})  Y_{4}(zq^{-6}t^{4})^{-1}: $
  \item[Monomial~8:] $ 1 \cdot:Y_{2}(zq^{-4}t^{4})^{-1}  Y_{4}(zq^{-2}t^{2}): $
  \item[Monomial~9:] $ 1 \cdot:Y_{1}(zq^{-1}t)  Y_{2}(zq^{-6}t^{4})  Y_{3}(zq^{-7}t^{5})^{-1}: $
  \item[Monomial~10:] $ 1 \cdot:Y_{2}(zq^{-4}t^{4})^{-1}  Y_{3}(zq^{-5}t^{3})  Y_{3}(zq^{-3}t^{3})  Y_{4}(zq^{-6}t^{4})^{-1}: $
  \item[Monomial~11:] $ 1 \cdot:Y_{1}(zq^{-3}t^{3})^{-1}  Y_{2}(zq^{-6}t^{4})  Y_{2}(zq^{-2}t^{2})  Y_{3}(zq^{-7}t^{5})^{-1}: $
  \item[Monomial~12:] $ 1 \cdot:Y_{1}(zq^{-7}t^{5})  Y_{1}(zq^{-1}t)  Y_{2}(zq^{-8}t^{6})^{-1}: $
  \item[Monomial~13:] $ \frac{{\left(q^{2} + 1\right)} {\left(q + t\right)} {\left(q - t\right)}}{{\left(q^{2} + t\right)} {\left(q^{2} - t\right)}} \cdot:Y_{2}(zq^{-6}t^{4})  Y_{2}(zq^{-4}t^{4})^{-1}  Y_{3}(zq^{-7}t^{5})^{-1}  Y_{3}(zq^{-3}t^{3}): $
  \item[Monomial~14:] $ \frac{{\left(q^{3} + t\right)} {\left(q^{3} - t\right)} {\left(q^{2} + t^{2}\right)} {\left(q + t\right)} {\left(q - t\right)}}{{\left(q^{3} + t^{2}\right)} {\left(q^{3} - t^{2}\right)} {\left(q^{2} + t\right)} {\left(q^{2} - t\right)}} \cdot:Y_{1}(zq^{-7}t^{5})  Y_{1}(zq^{-3}t^{3})^{-1}  Y_{2}(zq^{-8}t^{6})^{-1}  Y_{2}(zq^{-2}t^{2}): $
  \item[Monomial~15:] $ \frac{{\left(q^{4} + t^{2}\right)} {\left(q^{2} + q t + t^{2}\right)} {\left(q^{2} - q t + t^{2}\right)} {\left(q^{2} + t\right)} {\left(q^{2} - t\right)} {\left(q + t\right)} {\left(q - t\right)}}{{\left(q^{4} + t^{3}\right)} {\left(q^{4} - t^{3}\right)} {\left(q^{3} + t^{2}\right)} {\left(q^{3} - t^{2}\right)}} \cdot:Y_{1}(zq^{-9}t^{7})^{-1}  Y_{1}(zq^{-1}t): $
  \item[Monomial~16:] $ 1 \cdot:Y_{2}(zq^{-6}t^{4})  Y_{3}(zq^{-7}t^{5})^{-1}  Y_{3}(zq^{-5}t^{5})^{-1}  Y_{4}(zq^{-4}t^{4}): $
  \item[Monomial~17:] $ 1 \cdot:Y_{1}(zq^{-7}t^{5})  Y_{2}(zq^{-8}t^{6})^{-1}  Y_{2}(zq^{-4}t^{4})^{-1}  Y_{3}(zq^{-3}t^{3}): $
  \item[Monomial~18:] $ 1 \cdot:Y_{1}(zq^{-9}t^{7})^{-1}  Y_{1}(zq^{-3}t^{3})^{-1}  Y_{2}(zq^{-2}t^{2}): $
  \item[Monomial~19:] $ 1 \cdot:Y_{1}(zq^{-7}t^{5})  Y_{2}(zq^{-8}t^{6})^{-1}  Y_{3}(zq^{-5}t^{5})^{-1}  Y_{4}(zq^{-4}t^{4}): $
  \item[Monomial~20:] $ 1 \cdot:Y_{2}(zq^{-6}t^{4})  Y_{4}(zq^{-8}t^{6})^{-1}: $
  \item[Monomial~21:] $ 1 \cdot:Y_{1}(zq^{-9}t^{7})^{-1}  Y_{2}(zq^{-4}t^{4})^{-1}  Y_{3}(zq^{-3}t^{3}): $
  \item[Monomial~22:] $ 1 \cdot:Y_{1}(zq^{-7}t^{5})  Y_{2}(zq^{-8}t^{6})^{-1}  Y_{3}(zq^{-7}t^{5})  Y_{4}(zq^{-8}t^{6})^{-1}: $
  \item[Monomial~23:] $ 1 \cdot:Y_{1}(zq^{-9}t^{7})^{-1}  Y_{3}(zq^{-5}t^{5})^{-1}  Y_{4}(zq^{-4}t^{4}): $
  \item[Monomial~24:] $ 1 \cdot:Y_{1}(zq^{-9}t^{7})^{-1}  Y_{3}(zq^{-7}t^{5})  Y_{4}(zq^{-8}t^{6})^{-1}: $
  \item[Monomial~25:] $ 1 \cdot:Y_{1}(zq^{-7}t^{5})  Y_{3}(zq^{-9}t^{7})^{-1}: $
  \item[Monomial~26:] $ 1 \cdot:Y_{1}(zq^{-9}t^{7})^{-1}  Y_{2}(zq^{-8}t^{6})  Y_{3}(zq^{-9}t^{7})^{-1}: $
  \item[Monomial~27:] $ 1 \cdot:Y_{2}(zq^{-10}t^{8})^{-1}: $
\end{itemize}

\subsection{The fundamental field $T_3(z)\in\mathbf{W}_{q,t}(D_4)$:}

\begin{itemize}[leftmargin=3cm, labelsep=0.5cm]
  \item[Monomial~1:] $ 1 \cdot:Y_{3}(z): $
  \item[Monomial~2:] $ 1 \cdot:Y_{2}(zq^{-1}t)  Y_{3}(zq^{-2}t^{2})^{-1}: $
  \item[Monomial~3:] $ 1 \cdot:Y_{1}(zq^{-2}t^{2})  Y_{2}(zq^{-3}t^{3})^{-1}  Y_{4}(zq^{-2}t^{2}): $
  \item[Monomial~4:] $ 1 \cdot:Y_{1}(zq^{-4}t^{4})^{-1}  Y_{4}(zq^{-2}t^{2}): $
  \item[Monomial~5:] $ 1 \cdot:Y_{1}(zq^{-2}t^{2})  Y_{4}(zq^{-4}t^{4})^{-1}: $
  \item[Monomial~6:] $ 1 \cdot:Y_{1}(zq^{-4}t^{4})^{-1}  Y_{2}(zq^{-3}t^{3})  Y_{4}(zq^{-4}t^{4})^{-1}: $
  \item[Monomial~7:] $ 1 \cdot:Y_{2}(zq^{-5}t^{5})^{-1}  Y_{3}(zq^{-4}t^{4}): $
  \item[Monomial~8:] $ 1 \cdot:Y_{3}(zq^{-6}t^{6})^{-1}: $
\end{itemize}

\section{Proof of Lemma \ref{algorithmreverse} in type $B_2,G_2$}\label{appendixlem}
We prove here Lemma \ref{algorithmreverse} in type $B_2,G_2$. The proof is quite long, but it closely follows the one for type $A_2$. The extra length comes simply from the number of cases to check. Indeed, in type $B_2$, we have to study the nodes associated with both the long and short roots, as well as handle possible cancellations in $q$ and $t$, while in type $A_2$ the field is symmetric in $q^{-1}$ and $t$.
\subsection{In type $B_2$}
We have $r_1=1$, $r_2=2$, 
$$A_1(z)=:Y_1(zq^{-1}t)Y_1(zq^{}t^{-1})Y_2(z)^{-1}:$$
$$A_2(z)=:Y_2(zq^{-2}t)Y_2(zq^{2}t^{-1})Y_1(zq^{-1})^{-1}Y_1(zq)^{-1}:$$
\subsubsection{In the case $(i,j)=(2,1)$ }

\begin{proof}
    We prove it by induction on the height of the monomials.  We assume the Lie subalgebra generated by the nodes $i$ and $j$ is of type $B_2$, and $(i,j)=(2,1)$.

      \textit{\underline{Height $0$:}} The only monomial with height $0$ is the dominant generic monomial from which we start our algorithm. Hence the property holds at height $0$.\\

    \textit{\underline{Height $h+1$:}} We assume the property is true for all the monomials with heights $h'\leq h\in\N$. 
        Let us prove it for the monomials with height $h+1$. Let $X$ be a monomial with height $h+1$. 
        The monomial appears in the algorithm, so it has to come from a monomial $X'$ with height $h$. 
        Let $A_j(zbq^{r_j}t^{-1})^{-1}$ be the transformation involved. 
        It implies that $d_{j,b}(X)=-1$ and $d_{j,bq^{2r_j}t^{-2}}(X')=1$. 
        Let $i\in I$, $a\in q^\Z t^\Z$ such that $d_{i,a}(X)=-1$ and $d_{i,aq^{2r_i}}(X),d_{i,at^{-2}}(X)\ne -1$. 
        If $(i,a)=(j,b)$ then we have the result. 
        Else, we have
        $$X=X'A_{j}(zbq^{r_j}t^{-1})^{-1}.$$
        Hence,
        $$ \left\{\begin{array}{llll}
            d_{j,bq^{2r_j}t^{-2}}(X')=1\\
            d_{j,b}(X)=-1\\
            d_{j,c}(X)=d_{j,c}(X')\text{ for all $c\not\in\{b,bq^{2r_j}t^{-2}\}$}\\
            d_{k,c}(X)\geq d_{k,c}(X')\text{ for all $k\neq j$, $c\in q^\Z t^\Z$}
        \end{array}\right.$$
        We know that $d_{i,a}(X)=-1$, and $(i,a)\neq (j,b)$.  
        It implies $d_{i,a}(X')\leq-1$. The monomials are all generic, thus 
        $$ d_{i,a}(X')=-1$$
        Then by the induction assumption, one of the two following assertions is true:\\
        \begin{itemize}
            \item[(a)]  $d_{i,aq^{2r_i}}(X')=-1$  {or}  $d_{i,at^{-2}}(X')= -1$.
            \item[(b)]  There exists a monomial $X''$ appearing in the algorithm such that
        $$X''A_{i}(zaq^{r_i}t^{-1})^{-1}=X'.$$
        \end{itemize}

        \textit{a)} Firstly we assume (a) is true and $X'$ does contain $Y_i(zaq^{2r_i})^{-1}$ (i.e $d_{i,aq^{2r_i}}(X')=-1$). 
        Let us prove that this case leads to a contradiction, then we will consider the case where $X'$ does contain $Y_i(zat^{-2})^{-1}$
         (i.e $d_{i,at^{2}}(X')=-1$)\\
        By assumption, $d_{i,aq^{2r_i}}(X)=0$. 
        It implies that the transformation $A_j(zbq^{r_j}t^{-1})^{-1}$ simplifies $Y_i(zaq^{2r_i})^{-1}$ and in particular $i\ne j$. 
        Therefore, the expressions depend on the type of the Dynkin diagram generated by the nodes $i$ and $j$. \\

        Let us do it in type $B_2$, $(i,j)=(2,1)$.  
        $$A_j(zbq^{r_j}t^{-1})^{-1}=A_j(zbq^2t^{-1})^{-1}=Y_j(zbq^4t^{-2})^{-1}Y_j(zb)^{-1}Y_i(zbq^3t^{-1})Y_i(zbqt^{-1})$$  
        Then 
        $$ bq^3t^{-1}=aq^2 \qquad\text{or}\qquad bqt^{-1}=aq^2 $$   
        In the first case, the transformation $A_j(zbq^{r_j}t^{-1})^{-1}$ also simplifies $Y_i(zbqt^{-1})^{-1}=Y_i(za)^{-1}$. This is impossible.
        Thus, we are in the second case and as in type $A_2$ we get:
        $$b=aqt.$$    
        
        Let $k>0$ be the maximal integer such that $Y_i(zaq^{2s })^{-1}$ appears in $X'$ for all $0\le s\le k$. We want to prove that the graph representation
        of the algorithm contains the following subgraph:
            \begin{figure}[htbp]
            \begin{center}
            \begin{tikzpicture}[scale=0.7,
                >=Stealth,
                every node/.style={align=center, text width=4cm},
                every edge/.style={-{Stealth[scale=1.2]}, thick}
            ]
            \node (N3) at (0, -3) { $X_{k+1}$};
            \node (N5) at (0, -4.5) { \shortstack{$\vdots$}};
            \node (N6) at (0, -6) { \shortstack{$X'$}};
            \node (N7) at (0, -7.5) { \shortstack{$X$}};
            \draw[->] (N5) -- (N6) node[midway, right, yshift=0.2cm] {$ A_{i}(zaq^{2k+1}t^{-1})^{-1}$};
              \draw[->] (N6) -- (N7) node[midway, right, yshift=0.2cm] {$ A_j(zaq^3)^{-1} $};
             \draw[->] (N3) -- (N5) node[midway, right, yshift=0.2cm] {$ A_{i}(zaq^{}t^{-1})^{-1} $};
             \end{tikzpicture}
              \end{center} 
        \end{figure}
        By definition, $Y_i(zaq^{2(k+1)})^{-1}$ is not in $X'$. 
        Moreover, if $Y_i(zaq^{2k}t^{-2})^{-1}$ appears in $X'$, then $Y_i(zaq^{2(k-1) })^{-1}Y_i(zaq^{2 k}t^{-2})^{-1}$ also appears in $X'$. 
        The monomial has to be regular, so $Y_i(zaq^{2k })^{-1}Y_i(zaq^{2( k-1)}t^{-2})^{-1}$ appears in $X'$. \\
        Thus, $X'$ contains $Y_i(zaq^{2(k-2) })^{-1}Y_i(zaq^{2(k-1) }t^{-2})^{-1}$.
        We can iterate this reasoning until we get that $Y_i(zat^{-2})^{-1}$ appears in $X'$. 
        However it does not appear in $X$ and has to be simplified by the transformation $A_j(zbq^{r_j}t^{-1})^{-1}=A_j(zaq^3)^{-1}$. 
        It is a contradiction by definition of $A_j$. 
        Hence, we can apply the induction hypothesis and we obtain that there exists $X_1$ given by the algorithm such that $$X_1A_{i}(zaq^{2k+1}t^{-1})^{-1}=X'.$$ 
        We can define recursively 
        $X_s$, $2\le s\le k+1$ such that $$X_s A_{i}(zaq^{2k-2s+3}t^{-1})^{-1}=X_{s-1}.$$ 
        Indeed, for all $1\leq s\leq k+1$, $Y_i(zaq^{2u })^{-1}$ appears in $X_s$ for all $0\le u\le k+1-s$ and $Y_i(zaq^{2(u-1) }t^{-2})^{-1}$ does not appear in $X_s$ nor $Y_i(zaq^{2(k+2-s) })^{-1}$ so that we can apply the induction hypothesis. 
        In particular, we have constructed $X_{k+1}$ such that $$X_{k+1}A_{i}(zaq^{ }t^{-1})^{-1}=X_k.$$
        But $d_{j,aqt^{-1}}(A_{i}(zaq^{}t^{-1})^{-1})=1$. Thus, $$d_{j,aqt^{-1}}(X_{k+1})=d_{j,aqt^{-1}}(X_{k})-1=d_{j,aqt^{-1}}(X')-1<0$$
        because by regularity, $$d_{j,aqt^{-1}}(X')\in\{0,1\},$$
        and if $d_{j,aqt^{-1}}(X')=1$ then the transformation $A_j(zbq^2t^{-1})^{-1}=A_j(zaq^3)^{-1}$ would not have been admissible from $X'$. 
        Thus $d_{j,aqt^{-1}}(X')=0$ and $X_{k+1}$ contains $Y_{j}(zaqt^{-1})^{-1}$. Moreover, 
        \begin{equation}\label{eqnb1}
            X_{k+1}A_{i}(zaq^{}t^{-1})^{-1}A_{i}(zaq^{3}t^{-1})^{-1}\cdots  A_{i}(zaq^{2k+1}t^{-1})^{-1}=X'.
        \end{equation}
        We recall that $X'A_{j}(zaq^{3})^{-1}=X$. By definition of the algorithm, it implies $Y_{j}(zaq^{5}t^{-1})$ is admissible in $X'$.
        In particular, it implies that $d_{j,aq^5t^{-1}}(X')=1$ and $d_{j,aqt^{-1}}(X')=0$. But the equation $\eqref{eqnb1}$ gives
        $$d_{j,aqt^{-1}}(X')=d_{j,aqt^{-1}}(X_{k+1})+1\quad \text{and } \quad d_{j,aq^5t^{-1}}(X')\leq d_{j,aq^5t^{-1}}(X_{k+1})+1.$$
        Hence, $d_{j,aqt^{-1}}(X_{k+1})=-1$ and $d_{j,aq^5t^{-1}}(X_{k+1})\geq 0$. Thus, $X_{k+1}$ contains $Y_{j}(zaqt^{-1})^{-1}$ and does not contain
        $Y_{j}(zaq^5t^{-1})^{-1}$.\\

        Let $s\ge0$ be the maximal integer such that $X_{k+1}$ contains the monomial: $$Y_{j}(zaqt^{-1})^{-1}\dots Y_{j}(zaqt^{-2s-1})^{-1}.$$ 
        If for one integer $1\le u\le s$, $X_{k+1}$ contains $Y_{j}(zaq^5t^{-2u-1})^{-1}$, then it contains the monomial $Y_{j}(zaqt^{-2u+1})^{-1}Y_{j}(zaq^5t^{-2u-1})^{-1}$. 
        However, $X_{k+1}$ has to be regular, so $X_{k+1}$ also contains $Y_{j}(zaq^5t^{-2u+1})^{-1}$. 
        Iterating the same argument we would get that $X_{k+1}$ does contain $Y_{j}(zaq^5t^{-1})^{-1}$, which is a contradiction by the argument above. 
        Hence, we can use $s+1$ times the recurrence hypothesis until a monomial $X_{k+s+2}$. We get: 
        $$\forall 0\leq u \leq s, \quad X_{k+u+2}A_j(zaq^3t^{-2(s-u)-2})^{-1}=X_{k+u+1}$$
        At this step, we have constructed a path from $X_{k+ s +2}$ to $X$ with:
        \begin{equation}\label{eqnb2}
            \begin{split}
            X_{k+s+2}A_j(zaq^3t^{-2})^{-1}\cdots A_j(zaq^3t^{-2s-2})^{-1} \times\qquad\qquad\qquad \\[5mm]
            \qquad A_i(zaqt^{-1})^{-1}\cdots A_i(zaq^{2k+1}t^{-1})^{-1}A_j(zaq^3)^{-1} = X.
            \end{split}
        \end{equation}
        We get the following path in the graph representation of the algorithm:
            \begin{figure}[htbp]
                \begin{center}
                \begin{tikzpicture}[scale=0.7,
                    >=Stealth,
                    every node/.style={align=center, text width=4cm},
                    every edge/.style={-{Stealth[scale=1.2]}, thick}
                ]
                \node (N1) at (0, 0) { \shortstack{$X_{k+s+2}$}};
                \node (N2) at (0, -1.5) { \shortstack{$\vdots$}};
                \node (N3) at (0, -3) { $X_{k+1}$};
                \node (N5) at (0, -4.5) { \shortstack{$\vdots$}};
                \node (N6) at (0, -6) { \shortstack{$X'$}};
                \node (N7) at (0, -7.5) { \shortstack{$X$}};
                \draw[->] (N1) -- (N2) node[midway, right, yshift=0.2cm] {$ A_{j}(zaq^{3}t^{-2})^{-1}$};
                \draw[->] (N5) -- (N6) node[midway, right, yshift=0.2cm] {$ A_{i}(zaq^{2k+1}t^{-1})^{-1}$};
                \draw[->] (N6) -- (N7) node[midway, right, yshift=0.2cm] {$ A_j(zaq^3)^{-1} $};
                \draw[->] (N2) -- (N3) node[midway, right, yshift=0.2cm] {$ A_{j}(zaq^{3}t^{-2s-2})^{-1} $};
                \draw[->] (N3) -- (N5) node[midway, right, yshift=0.2cm] {$ A_{i}(zaq^{}t^{-1})^{-1} $};
                \end{tikzpicture}
                \end{center} 
                \caption{One path from $X_{k+s+2}$ to $X$}\label{gr1}
            \end{figure}
        Finally, 
        \begin{equation}\label{assumption1bb}
            \text{$X_{k+s+2}$ contains $Y_i(zaq^6t^{-2})\dots Y_i(zaq^{2k+2}t^{-2})Y_j(zaq^5t^{-3})\dots Y_j(zaq^5t^{-2s-3})$.}
        \end{equation}

        The aim is now to construct another path from $X_{k+s+2}$ to $X$ ending with the transformation $A_i(zaqt^{-1})^{-1}$ in order to conclude.\\
        We want to obtain the following subgraph:
                       \begin{center}
\begin{tikzpicture}[
    >=Stealth,
    every node/.style={align=center, text width=4cm},
    every edge/.style={-{Stealth[scale=1.2]}, thick}
]
\node (N1) at (0, 0) { \shortstack{$X_{k+s+2}$}};
\node (N2) at (0, -1.5) { \shortstack{$\vdots$}};
\node (N3) at (0, -3) { $X_{k+1}$};
\node (N4) at (-5, -1.5) { \shortstack{$\vdots$}};
\node (N5) at (0, -5.5) { \shortstack{$\vdots$}};
\node (N6) at (0, -7) { \shortstack{$X'$}};
\node (N7) at (0, -8.5) { \shortstack{$X$}};
\node (N9) at (-5, -3) { \shortstack{$Z_{s+3}$}};
\draw[->] (N1) -- (N2) node[midway, right, yshift=0.2cm] {$ A_{j}(zaq^{3}t^{-2})^{-1}$};
\draw[->] (N5) -- (N6) node[midway, right, yshift=0.2cm] {$ A_{i}(zaq^{2k+1}t^{-1})^{-1}$};
\draw[->] (N6) -- (N7) node[midway, right, yshift=0.2cm] {$ A_j(zaq^3)^{-1} $};
\draw[->] (N1) -- (N4) node[midway, left, yshift=0.2cm] {$ A_{i}(zaq^{5}t^{-1})^{-1} $};
\draw[->] (N4) -- (N9) node[midway, left, yshift=0.2cm] {$ A_{i}(zaq^{2k+1}t^{-1})^{-1} $};
\draw[->] (N2) -- (N3) node[midway, right, yshift=0.2cm] {$ A_{j}(zaq^{3}t^{-2s-2})^{-1} $};
\draw[->] (N3) -- (N5) node[midway, right, yshift=0.2cm] {$ A_{i}(zaq^{}t^{-1})^{-1} $};
\end{tikzpicture}
\end{center} 

        Firstly, let us prove that $k\geq2$. The transformation $X'\rightarrow X$ gives that $Y_j(zaq^5t^{-1})$ is admissible in $X'$. Thus, 
        $d_{j,aq^5t^{-1}}(X')=1$. By contradiction, if $k=1$ then from the graph above we deduce that 
        $$d_{j,aq^5t^{-1}}(X')=d_{j,aq^5t^{-1}}(X_{k+s+2})=1.$$
        Hence, $Y_j(zaq^5t^{-3})$ is not admissible in $X_{k+s+2}$, which is a contradiction. Hence, $k\geq2$.\\
        
        To apply the transformation $A_i(zaq^5t^{-1})^{-1}$, we need to check that $Y_i(zaq^{6}t^{-2})$ is admissible in $X_{k+s+2}$.\\
        We know that $X_{k+s+2}$ does not contain $Y_i(zaq^{4}t^{-2})$ as it would imply that $X_{k+s+1}$ contains $Y_i(zaq^{4}t^{-2})^2$ 
        which is a contradiction as all monomials are generic by assumption.
        Moreover, if $X_{k+s+2}$ contains $Y_i(zaq^{6})$ then the equation \eqref{eqnb2} implies:
        $$d_{i,aq^6}(X_{k+s+2})=d_{i,aq^6}(X_{k-1})=1.$$
        But the graph gives that $Y_{i}(zaq^6t^{-2})$ is admissible in $X_{k-1}$. It is a contradiction.
        Thus, $Y_i(zaq^{6}t^{-2})$ is admissible in $X_{k+s+2}$ and
        the algorithm applies the transformation $A_i(zaq^5t^{-1})^{-1}$ to $X_{k+s+2}$, giving a monomial $Z_{k+s+1}$ such that:
        \begin{equation}\label{eqnb4}
            X_{k+s+2}A_i(zaq^5t^{-1})^{-1}=Z_{k+s+1}.
        \end{equation}
        We assume there exists $2\le m\le k-1$ such that the algorithm gives $Z_{u+s+1}$ for all $m< u\le k$ such that:
        $$\forall m< u\le k, \quad Z_{u+s+2}A_i(zaq^{2(k-u+2)+1}t^{-1})^{-1}=Z_{u+s+1},$$
        setting $Z_{k+s+2}=X_{k+s+2}$.
        We want to prove that the algorithm gives $Z_{m+s+1}$ such that:
        $$Z_{m+s+2}A_i(zaq^{2(k-m+2)+1}t^{-1})^{-1}=Z_{m+s+1}.$$
        We know that
        \begin{equation}\label{eqnb3}
            Z_{m+s+2}=X_{k+s+2}A_i(zaq^{5}t^{-1})^{-1}\cdots A_i(zaq^{2(k-m+1)+1}t^{-1})^{-1}.
        \end{equation}
        Then, $Z_{m+s+2}$ contains $Y_i(zaq^{2(k-m+3)}t^{-2})\ldots Y_i(zaq^{2k+2}t^{-2})$.\\
        Let us check that $Y_i(zaq^{2(k-m+3)}t^{-2})$ is admissible in $Z_{m+s+2}$.\\
        It is clear that $Z_{m+s+2}$ does not contain $Y_i(zaq^{2(k-m+2)}t^{-2})$ as it has been removed at the previous step.
        Moreover, if $Z_{m+s+2}$ contains $Y_i(zaq^{2(k-m+3)})$. then by \eqref{eqnb3}, $X_{k+s+2}$ contains $Y_i(zaq^{2(k-m+3)})$.
        Thus, according to the graph, $X_{m-1}$ contains $Y_i(zaq^{2(k-m+3)})$. 
        However, we read in the same figure that $Y_i(zaq^{2(k-m+3)}t^{-2})$ is admissible in $X_{m-1}$. It is a contradiction.
        Hence, $Y_i(zaq^{2(k-m+3)}t^{-2})$ is admissible in $Z_{m+s+2}$ and the algorithm gives a monomial $Z_{m+s+1}$ such that:
        $$Z_{m+s+2}A_i(zaq^{2(k-m+2)+1}t^{-1})^{-1}=Z_{m+s+1}.$$
        By induction, we get monomials $Z_{s+3},\ldots,Z_{k+s+1}$ verifying the equation \eqref{eqnb3}.\\

        Now we want to complete our graph to the following:
                       \begin{center}
\begin{tikzpicture}[
    >=Stealth,
    every node/.style={align=center, text width=4cm},
    every edge/.style={-{Stealth[scale=1.2]}, thick}
]
\node (N1) at (0, 0) { \shortstack{$X_{k+s+2}$}};
\node (N2) at (0, -1.5) { \shortstack{$\vdots$}};
\node (N3) at (0, -3) { $X_{k+1}$};
\node (N4) at (-5, -1.5) { \shortstack{$\vdots$}};
\node (N5) at (0, -5.5) { \shortstack{$\vdots$}};
\node (N6) at (0, -7) { \shortstack{$X'$}};
\node (N7) at (0, -8.5) { \shortstack{$X$}};
\node (N11) at (-5, -7) { \shortstack{$Z_{s+1}$}};
\node (N9) at (-5, -3) { \shortstack{$Z_{s+3}$}};
\node (N10) at (-5, -5) { \shortstack{$Z_{s+2}$}};
\draw[->] (N1) -- (N2) node[midway, right, yshift=0.2cm] {$ A_{j}(zaq^{3}t^{-2})^{-1}$};
\draw[->] (N5) -- (N6) node[midway, right, yshift=0.2cm] {$ A_{i}(zaq^{2k+1}t^{-1})^{-1}$};
\draw[->] (N6) -- (N7) node[midway, right, yshift=0.2cm] {$ A_j(zaq^3)^{-1} $};
\draw[->] (N9) -- (N10) node[midway, left, yshift=0.2cm] {$ A_j(zaq^3)^{-1} $};
\draw[->] (N1) -- (N4) node[midway, left, yshift=0.2cm] {$ A_{i}(zaq^{5}t^{-1})^{-1} $};
\draw[->] (N4) -- (N9) node[midway, left, yshift=0.2cm] {$ A_{i}(zaq^{2k+1}t^{-1})^{-1} $};
\draw[->] (N10) -- (N11) node[midway, left, yshift=0.2cm] {$ A_{j}(zaq^{3}t^{-2})^{-1}$};
\draw[->] (N2) -- (N3) node[midway, right, yshift=0.2cm] {$ A_{j}(zaq^{3}t^{-2s-2})^{-1} $};
\draw[->] (N3) -- (N5) node[midway, right, yshift=0.2cm] {$ A_{i}(zaq^{}t^{-1})^{-1} $};
\end{tikzpicture}
\end{center}  

        Furthermore, by \eqref{eqnb4}, we get:
        $$ d_{j,aq^5t^{-1}}(Z_{s+3})=d_{j,aq^5t^{-1}}(Z_{k+s+1})=d_{j,aq^5t^{-1}}(X_{k+s+2})+1.$$
        But we know by admissibility of $Y_j(zaq^5t^{-3})$ and regularity of $X_{k+s+2}$ that:
        $$d_{j,aq^5t^{-1}}(X_{k+s+2})=0.$$
        Hence, $d_{j,aq^5t^{-1}}(Z_{s+3})=1$.
        
        Thus, the monomial $Z_{s+3}$ contains $Y_j(zaq^5t^{-1})\ldots Y_j(zaq^5t^{-2s-3})$.
        We want to check that $Y_j(zaq^5t^{-1})$ is admissible in $Z_{s+3}$.\\
        If $Z_{s+3}$ contains $Y_j(zaqt^{-1})$ then it contains $Y_j(zaq^5t^{-3})Y_j(zaqt^{-1})$ (this is of the form $Y_j(zc)Y_j(zcq^{-2r_j}t^2)$).
        By regularity, it also contains $Y_j(zaqt^{-3})$.
        By \eqref{eqnb3}, $X_{k+s+2}$ contains $Y_j(zaqt^{-3})$, and $Y_j(zaq^5t^{-3})$ is not admissible in $X_{k+s+2}$. It is a contradiction.\\
        If $Z_{s+3}$ contains $Y_j(zaq^5t)$ then by \eqref{eqnb3}, $X_{k+s+2}$ contains $Y_j(zaq^5t)$. 
        According to Figure \ref{gr1}, this implies that $X'$ contains $Y_j(zaq^5t)$. But $Y_j(zaq^5t^{-1})$ is admissible in $X'$. 
        It is a contradiction.\\
        Hence, $Y_j(zaq^5t^{-1})$ is admissible in $Z_{s+3}$ and the algorithm gives a monomial $Z_{s+2}$ such that:
        \begin{equation}\label{eqnb5}
            Z_{s+3}A_j(zaq^3)^{-1}=Z_{s+2}.
        \end{equation}
        Now we want to construct $Z_{s+1}$ such that 
        $$ Z_{s+2}A_j(zaq^3t^{-2})^{-1}=Z_{s+1}. $$
        Let us prove that $Y_j(zaq^5t^{-3})$ is admissible in $Z_{s+2}$.\\
        By admissibility, 
        $$d_{j,aq^5t^{-3}}(Z_{s+2})=d_{j,aq^5t^{-3}}(X_{k+s+2})=1.$$
        Moreover, $$d_{j,aqt^{-3}}(Z_{s+2})=d_{j,aqt^{-3}}(X_{k+s+2})=0,$$
        and 
        $$d_{j,aq^5t^{-1}}(Z_{s+2})=d_{j,aq^5t^{-1}}(Z_{s+3})-1=1-1=0.$$
        Hence the admissibility and the monomial $Z_{s+1}$ appears in the algorithm.\\

        Finally, let us prove that $Z_{s+1}$ is not regular.\\ 
        By regularity and because $Y_i(zaq^2t^{-2})$ is admissible in $X_{k+1}$, we have $d_{i,aq^2}(X_{k+1})=0$. Thus,
        \begin{align*}
            d_{i,aq^2}(Z_{s+1})&=d_{i,aq^2}(Z_{s+2}),\\
            &=d_{i,aq^2}(Z_{s+3})+1,\\
            &=d_{i,aq^2}(X_{k+s+2})+1,\\
            &=d_{i,aq^2}(X_{k+1})+1,\\
            d_{i,aq^2}(Z_{s+1})&=1.
        \end{align*}
        Moreover, by admissibility, $d_{i,aq^{4}t^{-2}}(X_{k})=1$. Thus,
        \begin{align*}
            d_{i,aq^{4}t^{-2}}(Z_{s+1})&=d_{i,aq^{4}t^{-2}}(Z_{s+2})+1,\\
            &=d_{i,aq^{4}t^{-2}}(X_{k+s+2})+1,\\
            &=d_{i,aq^{4}t^{-2}}(X_{k+s+1})-1+1,\\
            &=d_{i,aq^{4}t^{-2}}(X_{k}),\\
            d_{i,aq^{4}t^{-2}}(Z_{s+1})&=1.
        \end{align*}
        However, by construction of the algorithm, the transformation $X_{k+s+2}\rightarrow Z_{k+s+1}$ gives that $d_{i,aq^4}(Z_{k+s+1})=-1$.
        Thus the graph implies 
        \begin{align*}
            d_{i,aq^{4}}(Z_{s+1})&=d_{i,aq^{4}}(Z_{s+2}),\\
            &=d_{i,aq^{4}}(Z_{s+3})+1,\\
            &=d_{i,aq^{4}}(Z_{k+s+1})+1,\\
            d_{i,aq^{4}}(Z_{s+1})&=0.
        \end{align*}
        Hence setting $c=aq^4t^{-2}$, $Z_{s+1}$ contains $Y_i(zc)Y_i(zcq^{-2r_i}t^2)$ but does not contain $Y_i(zct^2)$. 
        It is a contradiction as assumed that the algorithm does not fail.\\
        \vskip 1cm

        Now, we assume $d_{i,at^{-2}}(X')=-1$. We know that $d_{i,at^{-2}}(X)=0$, then it implies 
        $$d_{i,at^{-2}}(A_j(zbq^{2}t^{-1})^{-1})=1.$$
        Hence, $$bq^3t^{-1}=at^{-2} \qquad \text{or}\qquad bqt^{-1}=at^{-2}.$$
        We assume $bqt^{-1}=at^{-2}$. Firstly, it is clear that $d_{i,a}(X')=-1$ and $d_{i,aq^2}(X')=0$. 
        Then by regularity, $d_{i,aq^2t^{-2}}(X')\neq-1$. Moreover, $$d_{i,aq^2t^{-2}}(A_j(zbq^{2}t^{-1})^{-1})=1.$$
        Then, $d_{i,aq^2t^{-2}}(X')=d_{i,aq^2t^{-2}}(X)-1$. Thus, $d_{i,aq^2t^{-2}}(X)=1$ (all degrees are in $\{-1,0,1\}$ by genericity).
        Hence, $X$ contains $Y_i(zaq^2t^{-2})Y_i(za)^{-1}$. Let us prove that $Y_i(zaq^2t^{-2})$ is admissible in $X$.
        If $d_{i,aq^2}(X)=1$ (resp. $d_{i,at^{-2}}(X)=1$) then $X$ is not regular as it contains a 
        term of the form $Y_i(zc)Y_i(zcq^{-2r_i})^{-1}$ (resp. $Y_i(zc)Y_i(zct^{2})^{-1}$). Hence the admissibility of $Y_i(zaq^2t^{-2})$in $X$.
        So, the algorithm implies a transformation to obtain the monomial 
        $$X^*=XA_i(zaqt^{-1})^{-1},$$
        and $d_{i,a}(X^*)=-2$ which contradicts the genericity.\\

        Hence, $b=aq^{-3}t^{-1}$. As in the previous case, we want to apply the induction hypothesis and get the following graph for a $k$ defined above:

                    \begin{figure}[htbp]
                \begin{center}
                \begin{tikzpicture}[scale=0.7,
                    >=Stealth,
                    every node/.style={align=center, text width=4cm},
                    every edge/.style={-{Stealth[scale=1.2]}, thick}
                ]
                \node (N3) at (0, -3) { $X_{k+1}$};
                \node (N5) at (0, -4.5) { \shortstack{$\vdots$}};
                \node (N6) at (0, -6) { \shortstack{$X'$}};
                \node (N7) at (0, -7.5) { \shortstack{$X$}};
                \draw[->] (N5) -- (N6) node[midway, right, yshift=0.2cm] {$ A_{i}(zaq^{}t^{-2k-1})^{-1}$};
                \draw[->] (N6) -- (N7) node[midway, right, yshift=0.2cm] {$ A_j(zaq^{-1}t^{-2})^{-1} $};
                \draw[->] (N3) -- (N5) node[midway, right, yshift=0.2cm] {$ A_{i}(zaq^{}t^{-1})^{-1} $};
                \end{tikzpicture}
                \end{center}
            \end{figure}

        Let $k>0$ be the maximal integer such that $Y_i(zat^{-2s })^{-1}$ appears in $X'$ for all $0\le s\le k$.\\
        By definition, $Y_i(zat^{-2(k+1)})^{-1}$ is not in $X'$. 
        Moreover, if $Y_i(zaq^2t^{-2k})^{-1}$ appears in $X'$, then $Y_i(zat^{-2(k-1) })^{-1}Y_i(zaq^2t^{-2 k})^{-1}$ also appears in $X'$. 
        The monomial has to be regular, so $Y_i(zat^{-2k })^{-1}Y_i(zaq^2t^{-2( k-1)})^{-1}$ appears in $X'$. \\
        Thus, $X'$ contains $Y_i(zat^{-2(k-2) })^{-1}Y_i(zaq^2t^{-2(k-1) })^{-1}$.
        We can iterate this reasoning until we get that $Y_i(zaq^{2})^{-1}$ appears in $X'$. 
        However it does not appear in $X$ and has to be simplified by the transformation $A_j(zbq^2t^{-1})^{-1}=A_j(zaq^{-1}t^{-2})^{-1}$. 
        It is a contradiction by definition of $A_j$. 
        Hence, we can apply the induction hypothesis and we obtain that there exists $X_1$ given by the algorithm such that 
        $$X_1A_{i}(zaqt^{-(2k+1)})^{-1}=X'.$$ 
        We can define recursively 
        $X_s$, $2\le s\le k+1$ such that $$X_s A_{i}(zaqt^{-(2k-2s+3)})^{-1}=X_{s-1}.$$ 
        Indeed, for all $1\leq s\leq k+1$, $Y_i(zat^{-2u })^{-1}$ appears in $X_s$ for all $0\le u\le k+1-s$ and $Y_i(zaq^2t^{-2(u-1) })^{-1}$ does not appear in $X_s$ nor $Y_i(zat^{-2(k+2-s) })^{-1}$ so that we can apply the induction hypothesis. 
        In particular, we have constructed $X_{k+1}$ such that $$X_{k+1}A_{i}(zaqt^{-1})^{-1}=X_k.$$
        Thus, 
        \begin{equation}\label{eqnbb1}
            X_{k+1}A_{i}(zaqt^{-1})^{-1}A_{i}(zaqt^{-3})^{-1}\cdots  A_{i}(zaqt^{-(2k+1)})^{-1}=X'.
        \end{equation}

        By definition of the algorithm, the arrow $X'\rightarrow X$ implies that $Y_{j}(zaqt^{-3})$ is admissible in $X'$.
        In particular, it implies that $d_{j,aqt^{-3}}(X')=1$ and $d_{j,aqt^{-1}}(X')=0$. But the equation $\eqref{eqnbb1}$ gives
        $$d_{j,aqt^{-1}}(X')=d_{j,aqt^{-1}}(X_{k+1})+1\quad \text{and } \quad d_{j,aqt^{-3}}(X')= d_{j,aqt^{-3}}(X_{k+1})+1.$$
        Hence, $d_{j,aqt^{-1}}(X_{k+1})=-1$ and $d_{j,aqt^{-3}}(X_{k+1})= 0$. Thus, $X_{k+1}$ contains $Y_{j}(zaqt^{-1})^{-1}$ and does not contain
        $Y_{j}(zaqt^{-3})^{-1}$.\\

        Let $s\ge0$ be the maximal integer such that $X_{k+1}$ contains the monomial: 
        $$Y_{j}(zaqt^{-1})^{-1}\dots Y_{j}(zaq^{4s+1}t^{-1})^{-1}.$$ 
        If for one integer $1\le u\le s$, $X_{k+1}$ contains $Y_{j}(zaq^{4u+1}t^{-3})^{-1}$, then it contains the monomial $Y_{j}(zaq^{4(u-1)+1}t^{-1})^{-1}Y_{j}(zaq^{4u+1}t^{-3})^{-1}$. 
        However, $X_{k+1}$ has to be regular, so $X_{k+1}$ also contains $Y_{j}(zaq^{4(u-1)+1}t^{-3})^{-1}$. 
        Iterating the same argument we would get that $X_{k+1}$ does contain $Y_{j}(zaqt^{-3})^{-1}$, which is a contradiction by the argument above. 
        Hence, we can use $s+1$ times the recurrence hypothesis until a monomial $X_{k+s+2}$. We get: 
        $$\forall 0\leq u \leq s, \quad X_{k+u+2}A_j(zaq^{4(s-u)+3}t^{-2})^{-1}=X_{k+u+1}$$
        At this step, we have constructed a path from $X_{k+ s +2}$ to $X$ with:
        \begin{equation}\label{eqnb22}
            \begin{split}
            X_{k+s+2}A_j(zaq^3t^{-2})^{-1}\cdots A_j(zaq^{4s+3}t^{-2})^{-1} \times\qquad\qquad\qquad \\[5mm]
            \qquad A_i(zaqt^{-1})^{-1}\cdots A_i(zaqt^{-2k-1})^{-1}A_j(zaq^{-1}t^{-2})^{-1} = X.
            \end{split}
        \end{equation}

        We get the following path in the graph representation of the algorithm:
                \begin{center}
                \begin{tikzpicture}[scale=0.7,
                    >=Stealth,
                    every node/.style={align=center, text width=4cm},
                    every edge/.style={-{Stealth[scale=1.2]}, thick}
                ]
                \node (N1) at (0, 0) { \shortstack{$X_{k+s+2}$}};
                \node (N2) at (0, -1.5) { \shortstack{$\vdots$}};
                \node (N3) at (0, -3) { $X_{k+1}$};
                \node (N5) at (0, -4.5) { \shortstack{$\vdots$}};
                \node (N6) at (0, -6) { \shortstack{$X'$}};
                \node (N7) at (0, -7.5) { \shortstack{$X$}};
                \draw[->] (N1) -- (N2) node[midway, right, yshift=0.2cm] {$ A_{j}(zaq^{3}t^{-2})^{-1}$};
                \draw[->] (N5) -- (N6) node[midway, right, yshift=0.2cm] {$ A_i(zaqt^{-2k-1})^{-1}$};
                \draw[->] (N6) -- (N7) node[midway, right, yshift=0.2cm] {$ A_j(zaq^{-1}t^{-2})^{-1} $};
                \draw[->] (N2) -- (N3) node[midway, right, yshift=0.2cm] {$ A_j(zaq^{4s+3}t^{-2})^{-1} $};
                \draw[->] (N3) -- (N5) node[midway, right, yshift=0.2cm] {$ A_{i}(zaq^{}t^{-1})^{-1} $};
                \end{tikzpicture}
                \end{center} 
        Finally, 
        \begin{equation}\label{assumption1cc}
            \text{$X_{k+s+2}$ contains $Y_i(zaq^2t^{-4})\dots Y_i(zaq^{2}t^{-2k-2})Y_j(zaq^5t^{-3})\dots Y_j(zaq^{4s+5}t^{-3})$.}
        \end{equation}

        The aim is now to construct another path from $X_{k+s+2}$ to $X$ ending with the transformation $A_i(zaqt^{-1})^{-1}$ in order to conclude.\\
        Firstly, $$d_{i,aq^2t^{-4}}(X_{k+s+2})=d_{i,aq^2t^{-4}}(X_{k})=1,$$
        because as shown in the graph above, $Y_i(zaq^2t^{-4})$ is admissible in $X_{k}$. 
        Moreover, $d_{i,at^{-4}}(X_{k+s+2})=d_{i,at^{-4}}(X_{k})=0$ for the same argument of admissibility. 
        Additionally, if $d_{i,aq^2t^{-2}}(X_{k+s+2})=1$ then  by definition of $A_j(zaq^3t^{-2})^{-1}$ it implies
        $d_{i,aq^2t^{-2}}(X_{k+s+1})=2$ which is impossible by the genericity hypothesis.\\
        Hence, $Y_i(zaq^2t^{-4})$ is admissible in $X_{k+s+2}$.\\
        We construct $$Z_{k+s+1}:=X_{k+s+2}A_i(zaqt^{-3})^{-1}.$$
        We assume there exists $1\leq m \leq k$ such that we constructed $Z_{s+u+1}$, for $m< u\leq k$ such that 
        $$Z_{u+s+1}=Z_{u+s+2}A_i(zaqt^{-2(k-u)-3})^{-1}.$$
        Let us construct $Z_{s+m+1}$.\\
        Firstly, according to the graph, $$d_{i,aq^2t^{-2(k-m)-4}}(Z_{s+m+2})=d_{i,aq^2t^{-2(k-m)-4}}(X_m)=1.$$
        Moreover, $$d_{i,aq^2t^{-2(k-m)-2}}(Z_{s+m+2})=d_{i,aq^2t^{-2(k-m)-2}}(Z_{s+m+3})-1=d_{i,aq^2t^{-2(k-m)-2}}(X_{k+s+2})-1=0,$$
        because of the transformation $Z_{s+m+3}\rightarrow Z_{s+m+2}$.\\
        Additionally, $Y_i(zaq^2t^{-2(k-m)-4})$ is admissible in $X_m$. Thus, 
        $$d_{i,at^{-2(k-m)-4}}(Z_{s+m+2})=d_{i,at^{-2(k-m)-4}}(X_{k+s+2})=d_{i,at^{-2(k-m)-4}}(X_m)=0.$$
        Hence the admissibility of $Y_i(zaq^2t^{-2(k-m)-4})$ in $Z_{s+m+2}$ and the existence of the following
         transformation given by the algorithm:
         $$ Z_{m+s+1}=Z_{m+s+2}A_i(zaqt^{-2(k-m)-3})^{-1}. $$
         We obtain the following graph:

        \begin{center}
\begin{tikzpicture}[
    >=Stealth,
    every node/.style={align=center, text width=4cm},
    every edge/.style={-{Stealth[scale=1.2]}, thick}
]
\node (N1) at (0, 0) { \shortstack{$X_{k+s+2}$}};
\node (N2) at (0, -1.5) { \shortstack{$\vdots$}};
\node (N3) at (0, -3) { $X_{k+1}$};
\node (N4) at (-5, -1.5) { \shortstack{$\vdots$}};
\node (N5) at (0, -5.5) { \shortstack{$\vdots$}};
\node (N6) at (0, -7) { \shortstack{$X'$}};
\node (N7) at (0, -8.5) { \shortstack{$X$}};
\node (N9) at (-5, -3) { \shortstack{$Z_{s+2}$}};
\draw[->] (N1) -- (N2) node[midway, right, yshift=0.2cm] {$ A_{j}(zaq^{3}t^{-2})^{-1}$};
\draw[->] (N5) -- (N6) node[midway, right, yshift=0.2cm] {$ A_i(zaqt^{-2k-1})^{-1}$};
\draw[->] (N6) -- (N7) node[midway, right, yshift=0.2cm] {$ A_j(zaq^{-1}t^{-2})^{-1} $};
\draw[->] (N1) -- (N4) node[midway, left, yshift=0.2cm] {$ A_{i}(zaq^{}t^{-3})^{-1} $};
\draw[->] (N4) -- (N9) node[midway, left, yshift=0.2cm] {$ A_i(zaqt^{-2k-1})^{-1} $};
\draw[->] (N2) -- (N3) node[midway, right, yshift=0.2cm] {$ A_j(zaq^{4s+3}t^{-2})^{-1} $};
\draw[->] (N3) -- (N5) node[midway, right, yshift=0.2cm] {$ A_{i}(zaq^{}t^{-1})^{-1} $};
\end{tikzpicture}
        \end{center}

        Again, we want to construct a path from $Z_{s+2}$ to $X$ in the algorithm graph.\\
        Firstly, 
        \begin{align*}
            d_{j,aqt^{-3}}(Z_{s+2})&=d_{j,aqt^{-3}}(X_{k+s+2})+1,\\
            &=d_{j,aqt^{-3}}(X_{k})+1,\\
            &=d_{j,aqt^{-3}}(X_{k-1})-1+1,\\
            &=d_{j,aqt^{-3}}(X')=1.
        \end{align*}
        Because $Y_j(zaqt^{-3})$ is admissible in $X'$.\\
        Moreover, $$d_{j,aq^{-3}t^{-3}}(Z_{s+2})=d_{j,aq^{-3}t^{-3}}(X')=0,$$
        as $Y_j(zaqt^{-3})$ is admissible in $X'$. Additionally, we deduce from the graph the following computation
        \begin{align*}
            d_{j,aqt^{-1}}(Z_{s+2})&=d_{j,aqt^{-1}}(X_{k+s+2}),\\
            &=d_{j,aqt^{-1}}(X_{k+s+1})-1,\\
            &=d_{j,aqt^{-1}}(X_{k})+1-1,\\
            &=d_{j,aqt^{-1}}(X')=0.
        \end{align*}
        Hence, $Y_j(zaqt^{-3})$ is admissible in $Z_{s+2}$. 
        Thus, the algorithm constructs the following monomial:
        $$ Z_{s+1}=Z_{s+2}A_j(zaq^{-1}t^{-2})^{-1}.$$
        Now, we assume there exists $0\leq m \leq s$ such that we constructed $Z_{u+1}$, for $m\leq u\leq s$ such that 
        $$Z_{u+1}=Z_{u+2}A_j(zaq^{4(s-u)-1}t^{-2})^{-1}.$$
        Let us construct $Z_{m}$.\\
        Let us prove that $Y_j(zaq^{4(s-m)+5}t^{-3})$ is admissible in $Z_{m+1}$.\\
        Firstly, according to the graph and by admissibility, $$d_{j,aq^{4(s-m)+5}t^{-3}}(Z_{m+1})=d_{j,aq^{4(s-m)+5}t^{-3}}(X_{k+m+2})=1.$$
        Moreover, $$d_{j,aq^{4(s-m)+1}t^{-3}}(Z_{m+1})=d_{j,aq^{4(s-m)+1}t^{-3}}(Z_{m+2})-1=0,$$
        by admissibility and because of the transformation $Z_{m+2}\rightarrow Z_{m+1}$.\\
        Additionally, $Y_j(zaq^{4(s-m)+5}t^{-3})$ is admissible in $X_{k+m+2}$. Thus, 
        $$d_{j,aq^{4(s-m)+5}t^{-1}}(Z_{m+1})=d_{j,aq^{4(s-m)+5}t^{-1}}(X_{k+m+2})=0.$$
        Hence the admissibility of $Y_j(zaq^{4(s-m)+5}t^{-3})$ in $Z_{m+1}$ and the existence of the following
         transformation given by the algorithm:
         $$ Z_{m}=Z_{m+1}A_j(zaq^{4(s-m)+3}t^{-2})^{-1}. $$
         To conclude, we have to prove that $Y_i(zaq^2t^{-2})$ is admissible in $Z_0$.\\
         We have the equality $Z_0A_i(zaqt^{-1})^{-1}=X$. Then, by identification of the degrees it is clear. \\
         Hence the following paths of transformations in the algorithm:
         \begin{center}
\begin{tikzpicture}[
    >=Stealth,
    every node/.style={align=center, text width=4cm},
    every edge/.style={-{Stealth[scale=1.2]}, thick}
]
\node (N1) at (0, 0) { \shortstack{$X_{k+s+2}$}};
\node (N2) at (0, -1.5) { \shortstack{$\vdots$}};
\node (N3) at (0, -3) { $X_{k+1}$};
\node (N4) at (-5, -1.5) { \shortstack{$\vdots$}};
\node (N5) at (0, -5.5) { \shortstack{$\vdots$}};
\node (N6) at (0, -7) { \shortstack{$X'$}};
\node (N7) at (0, -8.5) { \shortstack{$X$}};
\node (N9) at (-5, -3) { \shortstack{$Z_{s+2}$}};
\node (N10) at (-5, -5.5) { \shortstack{$\vdots$}};
\node (N11) at (-5, -7) { \shortstack{$Z_0$}};
\draw[->] (N1) -- (N2) node[midway, right, yshift=0.2cm] {$ A_{j}(zaq^{3}t^{-2})^{-1}$};
\draw[->] (N5) -- (N6) node[midway, right, yshift=0.2cm] {$ A_i(zaqt^{-2k-1})^{-1}$};
\draw[->] (N6) -- (N7) node[midway, right, yshift=0.2cm] {$ A_j(zaq^{-1}t^{-2})^{-1} $};
\draw[->] (N9) -- (N10) node[midway, left, yshift=0.2cm] {$ A_j(zaq^{-1}t^{-2})^{-1} $};
\draw[->] (N1) -- (N4) node[midway, left, yshift=0.2cm] {$ A_{i}(zaq^{}t^{-3})^{-1} $};
\draw[->] (N4) -- (N9) node[midway, left, yshift=0.2cm] {$ A_i(zaqt^{-2k-1})^{-1} $};
\draw[->] (N2) -- (N3) node[midway, right, yshift=0.2cm] {$ A_j(zaq^{4s+3}t^{-2})^{-1} $};
\draw[->] (N10) -- (N11) node[midway, left, yshift=0.2cm] {$ A_j(zaq^{4s+3}t^{-2})^{-1} $};
\draw[->] (N3) -- (N5) node[midway, right, yshift=0.2cm] {$ A_{i}(zaq^{}t^{-1})^{-1} $};
\draw[->] (N11) -- (N7) node[midway, left, yshift=0cm] {$ A_{i}(zaq^{}t^{-1})^{-1} $};
\end{tikzpicture}
        \end{center}

        \textit{b)} We assume there exists a monomial $X''$ also appearing in the algorithm such that $$X''A_{i}(zaq^{r_i}t^{-1})^{-1}=X',$$
        then $X''$ does contain $Y_i(zaq^{2r_i}t^{-2})$.\\ 

        Here there are two cases to consider:
        \begin{itemize}
            \item $d_{j,bq^{2r_j}t^{-2}}(A_{i}(zaq^{r_i}t^{-1})^{-1})=1$ and the term $Y_j(zbq^{2r_j}t^{-2})$ in $X'$ comes from 
            the transformation $A_{i}(zaq^{r_i}t^{-1})^{-1}$.
            \item $X''$ contains $Y_j(zbq^{2r_j}t^{-2})$.
        \end{itemize}

        Firstly, we assume $d_{j,bq^{2r_j}t^{-2}}(A_{i}(zaq^{r_i}t^{-1})^{-1})=1$ and the term $Y_i(zaq^{2r_i}t^{-2})$ comes from 
        the transformation $A_{i}(zaq^{r_i}t^{-1})^{-1}$. Hence, $i\neq j$, and we assume the Dynkin subdiagram with nodes 
        $\{i,j\}$ is of type $B_2$ with $(i,j)=(2,1)$.
        Thus, $bq^4t^{-2}=aqt^{-1}$, and $bq^3t^{-1}=a$. Hence, 
        $$d_{i,a}(A_j(zbq^2t^{-1})^{-1})=1.$$
        However, $d_{i,a}(X)=-1$ and $X'A_{j}(zbq^{2}t^{-1})^{-1}=X$. Thus, $d_{i,a}(X')=-2$ which contradicts the genericity. It is a contradiction.\\
        
        Hence, we are in the second case and $X''$ contains $Y_j(zbq^{2r_j}t^{-2})$.

        Let us prove that $Y_j(zbq^{2r_j}t^{-2})$ is admissible in $X''$.\\
        If $i=j$ (so that $r_i=r_j$), then by contradiction we assume that $Y_j(zbq^{2r_j}t^{-2})$ is 
        not admissible in $X''$. We know that $Y_j(zbq^{2r_j}t^{-2})$ is admissible in $X'$ so that the presence of 
        $Y_i(zaq^{2r_i}t^{-2})$  in $X''$ prevents from applying the transformation $A_j(zbq^{r_j}t^{-1})^{-1}$. 
        It implies that $a=bq^{-2r_i}$ or $a=bt^2$. 
        This implies $X$ contains $Y_i(zb)^{-1}=Y_i(zaq^{2r_i})^{-1}$ or $Y_i(zat^{-2})^{-1}$. 
        It is a contradiction by the initial assumption on $X$. 
        Thus, $Y_j(zbq^{2r_j}t^{-2})$ is admissible in $X''$.\\
        If $i\neq j$ then it is clear that the admissibility of $Y_j(zbq^{2r_j}t^{-2})$ in $X'$ implies its admissibility in $X''$. \\
        Finally, there exists a monomial $$X_{new}=X''A_{j}(zbq^{r_j}t^{-1})^{-1}$$ given by the algorithm and $X_{new}$ contains $Y_i(zaq^{2r_i}t^{-2})$.
        Moreover, if the transformation $A_{i}(zaq^{r_i}t^{-1})^{-1}$ is not admissible from $X_{new}$ then there exists a factor 
        $Y_i(zat^{-2})$ (resp. $Y_i(zaq^{2r_i})$) blocking this transformation. 
        But this would imply that this factor also appears in $X$ as we have the following equality of monomials:
        $$X=X_{new}A_{i}(zaq^{r_i}t^{-1})^{-1}.$$
        We remark that a priori, this equality is not sufficient to get an arrow from $X_{new}$ to $X$. This equality implies:
        \begin{align*}
            d_{i,aq^{2r_i}}(X_{new})&=d_{i,aq^{2r_i}}(X),
        \end{align*}
        and same for $at^{-2}$. 
        And hence $X$ does contain $Y_i(zc)$ and $Y_i(zct^2)^{-1}$ with $c=at^{-2}$ (resp. $Y_i(zc)Y_i(zcq^{-2r_i})^{-1}$ with $c=aq^{2r_i}$) 
        which contradicts the regularity of $X$. It is impossible. Hence the admissibility of $Y_i(zaq^{2r_i}t^{-2})$ in $X_{new}$.\\
        Finally we get the following paths:
        \begin{center}
\begin{tikzpicture}[
    >=Stealth,
    every node/.style={align=center, text width=4cm},
    every edge/.style={-{Stealth[scale=1.2]}, thick}
]
\node (N1) at (0, 0) { \shortstack{$X''$}};
\node (N2) at (0, -3) { \shortstack{$X'$}};
\node (N3) at (0, -6) { \shortstack{$X$}};
\node (N4) at (-3, -3) { \shortstack{$X_{new}$}};
\draw[->] (N1) -- (N2) node[midway, right, yshift=0.2cm] {$ A_{i}(zaq^{r_i}t^{-1})^{-1}$};
\draw[->] (N2) -- (N3) node[midway, right, yshift=0.2cm] {$ A_j(zbq^{r_j}t^{-1})^{-1} $};
\draw[->] (N4) -- (N3) node[midway, left, yshift=0.2cm] {$ A_{i}(zaq^{r_i}t^{-1})^{-1} $};
\draw[->] (N1) -- (N4) node[midway, left, yshift=0.2cm] {$ A_j(zbq^{r_j}t^{-1})^{-1}$};
\end{tikzpicture}
\end{center}

and hence the result.
        This concludes the proof in type $B_2$ for $(i,j)=(2,1)$.

        \end{proof}

\subsubsection{In the case $(i,j)=(1,2)$ }

\begin{proof}
    We prove it by induction on the height of the monomials. We assume the Lie subalgebra generated by the nodes $i$ and $j$ is of type $B_2$, and $(i,j)=(1,2)$.
         \textit{\underline{Height $0$:}} The only monomial with height $0$ is the dominant generic monomial from which we start our algorithm. 
        Hence the property holds at height $0$.\\

         \textit{\underline{Height $h+1$:}} We assume the property is true for all the monomials with heights $h'\leq h\in\N$. 
        Let us prove it for the monomials with height $h+1$. Let $X$ be a monomial with height $h+1$. 
        The monomial appears in the algorithm, so it has to come from a monomial $X'$ with height $h$. 
        Let $A_j(zbq^{r_j}t^{-1})^{-1}$ be the transformation involved. 
        It implies that $d_{j,b}(X)=-1$ and $d_{j,bq^{2r_j}t^{-2}}(X')=1$. 
        Let $i\in I$, $a\in q^\Z t^\Z$ such that $d_{i,a}(X)=-1$ and $d_{i,aq^{2r_i}}(X),d_{i,at^{-2}}(X)\ne -1$. 
        If $(i,a)=(j,b)$ then we have the result. 
        Else, we have
        $$X=X'A_{j}(zbq^{r_j}t^{-1})^{-1}.$$
        Hence,
        $$ \left\{\begin{array}{llll}
            d_{j,bq^{2r_j}t^{-2}}(X')=1\\
            d_{j,b}(X)=-1\\
            d_{j,c}(X)=d_{j,c}(X')\text{ for all $c\not\in\{b,bq^{2r_j}t^{-2}\}$}\\
            d_{k,c}(X)\geq d_{k,c}(X')\text{ for all $k\neq j$, $c\in q^\Z t^\Z$}
        \end{array}\right.$$
        We know that $d_{i,a}(X)=-1$, and $(i,a)\neq (j,b)$.  
        It implies $d_{i,a}(X')\leq-1$. The monomials are all generic, thus 
        $$ d_{i,a}(X')=-1$$
        Then by the induction assumption, one of the two following assertions is true:\\
        \begin{itemize}
            \item[(a)]  $d_{i,aq^{2r_i}}(X')=-1$  {or}  $d_{i,at^{-2}}(X')= -1$.
            \item[(b)]  There exists a monomial $X''$ appearing in the algorithm such that
        $$X''A_{i}(zaq^{r_i}t^{-1})^{-1}=X'.$$
        \end{itemize}

        \textit{a)} Firstly we assume (a) is true and $X'$ does contain $Y_i(zaq^{2r_i})^{-1}$ (i.e $d_{i,aq^{2r_i}}(X')=-1$). 
        By assumption, $X$ does not. 
        It implies that the transformation $A_j(zbq^{r_j}t^{-1})^{-1}$ simplifies $Y_i(zaq^{2r_i})^{-1}$ and in particular $i\ne j$. 
        Therefore, the expressions depend on the type of the Dynkin diagram generated by the nodes $i$ and $j$. \\

        Let us do it in type $B_2$, $(i,j)=(1,2)$.  Thus, $(r_i,r_j)=(2,1)$. 
        Then $ bqt^{-1}=aq^4  $ and $b=aq^3t$
        
        Let $k>0$ be the maximal integer such that $Y_i(zaq^{4s })^{-1}$ appears in $X'$ for all $0\le s\le k$. 
        We want to prove that the graph representation
        of the algorithm contains the following subgraph:
            \begin{figure}[htbp]
            \begin{center}
            \begin{tikzpicture}[scale=0.7,
                >=Stealth,
                every node/.style={align=center, text width=4cm},
                every edge/.style={-{Stealth[scale=1.2]}, thick}
            ]
            \node (N3) at (0, -3) { $X_{k+1}$};
            \node (N5) at (0, -4.5) { \shortstack{$\vdots$}};
            \node (N6) at (0, -6) { \shortstack{$X'$}};
            \node (N7) at (0, -7.5) { \shortstack{$X$}};
            \draw[->] (N5) -- (N6) node[midway, right, yshift=0.2cm] {$ A_{i}(zaq^{4k+2}t^{-1})^{-1}$};
              \draw[->] (N6) -- (N7) node[midway, right, yshift=0.2cm] {$ A_j(zaq^4)^{-1} $};
             \draw[->] (N3) -- (N5) node[midway, right, yshift=0.2cm] {$ A_{i}(zaq^{2}t^{-1})^{-1} $};
             \end{tikzpicture}
              \end{center} 
        \end{figure}
        By definition, $Y_i(zaq^{4(k+1)})^{-1}$ is not in $X'$. 
        Moreover, if $Y_i(zaq^{4k}t^{-2})^{-1}$ appears in $X'$, then $Y_i(zaq^{4(k-1) })^{-1}Y_i(zaq^{4 k}t^{-2})^{-1}$ also appears in $X'$. 
        The monomial has to be regular, so $Y_i(zaq^{4k })^{-1}Y_i(zaq^{4( k-1)}t^{-2})^{-1}$ appears in $X'$. \\
        Thus, $X'$ contains $Y_i(zaq^{4(k-2) })^{-1}Y_i(zaq^{4(k-1) }t^{-2})^{-1}$.
        We can iterate this reasoning until we get that $Y_i(zat^{-2})^{-1}$ appears in $X'$. 
        However it does not appear in $X$ and has to be simplified by the transformation $A_j(zbq^{r_j}t^{-1})^{-1}=A_j(zaq^4)^{-1}$. 
        It is a contradiction by definition of $A_j$. 
        Hence, we can apply the induction hypothesis and we obtain that there exists $X_1$ given by the algorithm such that 
        $$X_1A_{i}(zaq^{4k+2}t^{-1})^{-1}=X'.$$ 
        We can define recursively 
        $X_s$, $2\le s\le k+1$ such that $$X_s A_{i}(zaq^{4k-4s+6}t^{-1})^{-1}=X_{s-1}.$$ 
        Indeed, for all $1\leq s\leq k+1$, $Y_i(zaq^{4u })^{-1}$ appears in $X_s$ for all $0\le u\le k+1-s$ and 
        $Y_i(zaq^{4(u-1) }t^{-2})^{-1}$ does not appear in $X_s$ nor 
        $Y_i(zaq^{4(k+2-s) })^{-1}$ so that we can apply the induction hypothesis. 
        In particular, we have constructed $X_{k+1}$ such that $$X_{k+1}A_{i}(zaq^{2 }t^{-1})^{-1}=X_k.$$
        Moreover, 
        \begin{equation}\label{eqnbbb1}
            X_{k+1}A_{i}(zaq^{2}t^{-1})^{-1}A_{i}(zaq^{6}t^{-1})^{-1}\cdots  A_{i}(zaq^{4k+2}t^{-1})^{-1}=X'.
        \end{equation}
        We recall that $X'A_{j}(zaq^{4})^{-1}=X$. By definition of the algorithm, it implies $Y_{j}(zaq^{5}t^{-1})$ is admissible in $X'$.
        In particular, it implies that $d_{j,aq^5t^{-1}}(X')=1$ and $d_{j,aq^3t^{-1}}(X')=0$. But the equation $\eqref{eqnbbb1}$ gives
        $$d_{j,aq^3t^{-1}}(X')=d_{j,aq^3t^{-1}}(X_{k+1})+1\quad \text{and } \quad d_{j,aq^5t^{-1}}(X')\leq d_{j,aq^5t^{-1}}(X_{k+1})+1.$$
        Hence, $d_{j,aq^3t^{-1}}(X_{k+1})=-1$ and $d_{j,aq^5t^{-1}}(X_{k+1})= 0$. Thus, $X_{k+1}$ contains $Y_{j}(zaq^3t^{-1})^{-1}$ and does not contain
        $Y_{j}(zaq^5t^{-1})^{-1}$.\\

        Let $s\ge0$ be the maximal integer such that $X_{k+1}$ contains the monomial: $$Y_{j}(zaq^3t^{-1})^{-1}\dots Y_{j}(zaq^3t^{-2s-1})^{-1}.$$ 
        If for one integer $1\le u\le s$, $X_{k+1}$ contains $Y_{j}(zaq^5t^{-2u-1})^{-1}$, then it contains the monomial 
        $Y_{j}(zaq^3t^{-2u+1})^{-1}Y_{j}(zaq^5t^{-2u-1})^{-1}$. 
        However, $X_{k+1}$ has to be regular, so $X_{k+1}$ also contains $Y_{j}(zaq^5t^{-2u+1})^{-1}$. 
        Iterating the same argument we would get that $X_{k+1}$ does contain $Y_{j}(zaq^5t^{-1})^{-1}$, 
        which is a contradiction by the argument above. 
        Hence, we can use $s+1$ times the recurrence hypothesis until a monomial $X_{k+s+2}$. We get: 
        $$\forall 0\leq u \leq s, \quad X_{k+u+2}A_j(zaq^4t^{-2(s-u)-2})^{-1}=X_{k+u+1}$$
        At this step, we have constructed a path from $X_{k+ s +2}$ to $X$ with:
        \begin{equation}\label{eqnbbb2}
            \begin{split}
            X_{k+s+2}A_j(zaq^4t^{-2})^{-1}\cdots A_j(zaq^4t^{-2s-2})^{-1} \times\qquad\qquad\qquad \\[5mm]
            \qquad A_i(zaq^2t^{-1})^{-1}\cdots A_i(zaq^{4k+2}t^{-1})^{-1}A_j(zaq^4)^{-1} = X.
            \end{split}
        \end{equation}
        We get the following path in the graph representation of the algorithm:
            \begin{figure}[htbp]
                \begin{center}
                \begin{tikzpicture}[scale=0.7,
                    >=Stealth,
                    every node/.style={align=center, text width=4cm},
                    every edge/.style={-{Stealth[scale=1.2]}, thick}
                ]
                \node (N1) at (0, 0) { \shortstack{$X_{k+s+2}$}};
                \node (N2) at (0, -1.5) { \shortstack{$\vdots$}};
                \node (N3) at (0, -3) { $X_{k+1}$};
                \node (N5) at (0, -4.5) { \shortstack{$\vdots$}};
                \node (N6) at (0, -6) { \shortstack{$X'$}};
                \node (N7) at (0, -7.5) { \shortstack{$X$}};
                \draw[->] (N1) -- (N2) node[midway, right, yshift=0.2cm] {$ A_{j}(zaq^{4}t^{-2})^{-1}$};
                \draw[->] (N5) -- (N6) node[midway, right, yshift=0.2cm] {$ A_{i}(zaq^{4k+2}t^{-1})^{-1}$};
                \draw[->] (N6) -- (N7) node[midway, right, yshift=0.2cm] {$ A_j(zaq^4)^{-1} $};
                \draw[->] (N2) -- (N3) node[midway, right, yshift=0.2cm] {$ A_{j}(zaq^{4}t^{-2s-2})^{-1} $};
                \draw[->] (N3) -- (N5) node[midway, right, yshift=0.2cm] {$ A_{i}(zaq^{2}t^{-1})^{-1} $};
                \end{tikzpicture}
                \end{center} 
                \caption{One path from $X_{k+s+2}$ to $X$}\label{grr1}
            \end{figure}
        Finally, 
        \begin{equation}\label{assumption1dd}
            \text{$X_{k+s+2}$ contains $Y_i(zaq^8t^{-2})\dots Y_i(zaq^{4k+4}t^{-2})Y_j(zaq^5t^{-3})\dots Y_j(zaq^5t^{-2s-3})$.}
        \end{equation}

        The aim is now to construct another path from $X_{k+s+2}$ to $X$ ending with the transformation 
        $A_i(zaq^2t^{-1})^{-1}$ in order to conclude.\\
        We want to obtain the following subgraph:
                       \begin{center}
\begin{tikzpicture}[
    >=Stealth,
    every node/.style={align=center, text width=4cm},
    every edge/.style={-{Stealth[scale=1.2]}, thick}
]
\node (N1) at (0, 0) { \shortstack{$X_{k+s+2}$}};
\node (N2) at (0, -1.5) { \shortstack{$\vdots$}};
\node (N3) at (0, -3) { $X_{k+1}$};
\node (N4) at (-5, -1.5) { \shortstack{$\vdots$}};
\node (N5) at (0, -5.5) { \shortstack{$\vdots$}};
\node (N6) at (0, -7) { \shortstack{$X'$}};
\node (N7) at (0, -8.5) { \shortstack{$X$}};
\node (N9) at (-5, -3) { \shortstack{$Z_{s+2}$}};
\draw[->] (N1) -- (N2) node[midway, right, yshift=0.2cm] {$ A_{j}(zaq^{4}t^{-2})^{-1}$};
\draw[->] (N5) -- (N6) node[midway, right, yshift=0.2cm] {$ A_{i}(zaq^{4k+2}t^{-1})^{-1}$};
\draw[->] (N6) -- (N7) node[midway, right, yshift=0.2cm] {$ A_j(zaq^4)^{-1} $};
\draw[->] (N1) -- (N4) node[midway, left, yshift=0.2cm] {$ A_{i}(zaq^{6}t^{-1})^{-1} $};
\draw[->] (N4) -- (N9) node[midway, left, yshift=0.2cm] {$ A_{i}(zaq^{4k+2}t^{-1})^{-1} $};
\draw[->] (N2) -- (N3) node[midway, right, yshift=0.2cm] {$ A_{j}(zaq^{4}t^{-2s-2})^{-1} $};
\draw[->] (N3) -- (N5) node[midway, right, yshift=0.2cm] {$ A_{i}(zaq^{2}t^{-1})^{-1} $};
\end{tikzpicture}
\end{center}

        To apply the transformation $A_i(zaq^6t^{-1})^{-1}$, we need to check that $Y_i(zaq^{8}t^{-2})$ is admissible in $X_{k+s+2}$.\\
        We know that $X_{k+s+2}$ does not contain $Y_i(zaq^{4}t^{-2})$ as it would imply that $X_{k+s+1}$ contains $Y_i(zaq^{4}t^{-2})^2$ 
        which is a contradiction as all monomials are generic by assumption.
        Moreover, if $X_{k+s+2}$ contains $Y_i(zaq^{8})$ then the equation \eqref{eqnbbb2} implies:
        $$d_{i,aq^8}(X_{k+s+2})=d_{i,aq^8}(X_{k})=1.$$
        But the graph gives that $Y_{i}(zaq^8t^{-2})$ is admissible in $X_{k}$. It is a contradiction.
        Thus, $Y_i(zaq^{8}t^{-2})$ is admissible in $X_{k+s+2}$ and
        the algorithm applies the transformation $A_i(zaq^6t^{-1})^{-1}$ to $X_{k+s+2}$, giving a monomial $Z_{k+s+1}$ such that:
        \begin{equation}\label{eqnbbb4}
            X_{k+s+2}A_i(zaq^6t^{-1})^{-1}=Z_{k+s+1}.
        \end{equation}
        We assume there exists $1\le m\le k-1$ such that the algorithm gives $Z_{u+s+1}$ for all $m< u\le k$ such that:
        $$\forall m< u\le k, \quad Z_{u+s+2}A_i(zaq^{4(k-u+1)+2}t^{-1})^{-1}=Z_{u+s+1},$$
        setting $Z_{k+s+2}=X_{k+s+2}$.
        We want to prove that the algorithm gives $Z_{m+s+1}$ such that:
        $$Z_{m+s+2}A_i(zaq^{4(k-m+1)+2}t^{-1})^{-1}=Z_{m+s+1}.$$
        We know that
        \begin{equation}\label{eqnbbb3}
            Z_{m+s+2}=X_{k+s+2}A_i(zaq^{6}t^{-1})^{-1}\cdots A_i(zaq^{4(k-m)+2}t^{-1})^{-1}.
        \end{equation}
        Then, $Z_{m+s+2}$ contains $Y_i(zaq^{4(k-m+2)}t^{-2})\ldots Y_i(zaq^{4k+4}t^{-2})$.\\
        Let us check that $Y_i(zaq^{4(k-m+2)}t^{-2})$ is admissible in $Z_{m+s+2}$.\\
        It is clear that $Z_{m+s+2}$ does not contain $Y_i(zaq^{4(k-m+1)}t^{-2})$ as it has been removed at the previous step.
        Moreover, if $Z_{m+s+2}$ contains $Y_i(zaq^{4(k-m+2)})$. then by \eqref{eqnbbb3}, $X_{k+s+2}$ contains $Y_i(zaq^{4(k-m+2)})$.
        Thus, according to the graph, $X_m$ contains $Y_i(zaq^{4(k-m+2)})$. 
        However, we read in the same figure that $Y_i(zaq^{4(k-m+2)}t^{-2})$ is admissible in $X_{m}$. It is a contradiction.
        Hence, $Y_i(zaq^{4(k-m+2)}t^{-2})$ is admissible in $Z_{m+s+2}$ and the algorithm gives a monomial $Z_{m+s+1}$ such that:
        $$Z_{m+s+2}A_i(zaq^{4(k-m+1)+2}t^{-1})^{-1}=Z_{m+s+1}.$$
        By induction, we get monomials $Z_{s+2},\ldots,Z_{k+s+1}$ verifying the equation \eqref{eqnbbb3}.\\

        Now we want to complete our graph to the following:
                       \begin{center}
\begin{tikzpicture}[
    >=Stealth,
    every node/.style={align=center, text width=4cm},
    every edge/.style={-{Stealth[scale=1.2]}, thick}
]
\node (N1) at (0, 0) { \shortstack{$X_{k+s+2}$}};
\node (N2) at (0, -1.5) { \shortstack{$\vdots$}};
\node (N3) at (0, -3) { $X_{k+1}$};
\node (N4) at (-5, -1.5) { \shortstack{$\vdots$}};
\node (N5) at (0, -5.5) { \shortstack{$\vdots$}};
\node (N6) at (0, -7) { \shortstack{$X'$}};
\node (N7) at (0, -8.5) { \shortstack{$X$}};
\node (N11) at (-5, -7) { \shortstack{$Z_0$}};
\node (N9) at (-5, -3) { \shortstack{$Z_{s+2}$}};
\node (N10) at (-5, -5.5) { \shortstack{$\vdots$}};
\draw[->] (N1) -- (N2) node[midway, right, yshift=0.2cm] {$ A_{j}(zaq^{4}t^{-2})^{-1}$};
\draw[->] (N5) -- (N6) node[midway, right, yshift=0.2cm] {$ A_{i}(zaq^{4k+2}t^{-1})^{-1}$};
\draw[->] (N6) -- (N7) node[midway, right, yshift=0.2cm] {$ A_j(zaq^4)^{-1} $};
\draw[->] (N9) -- (N10) node[midway, left, yshift=0.2cm] {$ A_j(zaq^4)^{-1} $};
\draw[->] (N1) -- (N4) node[midway, left, yshift=0.2cm] {$ A_{i}(zaq^{6}t^{-1})^{-1} $};
\draw[->] (N4) -- (N9) node[midway, left, yshift=0.2cm] {$ A_{i}(zaq^{4k+2}t^{-1})^{-1} $};
\draw[->] (N10) -- (N11) node[midway, left, yshift=0.2cm] {$ A_{j}(zaq^{4}t^{-2s-2})^{-1}$};
\draw[->] (N2) -- (N3) node[midway, right, yshift=0.2cm] {$ A_{j}(zaq^{4}t^{-2s-2})^{-1} $};
\draw[->] (N3) -- (N5) node[midway, right, yshift=0.2cm] {$ A_{i}(zaq^{2}t^{-1})^{-1} $};
\end{tikzpicture}
\end{center}  

        Furthermore, by \eqref{eqnbbb4}, we get:
        $$ d_{j,aq^5t^{-1}}(Z_{s+2})=d_{j,aq^5t^{-1}}(Z_{k+s+1})=d_{j,aq^5t^{-1}}(X_{k+s+2})+1.$$
        But we know by admissibility of $Y_j(zaq^5t^{-3})$ and regularity of $X_{k+s+2}$ that:
        $$d_{j,aq^5t^{-1}}(X_{k+s+2})=0.$$
        Hence, $d_{j,aq^5t^{-1}}(Z_{s+2})=1$.
        
        Thus, the monomial $Z_{s+2}$ contains $Y_j(zaq^5t^{-1})\ldots Y_j(zaq^5t^{-2s-3})$.
        We want to check that $Y_j(zaq^5t^{-1})$ is admissible in $Z_{s+2}$.\\
        If $Z_{s+2}$ contains $Y_j(zaq^3t^{-1})$ then it contains $Y_j(zaq^5t^{-3})Y_j(zaq^3t^{-1})$ (this is of the form $Y_j(zc)Y_j(zcq^{-2r_j}t^2)$).
        By regularity, it also contains $Y_j(zaq^3t^{-3})$.
        By \eqref{eqnbbb3}, $X_{k+s+2}$ contains $Y_j(zaq^3t^{-3})$, and $Y_j(zaq^5t^{-3})$ is not admissible in $X_{k+s+2}$. It is a contradiction.\\
        If $Z_{s+2}$ contains $Y_j(zaq^5t)$ then by \eqref{eqnbbb3}, $X_{k+s+2}$ contains $Y_j(zaq^5t)$. , 
        According to the graph, this implies that $X'$ contains $Y_j(zaq^5t)$. But $Y_j(zaq^5t^{-1})$ is admissible in $X'$. 
        It is a contradiction.\\
        Hence, $Y_j(zaq^5t^{-1})$ is admissible in $Z_{s+2}$ and the algorithm gives a monomial $Z_{s+1}$ such that:
        \begin{equation}\label{eqnbbb5}
            Z_{s+2}A_j(zaq^4)^{-1}=Z_{s+1}.
        \end{equation}
        We assume there exists $0\le m< s+1$ such that we constructed $Z_{u+1}$ for all $m< u\le s+1$ such that:
        $$\forall m< u\le s+1, \quad Z_{u+1}A_j(zaq^4t^{-2(s+1-u)})^{-1}=Z_{u}.$$
        We want to prove that $Y_j(zaq^5t^{-2(s-m+1)-1})$ is admissible in $Z_{m+1}$ so that the algorithm gives the right monomial $Z_m$.\\
        It is clear that $Z_{m+1}$ does not contain $Y_j(zaq^5t^{-2(s-m+1)+1})$ as it has been removed at the previous step.\\
        Moreover, if $Z_{m+1}$ contains $Y_j(zaq^3t^{-2(s-m+1)-1})$ then by construction, $X_{k+s+2}$ also contains $Y_j(zaq^3t^{-2(s-m+1)-1})$.
        According to the graph, this implies that $X_{k+m+2}$ contains $Y_j(zaq^3t^{-2(s-m+1)-1})$. 
        However, we read in the graph that $Y_j(zaq^5t^{-2(s-m+1)-1})$ is admissible in $X_{k+m+2}$. It is a contradiction.
        Hence, $Y_j(zaq^5t^{-2(s-m+1)-1})$ is admissible in $Z_{m+1}$ and the algorithm gives a monomial $Z_m$ such that:
        $$Z_{m+1}A_j(zaq^4t^{-2(s-m+1)})^{-1}=Z_{m}.$$\\

        Finally, let us prove that $Y_i(zaq^4t^{-2})$ is admissible in $Z_0$.\\ 
        We have the equality $Z_0A_i(zaq^2t^{-1})^{-1}=X$. Then, by identification of the degrees the admissibility is clear. \\
        Hence the admissibility, and the following graph:
        \begin{center}
\begin{tikzpicture}[
    >=Stealth,
    every node/.style={align=center, text width=4cm},
    every edge/.style={-{Stealth[scale=1.2]}, thick}
]
\node (N1) at (0, 0) { \shortstack{$X_{k+s+2}$}};
\node (N2) at (0, -1.5) { \shortstack{$\vdots$}};
\node (N3) at (0, -3) { $X_{k+1}$};
\node (N4) at (-5, -1.5) { \shortstack{$\vdots$}};
\node (N5) at (0, -5.5) { \shortstack{$\vdots$}};
\node (N6) at (0, -7) { \shortstack{$X'$}};
\node (N7) at (0, -8.5) { \shortstack{$X$}};
\node (N11) at (-5, -7) { \shortstack{$Z_0$}};
\node (N9) at (-5, -3) { \shortstack{$Z_{s+2}$}};
\node (N10) at (-5, -5.5) { \shortstack{$\vdots$}};
\draw[->] (N1) -- (N2) node[midway, right, yshift=0.2cm] {$ A_{j}(zaq^{4}t^{-2})^{-1}$};
\draw[->] (N5) -- (N6) node[midway, right, yshift=0.2cm] {$ A_{i}(zaq^{4k+2}t^{-1})^{-1}$};
\draw[->] (N6) -- (N7) node[midway, right, yshift=0.2cm] {$ A_j(zaq^4)^{-1} $};
\draw[->] (N9) -- (N10) node[midway, left, yshift=0.2cm] {$ A_j(zaq^4)^{-1} $};
\draw[->] (N1) -- (N4) node[midway, left, yshift=0.2cm] {$ A_{i}(zaq^{6}t^{-1})^{-1} $};
\draw[->] (N4) -- (N9) node[midway, left, yshift=0.2cm] {$ A_{i}(zaq^{4k+2}t^{-1})^{-1} $};
\draw[->] (N10) -- (N11) node[midway, left, yshift=0.2cm] {$ A_{j}(zaq^{4}t^{-2s-2})^{-1}$};
\draw[->] (N2) -- (N3) node[midway, right, yshift=0.2cm] {$ A_{j}(zaq^{4}t^{-2s-2})^{-1} $};
\draw[->] (N3) -- (N5) node[midway, right, yshift=0.2cm] {$ A_{i}(zaq^{2}t^{-1})^{-1} $};
\draw[->] (N11) -- (N7) node[midway, left, yshift=0cm] {$ A_{i}(zaq^{2}t^{-1})^{-1} $};
\end{tikzpicture}
\end{center} 

    \vskip 1cm

    Now we assume $d_{i,at^{-2}}(X')=-1$. We know that $d_{i,at^{-2}}(X)=0$, then it implies 
        $$d_{i,at^{-2}}(A_j(zbq^{}t^{-1})^{-1})=1.$$
        Hence, $bqt^{-1}=at^{-2}$ and hence $$b=aq^{-1}t^{-1}.$$
        
        As in the previous case, we want to apply the induction hypothesis and get the following graph for a $k$ defined above:

                \begin{center}
                \begin{tikzpicture}[scale=0.7,
                    >=Stealth,
                    every node/.style={align=center, text width=4cm},
                    every edge/.style={-{Stealth[scale=1.2]}, thick}
                ]
                \node (N3) at (0, -3) { $X_{k+1}$};
                \node (N5) at (0, -4.5) { \shortstack{$\vdots$}};
                \node (N6) at (0, -6) { \shortstack{$X'$}};
                \node (N7) at (0, -7.5) { \shortstack{$X$}};
                \draw[->] (N5) -- (N6) node[midway, right, yshift=0.2cm] {$ A_{i}(zaq^{2}t^{-2k-1})^{-1}$};
                \draw[->] (N6) -- (N7) node[midway, right, yshift=0.2cm] {$ A_j(zat^{-2})^{-1} $};
                \draw[->] (N3) -- (N5) node[midway, right, yshift=0.2cm] {$ A_{i}(zaq^{2}t^{-1})^{-1} $};
                \end{tikzpicture}
                \end{center}

        Let $k>0$ be the maximal integer such that $Y_i(zat^{-2s })^{-1}$ appears in $X'$ for all $0\le s\le k$.\\
        By definition, $Y_i(zat^{-2(k+1)})^{-1}$ is not in $X'$. 
        Moreover, if $Y_i(zaq^4t^{-2k})^{-1}$ appears in $X'$, then $Y_i(zat^{-2(k-1) })^{-1}Y_i(zaq^4t^{-2 k})^{-1}$ also appears in $X'$. 
        The monomial has to be regular, so $Y_i(zat^{-2k })^{-1}Y_i(zaq^4t^{-2( k-1)})^{-1}$ appears in $X'$. \\
        Thus, $X'$ contains $Y_i(zat^{-2(k-2) })^{-1}Y_i(zaq^4t^{-2(k-1) })^{-1}$.
        We can iterate this reasoning until we get that $Y_i(zaq^{4})^{-1}$ appears in $X'$. 
        However it does not appear in $X$ and has to be simplified by the transformation $A_j(zbqt^{-1})^{-1}=A_j(zat^{-2})^{-1}$. 
        It is a contradiction by definition of $A_j$. 
        Hence, we can apply the induction hypothesis and we obtain that there exists $X_1$ given by the algorithm such that 
        $$X_1A_{i}(zaq^2t^{-(2k+1)})^{-1}=X'.$$ 
        We can define recursively 
        $X_s$, $2\le s\le k+1$ such that $$X_s A_{i}(zaq^2t^{-(2k-2s+3)})^{-1}=X_{s-1}.$$ 
        Indeed, for all $1\leq s\leq k+1$, $Y_i(zat^{-2u })^{-1}$ appears in $X_s$ for all $0\le u\le k+1-s$ and $Y_i(zaq^4t^{-2(u-1) })^{-1}$ 
        does not appear in $X_s$ nor $Y_i(zat^{-2(k+2-s) })^{-1}$ so that we can apply the induction hypothesis. 
        In particular, we have constructed $X_{k+1}$ such that $$X_{k+1}A_{i}(zaq^2t^{-1})^{-1}=X_k.$$
        Moreover, 
        \begin{equation}\label{eqnbbbb1}
            X_{k+1}A_{i}(zaq^2t^{-1})^{-1}A_{i}(zaq^2t^{-3})^{-1}\cdots  A_{i}(zaq^2t^{-(2k+1)})^{-1}=X'.
        \end{equation}

        By definition of the algorithm, the arrow $X'\rightarrow X$ implies that $Y_{j}(zaqt^{-3})$ is admissible in $X'$.
        In particular, it implies that $d_{j,aqt^{-3}}(X')=1$ and $d_{j,aqt^{-1}}(X')=0$. But the equation $\eqref{eqnbbbb1}$ gives
        $$d_{j,aqt^{-1}}(X')=d_{j,aqt^{-1}}(X_{k+1})+1\quad \text{and } \quad d_{j,aqt^{-3}}(X')= d_{j,aqt^{-3}}(X_{k+1})+1.$$
        Hence, $d_{j,aqt^{-1}}(X_{k+1})=-1$ and $d_{j,aqt^{-3}}(X_{k+1})= 0$. Thus, $X_{k+1}$ contains 
        $Y_{j}(zaqt^{-1})^{-1}$ and does not contain
        $Y_{j}(zaqt^{-3})^{-1}$.\\

        Let $s\ge0$ be the maximal integer such that $X_{k+1}$ contains the monomial: 
        $$Y_{j}(zaqt^{-1})^{-1}\dots Y_{j}(zaq^{2s+1}t^{-1})^{-1}.$$ 
        If for one integer $1\le u\le s$, $X_{k+1}$ contains $Y_{j}(zaq^{2u+1}t^{-3})^{-1}$, then it contains the monomial 
        $Y_{j}(zaq^{2(u-1)+1}t^{-1})^{-1}Y_{j}(zaq^{2u+1}t^{-3})^{-1}$. 
        However, $X_{k+1}$ has to be regular, so $X_{k+1}$ also contains $Y_{j}(zaq^{2(u-1)+1}t^{-3})^{-1}$. 
        Iterating the same argument we would get that $X_{k+1}$ does contain $Y_{j}(zaqt^{-3})^{-1}$, 
        which is a contradiction by the argument above. 
        Hence, we can use $s+1$ times the recurrence hypothesis until a monomial $X_{k+s+2}$. We get: 
        $$\forall 0\leq u \leq s, \quad X_{k+u+2}A_j(zaq^{2(s-u)+2}t^{-2})^{-1}=X_{k+u+1}$$
        At this step, we have constructed a path from $X_{k+ s +2}$ to $X$ with:
        \begin{equation}\label{eqnb222}
            \begin{split}
            X_{k+s+2}A_j(zaq^2t^{-2})^{-1}\cdots A_j(zaq^{2s+2}t^{-2})^{-1} \times\qquad\qquad\qquad \\[5mm]
            \qquad A_i(zaq^2t^{-1})^{-1}\cdots A_i(zaq^2t^{-2k-1})^{-1}A_j(zat^{-2})^{-1} = X.
            \end{split}
        \end{equation}

        We get the following path in the graph representation of the algorithm:
                \begin{center}
                \begin{tikzpicture}[scale=0.7,
                    >=Stealth,
                    every node/.style={align=center, text width=4cm},
                    every edge/.style={-{Stealth[scale=1.2]}, thick}
                ]
                \node (N1) at (0, 0) { \shortstack{$X_{k+s+2}$}};
                \node (N2) at (0, -1.5) { \shortstack{$\vdots$}};
                \node (N3) at (0, -3) { $X_{k+1}$};
                \node (N5) at (0, -4.5) { \shortstack{$\vdots$}};
                \node (N6) at (0, -6) { \shortstack{$X'$}};
                \node (N7) at (0, -7.5) { \shortstack{$X$}};
                \draw[->] (N1) -- (N2) node[midway, right, yshift=0.2cm] {$ A_{j}(zaq^{2}t^{-2})^{-1}$};
                \draw[->] (N5) -- (N6) node[midway, right, yshift=0.2cm] {$ A_i(zaq^2t^{-2k-1})^{-1}$};
                \draw[->] (N6) -- (N7) node[midway, right, yshift=0.2cm] {$ A_j(zat^{-2})^{-1} $};
                \draw[->] (N2) -- (N3) node[midway, right, yshift=0.2cm] {$ A_j(zaq^{2s+2}t^{-2})^{-1} $};
                \draw[->] (N3) -- (N5) node[midway, right, yshift=0.2cm] {$ A_{i}(zaq^{2}t^{-1})^{-1} $};
                \end{tikzpicture}
                \end{center} 
        Finally, 
        \begin{equation}\label{assumption1ee}
            \text{$X_{k+s+2}$ contains $Y_i(zaq^4t^{-4})\dots Y_i(zaq^{4}t^{-2k-2})Y_j(zaq^3t^{-3})\dots Y_j(zaq^{2s+3}t^{-3})$.}
        \end{equation}

        Let us prove that $s>0$. By contradiction, if $s=0$ then  $Y_i(zaq^4t^{-2})$ is admissible in $X_{k+s+2}$. Indeed, 
        $$d_{i,aq^4t^{-2}}(X_{k+s+2})=d_{i,aq^4t^{-2}}(X_{k+1})=1,$$
        $$d_{i,aq^4}(X_{k+s+2})=d_{i,aq^4}(X_{k+1})=0,$$
        $$d_{i,at^{-2}}(X_{k+s+2})=d_{i,at^{-2}}(X_{k+1})=0.$$
        Hence, the algorithm applies the transformation $A_i(zaq^2t^{-1})^{-1}$ to $X_{k+s+2}$ and we get a monomial $X^*$ such that:
        $$X_{k+s+2}A_i(zaq^2t^{-1})^{-1}=X^*.$$
        Then in particular, 
        $$d_{j,aqt^{-1}}(X^*)=d_{j,aqt^{-1}}(X_{k+s+2})+1=d_{j,aqt^{-1}}(X_{k+s+1})+1+1=1.$$
        Indeed, by genericity, $d_{j,aqt^{-1}}(X_{k+s+1})$ has to be equal to $-1$.  
        Moreover, by admissibility of $Y_j(zaq^3t^{-3})$ in $X_{k+s+2}$ and regularity, 
        $$d_{j,aq^3t^{-3}}(X^*)=d_{j,aq^3t^{-3}}(X_{k+s+2})=1,$$
        $$d_{j,aqt^{-3}}(X^*)=d_{j,aqt^{-3}}(X_{k+s+2})=0,$$
        Then $X^*$ contains a term of the form $Y_j(zc)Y_j(zcq^{-2r_j}t^2)$ but does not contain $Y_j(zcq^{-2r_j})$ (with $c=aq^3t^{-3}$)
        and $X^*$ is not regular. It is a contradiction. Hence, $s>0$.\\
        
        The aim is now to construct another path from $X_{k+s+2}$ to $X$ ending with the transformation $A_i(zaq^2t^{-1})^{-1}$ 
        in order to conclude.\\
        Firstly, $$d_{i,aq^4t^{-4}}(X_{k+s+2})=d_{i,aq^4t^{-4}}(X_{k})=1,$$
        because as shown in the graph above, $Y_i(zaq^4t^{-4})$ is admissible in $X_{k}$. 
        Moreover, $d_{i,at^{-4}}(X_{k+s+2})=d_{i,at^{-4}}(X_{k})=0$ for the same argument of admissibility. 
        Additionally, if $d_{i,aq^4t^{-2}}(X_{k+s+2})=1$ then  by definition of $A_j(zaq^4t^{-2})^{-1}$ it implies
        $d_{i,aq^4t^{-2}}(X_{k+s})=2$ which is impossible by the genericity hypothesis.\\
        Hence, $Y_i(zaq^4t^{-4})$ is admissible in $X_{k+s+2}$.\\

        We construct $$Z_{k+s+1}:=X_{k+s+2}A_i(zaq^2t^{-3})^{-1}.$$
        We assume there exists $1\leq m \leq k$ such that we constructed $Z_{s+u+1}$, for $m< u\leq k$ such that 
        $$Z_{u+s+1}=Z_{u+s+2}A_i(zaq^2t^{-2(k-u)-3})^{-1}.$$
        Let us construct $Z_{s+m+1}$.\\
        Firstly, according to the graph, $$d_{i,aq^4t^{-2(k-m)-4}}(Z_{s+m+2})=d_{i,aq^4t^{-2(k-m)-4}}(X_m)=1.$$
        Moreover, $$d_{i,aq^4t^{-2(k-m)-2}}(Z_{s+m+2})=d_{i,aq^4t^{-2(k-m)-2}}(Z_{s+m+3})-1=d_{i,aq^4t^{-2(k-m)-2}}(X_{m+1})-1=0,$$
        because of the transformation $Z_{s+m+3}\rightarrow Z_{s+m+2}$ and because $Y_i(zaq^4t^{-2(k-m)-2})$ is admissible in $X_{m+1}$.\\
        Additionally, $Y_i(zaq^4t^{-2(k-m)-4})$ is admissible in $X_m$. Thus, 
        $$d_{i,at^{-2(k-m)-4}}(Z_{s+m+2})=d_{i,at^{-2(k-m)-4}}(X_{k+s+2})=d_{i,at^{-2(k-m)-4}}(X_m)=0.$$
        Hence the admissibility of $Y_i(zaq^4t^{-2(k-m)-4})$ in $Z_{s+m+2}$ and the existence of the following
         transformation given by the algorithm:
         $$ Z_{m+s+1}=Z_{m+s+2}A_i(zaq^2t^{-2(k-m)-3})^{-1}. $$
         We obtain the following graph:

        \begin{center}
\begin{tikzpicture}[
    >=Stealth,
    every node/.style={align=center, text width=4cm},
    every edge/.style={-{Stealth[scale=1.2]}, thick}
]
\node (N1) at (0, 0) { \shortstack{$X_{k+s+2}$}};
\node (N2) at (0, -1.5) { \shortstack{$\vdots$}};
\node (N3) at (0, -3) { $X_{k+1}$};
\node (N4) at (-5, -1.5) { \shortstack{$\vdots$}};
\node (N5) at (0, -5.5) { \shortstack{$\vdots$}};
\node (N6) at (0, -7) { \shortstack{$X'$}};
\node (N7) at (0, -8.5) { \shortstack{$X$}};
\node (N9) at (-5, -3) { \shortstack{$Z_{s+2}$}};
\draw[->] (N1) -- (N2) node[midway, right, yshift=0.2cm] {$ A_{j}(zaq^{2}t^{-2})^{-1}$};
\draw[->] (N5) -- (N6) node[midway, right, yshift=0.2cm] {$ A_i(zaq^2t^{-2k-1})^{-1}$};
\draw[->] (N6) -- (N7) node[midway, right, yshift=0.2cm] {$ A_j(zat^{-2})^{-1} $};
\draw[->] (N1) -- (N4) node[midway, left, yshift=0.2cm] {$ A_{i}(zaq^{2}t^{-3})^{-1} $};
\draw[->] (N4) -- (N9) node[midway, left, yshift=0.2cm] {$ A_i(zaq^2t^{-2k-1})^{-1} $};
\draw[->] (N2) -- (N3) node[midway, right, yshift=0.2cm] {$ A_j(zaq^{2s+2}t^{-2})^{-1} $};
\draw[->] (N3) -- (N5) node[midway, right, yshift=0.2cm] {$ A_{i}(zaq^{2}t^{-1})^{-1} $};
\end{tikzpicture}
        \end{center}

        Again, we want to construct a path from $Z_{s+2}$ to $X$ in the algorithm graph.\\
        Firstly, by definition of the variables $A_i$ and by admissibility in $X'$,
        \begin{align*}
            d_{j,aqt^{-3}}(Z_{s+2})&=d_{j,aqt^{-3}}(X_{k+s+2})+1,\\
            &=d_{j,aqt^{-3}}(X_{k})+1,\\
            &=d_{j,aqt^{-3}}(X_{k-1})-1+1,\\
            &=d_{j,aqt^{-3}}(X')=1.
        \end{align*}
        Moreover, 
        $$d_{j,aq^{-1}t^{-3}}(Z_{s+2})=d_{j,aq^{-1}t^{-3}}(X')=0,$$
        as $Y_j(zaqt^{-3})$ is admissible in $X'$. 
        Additionally, we deduce from the graph the following computation
        \begin{align*}
            d_{j,aqt^{-1}}(Z_{s+2})&=d_{j,aqt^{-1}}(X_{k+s+2})\\
            &=d_{j,aqt^{-1}}(X_{k+s+1})-1\\
            &=d_{j,aqt^{-1}}(X_{k})+1-1\\
            &=d_{j,aqt^{-1}}(X')=0
        \end{align*}
        Hence, $Y_j(zaqt^{-3})$ is admissible in $Z_{s+2}$. 
        Thus, the algorithm constructs the following monomial:
        $$ Z_{s+1}=Z_{s+2}A_j(zat^{-2})^{-1}.$$
        Now, we assume there exists $0\leq m \leq s$ such that we constructed $Z_{u+1}$, for $m\leq u\leq s$ such that 
        $$Z_{u+1}=Z_{u+2}A_j(zaq^{2(s-u)}t^{-2})^{-1}.$$
        Let us construct $Z_{m}$.\\
        Let us prove that $Y_j(zaq^{2(s+1-m)+1}t^{-3})$ is admissible in $Z_{m+1}$.\\
        Firstly, according to the graph and by admissibility, $$d_{j,aq^{2(s+1-m)+1}t^{-3}}(Z_{m+1})=d_{j,aq^{2(s+1-m)+1}t^{-3}}(X_{k+m+2})=1.$$
        Moreover, $$d_{j,aq^{2(s-m)+1}t^{-3}}(Z_{m+1})=d_{j,aq^{2(s-m)+1}t^{-3}}(Z_{m+2})-1=0,$$
        by admissibility and because of the transformation $Z_{m+2}\rightarrow Z_{m+1}$.\\
        Additionally, $Y_j(zaq^{2(s+1-m)+1}t^{-3})$ is admissible in $X_{k+m+2}$. Thus, 
        $$d_{j,aq^{2(s+1-m)+1}t^{-1}}(Z_{m+1})=d_{j,aq^{2(s+1-m)+1}t^{-1}}(X_{k+m+2})=0.$$
        Hence the admissibility of $Y_j(zaq^{2(s+1-m)+1}t^{-3})$ in $Z_{m+1}$ and the existence of the following
         transformation given by the algorithm:
         $$ Z_{m}=Z_{m+1}A_j(zaq^{2(s+1-m)}t^{-2})^{-1}. $$
         To conclude, we have to prove that $Y_i(zaq^4t^{-2})$ is admissible in $Z_0$.\\
         We have the equality $Z_0A_i(zaq^2t^{-1})^{-1}=X$. Then, by identification of the degrees the admissibility is clear. \\
         Hence the following paths of transformations in the algorithm:
         \begin{center}
\begin{tikzpicture}[
    >=Stealth,
    every node/.style={align=center, text width=4cm},
    every edge/.style={-{Stealth[scale=1.2]}, thick}
]
\node (N1) at (0, 0) { \shortstack{$X_{k+s+2}$}};
\node (N2) at (0, -1.5) { \shortstack{$\vdots$}};
\node (N3) at (0, -3) { $X_{k+1}$};
\node (N4) at (-5, -1.5) { \shortstack{$\vdots$}};
\node (N5) at (0, -5.5) { \shortstack{$\vdots$}};
\node (N6) at (0, -7) { \shortstack{$X'$}};
\node (N7) at (0, -8.5) { \shortstack{$X$}};
\node (N9) at (-5, -3) { \shortstack{$Z_{s+2}$}};
\node (N10) at (-5, -5.5) { \shortstack{$\vdots$}};
\node (N11) at (-5, -7) { \shortstack{$Z_0$}};
\draw[->] (N1) -- (N2) node[midway, right, yshift=0.2cm] {$ A_{j}(zaq^{2}t^{-2})^{-1}$};
\draw[->] (N5) -- (N6) node[midway, right, yshift=0.2cm] {$ A_i(zaq^2t^{-2k-1})^{-1}$};
\draw[->] (N6) -- (N7) node[midway, right, yshift=0.2cm] {$ A_j(zat^{-2})^{-1} $};
\draw[->] (N9) -- (N10) node[midway, left, yshift=0.2cm] {$ A_j(zat^{-2})^{-1} $};
\draw[->] (N1) -- (N4) node[midway, left, yshift=0.2cm] {$ A_{i}(zaq^{2}t^{-3})^{-1} $};
\draw[->] (N4) -- (N9) node[midway, left, yshift=0.2cm] {$ A_i(zaq^2t^{-2k-1})^{-1} $};
\draw[->] (N2) -- (N3) node[midway, right, yshift=0.2cm] {$ A_j(zaq^{2s+2}t^{-2})^{-1} $};
\draw[->] (N10) -- (N11) node[midway, left, yshift=0.2cm] {$ A_j(zaq^{2s+2}t^{-2})^{-1} $};
\draw[->] (N3) -- (N5) node[midway, right, yshift=0.2cm] {$ A_{i}(zaq^{2}t^{-1})^{-1} $};
\draw[->] (N11) -- (N7) node[midway, left, yshift=0cm] {$ A_{i}(zaq^{2}t^{-1})^{-1} $};
\end{tikzpicture}
        \end{center} 
        \textit{b)} We assume there exists a monomial $X''$ also appearing in the algorithm such that $$X''A_{i}(zaq^{r_i}t^{-1})^{-1}=X',$$
        then $X''$ does contain $Y_i(zaq^{2r_i}t^{-2})$.\\ 

        Here there are two cases to consider:
        \begin{itemize}
            \item $d_{j,bq^{2r_j}t^{-2}}(A_{i}(zaq^{r_i}t^{-1})^{-1})=1$ and the term $Y_j(zbq^{2r_j}t^{-2})$ comes from 
            the transformation $A_{i}(zaq^{r_i}t^{-1})^{-1}$.
            \item $X''$ contains $Y_j(zbq^{2r_j}t^{-2})$.
        \end{itemize}

        Firstly, we assume $d_{j,bq^{2r_j}t^{-2}}(A_{i}(zaq^{r_i}t^{-1})^{-1})=1$ and the term $Y_j(zbq^{2r_j}t^{-2})$ in $X'$ comes from 
        the transformation $A_{i}(zaq^{r_i}t^{-1})^{-1}$. Hence, $i\neq j$, and 
        we assume the Dynkin subdiagram with nodes $\{i,j\}$ is of type $B_2$ with $(i,j)=(1,2)$.\\
        Thus, $$bq^2t^{-2}=aqt^{-1}\qquad\text{or}\qquad bq^2t^{-2}=aq^3t^{-1}$$
        If $bq^2t^{-2}=aqt^{-1}$ then $a=bqt^{-1}$ and $d_{i,a}(A_{j}(zbqt^{-1})^{-1})=1$. 
        However, $d_{i,a}(X)=-1$ and $X'A_{j}(zbq^{r_j}t^{-1})^{-1}=X$. 
        Thus, $d_{i,a}(X')=-2$ which contradicts the genericity.\\
        Hence, $bq^2t^{-2}=aq^3t^{-1}$. By admissibility, $d_{j,aq^3t^{-1}}(X')=1$ and $d_{j,aqt^{-1}}(X')=0$. \\
        But $d_{j,aqt^{-1}}(A_{i}(zaq^{2}t^{-1})^{-1})=1$. Thus, $d_{j,aqt^{-1}}(X'')=-1$.\\
        Let $s\geq0$ be the maximal integer such that $X''$ contains the monomial:
        $$Y_j(zaqt^{-1})^{-1}\ldots Y_j(zaqt^{-2s-1})^{-1}.$$
        If for one $0\leq u\leq s$, $X''$ contains $Y_j(zaq^{3}t^{-2u-1})^{-1}$ by regularity we deduce that 
        $X''$ contains $Y_j(zaq^3t^{-1})^{-1}$. But we have: 
        $$d_{j,aq^3t^{-1}}(X'')=d_{j,aq^3t^{-1}}(X')-d_{j,aq^3t^{-1}}(A_{i}(zaq^{2}t^{-1})^{-1})=1-1=0.$$
        Hence, it is a contradiction. Thus, the induction hypothesis allows to construct $X_1,\ldots,X_{s+1}$ such that we get the following subgraph in the algorithm graph:
            \begin{center}
                \begin{tikzpicture}[
                    >=Stealth,
                    every node/.style={align=center, text width=4cm},
                    every edge/.style={-{Stealth[scale=1.2]}, thick}
                ]
                \node (N1) at (0, 0) { \shortstack{$X''$}};
                \node (N2) at (0, -1) { \shortstack{$X'$}};
                \node (N3) at (0, -2) { \shortstack{$X$}};
                \node (N4) at (0, 1) { \shortstack{$X_1$}};
                \node (N5) at (0, 2) { \shortstack{$\vdots$}};
                \node (N6) at (0, 3) { \shortstack{$X_{s+1}$}};
                \draw[->] (N1) -- (N2) node[midway, right, yshift=0.2cm] {$ A_{i}(zaq^{2}t^{-1})^{-1}$};
                \draw[->] (N2) -- (N3) node[midway, right, yshift=0.2cm] {$ A_j(zaq^{2})^{-1} $};
                \draw[->] (N4) -- (N1) node[midway, right, yshift=0.2cm] {$ A_{j}(zaq^{2}t^{-2s-2})^{-1} $};
                \draw[->] (N6) -- (N5) node[midway, right, yshift=0.2cm] {$ A_j(zaq^{2}t^{-2})^{-1}$};
                \draw[->] (N5) -- (N4) node[midway, left, yshift=0.2cm] {};
                \end{tikzpicture}
            \end{center}
        In particular, by admissibility and regularity, 
        $$d_{j,aq^3t^{-3}}(X_{s+1})=1,\qquad d_{j,aqt^{-1}}(X_{s+1})=0, \qquad d_{j,aqt^{-3}}(X_{s+1})=0.$$
        Now, $d_{i,c}(X_{s+1})=d_{i,c}(X'')$ for $c=aq^4t^{-2},aq^4,at^{-2}$. 
        Hence, the admissibility of $Y_i(zaq^4t^{-2})$ in $X''$ implies its admissibility in $X_{s+1}$, and the creation of 
        $$Z=X_{s+1}A_{i}(zaq^{2}t^{-1})^{-1}$$
        in the algorithm.
        Hence, $d_{j,aqt^{-1}}(Z)=d_{j,aq^3t^{-3}}(Z)=1$ and $d_{j,aqt^{-3}}(Z)=0$. This contradicts the regularity of $Z$.\\

        Hence, we are in the second case and $X''$ contains $Y_j(zbq^{2r_j}t^{-2})$.

        Let us prove that $Y_j(zbq^{2r_j}t^{-2})$ is admissible in $X''$.\\
        If $i=j$ (so that $r_i=r_j$), then by contradiction we assume that $Y_j(zbq^{2r_j}t^{-2})$ is 
        not admissible in $X''$. We know that $Y_j(zbq^{2r_j}t^{-2})$ is admissible in $X'$ so that the presence of 
        $Y_i(zaq^{2r_i}t^{-2})$  in $X''$ prevents from applying the transformation $A_j(zbq^{r_j}t^{-1})^{-1}$. 
        It implies that $a=bq^{-2r_i}$ or $a=bt^2$. 
        This implies $X$ contains $Y_i(zb)^{-1}=Y_i(zaq^{2r_i})^{-1}$ or $Y_i(zat^{-2})^{-1}$. 
        It is a contradiction by the initial assumption on $X$. 
        Thus, $Y_j(zbq^{2r_j}t^{-2})$ is admissible in $X''$.\\
        If $i\neq j$ then it is clear that the admissibility of $Y_j(zbq^{2r_j}t^{-2})$ in $X'$ implies its admissibility in $X''$. \\
        Finally, there exists a monomial $$X_{new}=X''A_{j}(zbq^{r_j}t^{-1})^{-1}$$ given by the algorithm and $X_{new}$ contains $Y_i(zaq^{2r_i}t^{-2})$.
        Moreover, if the transformation $A_{i}(zaq^{r_i}t^{-1})^{-1}$ is not admissible from $X_{new}$ then there exists a factor 
        $Y_i(zat^{-2})$ (resp. $Y_i(zaq^{2r_i})$) blocking this transformation. 
        But this would imply that this factor also appears in $X$ as we have the following equality of monomials:
        $$X=X_{new}A_{i}(zaq^{r_i}t^{-1})^{-1}.$$
        We remark that a priori, this equality is not sufficient to get an arrow from $X_{new}$ to $X$. This equality implies:
        \begin{align*}
            d_{i,aq^{2r_i}}(X_{new})&=d_{i,aq^{2r_i}}(X),
        \end{align*}
        and same for $at^{-2}$. 
        And hence $X$ does contain $Y_i(zc)$ and $Y_i(zct^2)^{-1}$ with $c=at^{-2}$ (resp. $Y_i(zc)Y_i(zcq^{-2r_i})^{-1}$ with $c=aq^{2r_i}$) 
        which contradicts the regularity of $X$. It is impossible. Hence the admissibility of $Y_i(zaq^{2r_i}t^{-2})$ in $X_{new}$.\\
        Finally we get the following paths:
        \begin{center}
\begin{tikzpicture}[
    >=Stealth,
    every node/.style={align=center, text width=4cm},
    every edge/.style={-{Stealth[scale=1.2]}, thick}
]
\node (N1) at (0, 0) { \shortstack{$X''$}};
\node (N2) at (0, -3) { \shortstack{$X'$}};
\node (N3) at (0, -6) { \shortstack{$X$}};
\node (N4) at (-3, -3) { \shortstack{$X_{new}$}};
\draw[->] (N1) -- (N2) node[midway, right, yshift=0.2cm] {$ A_{i}(zaq^{r_i}t^{-1})^{-1}$};
\draw[->] (N2) -- (N3) node[midway, right, yshift=0.2cm] {$ A_j(zbq^{r_j}t^{-1})^{-1} $};
\draw[->] (N4) -- (N3) node[midway, left, yshift=0.2cm] {$ A_{i}(zaq^{r_i}t^{-1})^{-1} $};
\draw[->] (N1) -- (N4) node[midway, left, yshift=0.2cm] {$ A_j(zbq^{r_j}t^{-1})^{-1}$};
\end{tikzpicture}
\end{center}
        This concludes the proof in type $B_2$.

        \end{proof}

\subsection{In type $G_2$}
We have $r_1=3$, $r_2=1$, 
$$A_1(z)=:Y_1(zq^{-3}t)Y_1(zq^{3}t^{-1})Y_2(zq^{-2})^{-1}Y_2(z)^{-1}Y_2(zq^2)^{-1}:$$
$$A_2(z)=:Y_2(zq^{-1}t)Y_2(zq^{}t^{-1})Y_1(z)^{-1}:$$
\subsubsection{In the case $(i,j)=(2,1)$ }

\begin{proof}
    We prove it by induction on the height of the monomials.  We assume the Lie subalgebra generated by the nodes $i$ and $j$ is of type $G_2$, and $(i,j)=(2,1)$.
       \textit{\underline{Height $0$:}} The only monomial with height $0$ is the dominant generic monomial from which we start our algorithm. Hence the property holds at height $0$.\\

         \textit{\underline{Height $h+1$:}} We assume the property is true for all the monomials with heights $h'\leq h\in\N$. 
        Let us prove it for the monomials with height $h+1$. Let $X$ be a monomial with height $h+1$. 
        The monomial appears in the algorithm, so it has to come from a monomial $X'$ with height $h$. 
        Let $A_j(zbq^{r_j}t^{-1})^{-1}$ be the transformation involved. 
        It implies that $d_{j,b}(X)=-1$ and $d_{j,bq^{2r_j}t^{-2}}(X')=1$. 
        Let $i\in I$, $a\in q^\Z t^\Z$ such that $d_{i,a}(X)=-1$ and $d_{i,aq^{2r_i}}(X),d_{i,at^{-2}}(X)\ne -1$. 
        If $(i,a)=(j,b)$ then we have the result. 
        Else, we have
        $$X=X'A_{j}(zbq^{r_j}t^{-1})^{-1}.$$
        Hence,
        $$ \left\{\begin{array}{llll}
            d_{j,bq^{2r_j}t^{-2}}(X')=1\\
            d_{j,b}(X)=-1\\
            d_{j,c}(X)=d_{j,c}(X')\text{ for all $c\not\in\{b,bq^{2r_j}t^{-2}\}$}\\
            d_{k,c}(X)\geq d_{k,c}(X')\text{ for all $k\neq j$, $c\in q^\Z t^\Z$}
        \end{array}\right.$$
        We know that $d_{i,a}(X)=-1$, and $(i,a)\neq (j,b)$.  
        It implies $d_{i,a}(X')\leq-1$. The monomials are all generic, thus 
        $$ d_{i,a}(X')=-1$$
        Then by the induction assumption, one of the two following assertions is true:\\
        \begin{itemize}
            \item[(a)]  $d_{i,aq^{2r_i}}(X')=-1$  {or}  $d_{i,at^{-2}}(X')= -1$.
            \item[(b)]  There exists a monomial $X''$ appearing in the algorithm such that
        $$X''A_{i}(zaq^{r_i}t^{-1})^{-1}=X'.$$
        \end{itemize}

        \textit{a)} Firstly we assume (a) is true and $X'$ does contain $Y_i(zaq^{2r_i})^{-1}$ (i.e $d_{i,aq^{2r_i}}(X')=-1$). 
        Let us prove that this case leads to a contradiction, then we will consider the case where $X'$ does contain $Y_i(zat^{-2})^{-1}$
         (i.e $d_{i,at^{2}}(X')=-1$)\\
        By assumption, $d_{i,aq^{2r_i}}(X)=0$. 
        It implies that the transformation $A_j(zbq^{r_j}t^{-1})^{-1}$ simplifies $Y_i(zaq^{2r_i})^{-1}$ and in particular $i\ne j$. 
        Therefore, the expressions depend on the type of the Dynkin diagram generated by the nodes $i$ and $j$. \\

        Let us do it in type $G_2$, $(i,j)=(2,1)$. 
        \begin{align*}
            A_j(zbq^{r_j}t^{-1})^{-1}&=A_j(zbq^3t^{-1})^{-1}\\
            &=Y_j(zbq^6t^{-2})^{-1}Y_j(zb)^{-1}Y_i(zbqt^{-1})Y_i(zbq^3t^{-1})Y_i(zbq^5t^{-1}) 
        \end{align*} 
        Then 
        $$ bqt^{-1}=aq^2, \qquad\text{or}\qquad bq^3t^{-1}=aq^2,\qquad\text{or}\qquad bq^5t^{-1}=aq^2. $$   
        In the two last cases, the transformation $A_j(zbq^{r_j}t^{-1})^{-1}$ also simplifies $Y_i(za)^{-1}$. This is impossible.
        Thus, we are in the first case and as in types $A_2,B_2$ we get:
        $$b=aqt.$$    
        
        Let $k>0$ be the maximal integer such that $Y_i(zaq^{2s })^{-1}$ appears in $X'$ for all $0\le s\le k$. We want to prove that the graph 
        representation
        of the algorithm contains the following subgraph:
            \begin{figure}[htbp]
            \begin{center}
            \begin{tikzpicture}[scale=0.7,
                >=Stealth,
                every node/.style={align=center, text width=4cm},
                every edge/.style={-{Stealth[scale=1.2]}, thick}
            ]
            \node (N3) at (0, -3) { $X_{k+1}$};
            \node (N5) at (0, -4.5) { \shortstack{$\vdots$}};
            \node (N6) at (0, -6) { \shortstack{$X'$}};
            \node (N7) at (0, -7.5) { \shortstack{$X$}};
            \draw[->] (N5) -- (N6) node[midway, right, yshift=0.2cm] {$ A_{i}(zaq^{2k+1}t^{-1})^{-1}$};
              \draw[->] (N6) -- (N7) node[midway, right, yshift=0.2cm] {$ A_j(zaq^4)^{-1} $};
             \draw[->] (N3) -- (N5) node[midway, right, yshift=0.2cm] {$ A_{i}(zaq^{}t^{-1})^{-1} $};
             \end{tikzpicture}
              \end{center} 
        \end{figure}
        By definition, $Y_i(zaq^{2(k+1)})^{-1}$ is not in $X'$. 
        Moreover, if $Y_i(zaq^{2k}t^{-2})^{-1}$ appears in $X'$, then $Y_i(zaq^{2(k-1) })^{-1}Y_i(zaq^{2 k}t^{-2})^{-1}$ also appears in $X'$. 
        The monomial has to be regular, so $Y_i(zaq^{2k })^{-1}Y_i(zaq^{2( k-1)}t^{-2})^{-1}$ appears in $X'$. \\
        Thus, $X'$ contains $Y_i(zaq^{2(k-2) })^{-1}Y_i(zaq^{2(k-1) }t^{-2})^{-1}$.
        We can iterate this reasoning until we get that $Y_i(zat^{-2})^{-1}$ appears in $X'$. 
        However it does not appear in $X$ and has to be simplified by the transformation $A_j(zbq^{r_j}t^{-1})^{-1}=A_j(zaq^4)^{-1}$. 
        It is a contradiction by definition of $A_j$. 
        Hence, we can apply the induction hypothesis and we obtain that there exists $X_1$ given by the algorithm such that $$X_1A_{i}(zaq^{2k+1}t^{-1})^{-1}=X'.$$ 
        We can define recursively 
        $X_s$, $2\le s\le k+1$ such that $$X_s A_{i}(zaq^{2k-2s+3}t^{-1})^{-1}=X_{s-1}.$$ 
        Indeed, for all $1\leq s\leq k+1$, $Y_i(zaq^{2u })^{-1}$ appears in $X_s$ for all $0\le u\le k+1-s$ and $Y_i(zaq^{2(u-1) }t^{-2})^{-1}$ does not appear in $X_s$ nor $Y_i(zaq^{2(k+2-s) })^{-1}$ so that we can apply the induction hypothesis. 
        In particular, we have constructed $X_{k+1}$ such that $$X_{k+1}A_{i}(zaq^{ }t^{-1})^{-1}=X_k.$$
        But $d_{j,aqt^{-1}}(A_{i}(zaq^{}t^{-1})^{-1})=1$. Thus, $$d_{j,aqt^{-1}}(X_{k+1})=d_{j,aqt^{-1}}(X_{k})-1=d_{j,aqt^{-1}}(X')-1<0$$
        because by regularity, $$d_{j,aqt^{-1}}(X')\in\{0,1\},$$
        and if $d_{j,aqt^{-1}}(X')=1$ then the transformation $A_j(zbq^3t^{-1})^{-1}=A_j(zaq^4)^{-1}$ would not have been admissible 
        from $X'$. 
        Thus $d_{j,aqt^{-1}}(X')=0$ and $X_{k+1}$ contains $Y_{j}(zaqt^{-1})^{-1}$. Moreover, 
        \begin{equation}\label{eqnc1}
            X_{k+1}A_{i}(zaq^{}t^{-1})^{-1}A_{i}(zaq^{3}t^{-1})^{-1}\cdots  A_{i}(zaq^{2k+1}t^{-1})^{-1}=X'.
        \end{equation}
        We recall that $X'A_{j}(zaq^{4})^{-1}=X$. By definition of the algorithm, it implies $Y_{j}(zaq^{7}t^{-1})$ is admissible in $X'$.
        In particular, it implies that $d_{j,aq^7t^{-1}}(X')=1$ and $d_{j,aqt^{-1}}(X')=0$. But the equation $\eqref{eqnc1}$ gives
        $$d_{j,aqt^{-1}}(X')=d_{j,aqt^{-1}}(X_{k+1})+1\quad \text{and } \quad d_{j,aq^7t^{-1}}(X')\leq d_{j,aq^7t^{-1}}(X_{k+1})+1.$$
        Hence, $d_{j,aqt^{-1}}(X_{k+1})=-1$ and $d_{j,aq^7t^{-1}}(X_{k+1})\geq 0$. Thus, $X_{k+1}$ contains $Y_{j}(zaqt^{-1})^{-1}$ and 
        does not contain
        $Y_{j}(zaq^7t^{-1})^{-1}$.\\

        Let $s\ge0$ be the maximal integer such that $X_{k+1}$ contains the monomial: $$Y_{j}(zaqt^{-1})^{-1}\dots Y_{j}(zaqt^{-2s-1})^{-1}.$$ 
        If for one integer $1\le u\le s$, $X_{k+1}$ contains $Y_{j}(zaq^7t^{-2u-1})^{-1}$, then it contains the monomial 
        $Y_{j}(zaqt^{-2u+1})^{-1}Y_{j}(zaq^7t^{-2u-1})^{-1}$. 
        However, $X_{k+1}$ has to be regular, so $X_{k+1}$ also contains $Y_{j}(zaq^7t^{-2u+1})^{-1}$. 
        Iterating the same argument we would get that $X_{k+1}$ does contain $Y_{j}(zaq^7t^{-1})^{-1}$, 
        which is a contradiction by the argument above. 
        Hence, we can use $s+1$ times the recurrence hypothesis until a monomial $X_{k+s+2}$. We get: 
        $$\forall 0\leq u \leq s, \quad X_{k+u+2}A_j(zaq^4t^{-2(s-u)-2})^{-1}=X_{k+u+1}$$
        At this step, we have constructed a path from $X_{k+ s +2}$ to $X$ with:
        \begin{equation}\label{eqnc2}
            \begin{split}
            X_{k+s+2}A_j(zaq^4t^{-2})^{-1}\cdots A_j(zaq^4t^{-2s-2})^{-1} \times\qquad\qquad\qquad \\[5mm]
            \qquad A_i(zaqt^{-1})^{-1}\cdots A_i(zaq^{2k+1}t^{-1})^{-1}A_j(zaq^4)^{-1} = X.
            \end{split}
        \end{equation}
        We get the following path in the graph representation of the algorithm:
                \begin{center}
                \begin{tikzpicture}[scale=0.7,
                    >=Stealth,
                    every node/.style={align=center, text width=4cm},
                    every edge/.style={-{Stealth[scale=1.2]}, thick}
                ]
                \node (N1) at (0, 0) { \shortstack{$X_{k+s+2}$}};
                \node (N2) at (0, -1.5) { \shortstack{$\vdots$}};
                \node (N3) at (0, -3) { $X_{k+1}$};
                \node (N5) at (0, -4.5) { \shortstack{$\vdots$}};
                \node (N6) at (0, -6) { \shortstack{$X'$}};
                \node (N7) at (0, -7.5) { \shortstack{$X$}};
                \draw[->] (N1) -- (N2) node[midway, right, yshift=0.2cm] {$ A_{j}(zaq^{4}t^{-2})^{-1}$};
                \draw[->] (N5) -- (N6) node[midway, right, yshift=0.2cm] {$ A_{i}(zaq^{2k+1}t^{-1})^{-1}$};
                \draw[->] (N6) -- (N7) node[midway, right, yshift=0.2cm] {$ A_j(zaq^4)^{-1} $};
                \draw[->] (N2) -- (N3) node[midway, right, yshift=0.2cm] {$ A_{j}(zaq^{4}t^{-2s-2})^{-1} $};
                \draw[->] (N3) -- (N5) node[midway, right, yshift=0.2cm] {$ A_{i}(zaq^{}t^{-1})^{-1} $};
                \end{tikzpicture}
                \end{center}
        Finally, 
        \begin{equation}\label{assumption1ff}
            \text{$X_{k+s+2}$ contains $Y_i(zaq^8t^{-2})\dots Y_i(zaq^{2k+2}t^{-2})Y_j(zaq^7t^{-3})\dots Y_j(zaq^7t^{-2s-3})$.}
        \end{equation}

        The aim is now to construct another path from $X_{k+s+2}$ to $X$ ending with the transformation $A_i(zaqt^{-1})^{-1}$ in order to conclude.\\
        We want to obtain the following subgraph:
                       \begin{center}
\begin{tikzpicture}[
    >=Stealth,
    every node/.style={align=center, text width=4cm},
    every edge/.style={-{Stealth[scale=1.2]}, thick}
]
\node (N1) at (0, 0) { \shortstack{$X_{k+s+2}$}};
\node (N2) at (0, -1.5) { \shortstack{$\vdots$}};
\node (N3) at (0, -3) { $X_{k+1}$};
\node (N4) at (-5, -1.5) { \shortstack{$\vdots$}};
\node (N5) at (0, -5.5) { \shortstack{$\vdots$}};
\node (N6) at (0, -7) { \shortstack{$X'$}};
\node (N7) at (0, -8.5) { \shortstack{$X$}};
\node (N9) at (-5, -3) { \shortstack{$Z_{s+4}$}};
\draw[->] (N1) -- (N2) node[midway, right, yshift=0.2cm] {$ A_{j}(zaq^{4}t^{-2})^{-1}$};
\draw[->] (N5) -- (N6) node[midway, right, yshift=0.2cm] {$ A_{i}(zaq^{2k+1}t^{-1})^{-1}$};
\draw[->] (N6) -- (N7) node[midway, right, yshift=0.2cm] {$ A_j(zaq^4)^{-1} $};
\draw[->] (N1) -- (N4) node[midway, left, yshift=0.2cm] {$ A_{i}(zaq^{7}t^{-1})^{-1} $};
\draw[->] (N4) -- (N9) node[midway, left, yshift=0.2cm] {$ A_{i}(zaq^{2k+1}t^{-1})^{-1} $};
\draw[->] (N2) -- (N3) node[midway, right, yshift=0.2cm] {$ A_{j}(zaq^{4}t^{-2s-2})^{-1} $};
\draw[->] (N3) -- (N5) node[midway, right, yshift=0.2cm] {$ A_{i}(zaq^{}t^{-1})^{-1} $};
\end{tikzpicture}
\end{center} 

        Firstly, let us prove that $k\geq3$. The transformation $X'\rightarrow X$ gives that $Y_j(zaq^7t^{-1})$ is admissible in $X'$. Thus, 
        $d_{j,aq^7t^{-1}}(X')=1$. By contradiction, if $k<3$ then from the graph above we deduce that 
        $$d_{j,aq^7t^{-1}}(X')=d_{j,aq^7t^{-1}}(X_{k+s+2})=1.$$
        Hence, $Y_j(zaq^7t^{-3})$ is not admissible in $X_{k+s+2}$, which is a contradiction. Hence, $k\geq3$.\\
        
        To apply the transformation $A_i(zaq^7t^{-1})^{-1}$, we need to check that $Y_i(zaq^{8}t^{-2})$ is admissible in $X_{k+s+2}$.\\
        We know that $X_{k+s+2}$ does not contain $Y_i(zaq^{6}t^{-2})$ as it would imply that $X_{k+s+1}$ contains $Y_i(zaq^{6}t^{-2})^2$ 
        which is a contradiction as all monomials are generic by assumption.
        Moreover, if $X_{k+s+2}$ contains $Y_i(zaq^{8})$ then the equation \eqref{eqnc2} implies:
        $$d_{i,aq^8}(X_{k+s+2})=d_{i,aq^8}(X_{k-2})=1.$$
        But the graph gives that $Y_{i}(zaq^8t^{-2})$ is admissible in $X_{k-2}$. It is a contradiction.
        Thus, $Y_i(zaq^{8}t^{-2})$ is admissible in $X_{k+s+2}$ and
        the algorithm applies the transformation $A_i(zaq^7t^{-1})^{-1}$ to $X_{k+s+2}$, giving a monomial $Z_{k+s+1}$ such that:
        \begin{equation}\label{eqnc4}
            X_{k+s+2}A_i(zaq^7t^{-1})^{-1}=Z_{k+s+1}.
        \end{equation}
        We assume there exists $3\le m\le k-1$ such that the algorithm gives $Z_{u+s+1}$ for all $m< u\le k$ such that:
        $$\forall m< u\le k, \quad Z_{u+s+2}A_i(zaq^{2(k-u+3)+1}t^{-1})^{-1}=Z_{u+s+1},$$
        setting $Z_{k+s+2}:=X_{k+s+2}$.
        We want to prove that the algorithm gives $Z_{m+s+1}$ such that:
        $$Z_{m+s+2}A_i(zaq^{2(k-m+3)+1}t^{-1})^{-1}=Z_{m+s+1}.$$
        We know that
        \begin{equation}\label{eqnc3}
            Z_{m+s+2}=X_{k+s+2}A_i(zaq^{7}t^{-1})^{-1}\cdots A_i(zaq^{2(k-m+2)+1}t^{-1})^{-1}.
        \end{equation}
        Then, $Z_{m+s+2}$ contains $Y_i(zaq^{2(k-m+4)}t^{-2})\ldots Y_i(zaq^{2k+2}t^{-2})$.\\
        Let us check that $Y_i(zaq^{2(k-m+4)}t^{-2})$ is admissible in $Z_{m+s+2}$.\\
        It is clear that $Z_{m+s+2}$ does not contain $Y_i(zaq^{2(k-m+3)}t^{-2})$ as it has been removed at the previous step.
        Moreover, if $Z_{m+s+2}$ contains $Y_i(zaq^{2(k-m+4)})$. then by \eqref{eqnc3}, $X_{k+s+2}$ contains $Y_i(zaq^{2(k-m+4)})$.
        Thus, according to the graph, $X_{m-2}$ contains $Y_i(zaq^{2(k-m+4)})$. 
        However, we read in the same figure that $Y_i(zaq^{2(k-m+4)}t^{-2})$ is admissible in $X_{m-2}$. It is a contradiction.
        Hence, $Y_i(zaq^{2(k-m+4)}t^{-2})$ is admissible in $Z_{m+s+2}$ and the algorithm gives a monomial $Z_{m+s+1}$ such that:
        $$Z_{m+s+2}A_i(zaq^{2(k-m+3)+1}t^{-1})^{-1}=Z_{m+s+1}.$$
        By induction, we get monomials $Z_{s+4},\ldots,Z_{k+s+1}$ verifying the equation \eqref{eqnc3}.\\

        Now we want to complete our graph to the following:
                       \begin{center}
\begin{tikzpicture}[
    >=Stealth,
    every node/.style={align=center, text width=4cm},
    every edge/.style={-{Stealth[scale=1.2]}, thick}
]
\node (N1) at (0, 0) { \shortstack{$X_{k+s+2}$}};
\node (N2) at (0, -1.5) { \shortstack{$\vdots$}};
\node (N3) at (0, -3) { $X_{k+1}$};
\node (N4) at (-5, -1.5) { \shortstack{$\vdots$}};
\node (N5) at (0, -5.5) { \shortstack{$\vdots$}};
\node (N6) at (0, -7) { \shortstack{$X'$}};
\node (N7) at (0, -8.5) { \shortstack{$X$}};
\node (N11) at (-5, -7) { \shortstack{$Z_{s+2}$}};
\node (N9) at (-5, -3) { \shortstack{$Z_{s+4}$}};
\node (N10) at (-5, -5) { \shortstack{$Z_{s+3}$}};
\draw[->] (N1) -- (N2) node[midway, right, yshift=0.2cm] {$ A_{j}(zaq^{4}t^{-2})^{-1}$};
\draw[->] (N5) -- (N6) node[midway, right, yshift=0.2cm] {$ A_{i}(zaq^{2k+1}t^{-1})^{-1}$};
\draw[->] (N6) -- (N7) node[midway, right, yshift=0.2cm] {$ A_j(zaq^4)^{-1} $};
\draw[->] (N9) -- (N10) node[midway, left, yshift=0.2cm] {$ A_j(zaq^4)^{-1} $};
\draw[->] (N1) -- (N4) node[midway, left, yshift=0.2cm] {$ A_{i}(zaq^{7}t^{-1})^{-1} $};
\draw[->] (N4) -- (N9) node[midway, left, yshift=0.2cm] {$ A_{i}(zaq^{2k+1}t^{-1})^{-1} $};
\draw[->] (N10) -- (N11) node[midway, left, yshift=0.2cm] {$ A_{j}(zaq^{4}t^{-2})^{-1}$};
\draw[->] (N2) -- (N3) node[midway, right, yshift=0.2cm] {$ A_{j}(zaq^{4}t^{-2s-2})^{-1} $};
\draw[->] (N3) -- (N5) node[midway, right, yshift=0.2cm] {$ A_{i}(zaq^{}t^{-1})^{-1} $};
\end{tikzpicture}
\end{center}  

        Furthermore, by \eqref{eqnc4}, we get:
        $$ d_{j,aq^7t^{-1}}(Z_{s+4})=d_{j,aq^7t^{-1}}(Z_{k+s+1})=d_{j,aq^7t^{-1}}(X_{k+s+2})+1.$$
        But we know by admissibility of $Y_j(zaq^7t^{-3})$ and regularity of $X_{k+s+2}$ that:
        $$d_{j,aq^7t^{-1}}(X_{k+s+2})=0.$$
        Hence, $d_{j,aq^7t^{-1}}(Z_{s+4})=1$.
        
        Thus, the monomial $Z_{s+4}$ contains $Y_j(zaq^7t^{-1})\ldots Y_j(zaq^7t^{-2s-3})$.
        We want to check that $Y_j(zaq^7t^{-1})$ is admissible in $Z_{s+4}$.\\
        If $Z_{s+4}$ contains $Y_j(zaqt^{-1})$ then it contains $Y_j(zaq^7t^{-3})Y_j(zaqt^{-1})$ (this is of the form $Y_j(zc)Y_j(zcq^{-2r_j}t^2)$).
        By regularity, it also contains $Y_j(zaqt^{-3})$.
        By \eqref{eqnc3}, $X_{k+s+2}$ contains $Y_j(zaqt^{-3})$, and $Y_j(zaq^7t^{-3})$ is not admissible in $X_{k+s+2}$. It is a contradiction.\\
        If $Z_{s+4}$ contains $Y_j(zaq^7t)$ then by \eqref{eqnc3}, $X_{k+s+2}$ contains $Y_j(zaq^7t)$.  
        According to the graph, this implies that $X'$ contains $Y_j(zaq^7t)$. But $Y_j(zaq^7t^{-1})$ is admissible in $X'$. 
        It is a contradiction.\\
        Hence, $Y_j(zaq^7t^{-1})$ is admissible in $Z_{s+4}$ and the algorithm gives a monomial $Z_{s+3}$ such that:
        \begin{equation}\label{eqnc5}
            Z_{s+4}A_j(zaq^4)^{-1}=Z_{s+3}.
        \end{equation}
        Now we want to construct $Z_{s+2}$ such that 
        $$ Z_{s+3}A_j(zaq^4t^{-2})^{-1}=Z_{s+2}. $$
        Let us prove that $Y_j(zaq^7t^{-3})$ is admissible in $Z_{s+3}$.\\
        By admissibility, 
        $$d_{j,aq^7t^{-3}}(Z_{s+3})=d_{j,aq^7t^{-3}}(X_{k+s+2})=1.$$
        Moreover, $$d_{j,aqt^{-3}}(Z_{s+3})=d_{j,aqt^{-3}}(X_{k+s+2})=0,$$
        and 
        $$d_{j,aq^7t^{-1}}(Z_{s+3})=d_{j,aq^7t^{-1}}(Z_{s+4})-1=1-1=0.$$
        Hence the admissibility and the monomial $Z_{s+2}$ appears in the algorithm.\\

        Finally, let us prove that $Z_{s+2}$ is not regular.\\ 
        By regularity and because $Y_i(zaq^4t^{-2})$ is admissible in $X_{k}$, we have $d_{i,aq^4}(X_{k})=0$. Thus, the graph gives
        \begin{align*}
            d_{i,aq^4}(Z_{s+2})&=d_{i,aq^4}(Z_{s+4})+1,\\
            &=d_{i,aq^4}(X_{k+s+2})+1,\\
            &=d_{i,aq^4}(X_{k})+1,\\
            d_{i,aq^4}(Z_{s+2})&=1.
        \end{align*}
        Moreover, by admissibility, $d_{i,aq^{6}t^{-2}}(X_{k-1})=1$. Thus,
        \begin{align*}
            d_{i,aq^{6}t^{-2}}(Z_{s+2})&=d_{i,aq^{6}t^{-2}}(Z_{s+3})+1,\\
            &=d_{i,aq^{6}t^{-2}}(X_{k+s+2})+1,\\
            &=d_{i,aq^{6}t^{-2}}(X_{k+s+1})-1+1,\\
            &=d_{i,aq^{6}t^{-2}}(X_{k-1}),\\
            d_{i,aq^{6}t^{-2}}(Z_{s+2})&=1.
        \end{align*}
        However, by construction of the algorithm, the transformation $X_{k+s+2}\rightarrow Z_{k+s+1}$ gives that $d_{i,aq^6}(Z_{k+s+1})=-1$.
        Thus the graph implies 
        \begin{align*}
            d_{i,aq^{6}}(Z_{s+2})&=d_{i,aq^{6}}(Z_{s+4})+1,\\
            &=d_{i,aq^{6}}(Z_{k+s+1})+1,\\
            d_{i,aq^{6}}(Z_{s+2})&=0.
        \end{align*}
        Hence setting $c=aq^6t^{-2}$, $Z_{s+2}$ contains $Y_i(zc)Y_i(zcq^{-2r_i}t^2)$ but does not contain $Y_i(zct^2)$. 
        It is a contradiction as assumed that the algorithm does not fail.\\
        \vskip 1cm

        Now, we assume $d_{i,at^{-2}}(X')=-1$. We know that $d_{i,at^{-2}}(X)=0$, then it implies 
        $$d_{i,at^{-2}}(A_j(zbq^{3}t^{-1})^{-1})=1.$$
        Hence, $$bqt^{-1}=at^{-2} \qquad \text{or}\qquad bq^3t^{-1}=at^{-2}\qquad \text{or}\qquad bq^5t^{-1}=at^{-2}.$$
        We assume $bqt^{-1}=at^{-2}$ or $bq^3t^{-1}=at^{-2}$. Firstly, it is clear that $d_{i,a}(X')=-1$ and $d_{i,aq^2}(X')=0$. 
        Then by regularity, $d_{i,aq^2t^{-2}}(X')\neq-1$. 
        Moreover, $$d_{i,aq^2t^{-2}}(A_j(zbq^{3}t^{-1})^{-1})=1.$$
        Then, $d_{i,aq^2t^{-2}}(X')=d_{i,aq^2t^{-2}}(X)-1$. Thus, $d_{i,aq^2t^{-2}}(X)=1$.
        Hence, $X$ contains $Y_i(zaq^2t^{-2})Y_i(za)^{-1}$. Let us prove that $Y_i(zaq^2t^{-2})$ is admissible in $X$.
        If $d_{i,aq^2}(X)=1$ (resp. $d_{i,at^{-2}}(X)=1$) then $X$ is not regular as it contains a 
        term of the form $Y_i(zc)Y_i(zcq^{-2r_i})^{-1}$ (resp. $Y_i(zc)Y_i(zct^{2})^{-1}$). Hence the admissibility of $Y_i(zaq^2t^{-2})$in $X$.
        So, the algorithm implies a transformation to obtain the monomial 
        $$X^*=XA_i(zaqt^{-1})^{-1},$$
        and $d_{i,a}(X^*)=-2$ which contradicts the genericity.\\

        Hence, $b=aq^{-5}t^{-1}$. 
        As in the previous case, we want to apply the induction hypothesis and get the following graph for a $k$ defined above:

                    \begin{figure}[htbp]
                \begin{center}
                \begin{tikzpicture}[scale=0.7,
                    >=Stealth,
                    every node/.style={align=center, text width=4cm},
                    every edge/.style={-{Stealth[scale=1.2]}, thick}
                ]
                \node (N3) at (0, -3) { $X_{k+1}$};
                \node (N5) at (0, -4.5) { \shortstack{$\vdots$}};
                \node (N6) at (0, -6) { \shortstack{$X'$}};
                \node (N7) at (0, -7.5) { \shortstack{$X$}};
                \draw[->] (N5) -- (N6) node[midway, right, yshift=0.2cm] {$ A_{i}(zaq^{}t^{-2k-1})^{-1}$};
                \draw[->] (N6) -- (N7) node[midway, right, yshift=0.2cm] {$ A_j(zaq^{-2}t^{-2})^{-1} $};
                \draw[->] (N3) -- (N5) node[midway, right, yshift=0.2cm] {$ A_{i}(zaq^{}t^{-1})^{-1} $};
                \end{tikzpicture}
                \end{center}
            \end{figure}

        Let $k>0$ be the maximal integer such that $Y_i(zat^{-2s })^{-1}$ appears in $X'$ for all $0\le s\le k$.\\
        By definition, $Y_i(zat^{-2(k+1)})^{-1}$ is not in $X'$. 
        Moreover, if $Y_i(zaq^2t^{-2k})^{-1}$ appears in $X'$, then $Y_i(zat^{-2(k-1) })^{-1}Y_i(zaq^2t^{-2 k})^{-1}$ also appears in $X'$. 
        The monomial has to be regular, so $Y_i(zat^{-2k })^{-1}Y_i(zaq^2t^{-2( k-1)})^{-1}$ appears in $X'$. \\
        Thus, $X'$ contains $Y_i(zat^{-2(k-2) })^{-1}Y_i(zaq^2t^{-2(k-1) })^{-1}$.
        We can iterate this reasoning until we get that $Y_i(zaq^{2})^{-1}$ appears in $X'$. 
        However it does not appear in $X$ and has to be simplified by the transformation $A_j(zbq^3t^{-1})^{-1}=A_j(zaq^{-2}t^{-2})^{-1}$. 
        It is a contradiction by definition of $A_j$. 
        Hence, we can apply the induction hypothesis and we obtain that there exists $X_1$ given by the algorithm such that 
        $$X_1A_{i}(zaqt^{-(2k+1)})^{-1}=X'.$$ 
        We can define recursively 
        $X_s$, $2\le s\le k+1$ such that $$X_s A_{i}(zaqt^{-(2k-2s+3)})^{-1}=X_{s-1}.$$ 
        Indeed, for all $1\leq s\leq k+1$, $Y_i(zat^{-2u })^{-1}$ appears in $X_s$ for all $0\le u\le k+1-s$ and $Y_i(zaq^2t^{-2(u-1) })^{-1}$ does not appear in $X_s$ nor $Y_i(zat^{-2(k+2-s) })^{-1}$ so that we can apply the induction hypothesis. 
        In particular, we have constructed $X_{k+1}$ such that 
        $$X_{k+1}A_{i}(zaqt^{-1})^{-1}=X_k.$$
        Thus, 
        \begin{equation}\label{eqncc1}
            X_{k+1}A_{i}(zaqt^{-1})^{-1}A_{i}(zaqt^{-3})^{-1}\cdots  A_{i}(zaqt^{-(2k+1)})^{-1}=X'.
        \end{equation}

        By definition of the algorithm, the arrow $X'\rightarrow X$ implies that $Y_{j}(zaqt^{-3})$ is admissible in $X'$.
        In particular, it implies that $d_{j,aqt^{-3}}(X')=1$ and $d_{j,aqt^{-1}}(X')=0$. But the equation $\eqref{eqncc1}$ gives
        $$d_{j,aqt^{-1}}(X')=d_{j,aqt^{-1}}(X_{k+1})+1\quad \text{and } \quad d_{j,aqt^{-3}}(X')= d_{j,aqt^{-3}}(X_{k+1})+1.$$
        Hence, $d_{j,aqt^{-1}}(X_{k+1})=-1$ and $d_{j,aqt^{-3}}(X_{k+1})= 0$. Thus, $X_{k+1}$ contains $Y_{j}(zaqt^{-1})^{-1}$ and does not contain
        $Y_{j}(zaqt^{-3})^{-1}$.\\

        Let $s\ge0$ be the maximal integer such that $X_{k+1}$ contains the monomial: 
        $$Y_{j}(zaqt^{-1})^{-1}\dots Y_{j}(zaq^{6s+1}t^{-1})^{-1}.$$ 
        If for one integer $1\le u\le s$, $X_{k+1}$ contains $Y_{j}(zaq^{6u+1}t^{-3})^{-1}$, then it contains the monomial 
        $Y_{j}(zaq^{6(u-1)+1}t^{-1})^{-1}Y_{j}(zaq^{6u+1}t^{-3})^{-1}$. 
        However, $X_{k+1}$ has to be regular, so $X_{k+1}$ also contains $Y_{j}(zaq^{6(u-1)+1}t^{-3})^{-1}$. 
        Iterating the same argument we would get that $X_{k+1}$ does contain $Y_{j}(zaqt^{-3})^{-1}$, which is a contradiction by the argument above. 
        Hence, we can use $s+1$ times the recurrence hypothesis until a monomial $X_{k+s+2}$. We get: 
        $$\forall 0\leq u \leq s, \quad X_{k+u+2}A_j(zaq^{6(s-u)+4}t^{-2})^{-1}=X_{k+u+1}$$
        At this step, we have constructed a path from $X_{k+ s +2}$ to $X$ with:
        \begin{equation}\label{eqncc2}
            \begin{split}
            X_{k+s+2}A_j(zaq^4t^{-2})^{-1}\cdots A_j(zaq^{6s+4}t^{-2})^{-1} \times\qquad\qquad\qquad \\[5mm]
            \qquad A_i(zaqt^{-1})^{-1}\cdots A_i(zaqt^{-2k-1})^{-1}A_j(zaq^{-2}t^{-2})^{-1} = X.
            \end{split}
        \end{equation}

        We get the following path in the graph representation of the algorithm:
                \begin{center}
                \begin{tikzpicture}[scale=0.7,
                    >=Stealth,
                    every node/.style={align=center, text width=4cm},
                    every edge/.style={-{Stealth[scale=1.2]}, thick}
                ]
                \node (N1) at (0, 0) { \shortstack{$X_{k+s+2}$}};
                \node (N2) at (0, -1.5) { \shortstack{$\vdots$}};
                \node (N3) at (0, -3) { $X_{k+1}$};
                \node (N5) at (0, -4.5) { \shortstack{$\vdots$}};
                \node (N6) at (0, -6) { \shortstack{$X'$}};
                \node (N7) at (0, -7.5) { \shortstack{$X$}};
                \draw[->] (N1) -- (N2) node[midway, right, yshift=0.2cm] {$ A_{j}(zaq^{4}t^{-2})^{-1}$};
                \draw[->] (N5) -- (N6) node[midway, right, yshift=0.2cm] {$ A_i(zaqt^{-2k-1})^{-1}$};
                \draw[->] (N6) -- (N7) node[midway, right, yshift=0.2cm] {$ A_j(zaq^{-2}t^{-2})^{-1} $};
                \draw[->] (N2) -- (N3) node[midway, right, yshift=0.2cm] {$ A_j(zaq^{6s+4}t^{-2})^{-1} $};
                \draw[->] (N3) -- (N5) node[midway, right, yshift=0.2cm] {$ A_{i}(zaq^{}t^{-1})^{-1} $};
                \end{tikzpicture}
                \end{center} 
        Finally, 
        \begin{equation}\label{assumptiong1}
            \text{$X_{k+s+2}$ contains $Y_i(zaq^2t^{-4})\dots Y_i(zaq^{2}t^{-2k-2})Y_j(zaq^7t^{-3})\dots Y_j(zaq^{6s+7}t^{-3})$.}
        \end{equation}

        The aim is now to construct another path from $X_{k+s+2}$ to $X$ ending with the transformation $A_i(zaqt^{-1})^{-1}$ 
        in order to conclude.\\
        Firstly, $$d_{i,aq^2t^{-4}}(X_{k+s+2})=d_{i,aq^2t^{-4}}(X_{k})=1,$$
        because as shown in the graph above, $Y_i(zaq^2t^{-4})$ is admissible in $X_{k}$. 
        Moreover, $d_{i,at^{-4}}(X_{k+s+2})=d_{i,at^{-4}}(X_{k})=0$ for the same argument of admissibility. 
        Additionally, if $d_{i,aq^2t^{-2}}(X_{k+s+2})=1$ then  by definition of $A_j(zaq^4t^{-2})^{-1}$ it implies
        $d_{i,aq^2t^{-2}}(X_{k+s+1})=2$ which is impossible by the genericity hypothesis.\\
        Hence, $Y_i(zaq^2t^{-4})$ is admissible in $X_{k+s+2}$.\\
        We construct $$Z_{k+s+1}:=X_{k+s+2}A_i(zaqt^{-3})^{-1}.$$
        We assume there exists $1\leq m \leq k$ such that we constructed $Z_{s+u+1}$, for $m< u\leq k$ such that 
        $$Z_{u+s+1}=Z_{u+s+2}A_i(zaqt^{-2(k-u)-3})^{-1}.$$
        Let us construct $Z_{s+m+1}$.\\
        Firstly, according to the graph, $$d_{i,aq^2t^{-2(k-m)-4}}(Z_{s+m+2})=d_{i,aq^2t^{-2(k-m)-4}}(X_m)=1.$$
        Moreover, $$d_{i,aq^2t^{-2(k-m)-2}}(Z_{s+m+2})=d_{i,aq^2t^{-2(k-m)-2}}(Z_{s+m+3})-1=d_{i,aq^2t^{-2(k-m)-2}}(X_{k+s+2})-1=0,$$
        because of the transformation $Z_{s+m+3}\rightarrow Z_{s+m+2}$.\\
        Additionally, $Y_i(zaq^2t^{-2(k-m)-4})$ is admissible in $X_m$. Thus, 
        $$d_{i,at^{-2(k-m)-4}}(Z_{s+m+2})=d_{i,at^{-2(k-m)-4}}(X_{k+s+2})=d_{i,at^{-2(k-m)-4}}(X_m)=0.$$
        Hence the admissibility of $Y_i(zaq^2t^{-2(k-m)-4})$ in $Z_{s+m+2}$ and the existence of the following
         transformation given by the algorithm:
         $$ Z_{m+s+1}=Z_{m+s+2}A_i(zaqt^{-2(k-m)-3})^{-1}. $$
         We obtain the following graph:

        \begin{center}
\begin{tikzpicture}[
    >=Stealth,
    every node/.style={align=center, text width=4cm},
    every edge/.style={-{Stealth[scale=1.2]}, thick}
]
\node (N1) at (0, 0) { \shortstack{$X_{k+s+2}$}};
\node (N2) at (0, -1.5) { \shortstack{$\vdots$}};
\node (N3) at (0, -3) { $X_{k+1}$};
\node (N4) at (-5, -1.5) { \shortstack{$\vdots$}};
\node (N5) at (0, -5.5) { \shortstack{$\vdots$}};
\node (N6) at (0, -7) { \shortstack{$X'$}};
\node (N7) at (0, -8.5) { \shortstack{$X$}};
\node (N9) at (-5, -3) { \shortstack{$Z_{s+2}$}};
\draw[->] (N1) -- (N2) node[midway, right, yshift=0.2cm] {$ A_{j}(zaq^{4}t^{-2})^{-1}$};
\draw[->] (N5) -- (N6) node[midway, right, yshift=0.2cm] {$ A_i(zaqt^{-2k-1})^{-1}$};
\draw[->] (N6) -- (N7) node[midway, right, yshift=0.2cm] {$ A_j(zaq^{-2}t^{-2})^{-1} $};
\draw[->] (N1) -- (N4) node[midway, left, yshift=0.2cm] {$ A_{i}(zaq^{}t^{-3})^{-1} $};
\draw[->] (N4) -- (N9) node[midway, left, yshift=0.2cm] {$ A_i(zaqt^{-2k-1})^{-1} $};
\draw[->] (N2) -- (N3) node[midway, right, yshift=0.2cm] {$ A_j(zaq^{6s+4}t^{-2})^{-1} $};
\draw[->] (N3) -- (N5) node[midway, right, yshift=0.2cm] {$ A_{i}(zaq^{}t^{-1})^{-1} $};
\end{tikzpicture}
        \end{center}

        Again, we want to construct a path from $Z_{s+2}$ to $X$ in the algorithm graph.\\
        Firstly,
        \begin{align*}
            d_{j,aqt^{-3}}(Z_{s+2})&=d_{j,aqt^{-3}}(X_{k+s+2})+1,\\
            &=d_{j,aqt^{-3}}(X_{k})+1,\\
            &=d_{j,aqt^{-3}}(X_{k-1})-1+1,\\
            &=d_{j,aqt^{-3}}(X')=1.
        \end{align*}
        Because $Y_j(zaqt^{-3})$ is admissible in $X'$.\\
        Moreover, 
        $$d_{j,aq^{-5}t^{-3}}(Z_{s+2})=d_{j,aq^{-5}t^{-3}}(X')=0,$$
        as $Y_j(zaqt^{-3})$ is admissible in $X'$. Additionally, we deduce from the graph the following computation
        \begin{align*}
            d_{j,aqt^{-1}}(Z_{s+2})&=d_{j,aqt^{-1}}(X_{k+s+2}),\\
            &=d_{j,aqt^{-1}}(X_{k+s+1})+1,\\
            &=d_{j,aqt^{-1}}(X_{k})-1+1,\\
            &=d_{j,aqt^{-1}}(X')=0.
        \end{align*}
        Hence, $Y_j(zaqt^{-3})$ is admissible in $Z_{s+2}$. 
        Thus, the algorithm constructs the following monomial:
        $$ Z_{s+1}=Z_{s+2}A_j(zaq^{-2}t^{-2})^{-1}.$$
        Now, we assume there exists $0\leq m \leq s$ such that we constructed $Z_{u+1}$, for $m\leq u\leq s$ such that 
        $$Z_{u+1}=Z_{u+2}A_j(zaq^{6(s-u)-2}t^{-2})^{-1}.$$
        Let us construct $Z_{m}$.\\
        Let us prove that $Y_j(zaq^{6(s-m)+7}t^{-3})$ is admissible in $Z_{m+1}$.\\
        Firstly, according to the graph and by admissibility, $$d_{j,aq^{6(s-m)+7}t^{-3}}(Z_{m+1})=d_{j,aq^{6(s-m)+7}t^{-3}}(X_{k+m+2})=1.$$
        Moreover, $$d_{j,aq^{6(s-m)+1}t^{-3}}(Z_{m+1})=d_{j,aq^{6(s-m)+1}t^{-3}}(Z_{m+2})-1=0,$$
        by admissibility and because of the transformation $Z_{m+2}\rightarrow Z_{m+1}$.\\
        Additionally, $Y_j(zaq^{6(s-m)+7}t^{-3})$ is admissible in $X_{k+m+2}$. Thus, 
        $$d_{j,aq^{6(s-m)+7}t^{-1}}(Z_{m+1})=d_{j,aq^{6(s-m)+7}t^{-1}}(X_{k+m+2})=0.$$
        Hence the admissibility of $Y_j(zaq^{6(s-m)+7}t^{-3})$ in $Z_{m+1}$ and the existence of the following
         transformation given by the algorithm:
         $$ Z_{m}=Z_{m+1}A_j(zaq^{6(s-m)+4}t^{-2})^{-1}. $$
         To conclude, we have to prove that $Y_i(zaq^2t^{-2})$ is admissible in $Z_0$.\\
         We have the equality $Z_0A_i(zaqt^{-1})^{-1}=X$. Then, by identification of the respective degrees it is clear. \\
         Hence the following paths of transformations in the algorithm:
         \begin{center}
\begin{tikzpicture}[
    >=Stealth,
    every node/.style={align=center, text width=4cm},
    every edge/.style={-{Stealth[scale=1.2]}, thick}
]
\node (N1) at (0, 0) { \shortstack{$X_{k+s+2}$}};
\node (N2) at (0, -1.5) { \shortstack{$\vdots$}};
\node (N3) at (0, -3) { $X_{k+1}$};
\node (N4) at (-5, -1.5) { \shortstack{$\vdots$}};
\node (N5) at (0, -5.5) { \shortstack{$\vdots$}};
\node (N6) at (0, -7) { \shortstack{$X'$}};
\node (N7) at (0, -8.5) { \shortstack{$X$}};
\node (N9) at (-5, -3) { \shortstack{$Z_{s+2}$}};
\node (N10) at (-5, -5.5) { \shortstack{$\vdots$}};
\node (N11) at (-5, -7) { \shortstack{$Z_0$}};
\draw[->] (N1) -- (N2) node[midway, right, yshift=0.2cm] {$ A_{j}(zaq^{4}t^{-2})^{-1}$};
\draw[->] (N5) -- (N6) node[midway, right, yshift=0.2cm] {$ A_i(zaqt^{-2k-1})^{-1}$};
\draw[->] (N6) -- (N7) node[midway, right, yshift=0.2cm] {$ A_j(zaq^{-2}t^{-2})^{-1} $};
\draw[->] (N9) -- (N10) node[midway, left, yshift=0.2cm] {$ A_j(zaq^{-2}t^{-2})^{-1} $};
\draw[->] (N1) -- (N4) node[midway, left, yshift=0.2cm] {$ A_{i}(zaq^{}t^{-3})^{-1} $};
\draw[->] (N4) -- (N9) node[midway, left, yshift=0.2cm] {$ A_i(zaqt^{-2k-1})^{-1} $};
\draw[->] (N2) -- (N3) node[midway, right, yshift=0.2cm] {$ A_j(zaq^{6s+4}t^{-2})^{-1} $};
\draw[->] (N10) -- (N11) node[midway, left, yshift=0.2cm] {$ A_j(zaq^{6s+4}t^{-2})^{-1} $};
\draw[->] (N3) -- (N5) node[midway, right, yshift=0.2cm] {$ A_{i}(zaq^{}t^{-1})^{-1} $};
\draw[->] (N11) -- (N7) node[midway, left, yshift=0cm] {$ A_{i}(zaq^{}t^{-1})^{-1} $};
\end{tikzpicture}
        \end{center}

        \textit{b)} We assume there exists a monomial $X''$ also appearing in the algorithm such that $$X''A_{i}(zaq^{r_i}t^{-1})^{-1}=X',$$
        then $X''$ does contain $Y_i(zaq^{2r_i}t^{-2})$.\\ 

        Here there are two cases to consider:
        \begin{itemize}
            \item $d_{j,bq^{2r_j}t^{-2}}(A_{i}(zaq^{r_i}t^{-1})^{-1})=1$ and the term $Y_j(zbq^{2r_j}t^{-2})$ in $X'$ comes from 
            the transformation $A_{i}(zaq^{r_i}t^{-1})^{-1}$.
            \item $X''$ contains $Y_j(zbq^{2r_j}t^{-2})$.
        \end{itemize}

        Firstly, we assume $d_{j,bq^{2r_j}t^{-2}}(A_{i}(zaq^{r_i}t^{-1})^{-1})=1$ and the term $Y_i(zbq^{2r_i}t^{-2})$ comes from 
        the transformation $A_{i}(zaq^{r_i}t^{-1})^{-1}$. Hence, $i\neq j$, and we assume the Dynkin subdiagram with nodes 
        $\{i,j\}$ is of type $G_2$ with $(i,j)=(2,1)$. We have $(r_i,r_j)=(1,3)$.
        Thus, $bq^6t^{-2}=aqt^{-1}$, and $bq^5t^{-1}=a$. Hence, 
        $$d_{i,a}(A_j(zbq^3t^{-1})^{-1})=1.$$
        However, $d_{i,a}(X)=-1$ and $X'A_{j}(zbq^{3}t^{-1})^{-1}=X$. Thus, $d_{i,a}(X')=-2$ which contradicts the genericity. It is a contradiction.\\
        
        Hence, we are in the second case and $X''$ contains $Y_j(zbq^{2r_j}t^{-2})$.

        Let us prove that $Y_j(zbq^{2r_j}t^{-2})$ is admissible in $X''$.\\
        If $i=j$ (so that $r_i=r_j$), then by contradiction we assume that $Y_j(zbq^{2r_j}t^{-2})$ is 
        not admissible in $X''$. We know that $Y_j(zbq^{2r_j}t^{-2})$ is admissible in $X'$ so that the presence of 
        $Y_i(zaq^{2r_i}t^{-2})$  in $X''$ prevents from applying the transformation $A_j(zbq^{r_j}t^{-1})^{-1}$. 
        It implies that $a=bq^{-2r_i}$ or $a=bt^2$. 
        This implies $X$ contains $Y_i(zb)^{-1}=Y_i(zaq^{2r_i})^{-1}$ or $Y_i(zat^{-2})^{-1}$. 
        It is a contradiction by the initial assumption on $X$. 
        Thus, $Y_j(zbq^{2r_j}t^{-2})$ is admissible in $X''$.\\
        If $i\neq j$ then it is clear that the admissibility of $Y_j(zbq^{2r_j}t^{-2})$ in $X'$ implies its admissibility in $X''$. \\
        Finally, there exists a monomial $$X_{new}=X''A_{j}(zbq^{r_j}t^{-1})^{-1}$$ given by the algorithm and $X_{new}$ contains $Y_i(zaq^{2r_i}t^{-2})$.
        Moreover, if the transformation $A_{i}(zaq^{r_i}t^{-1})^{-1}$ is not admissible from $X_{new}$ then there exists a factor 
        $Y_i(zat^{-2})$ (resp. $Y_i(zaq^{2r_i})$) blocking this transformation. 
        But this would imply that this factor also appears in $X$ as we have the following equality of monomials:
        $$X=X_{new}A_{i}(zaq^{r_i}t^{-1})^{-1}.$$
        We remark that a priori, this equality is not sufficient to get an arrow from $X_{new}$ to $X$. This equality implies:
        \begin{align*}
            d_{i,aq^{2r_i}}(X_{new})&=d_{i,aq^{2r_i}}(X),
        \end{align*}
        and same for $at^{-2}$. 
        And hence $X$ does contain $Y_i(zc)$ and $Y_i(zct^2)^{-1}$ with $c=at^{-2}$ (resp. $Y_i(zc)Y_i(zcq^{-2r_i})^{-1}$ with $c=aq^{2r_i}$) 
        which contradicts the regularity of $X$. It is impossible. Hence the admissibility of $Y_i(zaq^{2r_i}t^{-2})$ in $X_{new}$.\\
        Finally we get the following paths:
        \begin{center}
\begin{tikzpicture}[
    >=Stealth,
    every node/.style={align=center, text width=4cm},
    every edge/.style={-{Stealth[scale=1.2]}, thick}
]
\node (N1) at (0, 0) { \shortstack{$X''$}};
\node (N2) at (0, -3) { \shortstack{$X'$}};
\node (N3) at (0, -6) { \shortstack{$X$}};
\node (N4) at (-3, -3) { \shortstack{$X_{new}$}};
\draw[->] (N1) -- (N2) node[midway, right, yshift=0.2cm] {$ A_{i}(zaq^{r_i}t^{-1})^{-1}$};
\draw[->] (N2) -- (N3) node[midway, right, yshift=0.2cm] {$ A_j(zbq^{r_j}t^{-1})^{-1} $};
\draw[->] (N4) -- (N3) node[midway, left, yshift=0.2cm] {$ A_{i}(zaq^{r_i}t^{-1})^{-1} $};
\draw[->] (N1) -- (N4) node[midway, left, yshift=0.2cm] {$ A_j(zbq^{r_j}t^{-1})^{-1}$};
\end{tikzpicture}
\end{center}

and hence the result.
        This concludes the proof in type $G_2$ for $(i,j)=(2,1)$.
        \end{proof}

\subsubsection{In the case $(i,j)=(1,2)$ }

\begin{proof}
    We prove it by induction on the height of the monomials. We assume the Lie subalgebra generated by the nodes $i$ and $j$ is of type $G_2$, and $(i,j)=(1,2)$.
         \textit{\underline{Height $0$:}} The only monomial with height $0$ is the dominant generic monomial from which we start our algorithm. Hence the property holds at height $0$.\\

        \textit{\underline{Height $h+1$:}} We assume the property is true for all the monomials with heights $h'\leq h\in\N$. 
        Let us prove it for the monomials with height $h+1$. Let $X$ be a monomial with height $h+1$. 
        The monomial appears in the algorithm, so it has to come from a monomial $X'$ with height $h$. 
        Let $A_j(zbq^{r_j}t^{-1})^{-1}$ be the transformation involved. 
        It implies that $d_{j,b}(X)=-1$ and $d_{j,bq^{2r_j}t^{-2}}(X')=1$. 
        Let $i\in I$, $a\in q^\Z t^\Z$ such that $d_{i,a}(X)=-1$ and $d_{i,aq^{2r_i}}(X),d_{i,at^{-2}}(X)\ne -1$. 
        If $(i,a)=(j,b)$ then we have the result. 
        Else, we have
        $$X=X'A_{j}(zbq^{r_j}t^{-1})^{-1}.$$
        Hence,
        $$ \left\{\begin{array}{llll}
            d_{j,bq^{2r_j}t^{-2}}(X')=1\\
            d_{j,b}(X)=-1\\
            d_{j,c}(X)=d_{j,c}(X')\text{ for all $c\not\in\{b,bq^{2r_j}t^{-2}\}$}\\
            d_{k,c}(X)\geq d_{k,c}(X')\text{ for all $k\neq j$, $c\in q^\Z t^\Z$}
        \end{array}\right.$$
        We know that $d_{i,a}(X)=-1$, and $(i,a)\neq (j,b)$.  
        It implies $d_{i,a}(X')\leq-1$. The monomials are all generic, thus 
        $$ d_{i,a}(X')=-1$$
        Then by the induction assumption, one of the two following assertions is true:\\
        \begin{itemize}
            \item[(a)]  $d_{i,aq^{2r_i}}(X')=-1$  {or}  $d_{i,at^{-2}}(X')= -1$.
            \item[(b)]  There exists a monomial $X''$ appearing in the algorithm such that
        $$X''A_{i}(zaq^{r_i}t^{-1})^{-1}=X'.$$
        \end{itemize}

        \textit{a)} Firstly we assume (a) is true and $X'$ does contain $Y_i(zaq^{2r_i})^{-1}$ (i.e $d_{i,aq^{2r_i}}(X')=-1$). 
        By assumption, $X$ does not. 
        It implies that the transformation $A_j(zbq^{r_j}t^{-1})^{-1}$ simplifies $Y_i(zaq^{2r_i})^{-1}$ and in particular $i\ne j$. 
        Therefore, the expressions depend on the type of the Dynkin diagram generated by the nodes $i$ and $j$. \\

        Let us do it in type $G_2$, $(i,j)=(1,2)$.  Thus, $(r_i,r_j)=(3,1)$. 
        Then $ bqt^{-1}=aq^6  $ and $b=aq^5t$
        
        Let $k>0$ be the maximal integer such that $Y_i(zaq^{6s})^{-1}$ appears in $X'$ for all $0\le s\le k$. 
        We want to prove that the graph representation
        of the algorithm contains the following subgraph:
            \begin{figure}[htbp]
            \begin{center}
            \begin{tikzpicture}[scale=0.7,
                >=Stealth,
                every node/.style={align=center, text width=4cm},
                every edge/.style={-{Stealth[scale=1.2]}, thick}
            ]
            \node (N3) at (0, -3) { $X_{k+1}$};
            \node (N5) at (0, -4.5) { \shortstack{$\vdots$}};
            \node (N6) at (0, -6) { \shortstack{$X'$}};
            \node (N7) at (0, -7.5) { \shortstack{$X$}};
            \draw[->] (N5) -- (N6) node[midway, right, yshift=0.2cm] {$ A_{i}(zaq^{6k+3}t^{-1})^{-1}$};
              \draw[->] (N6) -- (N7) node[midway, right, yshift=0.2cm] {$ A_j(zaq^6)^{-1} $};
             \draw[->] (N3) -- (N5) node[midway, right, yshift=0.2cm] {$ A_{i}(zaq^{3}t^{-1})^{-1} $};
             \end{tikzpicture}
              \end{center} 
        \end{figure}
        By definition, $Y_i(zaq^{6(k+1)})^{-1}$ is not in $X'$. 
        Moreover, if $Y_i(zaq^{6k}t^{-2})^{-1}$ appears in $X'$, then $Y_i(zaq^{6(k-1) })^{-1}Y_i(zaq^{6 k}t^{-2})^{-1}$ also appears in $X'$. 
        The monomial has to be regular, so $Y_i(zaq^{6k })^{-1}Y_i(zaq^{6( k-1)}t^{-2})^{-1}$ appears in $X'$. \\
        Thus, $X'$ contains $Y_i(zaq^{6(k-2) })^{-1}Y_i(zaq^{6(k-1) }t^{-2})^{-1}$.
        We can iterate this reasoning until we get that $Y_i(zat^{-2})^{-1}$ appears in $X'$. 
        However it does not appear in $X$ and has to be simplified by the transformation $A_j(zbq^{r_j}t^{-1})^{-1}=A_j(zaq^6)^{-1}$. 
        It is a contradiction by definition of $A_j$. 
        Hence, we can apply the induction hypothesis and we obtain that there exists $X_1$ given by the algorithm such that 
        $$X_1A_{i}(zaq^{6k+3}t^{-1})^{-1}=X'.$$ 
        We can define recursively 
        $X_s$, $2\le s\le k+1$ such that $$X_s A_{i}(zaq^{6k-6s+9}t^{-1})^{-1}=X_{s-1}.$$ 
        Indeed, for all $1\leq s\leq k+1$, $Y_i(zaq^{6u })^{-1}$ appears in $X_s$ for all $0\le u\le k+1-s$ and 
        $Y_i(zaq^{6(u-1) }t^{-2})^{-1}$ does not appear in $X_s$ by the same argument as above nor 
        $Y_i(zaq^{6(k+2-s) })^{-1}$ so that we can apply the induction hypothesis. 
        In particular, we have constructed $X_{k+1}$ such that $$X_{k+1}A_{i}(zaq^{3 }t^{-1})^{-1}=X_k.$$
        Moreover, 
        \begin{equation}\label{eqnccc1}
            X_{k+1}A_{i}(zaq^{3}t^{-1})^{-1}A_{i}(zaq^{9}t^{-1})^{-1}\cdots  A_{i}(zaq^{6k+3}t^{-1})^{-1}=X'.
        \end{equation}
        We recall that $X'A_{j}(zaq^{6})^{-1}=X$. By definition of the algorithm, it implies $Y_{j}(zaq^{7}t^{-1})$ is admissible in $X'$.
        In particular, it implies that $d_{j,aq^7t^{-1}}(X')=1$ and $d_{j,aq^5t^{-1}}(X')=0$. But the equation $\eqref{eqnccc1}$ gives
        $$d_{j,aq^5t^{-1}}(X')=d_{j,aq^5t^{-1}}(X_{k+1})+1\quad \text{and } \quad d_{j,aq^7t^{-1}}(X')= d_{j,aq^7t^{-1}}(X_{k+1})+1.$$
        Hence, $d_{j,aq^5t^{-1}}(X_{k+1})=-1$ and $d_{j,aq^7t^{-1}}(X_{k+1})= 0$. Thus, $X_{k+1}$ contains $Y_{j}(zaq^5t^{-1})^{-1}$ and does not contain
        $Y_{j}(zaq^7t^{-1})^{-1}$.\\

        Let $s\ge0$ be the maximal integer such that $X_{k+1}$ contains the monomial: 
        $$Y_{j}(zaq^5t^{-1})^{-1}\dots Y_{j}(zaq^5t^{-2s-1})^{-1}.$$ 
        If for one integer $1\le u\le s$, $X_{k+1}$ contains $Y_{j}(zaq^7t^{-2u-1})^{-1}$, then it contains the monomial 
        $Y_{j}(zaq^5t^{-2u+1})^{-1}Y_{j}(zaq^7t^{-2u-1})^{-1}$. 
        However, $X_{k+1}$ has to be regular, so $X_{k+1}$ also contains $Y_{j}(zaq^7t^{-2u+1})^{-1}$. 
        Iterating the same argument we would get that $X_{k+1}$ does contain $Y_{j}(zaq^7t^{-1})^{-1}$, 
        which is a contradiction by the argument above. 
        Hence, we can use $s+1$ times the recurrence hypothesis until a monomial $X_{k+s+2}$. We get: 
        $$\forall 0\leq u \leq s, \quad X_{k+u+2}A_j(zaq^6t^{-2(s-u)-2})^{-1}=X_{k+u+1}$$
        At this step, we have constructed a path from $X_{k+ s +2}$ to $X$ with:
        \begin{equation}\label{eqnccc2}
            \begin{split}
            X_{k+s+2}A_j(zaq^6t^{-2})^{-1}\cdots A_j(zaq^6t^{-2s-2})^{-1} \times\qquad\qquad\qquad \\[5mm]
            \qquad A_i(zaq^3t^{-1})^{-1}\cdots A_i(zaq^{6k+3}t^{-1})^{-1}A_j(zaq^6)^{-1} = X.
            \end{split}
        \end{equation}
        We get the following path in the graph representation of the algorithm:
                \begin{center}
                \begin{tikzpicture}[scale=0.7,
                    >=Stealth,
                    every node/.style={align=center, text width=4cm},
                    every edge/.style={-{Stealth[scale=1.2]}, thick}
                ]
                \node (N1) at (0, 0) { \shortstack{$X_{k+s+2}$}};
                \node (N2) at (0, -1.5) { \shortstack{$\vdots$}};
                \node (N3) at (0, -3) { $X_{k+1}$};
                \node (N5) at (0, -4.5) { \shortstack{$\vdots$}};
                \node (N6) at (0, -6) { \shortstack{$X'$}};
                \node (N7) at (0, -7.5) { \shortstack{$X$}};
                \draw[->] (N1) -- (N2) node[midway, right, yshift=0.2cm] {$ A_{j}(zaq^{6}t^{-2})^{-1}$};
                \draw[->] (N5) -- (N6) node[midway, right, yshift=0.2cm] {$ A_{i}(zaq^{6k+3}t^{-1})^{-1}$};
                \draw[->] (N6) -- (N7) node[midway, right, yshift=0.2cm] {$ A_j(zaq^6)^{-1} $};
                \draw[->] (N2) -- (N3) node[midway, right, yshift=0.2cm] {$ A_{j}(zaq^{6}t^{-2s-2})^{-1} $};
                \draw[->] (N3) -- (N5) node[midway, right, yshift=0.2cm] {$ A_{i}(zaq^{3}t^{-1})^{-1} $};
                \end{tikzpicture}
                \end{center}
        Finally, 
        \begin{equation}\label{assumptioncc2}
            \text{$X_{k+s+2}$ contains $Y_i(zaq^{12}t^{-2})\dots Y_i(zaq^{6k+6}t^{-2})Y_j(zaq^7t^{-3})\dots Y_j(zaq^7t^{-2s-3})$.}
        \end{equation}

        The aim is now to construct another path from $X_{k+s+2}$ to $X$ ending with the transformation 
        $A_i(zaq^3t^{-1})^{-1}$ in order to conclude.\\
        We want to obtain the following subgraph:
                       \begin{center}
\begin{tikzpicture}[
    >=Stealth,
    every node/.style={align=center, text width=4cm},
    every edge/.style={-{Stealth[scale=1.2]}, thick}
]
\node (N1) at (0, 0) { \shortstack{$X_{k+s+2}$}};
\node (N2) at (0, -1.5) { \shortstack{$\vdots$}};
\node (N3) at (0, -3) { $X_{k+1}$};
\node (N4) at (-5, -1.5) { \shortstack{$\vdots$}};
\node (N5) at (0, -5.5) { \shortstack{$\vdots$}};
\node (N6) at (0, -7) { \shortstack{$X'$}};
\node (N7) at (0, -8.5) { \shortstack{$X$}};
\node (N9) at (-5, -3) { \shortstack{$Z_{s+2}$}};
\draw[->] (N1) -- (N2) node[midway, right, yshift=0.2cm] {$ A_{j}(zaq^{6}t^{-2})^{-1}$};
\draw[->] (N5) -- (N6) node[midway, right, yshift=0.2cm] {$ A_{i}(zaq^{6k+3}t^{-1})^{-1}$};
\draw[->] (N6) -- (N7) node[midway, right, yshift=0.2cm] {$ A_j(zaq^6)^{-1} $};
\draw[->] (N1) -- (N4) node[midway, left, yshift=0.2cm] {$ A_{i}(zaq^{9}t^{-1})^{-1} $};
\draw[->] (N4) -- (N9) node[midway, left, yshift=0.2cm] {$ A_{i}(zaq^{6k+3}t^{-1})^{-1} $};
\draw[->] (N2) -- (N3) node[midway, right, yshift=0.2cm] {$ A_{j}(zaq^{6}t^{-2s-2})^{-1} $};
\draw[->] (N3) -- (N5) node[midway, right, yshift=0.2cm] {$ A_{i}(zaq^{3}t^{-1})^{-1} $};
\end{tikzpicture}
\end{center}

        To apply the transformation $A_i(zaq^9t^{-1})^{-1}$, we need to check that $Y_i(zaq^{12}t^{-2})$ is admissible in $X_{k+s+2}$.\\
        We know that $X_{k+s+2}$ does not contain $Y_i(zaq^{6}t^{-2})$ as it would imply that $X_{k+s+1}$ contains $Y_i(zaq^{6}t^{-2})^2$ 
        which is a contradiction as all monomials are generic by assumption.
        Moreover, if $X_{k+s+2}$ contains $Y_i(zaq^{12})$ then the equation \eqref{eqnccc2} implies:
        $$d_{i,aq^{12}}(X_{k+s+2})=d_{i,aq^{12}}(X_{k})=1.$$
        But the graph gives that $Y_{i}(zaq^{12}t^{-2})$ is admissible in $X_{k}$. It is a contradiction.
        Thus, $Y_i(zaq^{12}t^{-2})$ is admissible in $X_{k+s+2}$ and
        the algorithm applies the transformation $A_i(zaq^9t^{-1})^{-1}$ to $X_{k+s+2}$, 
        giving a monomial $Z_{k+s+1}$ such that:
        \begin{equation}\label{eqnccc4}
            X_{k+s+2}A_i(zaq^9t^{-1})^{-1}=Z_{k+s+1}.
        \end{equation}
        We assume there exists $1\le m\le k-1$ such that the algorithm gives $Z_{u+s+1}$ for all $m< u\le k$ such that:
        $$\forall m< u\le k, \quad Z_{u+s+2}A_i(zaq^{6(k-u+1)+3}t^{-1})^{-1}=Z_{u+s+1},$$
        setting $Z_{k+s+2}=X_{k+s+2}$.
        We want to prove that the algorithm gives $Z_{m+s+1}$ such that:
        $$Z_{m+s+2}A_i(zaq^{6(k-m+1)+3}t^{-1})^{-1}=Z_{m+s+1}.$$
        We know that
        \begin{equation}\label{eqnccc3}
            Z_{m+s+2}=X_{k+s+2}A_i(zaq^{9}t^{-1})^{-1}\cdots A_i(zaq^{6(k-m)+3}t^{-1})^{-1}.
        \end{equation}
        Then, $Z_{m+s+2}$ contains $Y_i(zaq^{6(k-m+2)}t^{-2})\ldots Y_i(zaq^{6k+6}t^{-2})$.\\
        Let us check that $Y_i(zaq^{6(k-m+2)}t^{-2})$ is admissible in $Z_{m+s+2}$.\\
        It is clear that $Z_{m+s+2}$ does not contain $Y_i(zaq^{6(k-m+1)}t^{-2})$ as it has been removed at the previous step.
        Moreover, if $Z_{m+s+2}$ contains $Y_i(zaq^{6(k-m+2)})$. then by \eqref{eqnccc3}, $X_{k+s+2}$ contains $Y_i(zaq^{6(k-m+2)})$.
        Thus, according to the graph, $X_m$ contains $Y_i(zaq^{6(k-m+2)})$. 
        However, we read in the same figure that $Y_i(zaq^{6(k-m+2)}t^{-2})$ is admissible in $X_{m}$. It is a contradiction.
        Hence, $Y_i(zaq^{6(k-m+2)}t^{-2})$ is admissible in $Z_{m+s+2}$ and the algorithm gives a monomial $Z_{m+s+1}$ such that:
        $$Z_{m+s+2}A_i(zaq^{6(k-m+1)+3}t^{-1})^{-1}=Z_{m+s+1}.$$
        By induction, we get monomials $Z_{s+2},\ldots,Z_{k+s+1}$ verifying the equation \eqref{eqnccc3}.\\

        Now we want to complete our graph to the following:
                       \begin{center}
\begin{tikzpicture}[
    >=Stealth,
    every node/.style={align=center, text width=4cm},
    every edge/.style={-{Stealth[scale=1.2]}, thick}
]
\node (N1) at (0, 0) { \shortstack{$X_{k+s+2}$}};
\node (N2) at (0, -1.5) { \shortstack{$\vdots$}};
\node (N3) at (0, -3) { $X_{k+1}$};
\node (N4) at (-5, -1.5) { \shortstack{$\vdots$}};
\node (N5) at (0, -5.5) { \shortstack{$\vdots$}};
\node (N6) at (0, -7) { \shortstack{$X'$}};
\node (N7) at (0, -8.5) { \shortstack{$X$}};
\node (N11) at (-5, -7) { \shortstack{$Z_0$}};
\node (N9) at (-5, -3) { \shortstack{$Z_{s+2}$}};
\node (N10) at (-5, -5.5) { \shortstack{$\vdots$}};
\draw[->] (N1) -- (N2) node[midway, right, yshift=0.2cm] {$ A_{j}(zaq^{6}t^{-2})^{-1}$};
\draw[->] (N5) -- (N6) node[midway, right, yshift=0.2cm] {$ A_{i}(zaq^{6k+3}t^{-1})^{-1}$};
\draw[->] (N6) -- (N7) node[midway, right, yshift=0.2cm] {$ A_j(zaq^6)^{-1} $};
\draw[->] (N9) -- (N10) node[midway, left, yshift=0.2cm] {$ A_j(zaq^6)^{-1} $};
\draw[->] (N1) -- (N4) node[midway, left, yshift=0.2cm] {$ A_{i}(zaq^{9}t^{-1})^{-1} $};
\draw[->] (N4) -- (N9) node[midway, left, yshift=0.2cm] {$ A_{i}(zaq^{6k+3}t^{-1})^{-1} $};
\draw[->] (N10) -- (N11) node[midway, left, yshift=0.2cm] {$ A_{j}(zaq^{6}t^{-2s-2})^{-1}$};
\draw[->] (N2) -- (N3) node[midway, right, yshift=0.2cm] {$ A_{j}(zaq^{6}t^{-2s-2})^{-1} $};
\draw[->] (N3) -- (N5) node[midway, right, yshift=0.2cm] {$ A_{i}(zaq^{3}t^{-1})^{-1} $};
\end{tikzpicture}
\end{center}  

        Furthermore, by \eqref{eqnccc4}, we get:
        $$ d_{j,aq^7t^{-1}}(Z_{s+2})=d_{j,aq^7t^{-1}}(Z_{k+s+1})=d_{j,aq^7t^{-1}}(X_{k+s+2})+1.$$
        But we know by admissibility of $Y_j(zaq^7t^{-3})$ and regularity of $X_{k+s+2}$ that:
        $$d_{j,aq^7t^{-1}}(X_{k+s+2})=0.$$
        Hence, $d_{j,aq^7t^{-1}}(Z_{s+2})=1$.
        
        Thus, the monomial $Z_{s+2}$ contains $Y_j(zaq^7t^{-1})\ldots Y_j(zaq^7t^{-2s-3})$.
        We want to check that $Y_j(zaq^7t^{-1})$ is admissible in $Z_{s+2}$.\\
        If $Z_{s+2}$ contains $Y_j(zaq^5t^{-1})$ then it contains $Y_j(zaq^7t^{-3})Y_j(zaq^5t^{-1})$ 
        (this is of the form $Y_j(zc)Y_j(zcq^{-2r_j}t^2)$).
        By regularity, it also contains $Y_j(zaq^5t^{-3})$.
        By \eqref{eqnccc3}, $X_{k+s+2}$ contains $Y_j(zaq^5t^{-3})$, and $Y_j(zaq^7t^{-3})$ is not admissible in $X_{k+s+2}$. It is a contradiction.\\
        If $Z_{s+2}$ contains $Y_j(zaq^7t)$ then by \eqref{eqnccc3}, $X_{k+s+2}$ contains $Y_j(zaq^7t)$. , 
        According to the graph, this implies that $X'$ contains $Y_j(zaq^7t)$. But $Y_j(zaq^7t^{-1})$ is admissible in $X'$. 
        It is a contradiction.\\
        Hence, $Y_j(zaq^7t^{-1})$ is admissible in $Z_{s+2}$ and the algorithm gives a monomial $Z_{s+1}$ such that:
        \begin{equation}\label{eqnccc5}
            Z_{s+2}A_j(zaq^6)^{-1}=Z_{s+1}.
        \end{equation}
        We assume there exists $0\le m< s+1$ such that we constructed $Z_{u+1}$ for all $m< u\le s+1$ such that:
        $$\forall m< u\le s+1, \quad Z_{u+1}A_j(zaq^6t^{-2(s+1-u)})^{-1}=Z_{u}.$$
        We want to prove that $Y_j(zaq^7t^{-2(s-m+1)-1})$ is admissible in $Z_{m+1}$ so that the algorithm gives the right monomial $Z_m$.\\
        It is clear that $Z_{m+1}$ does not contain $Y_j(zaq^7t^{-2(s-m+1)+1})$ as it has been removed at the previous step.\\
        Moreover, if $Z_{m+1}$ contains $Y_j(zaq^5t^{-2(s-m+1)-1})$ then by construction, $X_{k+s+2}$ also contains $Y_j(zaq^5t^{-2(s-m+1)-1})$.
        According to the graph, this implies that $X_{k+m+2}$ contains $Y_j(zaq^5t^{-2(s-m+1)-1})$. 
        However, we read in the graph that $Y_j(zaq^7t^{-2(s-m+1)-1})$ is admissible in $X_{k+m+2}$. It is a contradiction.
        Hence, $Y_j(zaq^7t^{-2(s-m+1)-1})$ is admissible in $Z_{m+1}$ and the algorithm gives a monomial $Z_m$ such that:
        $$Z_{m+1}A_j(zaq^6t^{-2(s-m+1)})^{-1}=Z_{m}.$$\\

        Finally, let us prove that $Y_i(zaq^6t^{-2})$ is admissible in $Z_0$.\\ 
        We have the equality $Z_0A_i(zaq^3t^{-1})^{-1}=X$. Then, by identification of the degrees the admissibility is clear. \\
        Hence the following graph:
        \begin{center}
\begin{tikzpicture}[
    >=Stealth,
    every node/.style={align=center, text width=4cm},
    every edge/.style={-{Stealth[scale=1.2]}, thick}
]
\node (N1) at (0, 0) { \shortstack{$X_{k+s+2}$}};
\node (N2) at (0, -1.5) { \shortstack{$\vdots$}};
\node (N3) at (0, -3) { $X_{k+1}$};
\node (N4) at (-5, -1.5) { \shortstack{$\vdots$}};
\node (N5) at (0, -5.5) { \shortstack{$\vdots$}};
\node (N6) at (0, -7) { \shortstack{$X'$}};
\node (N7) at (0, -8.5) { \shortstack{$X$}};
\node (N11) at (-5, -7) { \shortstack{$Z_0$}};
\node (N9) at (-5, -3) { \shortstack{$Z_{s+2}$}};
\node (N10) at (-5, -5.5) { \shortstack{$\vdots$}};
\draw[->] (N1) -- (N2) node[midway, right, yshift=0.2cm] {$ A_{j}(zaq^{6}t^{-2})^{-1}$};
\draw[->] (N5) -- (N6) node[midway, right, yshift=0.2cm] {$ A_{i}(zaq^{6k+3}t^{-1})^{-1}$};
\draw[->] (N6) -- (N7) node[midway, right, yshift=0.2cm] {$ A_j(zaq^6)^{-1} $};
\draw[->] (N9) -- (N10) node[midway, left, yshift=0.2cm] {$ A_j(zaq^6)^{-1} $};
\draw[->] (N1) -- (N4) node[midway, left, yshift=0.2cm] {$ A_{i}(zaq^{9}t^{-1})^{-1} $};
\draw[->] (N4) -- (N9) node[midway, left, yshift=0.2cm] {$ A_{i}(zaq^{6k+3}t^{-1})^{-1} $};
\draw[->] (N10) -- (N11) node[midway, left, yshift=0.2cm] {$ A_{j}(zaq^{6}t^{-2s-2})^{-1}$};
\draw[->] (N2) -- (N3) node[midway, right, yshift=0.2cm] {$ A_{j}(zaq^{6}t^{-2s-2})^{-1} $};
\draw[->] (N3) -- (N5) node[midway, right, yshift=0.2cm] {$ A_{i}(zaq^{3}t^{-1})^{-1} $};
\draw[->] (N11) -- (N7) node[midway, left, yshift=0cm] {$ A_{i}(zaq^{3}t^{-1})^{-1} $};
\end{tikzpicture}
\end{center} 

    \vskip 1cm

    Now we assume $d_{i,at^{-2}}(X')=-1$. We know that $d_{i,at^{-2}}(X)=0$, then it implies 
        $$d_{i,at^{-2}}(A_j(zbq^{}t^{-1})^{-1})=1.$$
        Hence, $bqt^{-1}=at^{-2}$ and hence $$b=aq^{-1}t^{-1}.$$
        
        As in the previous case, we want to apply the induction hypothesis and get the following graph for a $k$ defined above:

                \begin{center}
                \begin{tikzpicture}[scale=0.7,
                    >=Stealth,
                    every node/.style={align=center, text width=4cm},
                    every edge/.style={-{Stealth[scale=1.2]}, thick}
                ]
                \node (N3) at (0, -3) { $X_{k+1}$};
                \node (N5) at (0, -4.5) { \shortstack{$\vdots$}};
                \node (N6) at (0, -6) { \shortstack{$X'$}};
                \node (N7) at (0, -7.5) { \shortstack{$X$}};
                \draw[->] (N5) -- (N6) node[midway, right, yshift=0.2cm] {$ A_{i}(zaq^{3}t^{-2k-1})^{-1}$};
                \draw[->] (N6) -- (N7) node[midway, right, yshift=0.2cm] {$ A_j(zat^{-2})^{-1} $};
                \draw[->] (N3) -- (N5) node[midway, right, yshift=0.2cm] {$ A_{i}(zaq^{3}t^{-1})^{-1} $};
                \end{tikzpicture}
                \end{center}

        Let $k>0$ be the maximal integer such that $Y_i(zat^{-2s })^{-1}$ appears in $X'$ for all $0\le s\le k$.\\
        By definition, $Y_i(zat^{-2(k+1)})^{-1}$ is not in $X'$. 
        Moreover, if $Y_i(zaq^6t^{-2k})^{-1}$ appears in $X'$, then $Y_i(zat^{-2(k-1) })^{-1}Y_i(zaq^6t^{-2 k})^{-1}$ also appears in $X'$. 
        The monomial has to be regular, so $Y_i(zat^{-2k })^{-1}Y_i(zaq^6t^{-2( k-1)})^{-1}$ appears in $X'$. \\
        Thus, $X'$ contains $Y_i(zat^{-2(k-2) })^{-1}Y_i(zaq^6t^{-2(k-1) })^{-1}$.
        We can iterate this reasoning until we get that $Y_i(zaq^{6})^{-1}$ appears in $X'$. 
        However it does not appear in $X$ and has to be simplified by the transformation $A_j(zbq^{r_j}t^{-1})^{-1}=A_j(zat^{-2})^{-1}$. 
        It is a contradiction by definition of $A_j$. 
        Hence, we can apply the induction hypothesis and we obtain that there exists $X_1$ given by the algorithm such that 
        $$X_1A_{i}(zaq^3t^{-(2k+1)})^{-1}=X'.$$ 
        We can define recursively 
        $X_s$, $2\le s\le k+1$ such that $$X_s A_{i}(zaq^3t^{-(2k-2s+3)})^{-1}=X_{s-1}.$$ 
        Indeed, for all $1\leq s\leq k+1$, $Y_i(zat^{-2u })^{-1}$ appears in $X_s$ for all $0\le u\le k+1-s$ and $Y_i(zaq^6t^{-2(u-1) })^{-1}$ 
        does not appear in $X_s$ by the same argument as above, nor $Y_i(zat^{-2(k+2-s) })^{-1}$ so that we can apply the induction hypothesis. 
        In particular, we have constructed $X_{k+1}$ such that $$X_{k+1}A_{i}(zaq^3t^{-1})^{-1}=X_k.$$
         Moreover, 
        \begin{equation}\label{eqncccc1}
            X_{k+1}A_{i}(zaq^3t^{-1})^{-1}A_{i}(zaq^3t^{-3})^{-1}\cdots  A_{i}(zaq^3t^{-(2k+1)})^{-1}=X'.
        \end{equation}

        By definition of the algorithm, the arrow $X'\rightarrow X$ implies that $Y_{j}(zaqt^{-3})$ is admissible in $X'$.
        In particular, it implies that $d_{j,aqt^{-3}}(X')=1$ and $d_{j,aqt^{-1}}(X')=0$. But the equation $\eqref{eqncccc1}$ gives
        $$d_{j,aqt^{-1}}(X')=d_{j,aqt^{-1}}(X_{k+1})+1\quad \text{and } \quad d_{j,aqt^{-3}}(X')= d_{j,aqt^{-3}}(X_{k+1})+1.$$
        Hence, $d_{j,aqt^{-1}}(X_{k+1})=-1$ and $d_{j,aqt^{-3}}(X_{k+1})= 0$. Thus, $X_{k+1}$ contains 
        $Y_{j}(zaqt^{-1})^{-1}$ and does not contain
        $Y_{j}(zaqt^{-3})^{-1}$.\\

        Let $s\ge0$ be the maximal integer such that $X_{k+1}$ contains the monomial: 
        $$Y_{j}(zaqt^{-1})^{-1}\dots Y_{j}(zaq^{2s+1}t^{-1})^{-1}.$$ 
        If for one integer $1\le u\le s$, $X_{k+1}$ contains $Y_{j}(zaq^{2u+1}t^{-3})^{-1}$, then it contains the monomial 
        $Y_{j}(zaq^{2(u-1)+1}t^{-1})^{-1}Y_{j}(zaq^{2u+1}t^{-3})^{-1}$. 
        However, $X_{k+1}$ has to be regular, so $X_{k+1}$ also contains $Y_{j}(zaq^{2(u-1)+1}t^{-3})^{-1}$. 
        Iterating the same argument we would get that $X_{k+1}$ does contain $Y_{j}(zaqt^{-3})^{-1}$, 
        which is a contradiction by the argument above. 
        Hence, we can use $s+1$ times the recurrence hypothesis until a monomial $X_{k+s+2}$. We get: 
        $$\forall 0\leq u \leq s, \quad X_{k+u+2}A_j(zaq^{2(s-u)+2}t^{-2})^{-1}=X_{k+u+1}$$
        At this step, we have constructed a path from $X_{k+ s +2}$ to $X$ with:
        \begin{equation*}
            \begin{split}
            X_{k+s+2}A_j(zaq^2t^{-2})^{-1}\cdots A_j(zaq^{2s+2}t^{-2})^{-1} \times\qquad\qquad\qquad \\[5mm]
            \qquad A_i(zaq^3t^{-1})^{-1}\cdots A_i(zaq^3t^{-2k-1})^{-1}A_j(zat^{-2})^{-1} = X.
            \end{split}
        \end{equation*}

        We get the following path in the graph representation of the algorithm:
                \begin{center}
                \begin{tikzpicture}[scale=0.7,
                    >=Stealth,
                    every node/.style={align=center, text width=4cm},
                    every edge/.style={-{Stealth[scale=1.2]}, thick}
                ]
                \node (N1) at (0, 0) { \shortstack{$X_{k+s+2}$}};
                \node (N2) at (0, -1.5) { \shortstack{$\vdots$}};
                \node (N3) at (0, -3) { $X_{k+1}$};
                \node (N5) at (0, -4.5) { \shortstack{$\vdots$}};
                \node (N6) at (0, -6) { \shortstack{$X'$}};
                \node (N7) at (0, -7.5) { \shortstack{$X$}};
                \draw[->] (N1) -- (N2) node[midway, right, yshift=0.2cm] {$ A_{j}(zaq^{2}t^{-2})^{-1}$};
                \draw[->] (N5) -- (N6) node[midway, right, yshift=0.2cm] {$ A_i(zaq^3t^{-2k-1})^{-1}$};
                \draw[->] (N6) -- (N7) node[midway, right, yshift=0.2cm] {$ A_j(zat^{-2})^{-1} $};
                \draw[->] (N2) -- (N3) node[midway, right, yshift=0.2cm] {$ A_j(zaq^{2s+2}t^{-2})^{-1} $};
                \draw[->] (N3) -- (N5) node[midway, right, yshift=0.2cm] {$ A_{i}(zaq^{3}t^{-1})^{-1} $};
                \end{tikzpicture}
                \end{center} 
        Finally, 
        \begin{equation}\label{assumptionf2}
            \text{$X_{k+s+2}$ contains $Y_i(zaq^6t^{-4})\dots Y_i(zaq^{6}t^{-2k-2})Y_j(zaq^3t^{-3})\dots Y_j(zaq^{2s+3}t^{-3})$.}
        \end{equation}

        Let us prove that $s>1$. By contradiction, if $s\leq1$ then  $Y_i(zaq^6t^{-2})$ is admissible in $X_{k+s+2}$. Indeed, 
        $$d_{i,aq^6t^{-2}}(X_{k+s+2})=d_{i,aq^6t^{-2}}(X_{k+1})=1,$$
        $$d_{i,aq^6}(X_{k+s+2})=d_{i,aq^6}(X_{k+1})=0,$$
        $$d_{i,at^{-2}}(X_{k+s+2})=d_{i,at^{-2}}(X_{k+1})=0.$$
        Hence, the algorithm applies the transformation $A_i(zaq^3t^{-1})^{-1}$ to $X_{k+s+2}$ and we get a monomial $X^*$ such that:
        $$X_{k+s+2}A_i(zaq^3t^{-1})^{-1}=X^*.$$
        Then in particular, 
        $$d_{j,aqt^{-1}}(X^*)=d_{j,aqt^{-1}}(X_{k+s+2})+1=d_{j,aqt^{-1}}(X_{k+s+1})+1+1=1.$$
        Indeed, by genericity, $d_{j,aqt^{-1}}(X_{k+s+1})$ has to be equal to $-1$.  
        Moreover, by admissibility of $Y_j(zaq^3t^{-3})$ in $X_{k+s+2}$ and regularity, 
        $$d_{j,aq^3t^{-3}}(X^*)=d_{j,aq^3t^{-3}}(X_{k+s+2})=1,$$
        $$d_{j,aqt^{-3}}(X^*)=d_{j,aqt^{-3}}(X_{k+s+2})=0,$$
        Then $X^*$ contains a term of the form $Y_j(zc)Y_j(zcq^{-2r_j}t^2)$ but does not contain $Y_j(zcq^{-2r_j})$ (with $c=aq^3t^{-3}$)
        and $X^*$ is not regular. It is a contradiction. Hence, $s>1$.\\

        The aim is now to construct another path from $X_{k+s+2}$ to $X$ ending with the transformation $A_i(zaq^3t^{-1})^{-1}$ 
        in order to conclude.\\
        Firstly, $$d_{i,aq^6t^{-4}}(X_{k+s+2})=d_{i,aq^6t^{-4}}(X_{k})=1,$$
        because as shown in the graph above, $Y_i(zaq^6t^{-4})$ is admissible in $X_{k}$. 
        Moreover, $d_{i,at^{-4}}(X_{k+s+2})=d_{i,at^{-4}}(X_{k})=0$ for the same argument of admissibility. 
        Additionally, if $d_{i,aq^6t^{-2}}(X_{k+s+2})=1$ then  by definition of $A_j(zaq^6t^{-2})^{-1}$ it implies
        $$d_{i,aq^6t^{-2}}(X_{k+s-1})=d_{i,aq^6t^{-2}}(X_{k+1})+1=2,$$ which is impossible by the genericity hypothesis.\\
        Hence, $Y_i(zaq^6t^{-4})$ is admissible in $X_{k+s+2}$.\\

        We construct $$Z_{k+s+1}:=X_{k+s+2}A_i(zaq^3t^{-3})^{-1}.$$
        We assume there exists $1\leq m \leq k$ such that we constructed $Z_{s+u+1}$, for $m< u\leq k$ such that 
        $$Z_{u+s+1}=Z_{u+s+2}A_i(zaq^3t^{-2(k-u)-3})^{-1}.$$
        Let us construct $Z_{s+m+1}$.\\
        Firstly, according to the graph, 
        $$d_{i,aq^6t^{-2(k-m)-4}}(Z_{s+m+2})=d_{i,aq^6t^{-2(k-m)-4}}(X_m)=1.$$
        Moreover, 
        $$d_{i,aq^6t^{-2(k-m)-2}}(Z_{s+m+2})=d_{i,aq^6t^{-2(k-m)-2}}(Z_{s+m+3})-1=d_{i,aq^6t^{-2(k-m)-2}}(X_{m+1})-1=0,$$
        because of the transformation $Z_{s+m+3}\rightarrow Z_{s+m+2}$ and because $Y_i(zaq^6t^{-2(k-m)-2})$ is admissible in $X_{m+1}$.\\
        Additionally, $Y_i(zaq^6t^{-2(k-m)-4})$ is admissible in $X_m$. Thus, 
        $$d_{i,at^{-2(k-m)-4}}(Z_{s+m+2})=d_{i,at^{-2(k-m)-4}}(X_{k+s+2})=d_{i,at^{-2(k-m)-4}}(X_m)=0.$$
        Hence the admissibility of $Y_i(zaq^6t^{-2(k-m)-4})$ in $Z_{s+m+2}$ and the existence of the following
         transformation given by the algorithm:
         $$ Z_{m+s+1}=Z_{m+s+2}A_i(zaq^3t^{-2(k-m)-3})^{-1}. $$
         We obtain the following graph:

        \begin{center}
\begin{tikzpicture}[
    >=Stealth,
    every node/.style={align=center, text width=4cm},
    every edge/.style={-{Stealth[scale=1.2]}, thick}
]
\node (N1) at (0, 0) { \shortstack{$X_{k+s+2}$}};
\node (N2) at (0, -1.5) { \shortstack{$\vdots$}};
\node (N3) at (0, -3) { $X_{k+1}$};
\node (N4) at (-5, -1.5) { \shortstack{$\vdots$}};
\node (N5) at (0, -5.5) { \shortstack{$\vdots$}};
\node (N6) at (0, -7) { \shortstack{$X'$}};
\node (N7) at (0, -8.5) { \shortstack{$X$}};
\node (N9) at (-5, -3) { \shortstack{$Z_{s+2}$}};
\draw[->] (N1) -- (N2) node[midway, right, yshift=0.2cm] {$ A_{j}(zaq^{2}t^{-2})^{-1}$};
\draw[->] (N5) -- (N6) node[midway, right, yshift=0.2cm] {$ A_i(zaq^3t^{-2k-1})^{-1}$};
\draw[->] (N6) -- (N7) node[midway, right, yshift=0.2cm] {$ A_j(zat^{-2})^{-1} $};
\draw[->] (N1) -- (N4) node[midway, left, yshift=0.2cm] {$ A_{i}(zaq^{3}t^{-3})^{-1} $};
\draw[->] (N4) -- (N9) node[midway, left, yshift=0.2cm] {$ A_i(zaq^3t^{-2k-1})^{-1} $};
\draw[->] (N2) -- (N3) node[midway, right, yshift=0.2cm] {$ A_j(zaq^{2s+2}t^{-2})^{-1} $};
\draw[->] (N3) -- (N5) node[midway, right, yshift=0.2cm] {$ A_{i}(zaq^{3}t^{-1})^{-1} $};
\end{tikzpicture}
        \end{center}

        Again, we want to construct a path from $Z_{s+2}$ to $X$ in the algorithm graph.\\
        Firstly, by definition of the variables $A_i$ and by admissibility in $X'$,
        \begin{align*}
            d_{j,aqt^{-3}}(Z_{s+2})&=d_{j,aqt^{-3}}(X_{k+s+2})+1,\\
            &=d_{j,aqt^{-3}}(X_{k})+1,\\
            &=d_{j,aqt^{-3}}(X_{k-1})-1+1,\\
            &=d_{j,aqt^{-3}}(X')=1.
        \end{align*}
        Moreover, the graph gives:
        $$d_{j,aq^{-1}t^{-3}}(Z_{s+2})=d_{j,aq^{-1}t^{-3}}(X')=0,$$
        as $Y_j(zaqt^{-3})$ is admissible in $X'$. 
        Additionally, we deduce from the graph the following computation
        \begin{align*}
            d_{j,aqt^{-1}}(Z_{s+2})&=d_{j,aqt^{-1}}(X_{k+s+2})\\
            &=d_{j,aqt^{-1}}(X_{k+s+1})-1\\
            &=d_{j,aqt^{-1}}(X_{k})+1-1\\
            &=d_{j,aqt^{-1}}(X')=0
        \end{align*}
        Hence, $Y_j(zaqt^{-3})$ is admissible in $Z_{s+2}$. 
        Thus, the algorithm constructs the following monomial:
        $$ Z_{s+1}=Z_{s+2}A_j(zat^{-2})^{-1}.$$
        Now, we assume there exists $0\leq m \leq s$ such that we constructed $Z_{u+1}$, for $m\leq u\leq s$ such that 
        $$Z_{u+1}=Z_{u+2}A_j(zaq^{2(s-u)}t^{-2})^{-1}.$$
        Let us construct $Z_{m}$.\\
        Let us prove that $Y_j(zaq^{2(s+1-m)+1}t^{-3})$ is admissible in $Z_{m+1}$.\\
        Firstly, according to the graph and by admissibility, 
        $$d_{j,aq^{2(s+1-m)+1}t^{-3}}(Z_{m+1})=d_{j,aq^{2(s+1-m)+1}t^{-3}}(X_{k+m+2})=1.$$
        Moreover, $$d_{j,aq^{2(s-m)+1}t^{-3}}(Z_{m+1})=d_{j,aq^{2(s-m)+1}t^{-3}}(Z_{m+2})-1=0,$$
        by admissibility and because of the transformation $Z_{m+2}\rightarrow Z_{m+1}$.\\
        Additionally, $Y_j(zaq^{2(s+1-m)+1}t^{-3})$ is admissible in $X_{k+m+2}$. Thus, 
        $$d_{j,aq^{2(s+1-m)+1}t^{-1}}(Z_{m+1})=d_{j,aq^{2(s+1-m)+1}t^{-1}}(X_{k+m+2})=0.$$
        Hence the admissibility of $Y_j(zaq^{2(s+1-m)+1}t^{-3})$ in $Z_{m+1}$ and the existence of the following
         transformation given by the algorithm:
         $$ Z_{m}=Z_{m+1}A_j(zaq^{2(s+1-m)}t^{-2})^{-1}. $$
         To conclude, we have to prove that $Y_i(zaq^6t^{-2})$ is admissible in $Z_0$.\\
         We have the equality $Z_0A_i(zaq^3t^{-1})^{-1}=X$. Then, by identification of the degrees the admissibility is clear. \\
         Hence the following paths of transformations in the algorithm:
         \begin{center}
\begin{tikzpicture}[
    >=Stealth,
    every node/.style={align=center, text width=4cm},
    every edge/.style={-{Stealth[scale=1.2]}, thick}
]
\node (N1) at (0, 0) { \shortstack{$X_{k+s+2}$}};
\node (N2) at (0, -1.5) { \shortstack{$\vdots$}};
\node (N3) at (0, -3) { $X_{k+1}$};
\node (N4) at (-5, -1.5) { \shortstack{$\vdots$}};
\node (N5) at (0, -5.5) { \shortstack{$\vdots$}};
\node (N6) at (0, -7) { \shortstack{$X'$}};
\node (N7) at (0, -8.5) { \shortstack{$X$}};
\node (N9) at (-5, -3) { \shortstack{$Z_{s+2}$}};
\node (N10) at (-5, -5.5) { \shortstack{$\vdots$}};
\node (N11) at (-5, -7) { \shortstack{$Z_0$}};
\draw[->] (N1) -- (N2) node[midway, right, yshift=0.2cm] {$ A_{j}(zaq^{2}t^{-2})^{-1}$};
\draw[->] (N5) -- (N6) node[midway, right, yshift=0.2cm] {$ A_i(zaq^3t^{-2k-1})^{-1}$};
\draw[->] (N6) -- (N7) node[midway, right, yshift=0.2cm] {$ A_j(zat^{-2})^{-1} $};
\draw[->] (N9) -- (N10) node[midway, left, yshift=0.2cm] {$ A_j(zat^{-2})^{-1} $};
\draw[->] (N1) -- (N4) node[midway, left, yshift=0.2cm] {$ A_{i}(zaq^{3}t^{-3})^{-1} $};
\draw[->] (N4) -- (N9) node[midway, left, yshift=0.2cm] {$ A_i(zaq^3t^{-2k-1})^{-1} $};
\draw[->] (N2) -- (N3) node[midway, right, yshift=0.2cm] {$ A_j(zaq^{2s+2}t^{-2})^{-1} $};
\draw[->] (N10) -- (N11) node[midway, left, yshift=0.2cm] {$ A_j(zaq^{2s+2}t^{-2})^{-1} $};
\draw[->] (N3) -- (N5) node[midway, right, yshift=0.2cm] {$ A_{i}(zaq^{3}t^{-1})^{-1} $};
\draw[->] (N11) -- (N7) node[midway, left, yshift=0cm] {$ A_{i}(zaq^{3}t^{-1})^{-1} $};
\end{tikzpicture}
        \end{center} 

        \textit{b)} We assume there exists a monomial $X''$ also appearing in the algorithm such that $$X''A_{i}(zaq^{r_i}t^{-1})^{-1}=X',$$
        then $X''$ does contain $Y_i(zaq^{2r_i}t^{-2})$.\\ 

        Here there are two cases to consider:
        \begin{itemize}
            \item $d_{j,bq^{2r_j}t^{-2}}(A_{i}(zaq^{r_i}t^{-1})^{-1})=1$ and the term $Y_j(zbq^{2r_j}t^{-2})$ comes from 
            the transformation $A_{i}(zaq^{r_i}t^{-1})^{-1}$.
            \item $X''$ contains $Y_j(zbq^{2r_j}t^{-2})$.
        \end{itemize}

        Firstly, we assume $d_{j,bq^{2r_j}t^{-2}}(A_{i}(zaq^{r_i}t^{-1})^{-1})=1$ and the term $Y_j(zbq^{2r_j}t^{-2})$ in $X'$ comes from 
        the transformation $A_{i}(zaq^{r_i}t^{-1})^{-1}$. Hence, $i\neq j$, and 
        we assume the Dynkin subdiagram with nodes $\{i,j\}$ is of type $G_2$ with $(i,j)=(1,2)$.\\
        Thus, $$bq^2t^{-2}=aqt^{-1}\qquad\text{or}\qquad bq^2t^{-2}=aq^3t^{-1}\qquad\text{or}\qquad bq^2t^{-2}=aq^5t^{-1}$$
        If $bq^2t^{-2}=aqt^{-1}$ then $a=bqt^{-1}$ and $d_{i,a}(A_{j}(zbqt^{-1})^{-1})=1$. 
        However, $d_{i,a}(X)=-1$ and $X'A_{j}(zbq^{r_j}t^{-1})^{-1}=X$. 
        Thus, $d_{i,a}(X')=-2$ which contradicts the genericity.\\
        Hence, $bq^2t^{-2}=aq^3t^{-1}$ or $bq^2t^{-2}=aq^5t^{-1}$. 
        By admissibility, $d_{j,bq^2t^{-2}}(X')=1$ and $d_{j,bt^{-2}}(X')=0$. \\
        But here, $d_{j,bt^{-2}}(A_{i}(zaq^{3}t^{-1})^{-1})=1$. Thus, $d_{j,bt^{-2}}(X'')=-1$.\\
        Let $s\geq0$ be the maximal integer such that $X''$ contains the monomial:
        $$Y_j(zbt^{-2})^{-1}\ldots Y_j(zbt^{-2s-2})^{-1}.$$
        If for one $0\leq u\leq s$, $X''$ contains $Y_j(zbq^{2}t^{-2u-2})^{-1}$ by regularity we deduce that 
        $X''$ contains $Y_j(zbq^2t^{-2})^{-1}$. But we have: 
        $$d_{j,bq^2t^{-2}}(X'')=d_{j,bq^2t^{-2}}(X')-d_{j,bq^2t^{-2}}(A_{i}(zaq^{3}t^{-1})^{-1})=1-1=0.$$
        Hence it is a contradiction. Thus, the induction hypothesis allows to construct $X_1,\ldots,X_{s+1}$ such that we get the following subgraph in the algorithm graph:
            \begin{center}
                \begin{tikzpicture}[
                    >=Stealth,
                    every node/.style={align=center, text width=4cm},
                    every edge/.style={-{Stealth[scale=1.2]}, thick}
                ]
                \node (N1) at (0, 0) { \shortstack{$X''$}};
                \node (N2) at (0, -1) { \shortstack{$X'$}};
                \node (N3) at (0, -2) { \shortstack{$X$}};
                \node (N4) at (0, 1) { \shortstack{$X_1$}};
                \node (N5) at (0, 2) { \shortstack{$\vdots$}};
                \node (N6) at (0, 3) { \shortstack{$X_{s+1}$}};
                \draw[->] (N1) -- (N2) node[midway, right, yshift=0.2cm] {$ A_{i}(zaq^{3}t^{-1})^{-1}$};
                \draw[->] (N2) -- (N3) node[midway, right, yshift=0.2cm] {$ A_j(zbqt^{-1})^{-1} $};
                \draw[->] (N4) -- (N1) node[midway, right, yshift=0.2cm] {$ A_{j}(zbqt^{-2s-3})^{-1} $};
                \draw[->] (N6) -- (N5) node[midway, right, yshift=0.2cm] {$ A_j(zbqt^{-3})^{-1}$};
                \draw[->] (N5) -- (N4) node[midway, left, yshift=0.2cm] {};
                \end{tikzpicture}
            \end{center}
        In particular, by admissibility and regularity, 
        $$d_{j,bq^2t^{-4}}(X_{s+1})=1,\qquad d_{j,bt^{-2}}(X_{s+1})=0, \qquad d_{j,bt^{-4}}(X_{s+1})=0.$$
        Now, $d_{i,c}(X_{s+1})=d_{i,c}(X'')$ for $c=aq^6t^{-2},aq^6,at^{-2}$
        Hence, the admissibility of $Y_i(zaq^6t^{-2})$ in $X''$ implies its admissibility in $X_{s+1}$, and the creation of 
        $$Z=X_{s+1}A_{i}(zaq^{3}t^{-1})^{-1}$$
        in the algorithm. 
        Moreover $d_{j,bt^{-2}}(A_{i}(zaq^{3}t^{-1})^{-1})=1$, 
        and $d_{j,bt^{-4}}(A_{i}(zaq^{3}t^{-1})^{-1})=0$.
        Hence, $d_{j,bq^2t^{-4}}(Z)=d_{j,bt^{-2}}(Z)=1$ and $d_{j,bt^{-4}}(Z)=0$. This contradicts the regularity of $Z$.\\

        Hence, we are in the second case and $X''$ contains $Y_j(zbq^{2r_j}t^{-2})$.

        Let us prove that $Y_j(zbq^{2r_j}t^{-2})$ is admissible in $X''$.\\
        If $i=j$ (so that $r_i=r_j$), then by contradiction we assume that $Y_j(zbq^{2r_j}t^{-2})$ is 
        not admissible in $X''$. We know that $Y_j(zbq^{2r_j}t^{-2})$ is admissible in $X'$ so that the presence of 
        $Y_i(zaq^{2r_i}t^{-2})$  in $X''$ prevents from applying the transformation $A_j(zbq^{r_j}t^{-1})^{-1}$. 
        It implies that $a=bq^{-2r_i}$ or $a=bt^2$. 
        This implies $X$ contains $Y_i(zb)^{-1}=Y_i(zaq^{2r_i})^{-1}$ or $Y_i(zat^{-2})^{-1}$. 
        It is a contradiction by the initial assumption on $X$. 
        Thus, $Y_j(zbq^{2r_j}t^{-2})$ is admissible in $X''$.\\
        If $i\neq j$ then it is clear that the admissibility of $Y_j(zbq^{2r_j}t^{-2})$ in $X'$ implies its admissibility in $X''$. \\
        Finally, there exists a monomial $$X_{new}=X''A_{j}(zbq^{r_j}t^{-1})^{-1}$$ given by the algorithm and $X_{new}$ contains $Y_i(zaq^{2r_i}t^{-2})$.
        Moreover, if the transformation $A_{i}(zaq^{r_i}t^{-1})^{-1}$ is not admissible from $X_{new}$ then there exists a factor 
        $Y_i(zat^{-2})$ (resp. $Y_i(zaq^{2r_i})$) blocking this transformation. 
        But this would imply that this factor also appears in $X$ as we have the following equality of monomials:
        $$X=X_{new}A_{i}(zaq^{r_i}t^{-1})^{-1}.$$
        We remark that a priori, this equality is not sufficient to get an arrow from $X_{new}$ to $X$. This equality implies:
        \begin{align*}
            d_{i,aq^{2r_i}}(X_{new})&=d_{i,aq^{2r_i}}(X),
        \end{align*}
        and same for $at^{-2}$. 
        And hence $X$ does contain $Y_i(zc)$ and $Y_i(zct^2)^{-1}$ with $c=at^{-2}$ (resp. $Y_i(zc)Y_i(zcq^{-2r_i})^{-1}$ with $c=aq^{2r_i}$) 
        which contradicts the regularity of $X$. It is impossible. Hence the admissibility of $Y_i(zaq^{2r_i}t^{-2})$ in $X_{new}$.\\
        Finally we get the following paths:
        \begin{center}
\begin{tikzpicture}[
    >=Stealth,
    every node/.style={align=center, text width=4cm},
    every edge/.style={-{Stealth[scale=1.2]}, thick}
]
\node (N1) at (0, 0) { \shortstack{$X''$}};
\node (N2) at (0, -3) { \shortstack{$X'$}};
\node (N3) at (0, -6) { \shortstack{$X$}};
\node (N4) at (-3, -3) { \shortstack{$X_{new}$}};
\draw[->] (N1) -- (N2) node[midway, right, yshift=0.2cm] {$ A_{i}(zaq^{r_i}t^{-1})^{-1}$};
\draw[->] (N2) -- (N3) node[midway, right, yshift=0.2cm] {$ A_j(zbq^{r_j}t^{-1})^{-1} $};
\draw[->] (N4) -- (N3) node[midway, left, yshift=0.2cm] {$ A_{i}(zaq^{r_i}t^{-1})^{-1} $};
\draw[->] (N1) -- (N4) node[midway, left, yshift=0.2cm] {$ A_j(zbq^{r_j}t^{-1})^{-1}$};
\end{tikzpicture}
\end{center}
        This concludes the proof in type $G_2$.
        \end{proof}

\end{appendices}

\bibliographystyle{amsalpha}
\bibliography{articwalg}

\end{document}